\documentclass[a4paper,12pt]{amsart}

\usepackage{amssymb,mathrsfs,bm}

\usepackage{graphics}
\usepackage[margin=2.5truecm]{geometry}
\usepackage{hyperref}

\newtheorem{theorem}{Theorem}[section]
\newtheorem{proposition}[theorem]{Proposition}
\newtheorem{lemma}[theorem]{Lemma}
\newtheorem{corollary}[theorem]{Corollary}

\theoremstyle{definition}
\newtheorem{definition}[theorem]{Definition}
\newtheorem{example}[theorem]{Example}
\newtheorem{question}[theorem]{Question}

\newtheorem{notation}[theorem]{Notation}

\theoremstyle{remark}
\newtheorem{remark}[theorem]{Remark}

\numberwithin{equation}{section}

\renewcommand{\theenumi}{\roman{enumi}}
\renewcommand{\labelenumi}{\rm (\theenumi)}

\makeatletter
\renewcommand{\p@enumii}{\empty}
\makeatother

\newcommand{\AC}{A}  

\newcommand{\R}{\mathbb{R}}
\newcommand{\C}{\mathbb{C}}

\newcommand{\N}{\mathbb{N}}

\newcommand\T{{\mathbb T}}  
\newcommand\D{\mathbb D} 
\newcommand{\extD}{\mathbb{D}^{\rm e}} 
\newcommand\cB{{\mathcal B}}  

\newcommand{\PickA}{\mathcal{P}}
\newcommand{\PickB}{\mathcal{P}^2}
\newcommand{\PickC}{\mathcal{P}^2_0}

\renewcommand{\Re}{\mathop{\rm Re}} 
\renewcommand{\Im}{\mathop{\rm Im}} 
\renewcommand{\epsilon}{\varepsilon}  
\newcommand{\id}{{\rm id}}
\newcommand{\supp}{\mathop{\rm supp}} 

\newcommand{\prob}{\mathbf{P}(\mathbb R)}
\newcommand{\mea}{\mathbf{M}(\mathbb R)}

\newcommand{\classical}{\mathfrak{c}}
\newcommand{\monotone}{\mathfrak{m}}
\newcommand{\free}{\mathfrak{f}}
\newcommand{\boole}{\mathfrak{b}}
\newcommand{\AR}{\mathrm{r}}
\newcommand{\ar}{r} 
\newcommand{\mean}{\mathrm{m}}
\newcommand{\var}{\mathrm{V}}
\newcommand{\Mean}{\mathrm{M}}
\newcommand{\genc}{\gamma}
\newcommand{\genm}{\eta}
\newcommand{\Genm}{\mathcal{H}}
\newcommand{\rgen}{\kappa}
\newcommand{\rGen}{\mathcal{K}}
\newcommand{\itm}{\lambda}
\newcommand{\para}{\mathbb{I}} 
\newcommand{\kk}{\psi} 

\begin{document}

\title[Additive processes and Loewner chains]{Additive processes on the real line \\ and Loewner chains}
\author[T. Hasebe]{Takahiro Hasebe}
\address{Department of Mathematics, Hokkaido University, North 10 West 8, Kita-ku, Sapporo 060-0810, Japan}
\email{thasebe@math.sci.hokudai.ac.jp}
\author[I. Hotta]{Ikkei Hotta}
\address{Department of Applied Science, Yamaguchi University 2-16-1 Tokiwadai, Ube 755-8611, Japan}
\email{ihotta@yamaguchi-u.ac.jp}
\author[T. Murayama]{Takuya Murayama}
\address{Faculty of Mathematics, Kyushu University, 744 Motooka, Nishi-ku, Fukuoka 819-0395, Japan. Present address: Department of Mathematics, Graduate School of Science, Kobe University, 1-1 Rokkodai, Nada-ku, Kobe 657-8501, Japan}
\email{murayama@math.kobe-u.ac.jp}

\date{\today}

\subjclass[2020]{46L53, 
46L54, 
60G51, 
30C99 
} 
\keywords{Loewner chain, additive process, convolution hemigroup, monotone convolution, free convolution, boolean convolution}

\begin{abstract}
This paper investigates additive processes with respect to several different independences in non-commutative probability in terms of the convolution hemigroups of the distributions of the increments of the processes.
In particular, we focus on the relation of monotone convolution hemigroups and Loewner chains, a special kind of family of conformal mappings, on the upper half-plane.
Generalizing the celebrated Loewner differential equation, we formulate an integral equation and the concept of ``generator'' for any Loewner chain of reciprocal Cauchy transforms.
This generalization enables us to remove the assumption of the absolute continuity of Loewner chains which had been imposed in the literature.
The locally uniform convergence of Loewner chains is then equivalent to a suitable convergence of generators.
Using generators, we define homeomorphisms between the aforementioned class of Loewner chains, the set of monotone convolution hemigroups, and the set of classical convolution hemigroups on the real line.
We also discuss similar homeomorphisms to free, boolean and anti-monotone convolution hemigroups on the real line.
\end{abstract}

\maketitle

\tableofcontents

\section{Introduction}

In this paper, we study additive processes, or stochastic processes with independent increments, from the perspectives of non-commutative probability and complex analysis.
The main topic is the relation of \emph{classical} and \emph{monotone} additive processes.
As shown in the sequel, monotone additive processes are associated with \emph{Loewner chains} of reciprocal Cauchy transforms.
One of our goals is to construct, by developing the theory of Loewner chains, a dynamical version of the \emph{Bercovici--Pata bijection}~\cite{AW14,BNT02,BP99} between classical and monotone additive processes, together with anti-monotone, free and boolean additive processes.

In what follows, we motivate and explain our formulation and results.
Section~\ref{sec:background} is devoted to a review of how the theory of Loewner chains has been interacting with classical and noncommutative probability; the reader who desires a quick access to technical details can skip it.
In Sections~\ref{sec:prel}--\ref{sec:results}, we state the precise definitions needed, more detailed objective, core idea, and main results of this paper.
Section~\ref{sec:outline_of_paper} provides the outline of the remaining contents through Sections~\ref{sec:classical_CH}--\ref{sec:univ_Cauchy} and Appendices~\ref{sec:conv_measures_appdx}--\ref{app:GHP}.

\subsection{Background}
\label{sec:background}

In complex function theory, it has been an important problem to establish sharp estimates on the Taylor coefficients of a given univalent function.
In 1923 C.\ Loewner introduced his famous differential equation, governing the time-evolution of a Loewner chain of certain conformal mappings, to offer an approach involving differential inequalities to this problem.
His approach turned out to be very powerful and, in particular, culminated in the solution of Bieberbach's conjecture by de Branges in 1984.
The interested reader is referred to the monograph \cite{BDDM86} for the history of this conjecture.
Moreover, Loewner's method has spread to several branches of mathematics outside function theory.
The reader can find such applications, for example, in dispersionless integrable systems \cite{ATZ21} and in the analysis of Laplacian growth \cite{GL21}.

In probability theory, the most striking application of the Loewner differential equation was initiated by O.\ Schramm in 2000: namely, the Schramm--Loewner evolution (SLE).
This stochastic process describes a random non-self-intersecting path, which appears in two-dimensional statistical physics, growing from a boundary point of the domain.
Indeed, the complement of a growing path determines a family of decreasing domains, to which the Loewner equation applies.
It is surprising that, in this case, the driving function in the equation becomes one-dimensional Brownian motion.
Thanks to this fact combined with It\^o's stochastic calculus, SLE is ``amenable'' enough to obtain many rigorous results predicted in physics; for an SLE overview, see the monograph \cite{Kem17}.

The huge success of SLE stimulated various studies based on Loewner's method in research fields around probability.
Among them, Bauer~\cite{Bau04} showed a connection between Loewner chains and non-commutative stochastic processes.
In his paper, a solution to the (chordal) Loewner differential equation, which is a family of holomorphic functions $f_t$, $t \ge 0$, on the upper half-plane $\C^+$ determines a family of probability distributions $\mathfrak{m}_t$, $t \ge 0$, on $\mathbb{R}$, and moreover, $\mathfrak{m}_t$ is regarded as the one-dimensional marginal distribution of a non-commutative stochastic process.
Here we note that a \emph{(reverse) evolution family}, a relative of the Loewner chain, had already appeared in the formula 
\begin{equation} \label{eq:Bia98_intro}
\text{``}F_{s,t}(K_{\mu_t}(z))=K_{\mu_s}\text{''}
\end{equation}
in Biane's work \cite[p.~161]{Bia98} on free additive processes.
Bauer focused on the relation of his idea to \emph{free independence} and to Biane's work only. Later, Schlei{\ss}inger \cite{Sch17} pointed out that $\mathfrak{m}_t$ is the marginal distribution of a monotone additive process\footnote{%
Before Schlei{\ss}inger's work, Franz \cite[Sections 1 and 5]{Fra09b} studied the relation of Formula \eqref{eq:Bia98_intro} to monotone probability and announced the paper~\cite{FHS20} that appeared much later.}.
This line of research was further developed by Jekel \cite{Jek20}, Franz, Hasebe and Schlei{\ss}inger \cite{FHS20}, and Hasebe and Hotta \cite{HH22}.
Jekel used Loewner chains on infinite-dimensional spaces to describe operator-valued monotone probability, while the other papers are based on usual Loewner chains.

The present work is a companion to the paper \cite{HH22} by the first and second named authors.
In that paper the authors formulate a bijection between additive processes \emph{on the unit circle} $\T$, (radial) Loewner chains \emph{on the unit disk} $\D$, and \emph{unitary multiplicative} processes with monotonically independent increments.
Key ideas there are introducing a suitable concept of ``generator'' for a Loewner chain and identifying it with the generator in the \emph{L\'evy--Khintchine representation} for a classical additive process.
We basically proceed in the same way, but we must take account for the difference between our setting and the previous one.
Roughly speaking, these differences arise from the (non)compactness of the state space in view of classical probability, the position of the unique attractive point of conformal mappings, called the Denjoy--Wolff point, in view of Loewner theory, and the (un)boundedness of operators in view of noncommutative probability.

Let us make a few comments on the domain and fixed point(s) of a Loewner chain.
In Bieberbach's conjecture, univalent functions are defined on $\D$ and normalized at the origin.
Therefore, the mainstream of research had been the radial case, in which the Denjoy--Wolff point is an interior point, until the end of the twentieth century.
However, the chordal case, mainly considered on $\C^+$, in which the Denjoy--Wolff point is located on the boundary, has also been studied classically; see the surveys \cite[\S4]{ABCDM10} \cite[\S5]{BCDMV14} for the historical development in this setting.
The present article also concerns the chordal case.
In applications including SLE, moreover, several variants of the Loewner equation such as ``whole-plane'', ``bilateral'', and ``covering'' ones appear and are equally important.
Even further, there is a modern framework motivated by the theory of holomorphic semigroups; see another part of the same surveys \cite[\S6]{ABCDM10} \cite[\S10]{BCDMV14}.
This framework includes both interior and boundary fixed points and, in addition, allows them to vary in time.
In this regard, Gumenyuk, Hasebe and P\'erez  \cite{GHP22+,GHP23+} recently developed Loewner theory for Bernstein functions in order to analyze both continuous and discrete state inhomogeneous branching processes in one dimension.
In the continuous state case, the domain for Loewner chains is the right half-plane and the Denjoy--Wolff points may vary from $0$ to $\infty$ on the positive real line, and in the discrete state case, the domain is the unit disk and the Denjoy--Wolff points may vary on $[0,1]$. 
The current article also relies on this modern framework as well as the classical chordal one.

\subsection{Convolutions, convolution hemigroups and additive processes} \label{sec:prel}

We now get into our subject.
Let $\prob$ be the set of Borel probability measures on $\R$.   
The classical convolution $\mu \ast \nu$ of probability measures $\mu, \nu \in \prob$ is the probability distribution of $X+Y$ where $X,Y$ are $\R$-valued independent random variables having distributions $\mu,\nu$, respectively. If we consider $\widehat\mu(z):=\int e^{iwz}\,\mu(dw)$, the characteristic function of $\mu$, then the identity
\begin{equation}
\widehat{\mu \ast \nu}(z) = \widehat{\mu}(z) \widehat{\nu}(z), \qquad z \in \R
\end{equation}
characterizes the classical convolution. 

In non-commutative probability theory, random variables can be regarded as self-adjoint operators on a Hilbert space, and the distribution of a random variable is defined properly.
For each independence the convolution is defined in the same spirit as the classical one; namely, the monotone convolution $\mu \rhd \nu$ (resp.\  anti-monotone convolution $\mu \lhd \nu$, free convolution $\mu \boxplus\nu$ and boolean convolution $\mu\uplus \nu$) of $\mu,\nu \in \prob$ is defined to be the probability distribution of $X+Y$ where $X,Y$ are monotonically (resp.\ anti-monotonically, freely and boolean) independent self-adjoint operators having the prescribed distributions $\mu,\nu$, respectively. Note that anti-monotone independence of $X,Y$ is just monotone independence of $Y,X$. 
We refer the reader to Bercovici and Voiculescu \cite{BV93} for unbounded operator models for free convolution and Franz \cite{Fra09b} for boolean and monotone convolutions.

The convolutions mentioned in the preceding paragraph are characterized in terms of the Cauchy transform and its relatives. Let $\C^+$ signify the upper half-plane. The \emph{Cauchy transform} of $\mu \in \prob$ is the function $G_\mu \colon \C^+ \to (-\C^+)$ defined by 
\begin{equation} \label{eq:def_Cauchy}
G_\mu(z) = \int_\R \frac1{z-x} \,\mu(dx). 
\end{equation}
We shall also write $G_\mu$ as $G[\mu]$ when $\mu$ is accompanied by superscripts or subscripts.
The \emph{reciprocal Cauchy transform} $F_\mu\colon \C^+\to\C^+$, also designated as $F[\mu]$, is then defined by 
\begin{equation} \label{eq:def_rCT}
F_\mu(z) = \frac1{G_\mu(z)}.
\end{equation}
Monotone convolution $\rhd$ \cite{Mur00} is characterized by  
\begin{equation} \label{eq:mono_convolution}
F_{\mu\rhd\nu} = F_\mu\circ F_\nu.  
\end{equation}
Characterizations of free and boolean convolutions will be reviewed in Section \ref{sec:c-f-real}. Anti-monotone convolution $\lhd$ is simply characterized by $\mu \lhd \nu = \nu \rhd \mu$. From this simple relation, most of the results on anti-monotone convolution can be reduced to those on monotone convolution, but some of them are still noteworthy. 
 
\begin{notation}[Time parameters]\label{not1}
Throughout the paper, we fix an interval $I$ with the form $I=[0,T_0)$ $(0< T_0 \le +\infty)$ or $I=[0,T_0]$ $(0< T_0 < +\infty)$, unless specified otherwise.
We set $I^2_{\le}:=\{\, (s,t) \in I^2 : s \le t \,\}$ and,
for the families $(y_t)_{t\in I}$ and $(y_{s,t})_{(s,t) \in I^2_{\le}}$ of arbitrary elements, write them as $(y_t)_{t}$ and $(y_{s,t})_{s \le t}$ or more simply $(y_t)$ and $(y_{s,t})$, respectively.
In the sequel, we use some properties usually stated for $I=[0,+\infty)$, but they are easily seen to be valid for a general $I$.
\end{notation}

For a given associative binary operation $\star$ on $\prob$, a \emph{convolution hemigroup with respect to $\star$} ($\star$-CH for brevity) is a family $(\mu_{s,t})_{s\le t}$ in $\prob$ such that $\mu_{s,s}=\delta_0$ and $\mu_{s,t}\star\mu_{t,u} = \mu_{s,u}$ for all $s \le t \le u$ and $(s,t) \mapsto \mu_{s,t}$ is weakly continuous\footnote{%
Here ``weakly'' means ``with respect to weak convergence,'' and a sequence $(\nu_n)_{n \in \N}$ of finite Borel measures on $\R$ is said to be weakly convergent to a finite measure $\nu$ if $\lim_{n \to \infty} \int_{\R} f \,d\nu_n=\int_{\R}f \,d\nu$ for any bounded continuous $f \colon \R \to \R$.
Related definitions and properties are summarized in Appendix~\ref{sec:weak_and_vague}.}. 
A $\star$-CH $(\mu_{s,t})_{0\le s\le t<+\infty}$ (where $I=[0,+\infty)$) is said to be \textit{time-homogeneous} if $\mu_{s,t}=\mu_{0,t-s}$ for all $0\le s\le t<+\infty$. In this case $\mu_t:=\mu_{0,t}$ becomes a \emph{$\star$-convolution semigroup}, i.e., $\mu_{s}\star\mu_t =\mu_{s+t}$ for all $0\le s,t<+\infty$.  

It is obvious that $\rhd$-convolution semigroups and $\lhd$-convolution semigroups are the same. In addition, if $I=[0,T_0]$, there is a canonical correspondence between $\rhd$-CHs and $\lhd$-CHs, given by $(\mu_{s,t}) \mapsto (\mu_{T_0-t, T_0-s})$.  

Usually, convolution hemigroups are studied in relation to infinite divisibility of probability distributions.
Given a convolution $\star \in \{\ast,\rhd,\lhd,\boxplus,\uplus\}$, $\mu \in\prob$ is said to be \emph{$\star$-infinitely divisible} if for every $n\in \N$ there exists $\mu_n \in \prob$ such that $\mu=\mu_n ^{\star n}$.
In the case $\star \in \{\ast,\boxplus,\uplus\}$, each element of a $\star$-CH is $\star$-infinitely divisible; for reference, see Remark~\ref{rem:AP_is_ID} below. 
Conversely, if $\mu \in \prob$ is $\star$-infinitely divisible then there exists a $\star$-CH $(\mu_{s,t})_{0\le s\le t<+\infty}$ such that $\mu_{0,1}=\mu$ (one can even take a time-homogeneous CH).
In the case $\star=\rhd$, however, the situation is much different: a $\rhd$-CH typically contains elements that are not $\rhd$-infinitely divisible \cite[Theorem 3.20 and Example 6.12]{FHS20}. 

\begin{remark} \label{rem:AP_is_ID}
In the preceding paragraph, we refer the reader to, e.g., Sato's book \cite[Theorems~9.1 and 9.7 (ii)]{Sat13} for the classical case.
The proof for free convolution is quite similar to the unit circle case \cite[Lemma 4.2]{HH22}; note that the Bawly--Khintchine theorem for $\boxplus$ is given by Bercovici and Pata \cite[Theorem 1]{BP00}.
For boolean convolution, there is nothing to prove because any probability measure is $\uplus$-infinitely divisible.  
\end{remark}

Let $\star\in \{\ast, \rhd, \boxplus\}$.
For every $\star$-CH $(\mu_{s,t})_{s\le t}$, there exists a $\star$-additive process whose increment from time $s$ to $t$ has the prescribed distribution $\mu_{s,t}$ (for uniqueness of such a process, see Remark~\ref{rem:CH_to_AP} below).
Conversely, given a $\star$-additive process, the distributions of its increments obviously form a $\star$-CH. In this paper, we work on convolution hemigroups and do not analyze additive processes directly. 

\begin{remark} \label{rem:CH_to_AP}
For classical and monotone convolutions, additive processes are unique up to finite dimensional distributions; see Sato \cite[Theorem~9.7 (iii)]{Sat13} and Franz, Hasebe and Schlei{\ss}inger \cite[Theorem 4.7]{FHS20}, respectively.
For the free case, existence of additive processes is known thanks to Biane \cite{Bia98} but uniqueness seems not to have been studied in the literature.
For boolean convolution, existence of additive processes (as unbounded operator processes) is not established in the literature.
\end{remark}

\subsection{Objective of this paper}

Our work originates from considerations of \emph{L\'evy--Khintchine representations} for classical and monotone CHs. Although free and boolean ones motivate our work as well, we postpone subjects on them to Section \ref{sec:c-f-real} because they are treated rather independently.

Let $\mea$ be the set of finite nonnegative Borel measures on $\R$. 

\begin{proposition}[L\'evy-Khintchine representation, see e.g.~\cite{Sat13}] \label{prop:CID_LKrep}
Let $\classical$ be a $\ast$-infinitely divisible distribution on $\R$. Then there exists a unique pair $(\genc,\genm) \in \R \times \mea$ such that 
\begin{equation}\label{eq:CLK0_intro}
\widehat \classical(\xi) = \exp\left( i \genc \xi + \int_\R \left( e^{i\xi x} -1 - \frac{i\xi x}{1+x^2} \right)\frac{1+x^2}{x^2}\,\genm(dx) \right), \qquad \xi \in \R.      
\end{equation}
Conversely, for every pair $(\genc,\genm) \in \R \times \mea$ there exists a unique $\ast$-infinitely divisible distribution $ \classical$ for which \eqref{eq:CLK0_intro} holds. 
\end{proposition}

In Proposition~\ref{prop:CID_LKrep}, the pair $(\genc,\genm)$ is called the \emph{generator} of $\classical$.
A $\ast$-infinitely divisible distribution $\classical$ with generator $(\genc,\genm)$ is naturally embedded into the convolution semigroup $(\classical^{\ast t})_{t}$ consisting of the $\ast$-infinitely divisible distributions with generator $(t\genc,t\genm)$ so that $\classical=\classical^{\ast 1}$. 

A representation formula for $\rhd$-infinitely divisible distributions is given by a differential equation.
Muraki \cite{Mur00} derived the finite variance case, and Belinschi \cite{Bel05} did the general case. 
\begin{proposition} \label{thm:MLK}
For a $\rhd$-infinitely divisible distribution $\monotone$ on $\R$, there exists a pair $(\genc,\genm)\in \R\times \mea$ 
such that the solution $(f_t)_{0\le t <+\infty}$ of the differential equation 
\begin{align}
\frac{d }{d t} f_t(z) &= p(f_t(z)), \qquad z \in\C^+,\quad 0\le t <+\infty, \label{eq:diff1}  \\
f_0(z) &= z, \qquad\qquad\quad z \in \C^+,  \label{eq:diff2}
\end{align}
where $p\colon \C^+\to \C^+ \cup \R$ is the holomorphic function
\begin{equation}\label{eq:MLK}
    p(z) = -\genc + \int_{\mathbb{R}}\frac{1+zx}{x-z} \,\genm(dx) ,\qquad z\in \C^+,   
\end{equation}
satisfies $f_1=F_\monotone$. The pair $(\genc,\genm)$ is uniquely determined by $\monotone$. 
Conversely, given a function $p$ of the form \eqref{eq:MLK} with $(\genc,\genm)\in \R\times \mea$, the solution $(f_t)_{0\le t <+\infty}$ to equation \eqref{eq:diff1} with initial value \eqref{eq:diff2} consists of the reciprocal Cauchy transforms of some $\rhd$-infinitely divisible distributions. In particular, $f_1=F_\monotone$ for some $\rhd$-infinitely divisible distribution $\monotone$.
\end{proposition}

In Proposition~\ref{thm:MLK}, the pair $(\genc,\genm)$ is called the \emph{generator} of $\monotone$.
The holomorphic vector field \eqref{eq:MLK} determined by this pair is known to be semi-complete \cite{BP78}; i.e., the solution $f_t\colon \C^+\to \C^+$ to the initial value problem \eqref{eq:diff1}--\eqref{eq:diff2} exists for all $0\le t<+\infty$ and defines a family of holomorphic functions.
For each $t>0$, the probability measure $\monotone_t$ can then be determined by $F[\monotone_t]=f_t$ and is $\rhd$-infinitely divisible with generator $(t\genc,t\genm)$.
Moreover, $(\monotone_t)_{0\le t <+\infty}$ forms a $\rhd$-convolution semigroup, i.e., $\monotone_0=\delta_0$ and $\monotone_{s}\rhd \monotone_t = \monotone_{s+t}$. 
We also note that $(f_t)$ satisfies another differential equation
\begin{equation}\label{eq:diff3}
  \frac{\partial }{\partial t} f_t(z) =   p(z) \frac{\partial }{\partial z}f_t(z)
\end{equation}
obtained by differentiating $f_{t+s}(z)=f_t(f_s(z))$ with respect to $s$ at $s=0$. 

By Propositions \ref{prop:CID_LKrep} and \ref{thm:MLK}, the sets of $\ast$- and $\rhd$-infinitely divisible distributions can be parametrized by common generators, and hence a bijection exists between these two sets.
This bijection is called the \emph{Bercovici--Pata bijection} (between classical and monotone convolutions) and naturally appears in limit theorems for i.i.d.\ random variables; see Anshelevich and Williams \cite{AW14}. 

A $\ast$-CH has a time-dependent L\'evy-Khintchine representation (see, e.g., Sato \cite[Theorem~9.8 and Section~56]{Sat13}).

\begin{proposition}
\label{prop:LK_repr}
For a $\ast$-CH $(\classical_{s,t})_{s\le t}$, there exist families $(\genc_t)_{t}$ in $\R$ and $(\genm_t)_{t}$ in $\mea$ such that  
\begin{equation}\label{eq:CLK}
\widehat\classical_{0,t}(\xi) = \exp\left( i \genc_t \xi + \int_\R \left( e^{i\xi x} -1 - \frac{i\xi x}{1+x^2} \right)\frac{1+x^2}{x^2}\,\genm_t(dx) \right), \quad \xi \in \R,    
\end{equation}
and the following hold:  
\begin{enumerate}
\item \label{LK1}
$\genc_0=0$, and $t \mapsto \genc_t$ is continuous;
\item  \label{LK2}
$\genm_0=0$, and $t \mapsto \genm_t$ is setwise  continuous and nondecreasing, i.e., the function $t \mapsto \genm_t(B)$ is non-decreasing continuous for each $B \in \cB(\R)$.  
\end{enumerate}
The family $(\genc_t,\genm_t)_{t}$ is unique. Conversely, for a family $(\genc_t,\genm_t)_{t}$ satisfying \eqref{LK1} and \eqref{LK2}, there exists a unique $\ast$-CH $(\classical_{s,t})_{s\le t}$ that satisfies \eqref{eq:CLK}. 
\end{proposition} 

In Proposition~\ref{prop:LK_repr}, we also call $(\genc_t,\genm_t)_{t}$ the generator of $(\classical_{s,t})_{s\le t}$.
If we disregard the $t$-dependence of the pair, the representation \eqref{eq:CLK} is just a direct application of Proposition~\ref{prop:CID_LKrep} to each $\classical_{0,t}$, which is $\ast$-infinitely divisible.

The main objective of the present paper is to answer the following question: 

\begin{question} \label{main_question}
What should be the right ``L\'evy--Khintchine representation'' for $\rhd$-CHs?  
\end{question} 

Different from the classical case, Proposition~\ref{thm:MLK} for $\rhd$-infinitely divisible distributions does not give an answer to Question~\ref{main_question}, because elements in $\rhd$-CHs are not $\rhd$-infinitely divisible in general, as mentioned in Section \ref{sec:prel}. A closely related fact is the non-equivalence of classical and monotone limit theorems for sums of independent random variables \cite[Proposition 6.37]{FHS20} \cite[Introduction]{HH22}.

\subsection{Core idea}
\label{subsec:idea}

To answer Question \ref{main_question}, we recall that 
a $\rhd$-CH $(\monotone_{s,t})_{s\le t}$ is associated with the \emph{Loewner chain} $(f_t:=F[\monotone_{0,t}])_{t}$.
Here by Loewner chain we mean that each $f_t$ is a univalent self-mapping of $\C^+$ and that $t\mapsto f_t(\C^+)$ is decreasing; for a precise definition, see Section \ref{sec:mono-CHs}. 
If we additionally assume certain absolute continuity on $t$, then by Franz, Hasebe and Schlei{\ss}inger \cite[Proposition 3.12]{FHS20}, the \emph{Loewner differential equation} 
\begin{equation} \label{eq:chordal_LDE_intro}
\frac{\partial f_t(z)}{\partial t}=\frac{\partial f_t(z)}{\partial z}p(z,t),
\qquad \text{a.e.}~t \in I,\, z \in \C^+, 
\end{equation}
holds for a function $p$, called a \emph{Herglotz vector field}, of the form 
\begin{equation}\label{eq:q_intro}
p(z,t) = -\dot{\genc}_t + \int_\R \frac{1+xz}{x-z} \,\dot{\genm}_t(dx),  \qquad z\in\C^+,\ t\in I.
\end{equation}
In \eqref{eq:q_intro}, $t\mapsto(\dot{\genc}_t, \dot{\genm}_t) \in\R\times \mea$ is such that $t\mapsto \dot{\genc}_t$ and $t\mapsto \dot{\genm}_t(B)$ are locally integrable for all $B\in \cB(\R)$.

\begin{remark} \label{rem:LDE_intro}
In terms of $f_{s,t}:=f_s^{-1}\circ f_t~(s\le t)$, the differential equation \eqref{eq:chordal_LDE_intro} can be generalized to   
\begin{equation} \label{eq:chordal_LDE2_intro}
\frac{\partial f_{s,t}(z)}{\partial t}=\frac{\partial f_{s,t}(z)}{\partial z}p(z,t),
\qquad \text{a.e.}~t \in I\cap[s,+\infty),\ z \in \C^+,
\end{equation}
for each $s\in I$.
Moreover, the differential equation
\begin{equation}
\label{eq:chordal_LDE3_intro}
\frac{\partial f_{s,t}(z)}{\partial s}=-p(f_{s,t}(z),s),
\qquad \text{a.e.}~s \in [0,t],\, t\in I,\, z \in \C^+,
\end{equation}
holds, which is a non-autonomous generalization of \eqref{eq:diff1}.
\end{remark}

The idea of this paper is to define the ``generator'' of $(\monotone_{s,t})_{s\le t}$ to be  
\begin{equation} \label{eq:LDE_generator}
    \genc_t:=\int_0^t \dot{\genc}_s\, ds \quad \text{and} \quad \genm_t:=\int_0^t \dot{\genm}_s \, ds.
\end{equation}
This definition is consistent with the time-homogeneous case mentioned after Proposition~\ref{thm:MLK}.
Indeed, putting $p(z,t)=p(z)$ with $\dot{\genc}_t \equiv \genc$ and $\dot{\genm}_t \equiv \genm$ in \eqref{eq:chordal_LDE_intro} and \eqref{eq:q_intro} gives the differential equation \eqref{eq:diff3}, and the generator $(\genc_t, \genm_t)$ coincides with $(t\genc,t\genm)$ of $\monotone_t$.
Another probabilistic viewpoint supporting the definition \eqref{eq:LDE_generator} will be given in Section \ref{subsec:generator_moments}. 

The main difficulty in introducing the generator is the absolute continuity in the time parameter assumed in \eqref{eq:chordal_LDE_intro}--\eqref{eq:LDE_generator}.
Such ``smoothness'' is not required in the L\'evy--Khintchine representation for $\ast$-CHs in Proposition~\ref{prop:LK_repr}.
To treat the non-smooth case, in which Loewner chains are only continuous with respect to $t\in I$, we generalize the Loewner differential equations.  A natural framework will be ``Loewner integro-differential equations'' and ``Loewner integral equations''.  
Indeed, this idea worked perfectly in the previous study \cite{HH22} on decreasing Loewner chains on the unit disk with Denjoy--Wolff fixed point at $0$ and CHs on the unit circle. 
We show in this paper that the same idea works on the real line as well. 

\begin{remark}
A kind of distributional treatment of the driving term, $(\dot{\eta}_t)_{t \ge 0}$ in our core idea, also appeared in Jekel~\cite[Section 5.1]{Jek20}, although it was due to the operator-valued setting rather than a lack of absolute continuity.
\end{remark}

\subsection{Role of second moment}
\label{sec:role_of_2nd}

To generalize the Loewner differential equations \eqref{eq:chordal_LDE3_intro} and \eqref{eq:chordal_LDE_intro} to integral/integro-differential equations, an important task is to construct a suitable time-change; in other words, we have to find an ``intrinsic'' time parameter rather than the ``extrinsic'' parameter $t$.
Although there may be some freedom of choice, the probabilistic meaning of the parameter that we shall adopt in the general case seems unclear.
However, if every element of the $\rhd$-CH $(\monotone_{s,t})_{s \le t}$ has finite second moment, we can choose the variance $r(t)$ of $\monotone_{0,t}$ as our desired parameter.
In this case, the corresponding Loewner chain $(f_t)_{t \in I}$ has the form
\[
f_t(z) = z - m_t - \frac{r(t)}{z} +o(z^{-1})
\] 
as $z\to \infty$ nontangentially; here $m_t$ is the mean of $\monotone_{0,t}$.

Most prior work in our context~\cite{Bau04, Sch17, Jek20} has focused on the case of finite second moment with the exception of \cite{FHS20}. 
In this paper, we also work on this specific case before going into the general case.
We do this for several reasons: preference for the clear probabilistic meaning, comparison of our results with prior ones, and difference in technique.
We can discuss the case of finite second moment using only a classical part of Loewner theory.
On the other hand, to treat the general case, we need several results in the modern theory.
The arguments will also be heavier. 
The authors believe that the discussion within the classical framework, which could be a good guide to the general case as well, is worth the space.

\subsection{Overview of main results}  
\label{sec:results}

Below is a summary of the results that we shall obtain, carrying out the ideas in Section~\ref{subsec:idea}.
The general and finite-second-moment cases are put parallel as mentioned in Section~\ref{sec:role_of_2nd}.

\subsubsection{Loewner chains and monotone CHs} \label{subsec:Loewner-monotone}

We use the following symbols (Definitions~\ref{def:reciprocal} and \ref{def:class_P}):
\begin{align*}
\PickA&=\{\, F_\mu : \mu \in \prob \,\}, \\
\PickB&=\left\{\, F_\mu : \mu \in \prob,\ \int_{\R} x^2\,\mu(dx)<+\infty \,\right\}, \\
\PickC&=\left\{\, F_\mu : \mu \in \prob,\ \int_{\R} x^2\,\mu(dx)<+\infty,\ \int_{\R} x\,\mu(dx)=0 \,\right\}. 
\end{align*}
For $\mathcal{F}=\PickA$, $\PickB$, or $\PickC$, we say that a family $(f_t)_t \subset \mathcal{F}$ of univalent self-mappings of $\C^+$ is a \emph{decreasing Loewner chain} in $\mathcal{F}$ or $\mathcal{F}$-DLC for brevity (Definition~\ref{def:DLC})
if the ranges $f_t(\C^+)$, $t\in I$, are non-increasing in $t$, $f_0 = \id$, and $t\mapsto f_t$ is continuous (in the topology of locally uniform convergence).
The relation $f_t=F[\monotone_{0,t}]$ gives a one-to-one correspondence between $\PickA$-DLCs and $\rhd$-CHs (Proposition~\ref{prop:additive_DLC}).
In this correspondence one can include what we call \emph{reverse evolution families} in $\PickA$ ($\PickA$-REFs for short), two-parameter families in $\PickA$ of (univalent) self-mappings of $\C^+$ determined by
\[
f_{s,t}=f_s^{-1} \circ f_t=F[\monotone_{s,t}],\qquad s \le t,
\]
from the other two (Propositions~\ref{prop:DLC_to_REF} and \ref{prop:additive_REF}).

In the case that every element $\monotone_{s,t}$ of a $\rhd$-CH $(\monotone_{s,t})_{s\le t}$ has finite second moment, its mean and variance appear in the Nevanlinna integral representation, whose general form is given in \eqref{eq:PNrep}, of the associated $\PickB$-DLC and $\PickB$-REF.
In fact, for a general $f=F_\mu \in \PickB$ there exist unique $\mean(f) \in \R$ and unique finite Borel measure $\rho_f$ such that
\begin{equation} \label{eq:PickB_rep_s1}
f(z) = z - \mean(f) + \int_\R\frac1{x-z} \,\rho_{f}(dx),
\end{equation}
and moreover, the mean and variance of $\mu$ are $\mean(f)$ and $\AR(f):=\rho_f(\R)$, respectively (cf.\ Proposition~\ref{prop:Cauchy_P_prime}).
The number $\AR(f)$ is called the \emph{angular residue at infinity} of $f$ (see Section~\ref{subsec:angular_residue}).
Applying \eqref{eq:PickB_rep_s1} to $\PickB$-DLCs we can deduce the following:

\begin{theorem}[Theorem \ref{th:capacity_conti_prime} below]  \label{thm:continuity_hcapacity}
For a $\PickB$-DLC $(f_t)_{t}$, the mapping $t\mapsto \mean(f_t)$ is continuous, and $t \mapsto \rho_{f_t}$ is weakly continuous.
In particular, $t \mapsto \AR(f_t)$ is continuous.
\end{theorem}

Theorem~\ref{thm:continuity_hcapacity} in particular implies that, given a $\rhd$-CH $(\monotone_{s,t})_{s\le t}$ with finite second moment, its mean and variance are continuous functions of $(s,t)$;
indeed, Lemma~\ref{lem:quasi-chordal_DLC_to_REF} yields
\begin{gather*}
\Mean(\monotone_{s,t}):=\int_{\R} x \,\monotone_{s,t}(dx)=\mean(f_{s,t})=\mean(f_t) - \mean(f_s), \\
\var(\monotone_{s,t}):=\int_{\R} (x-\Mean(\monotone_{s,t}))^2 \,\monotone_{s,t}(dx)=\AR(f_{s,t}) =\AR(f_t) - \AR(f_s).
\end{gather*}

\begin{remark} \label{rem:mean-var_conti}
The continuity of the mean and variance is \emph{not} a direct consequence of the weak continuity of $(\monotone_{s,t})$, because every non-constant polynomial is unbounded on $\R$ and hence cannot be taken as a test function paired with $\monotone_{s,t}$.
Even apart from the probabilistic context, the example $f_t (z) = z + (1/t-z)^{-1} \in \PickC$, $t>0$, shows that Theorem \ref{thm:continuity_hcapacity} is not very trivial.
Here $f_t$ converges to $\id$ as $t \to +0$, but $\rho_{f_t} = \delta_{1/t}$ does not converge weakly as $t \to +0$.
However, $(f_t)$ is not a DLC as $f_t(\C^+)$ is not decreasing.
\end{remark}

\subsubsection{Loewner integral equations}

Here is the general form of the Loewner integral equation for a $\PickA$-REF $f_{s,t}=F[\monotone_{s,t}]$, which answers our Question~\ref{main_question}.

\begin{theorem}[Theorem~{\ref{thm:integral}} below]
\label{th:REF_to_LIE_s1}
Let $(f_{s,t})_{s\le t}$ be a $\PickA$-REF. Then there exist a continuous function $I\ni t\mapsto \genc_t \in \R$ with $\genc_0=0$ and a Borel measure $\Genm$ on $\R\times I$ with $\Genm(\R\times[0,t])<\infty$ and $\Genm(\R\times \{t\})=0$ for all $t\in I$, such that 
\begin{equation}\label{eq:general_LIE_s1}
    f_{s,t} (z) = z  + \genc_s-\genc_t + \int_{\R \times [s,t]} \frac{1+x f_{r,t}(z)}{x-f_{r,t}(z)} \,\Genm(dx\,dr), \qquad 0 \le s \le t, ~z \in \C^+. 
\end{equation}
The pair $((\genc_t),\Genm)$ is unique. Conversely, given such a pair $((\genc_t),\Genm)$, there is a unique $\PickA$-REF $(f_{s,t})$ such that \eqref{eq:general_LIE_s1} holds.
\end{theorem}

If $(f_{s,t})_{s \le t}$ is a $\PickC$-REF (i.e., $\Mean(\monotone_{s,t})=0$ and $\var(\monotone_{s,t})<+\infty$), then \eqref{eq:general_LIE_s1} reduces to
\begin{equation} \label{eq:chordal_LIE_s1}
f_{s,t}(z)=z+\int_{\R\times[s,t]} \frac{1}{x-f_{r,t}(z)} \,\varTheta(dx\,dr),\qquad s\le t,\, z \in \C^+, 
\end{equation}
with additional property $\varTheta(\R\times[s,t]) = \AR(f_{s,t})=\var(\monotone_{s,t})$; see \eqref{eq:LIE2} in Theorem~\ref{th:DLC_to_LIE}.
As will be discussed in Sections~\ref{sec:chordal_LDE} and \ref{sec:LIE-additive}, in the case $\AR(f_{0,t})=t$ the integral equation \eqref{eq:chordal_LIE_s1} is proven to be essentially the same as the classical chordal Loewner differential equation \eqref{eq:chordal_LODE}. Moreover, for the $\PickC$-DLC $(f_t):=(f_{0,t})$ we have the integro-differential equation
\[
f_t(z)=z+\int_{\R\times [0,t]} \frac{1}{x-z} \,\frac{\partial f_s(z)}{\partial z}  \,\varTheta(dx\,ds), \qquad t \in I,\, z \in \C^+;
\]
see \eqref{eq:LIE-additive} and Proposition~\ref{prop:CMF}.

\begin{remark}
We do not establish an integro-differential equation for a general $\PickA$-DLC because there would appear the product of a continuous function and a Schwartz distribution that seems ill-defined; see Remark \ref{rem:ill}. Only $\PickA$-DLCs with better regularity, such as $\PickC$-DLCs as above and the ones in Remark \ref{rem:LIDE_general}, satisfy reasonable integro-differential equations.
\end{remark}

\subsubsection{Generators of monotone CHs and bijections}
\label{subsubsec:bijections}

Given a $\rhd$-CH $(\monotone_{s,t})$, the pair $((\genc_t),\Genm)$ is determined by Theorem~\ref{th:REF_to_LIE_s1}.
Set $\genm_t(B):=\Genm(B \times [0,t])$, $B \in \cB(\R)$, which defines a finite measure on $\R$ for each $t \in I$.
We call the pair $((\genc_t), (\genm_t))$ (or $((\genc_t),\Genm)$) the \emph{generator} of $(\monotone_{s,t})$ (and also of the associated DLC and REF).
The totality of such pairs $((\genc_t), (\genm_t))$ can be completely characterized by conditions \eqref{LK1}--\eqref{LK2} in Proposition \ref{prop:LK_repr}, see Section~\ref{subsec:LK}.
Now identifying the generator $((\genc_t), (\genm_t))$ of a $\rhd$-CH with the one in the L\'evy--Khintchine representation~\eqref{eq:CLK} for $\ast$-CHs, we can define a bijection between the sets of $\rhd$-CHs and of $\ast$-CHs.
This is what we have referred to as a ``dynamical version of the Bercovici--Pata bijection'' in the first paragraph of this paper.

The bijection introduced above is consistent with the idea in Section \ref{subsec:idea}.
Indeed, suppose that a $\PickA$-REF $(f_{s,t})$ has a generator $(\gamma_t,\genm_t)_t$ that is absolutely continuous in the sense that  
\[
\gamma_t = \int_0^t \dot{\gamma}_s \,ds \qquad  \text{and}\qquad  \genm_t = \int_0^t \dot{\genm}_s\, ds, 
\] 
 where $t\mapsto\dot{\gamma}_t$ and $t\mapsto \dot{\genm}_t(B)$ are both locally integrable for all $B\in \cB(\R)$. Then $\Genm(dx\,dt) = \dot{\genm}_t(dx)\,dt$ and so  the integral equation \eqref{eq:general_LIE_s1} is equivalent to \eqref{eq:chordal_LDE3_intro} with Herglotz vector field 
\eqref{eq:q_intro}. 

In the case that each $\monotone_{s,t}$ has finite second moment, the \emph{reduced generator} $((m_t), (\rgen_t))$ is defined in a similar way based on the reduced equation \eqref{eq:chordal_LIE_s1}; see Definition~\ref{def:monotone_MGF}.
The relation to the generator $((\genc_t), (\genm_t))$ is
\[
m_t=\genc_t+\int_{\R}x\,\genm_t(dx),\qquad \rgen_t(dx)=(1+x^2)\,\genm_t(dx),
\]
(cf.\ Definition~\ref{def:classical_MGF}), and the L\'evy--Khintchine representation for the corresponding $\ast$-CH\footnote{%
This expression is originally due to Kolmogorov; see Gnedenko and Kolmogorov~\cite[Chapter~3]{GK54}.} 
is
\begin{equation} \label{eq:CLK_reduced_intro}
\widehat\classical_{0,t}(\xi)=\exp\left(i m_t \xi + \int_{\R}\frac{e^{i\xi x}-1- i\xi x }{x^2} \, \rgen_t(dx)\right), \qquad \xi \in\R, ~t\in I.
\end{equation}
In this case, the following relations of moments are easily obtained but still noteworthy: for all $s\le t$, 
\begin{align}
\Mean(\classical_{0,t}) &= \Mean(\monotone_{0,t}) = \mean(f_t) = m_t, \\  
\var(\classical_{0,t}) &= \var(\monotone_{0,t})=\AR(f_t) = \rgen_t(\R), \label{eq:var_hcap} \\
\var(\classical_{s,t}) & =\var(\monotone_{s,t}) = \AR(f_{s,t})=\var(\monotone_{0,t})-\var(\monotone_{0,s}).   \label{eq:var_hcap2}
\end{align}

We can formulate a bijection similar to the above between classical, free and boolean CHs using the bijection already well-known for infinitely divisible distributions. 
This will be implemented in Section~\ref{sec:c-f-real}.

\subsubsection{Continuity of the bijection}
\label{sec:intro_conti_bijec}

In order to formulate the continuity of the bijection introduced in Section~\ref{subsubsec:bijections}, we have to define a suitable topology on the sets on which the bijection is defined.
For the set of $\PickA$-DLCs or of $\PickA$-REFs, a natural choice will be just the topology of locally uniform convergence (with respect to $(z,t)$ or to $(z,s,t)$).
Then, as a $\star$-CH $(\mu_{s,t})$ with $\star=\ast$ or $\rhd$ is a weakly continuous mapping $(s,t) \mapsto \mu_{s,t}$ from the parameter set $I^2_\le=\{\, (s,t) \in I^2 :  s \le t \,\}$ to $\mea$, the corresponding topology on the set of $\star$-CHs should be of ``locally uniform convergence'' of such \emph{measure-valued} functions on $I^2_\le$.
This convergence is, in fact, described concretely as follows.
Let $\para=I$ or $I^2_\le$ and $C(\para; \mea)$ denote the set of weakly continuous mappings from $\para$ to $\mea$.
Suppose that $(\mu_{\tau})_{\tau \in \para}$ and $(\mu_{\tau}^n)_{\tau \in \para}$, $n \in \N$, are elements of $C(\para; \mea)$.
Then we can say that $(\mu_{\tau}^n)_{\tau \in \para}$ \emph{converges weakly} to $(\mu_{\tau})_{\tau \in \para}$ \emph{locally uniformly} on $\para$ as $n \to \infty$ if
\begin{equation} \label{eq:luwc_intro}
\lim_{n\to \infty}\sup_{\tau \in K}\left\lvert \int_{\R}f(x) \,\mu^n_{\tau}(dx) - \int_{\R}f(x) \,\mu_{\tau}(dx) \right\rvert=0
\end{equation}
for any compact set $K \subset \para$ and any $f \in C_{\rm b}(\R)$, the set of bounded continuous functions.
The weak convergence is exactly the convergence in the compact-open topology of $C(\para; \mea)$; see Appendix~\ref{sec:topology_of_luwc} for details.

We continue to use the symbol $I^2_\le=\{\, (s,t) \in I^2 :  s \le t \,\}$ in the next theorem.

\begin{theorem}[Proposition~\ref{prop:two_sigmas}, Theorems~\ref{th:conv_classical}, \ref{th:conv_monotone} and \ref{th:conv_P-REF_gen} below] \label{th:conv_monotone_intro}
Let $(\genc_t, \genm_t)_{t}$ and $(\genc_t^n,\genm_t^n)_{t}$, $n\in \N$, be generators, i.e., families satisfying conditions \eqref{LK1} and \eqref{LK2} in Proposition \ref{prop:LK_repr}.
Consider the following objects that have $(\genc_t, \genm_t)_{t}$ and $(\genc_t^n, \genm_t^n)_{t}$ as their generators, respectively: $\ast$-CHs $(\classical_{s,t})_{s\le t}$ and $(\classical^n_{s,t})_{s\le t}$, $\rhd$-CHs $(\monotone_{s,t})_{s\le t}$ and $(\monotone^n_{s,t})_{s\le t}$, $\PickA$-DLCs $(f_t)_{t}$ and $(f^n_t)_{t}$, $\PickA$-REFs $(f_{s,t})_{s\le t}$ and $(f^n_{s,t})_{s\le t}$.
Then the following are equivalent as $n\to \infty$:
\begin{enumerate}
\item \label{C1_intro} $(\classical^n_{0,t})_{t}$ converges weakly to $(\classical_{0,t})_{t}$ locally uniformly on $I$;
\item \label{C2_intro}
$(\classical^n_{s,t})_{s\le t}$ converges weakly to $(\classical_{s,t})_{s\le t}$ locally uniformly on $I^2_\le$;
\item \label{M1_intro}
$(\monotone^n_{0,t})_{t}$ converges weakly to $(\monotone_{0,t})_{t}$ locally uniformly on $I$;
\item \label{M2_intro}
$(\monotone^n_{s,t})_{s\le t}$ converges weakly to $(\monotone_{s,t})_{s\le t}$ locally uniformly on $I^2_\le$;
\item \label{M3_intro}
$f^n_t(z)$ converges to $f_t(z)$ locally uniformly on $\C^+ \times I$;
\item \label{M4_intro}
$f^n_{s,t}(z)$ converges to $f_{s,t}(z)$ locally uniformly on $\C^+ \times   I^2_\le$;
\item \label{M5_intro}
$(\genc^n_t)_{t}$ converges to $(\genc_t)_{t}$ locally uniformly on $I$, and $(\genm^n_t)_{t}$ converges weakly to $(\genm_t)_{t}$ locally uniformly on $I$.
\item \label{M6_intro}
$(\genc^n_t)_{t}$ converges to $(\genc_t)_{t}$ locally uniformly on $I$, and $(\genm^n_t)_{t}$ converges weakly to $(\genm_t)_{t}$ for each $t \in I$.
\end{enumerate}
\end{theorem}

\begin{remark}
The compact-open topology of $C(\para; \mea)$, which we regard as the topology of ``locally uniform weak convergence,'' is metrizable (Corollary~\ref{cor:metrizability}).
Therefore, Theorem~\ref{th:conv_monotone_intro}, which contains only sequential properties, actually guarantees the genuine continuity of our bijection in this topology.
In the case of finite second moment, we can consider a similar but finer topology to include the convergence of the second moments into the equivalent conditions above; see Corollary~\ref{cor:moment_conv} and comments after that corollary.
\end{remark}

We note that Theorem~\ref{th:conv_monotone_intro} shows the continuous dependence of each object on its generator.
Indeed, as the statement itself suggests, studying this dependence comprises not a small part of the proof of the theorem.

It is much easier to formulate similar bijections between classical, free and boolean CHs. A bijection between monotone and anti-monotone CHs can also be defined via Theorem \ref{th:REF_to_LIE_s1}.
All these bijections are continuous. 
Therefore, we obtain homeomorphisms between the five kinds of CHs. 
We skip the details here and refer the reader to Theorems~\ref{th:luc_boole}, \ref{thm:free} and Remark~\ref{rem:M_AM} for actual statements.

\subsection{Outline of the subsequent sections}
\label{sec:outline_of_paper}

Section~\ref{sec:classical_CH} concerns classical CHs.
L\'evy's continuity theorem is generalized in Section~\ref{subsec:Levy_continuity} so that it can be applied to $\ast$-CHs and generators.
Recalling a functional limit theorem for additive processes from the literature, in Section~\ref{subsec:LK}, we confirm that a sequence of $\ast$-CHs converges weakly locally uniformly if and only if the sequence of their generators converges.
An analogous property is achieved for reduced generators, just by simple calculations, in Section~\ref{sec:mgf_for_classical}.

Section~\ref{sec:Pick_Cauchy} is devoted to the properties of the (reciprocal) Cauchy transform.
Although most of the results in this section are more or less known, some proofs are provided for completeness.
As a fundamental tool, the Nevanlinna integral representation for holomorphic functions from $\C^+$ to $\overline{\C^+}$ is reviewed in Section~\ref{sec:PN_to_Cauchy}.
Section~\ref{subsec:angular_residue} then collects several known conditions for a given function to be the reciprocal Cauchy transform of a probability measure with finite second moment.
In Section~\ref{sec:conti_Cauchy_trans}, we establish the continuity of the Cauchy transform with respect to locally uniform convergence of measure-valued functions.

Section~\ref{sec:mono-CHs} is a basic part of Loewner theory for monotone CHs.
Section~\ref{sec:REF_DLC} concerns general DLCs and REFs, and in Section~\ref{sec:ALC}, the results are specialized to $\PickA$-DLCs (or $\PickA$-REFs), which consist of reciprocal Cauchy transforms of probability measures.

In Section~\ref{sec:finite_2nd_case}, we give a closer look at $\PickB$- and $\PickC$-DLCs and REFs to achieve our main results for $\rhd$-CHs in the case of finite second moment.
For the reader only interested in the general case in Section~\ref{sec:general_case}, we point out that Section 6 makes limited use of the techniques developed in Section~\ref{sec:finite_2nd_case}. 
In Section~\ref{sec:capacity_conti}, we prove that the angular residues at infinity of the functions in a $\PickC$-DLC form a continuous function of the time parameter, which is a key to showing Theorem~\ref{thm:continuity_hcapacity}.
Section~\ref{sec:chordal_LDE} is a review on the chordal Loewner differential equation; as already mentioned, $\PickC$-DLCs are usually called chordal Loewner chains and have been studied classically.
Using the results so reviewed, in Section~\ref{sec:LIE-additive}, we derive the Loewner integral equation \eqref{eq:chordal_LIE_s1}.
By virtue of these equations, we can introduce the reduced generator of a $\rhd$-CH with finite second moment and hence of the associated $\PickB$-DLC in Section~\ref{subsec:reducedGF}.
In Section~\ref{subsec:conv_Loewner_rgf}, we prove that a sequence of $\PickB$-DLCs converges locally uniformly if and only if the sequence of reduced generators converges in a suitable sense, which is a variant of Theorem~\ref{th:conv_monotone_intro}. 
Section~\ref{subsec:generator_moments} gives a relation between the moments of $\ast$- and $\rhd$-CHs under the bijection.

In Section~\ref{sec:general_case}, we establish our main results for $\rhd$-CHs in the general case.
For this purpose we cite some facts from the modern Loewner theory in Section~\ref{sec:modern_prel}.
The Loewner integral equation \eqref{eq:general_LIE_s1} for $\PickA$-REFs is then derived in Section~\ref{sec:LIE_for_P-REF}.
This equation enables us to define generators in the general case.
In Section~\ref{sec:REF_gen_homeo} we prove that the correspondence between $\PickA$-REFs and generators is homeomorphic.

In Section~\ref{sec:free_boolean-CHs}, we merge free and boolean CHs into our context.
After providing some preliminary results on the Voiculescu transform in Section~\ref{subsec:luc_Voic-trans}, we introduce the bijection and show its bicontinuity  in Section~\ref{sec:c-f-real}, using the L\'evy--Khintchine representations for free and boolean CHs that have already been known.
Section~\ref{subsec:interpret_bij_free_etc} is devoted to calculations of infinitesimal growth of the moments of classical, free, and boolean CHs under the bijection corresponding to Section~\ref{subsec:generator_moments}.
In Section~\ref{subsec:embed}, we study a related phenomenon, the subordination property of free CHs, originally observed in Biane's work \cite{Bia98} cited in Section~\ref{sec:background}.

In Section~\ref{sec:univ_Cauchy}, we investigate probability measures with compact support in view of geometric function theory.
In Section~\ref{sec:Q-ALC}, we relate the support of probability measures in a $\rhd$-CH to the support of the generator and also the continuous extendibility to the real line of the associated $\PickB$-DLC.
Through Sections~\ref{sec:transfinite_diameter}--\ref{sec:hcap_area}, we state the relation of the variance of measures to certain capacity-like quantities.

We have five appendices.
Appendix~\ref{sec:conv_measures_appdx} concerns topological treatments of finite Borel measures and functions taking values in the space of such measures.
Appendix~\ref{sec:anglim_and_deriv} concerns the angular limit and derivative of holomorphic functions.
Appendix~\ref{sec:Pick_appdx} collects the proofs skipped in Section~\ref{sec:Pick_Cauchy}. 
Appendix~\ref{app:GHP} provides a rather self-contained proof of a crucial lemma needed in Section \ref{sec:general_case} on absolute continuity of REFs. 
Appendix~\ref{app:conv} gives proofs of known results on the reciprocal Cauchy transforms in free probability theory.

\section*{Acknowledgments}
The authors are grateful to Sebastian Schlei{\ss}inger for many fruitful discussions, including suggestions on Theorem \ref{thm:continuity_hcapacity}.
Also, the authors are greatly indebted to the anonymous referee for the careful reading and valuable suggestions; in fact, the referee encouraged us to generalize the results under no second moment assumption, suggested the use of Harnack's inequality, which is adopted in Appendix~\ref{app:GHP}, and pointed out the lack of some references.
This work is supported by JSPS KAKENHI Grant Numbers 18H01115, 19K14546, 20H01807, 20K03632, 22K20341, 23K25775, 24K16935, and JSPS Open Partnership Joint Research Projects grant no.~JPJSBP120209921.

\section{Classical convolution hemigroups}
\label{sec:classical_CH}

In this section, we begin with a generalization of L\'evy's continuity theorem\footnote{%
See a basic textbook on probability, e.g., \cite[Theorem~6.11]{BW16}\cite[Theorem~26.3]{Bil12}\cite[Theorem~9.8.2]{Dud02}.}
and then apply it to $\ast$-CHs and their generators. The results are suitably modified in the case of finite second moment.

\subsection{A generalization of L\'evy's continuity theorem}
\label{subsec:Levy_continuity}

In the next proposition, let $\para=I$ or $I^2_{\le}$.
$\rho$ denotes an arbitrary distance on $\mea$ which induces the weak topology, such as the L\'evy--Prokhorov distance. Weak convergence is indicated by $\xrightarrow{\rm w}$.

\begin{proposition} \label{prop:gLevy_conti}
Suppose that $\mu_{\tau}^n \in \mea$ for all $\tau \in \para$ and $n \in \N$ and that $(\mu_{\tau})_{\tau \in \para} \in C(\para; \mea)$.
Then the following are equivalent:
\begin{enumerate}
\item \label{i:gLevy_luwc}
$\lim_{n\to \infty}\sup_{\tau \in K}\left\lvert \int_{\R}f \,d\mu^n_{\tau} - \int_{\R}f \,d\mu_{\tau} \right\rvert=0$ for every compact set $K \subset \para$ and for every $f \in C_{\rm b}(\R)$;
\item \label{i:gLevy_cf-luc}
$\widehat{\mu^n_{\tau}}(\xi) \to \widehat{\mu_{\tau}}(\xi)$ ($n \to \infty$) uniformly in $(\xi, \tau)$ on each compact subset of $\R \times \para$;
\item \label{i:gLevy_cf-pwc}
for each fixed $\xi \in \R$, $\widehat{\mu^n_{\tau}}(\xi) \to \widehat{\mu_{\tau}}(\xi)$ ($n \to \infty$) uniformly in $\tau$ on each compact subset of $\para$.
\end{enumerate}
In addition, \eqref{i:gLevy_luwc}--\eqref{i:gLevy_cf-pwc} are further equivalent to the condition
\begin{enumerate}
\setcounter{enumi}{3}
\item \label{i:gLevy_metric-luc}
$\lim_{n \to \infty}\sup_{\tau \in K} \rho(\mu^n_{\tau}, \mu_{\tau})=0$ for every compact set $K \subset \para$.
\end{enumerate}
\end{proposition}

\begin{proof}
We first make a simple observation.
Let $K \subset \para$ be compact and $(\tau(n))_{n \in \N}$ be a sequence in $K$ which converges to some $u$.
By assumption, $\mu_{\tau(n)} \xrightarrow{\rm w} \mu_u$ as $n \to \infty$.
Then each one of \eqref{i:gLevy_luwc}--\eqref{i:gLevy_metric-luc} yields $\mu^n_{\tau(n)} \xrightarrow{\rm w} \mu_u$.
For example, if \eqref{i:gLevy_cf-pwc} holds, then for each $\xi \in \R$,
\begin{align*}
\lvert \widehat{\mu^n_{\tau(n)}}(\xi) - \widehat{\mu_u}(\xi) \vert
&\le \lvert \widehat{\mu^n_{\tau(n)}}(\xi) - \widehat{\mu_{\tau(n)}}(\xi) \vert + \lvert \widehat{\mu_{\tau(n)}}(\xi) - \widehat{\mu_u}(\xi) \vert \\
&\le \sup_{\tau \in K}\lvert \widehat{\mu^n_\tau}(\xi) - \widehat{\mu_\tau}(\xi) \rvert + \lvert \widehat{\mu_{\tau(n)}}(\xi) - \widehat{\mu_u}(\xi) \vert
\to 0,
\end{align*}
which implies $\mu^n_{\tau(n)} \xrightarrow{\rm w} \mu_u$ by L\'evy's continuity theorem.

Now assume that one of \eqref{i:gLevy_luwc}--\eqref{i:gLevy_metric-luc} fails.
Then there exist a compact set $K \subset \para$, a positive number $\varepsilon$, and an increasing sequence $(n(k))_{k \in \N}$ in $\N$ such that $\sup_{\tau \in K}D(\mu^{n(k)}_\tau, \mu_\tau) \ge 2\varepsilon$.
Here, $D(\mu, \nu)$ is one of
\[
\left\lvert \int f \,d\mu - \int f \,d\nu \right\rvert,\ \sup_{\xi \in B}\lvert \hat{\mu}(\xi) - \hat{\nu}(\xi) \rvert,\ \lvert \hat{\mu}(\xi) - \hat{\nu}(\xi) \rvert,\ \rho(\mu, \nu)
\]
for some $f \in C_{\rm b}(S)$ in the first and for some compact $B \subset \R$ in the second expression, according to which of \eqref{i:gLevy_luwc}--\eqref{i:gLevy_metric-luc} is assumed to fail.
Then, taking a further subsequence of $(n(k))_{k \in \N}$ if necessary, we can choose a sequence $(\tau(k))_{k \in \N}$ in $K$ which converges to some $u$ so that $D(\mu^{n(k)}_{\tau(k)}, \mu_{\tau(k)}) \ge \varepsilon$.
This implies that all of \eqref{i:gLevy_luwc}--\eqref{i:gLevy_metric-luc} fail by the previous observation.
Indeed, if $\mu_{\tau(k)} \xrightarrow{\rm w} \mu_u$ and $\mu^{n(k)}_{\tau(k)} \xrightarrow{\rm w} \mu_u$ occur at the same time, then $D(\mu^{n(k)}_{\tau(k)}, \mu_{\tau(k)})$ must go to zero as $k \to \infty$.
Thus, taking the contraposition yields the conclusion.
\end{proof}

\subsection{Convergence of classical additive processes}  \label{subsec:LK}  

Using Proposition~\ref{prop:gLevy_conti}, we translate an existing functional limit theorem for classical additive processes in terms of $\ast$-CHs.
To do this, let us restate Proposition~\ref{prop:LK_repr} with words of stochastic processes.
Let $X=(X_t)_t$ be a one-dimensional additive process, that is, an a.s.-RCLL\footnote{%
meaning ``right continuous with left limits''; also called ``c\`adl\`ag''}
and stochastically continuous process with independent increments.
In terms of the process $X$ with marginal laws $(\mu_t)_t$, \eqref{eq:CLK} is the same as the representation
\begin{equation} \label{eq:L-K_1d}
\mathbb{E}[e^{i\xi X_t}]
=\exp\left[ i \gamma_t \xi + \int_{\R} \left(e^{i\xi x}-1-\frac{i\xi x}{1+x^2}\right)\frac{1+x^2}{x^2} \,\eta_t(dx) \right],
\qquad \xi \in \R.
\end{equation}
As in Proposition~\ref{prop:LK_repr}, we call the family of $(\gamma_t, \eta_t) \in \R \times \mea$, $t \in I$, the \emph{generator} of $X$.

\begin{remark}
Compared to the multi-dimensional formula \eqref{eq:L-K_md} below, the Gaussian coefficient in \eqref{eq:L-K_1d} is given by $\eta_t(\{0\})$; here, the value of the integrand at $x=0$ is assigned as the limit
\[
\lim_{x\to 0}\left(e^{i\xi x}-1-\frac{i\xi x}{1+x^2}\right)\frac{1+x^2}{x^2}=-\frac{\xi^2}{2}.
\]
\end{remark}

The next proposition is taken from Jacod and Shiryaev~\cite[Theorem~VII.3.4 and Corollary~VII.4.43]{JS10}, adapted to the one-dimensional case.

\begin{proposition} \label{prop:JS10}
Let $X^n=(X^n_t)_t$, $n \in \N$, and $X=(X_t)_t$ be one-dimensional additive processes and $(\gamma^n_t, \eta^n_t)_t$ and $(\gamma_t, \eta_t)_t$ be their generators, respectively.
Then the following are equivalent:
\begin{enumerate}
\item \label{JS10:J_1}
$X^n \xrightarrow{\rm d} X$ in Skorokhod's $J_1$-topology as $n \to \infty$;
\item \label{JS10:pair_pw}
The generators satisfy
\begin{gather}
\lim_{n\to \infty}\sup_{t\in [0,T]}\lvert \gamma^n_t-\gamma_t\rvert=0
\qquad \text{for each}\ T \in I\ \text{and}
\label{eq:luc_gamma} \\
\operatorname*{w-lim}_{n \to \infty}\eta^n_t=\eta_t 
\qquad \text{for each}\ t \in I; \label{eq:pwc_eta}
\end{gather}
\item \label{JS10:pair_lu}
The generators satisfy \eqref{eq:luc_gamma} and
\begin{equation} \label{eq:luc_eta}
\begin{gathered}
\lim_{n \to \infty}\sup_{t \in [0,T]}\left\lvert \int_{\R} f(x) \, \eta^n_t(dx) - \int_{\R} f(x) \, \eta_t(dx) \right\rvert=0 \\
\text{for each}\ T \in I\ \text{and for each}\ f \in C_{\rm b}(\R);
\end{gathered}
\end{equation}
\item \label{JS10:cf_lu}
The characteristic function $\mathbb{E}[\exp(i\xi X^n_t)]$ converges to $\mathbb{E}[\exp(i\xi X_t)]$ locally uniformly in $(\xi,t) \in \R \times I$ as $n \to \infty$;
\item \label{JS10:cf_pw}
The characteristic function $\mathbb{E}[\exp(i\xi X^n_t)]$ converges to $\mathbb{E}[\exp(i\xi X_t)]$ locally uniformly in $t \in I$ as $n \to \infty$ for each $\xi \in \R$.
\end{enumerate}
\end{proposition}

In Proposition~\ref{prop:JS10}, let $\mu^n_t$ and $\mu_t$ be the distributions of $X^n_t$ and $X_t$, respectively.
Then condition \eqref{JS10:cf_lu} or \eqref{JS10:cf_pw} is equivalent to $(\mu^n_t)_t \to (\mu_t)_t$ in $C(I; \prob)$ by Proposition~\ref{prop:gLevy_conti}.
In addition, \eqref{eq:luc_eta} means that $(\eta^n_t)_t \to (\eta_t)_t$ in $C(I; \mea)$.
Hence, if a sequence of additive processes converges to an additive process, then the finite measures in the L\'evy--Khintchine representation as well as the marginal distributions of the processes converge locally uniformly in time parameter.

\begin{remark} \label{rem:Polya}
In Proposition~\ref{prop:JS10}, conditions \eqref{eq:pwc_eta} and \eqref{eq:luc_eta} are equivalent because $\eta^n_t$ and $\eta_t$ are setwise continuous and non-decreasing in $t$.
In fact, a theorem of P\'olya\footnote{%
See the original paper \cite{Pol20} or a suitable book on probability, e.g., \cite[Theorem 4.5]{BW16}.}
ensures that, if a sequence of right-continuous non-decreasing functions $F_n$, $n \in \N$, on $I$ converges to a continuous function $F$ pointwise with $F_n(0)=F(0)=0$, then $(F_n)_{n \in \N}$ converges to $F$ locally uniformly on $I$.
From this theorem we easily see that \eqref{eq:pwc_eta} implies \eqref{eq:luc_eta}, decomposing a test function $f \in C_{\rm b}(\R)$ into the positive and negative parts.
Since the marginal distributions $\mu^n_t$ of the processes $X^n_t$ have no monotonicity in $t$, we cannot drop $t$-uniformity in the convergence of $((\mu^n_t)_t)_{n \in \N}$.
\end{remark}

\begin{remark}
Here are comments on analogous results.
\begin{enumerate}
\item
Although it is not a focus of the current paper, for multi-dimensional additive processes the same property as Proposition~\ref{prop:JS10} holds true, except that the pair $(\gamma_t, \eta_t)$ in \eqref{eq:L-K_1d} should be replaced by the characteristic triplet.
Namely, the following expression is adopted:
\begin{equation} \label{eq:L-K_md}
\mathbb{E}[e^{i\xi \cdot X_t}]
=\exp\left[ i \gamma_t \cdot \xi - \frac{1}{2}\xi \cdot C_t\xi + \int_{\R^d} \left(e^{i\xi \cdot x}-1-i\xi \cdot h(x) \right) \nu_t(dx) \right],
\qquad \xi \in \R^d.
\end{equation}
Here, $C_t$ is the Gaussian covariance matrix, $h \colon \R^d \to \R^d$ is a truncation function~\cite[p.75]{JS10}, and $\nu_t$ is a L\'evy measure.
Then
\[
\tilde{C}_t := C_t - \int_{\R^d} h(x) \otimes h(x) \,\nu_t(dx)
\]
is called the modified second characteristic~\cite[p.115]{JS10}.
The multi-dimensional version of Proposition~\ref{prop:JS10} states that the locally uniform convergences of additive processes, of modified triplets, and of marginal distributions (of the processes) are all equivalent; for the precise statement, see Jacod and Shiryaev~\cite[Theorem~VII.3.4 and Corollary~VII.4.43]{JS10}.

\item
An analogue to Proposition~\ref{prop:JS10} is formulated for a sequence of random walks constructed by an infinitesimal triangular array in Hata's thesis~\cite{Ha23}.
Similar results were proved even on a Lie group; see Feinsilver~\cite{Fe78}.
\end{enumerate}
\end{remark}

We now switch our language from processes to $\ast$-CHs and make further discussions on conditions \eqref{eq:pwc_eta} and \eqref{eq:luc_eta}, which are equivalent by Remark~\ref{rem:Polya}.
For this purpose we slightly rewrite the L\'evy--Khintchine representation \eqref{eq:CLK}.
As $\widehat\classical_{s,t}=\widehat\classical_{0,t}/\widehat\classical_{0,s}$, setting $\genc_{s,t}:=\genc_{t}-\genc_s$ and $\genm_{s,t}:=\genm_t-\genm_s$ yields
\begin{equation}\label{eq:CLK2}
\widehat\classical_{s,t}(\xi) = \exp\left( i \genc_{s,t}\xi + \int_\R \left( e^{i\xi x} -1 - \frac{i\xi x}{1+x^2} \right)\frac{1+x^2}{x^2}\,\genm_{s,t}(dx) \right). 
\end{equation}
We then define
\begin{align} 
&\Genm(B\times (s,t]):=\genm_{s,t}(B)=\genm_t(B)-\genm_s(B), \qquad s\le t,\ B\in \cB(\R), \label{eq:two_sigmas} \\
&\Genm(\R \times \{0\}) := 0. \label{eq:sigma2}
\end{align}  
Since $(x,t) \mapsto \Genm((-\infty, x] \times [0,t])$ is a two-dimensional distribution function\footnote{%
See, e.g., Billingsley~\cite[Theorem~12.5]{Bil12} for multi-dimensional distribution functions.}, 
$\Genm$ extends to a $\sigma$-finite Borel measure on $\R \times I$.
Since $(\genm_t)_{t}$ is setwise nondecreasing and continuous, we see that
\begin{equation} \label{LK3}
\Genm(\R \times [0, t])<\infty
\quad \text{and}\quad
\Genm(\R \times \{t\})=0
\quad \text{for every}\ t \in I.
\end{equation}
Conversely, given a Borel measure $\Genm$ on $\R \times I$ satisfying \eqref{LK3}, the relation \eqref{eq:two_sigmas} defines a unique family $(\genm_t)_{t}$ in $\mea$ that is setwise nondecreasing and continuous with $\genm_0=0$.
In terms of $\Genm$, \eqref{eq:CLK2} is written as
\begin{equation} \label{eq:CLK3}
\widehat\classical_{s,t}(\xi) = \exp\left( i \genc_{s,t}\xi + \int_{\R\times (s,t]} \left( e^{i\xi x} -1 - \frac{i\xi x}{1+x^2} \right)\frac{1+x^2}{x^2}\,\Genm(dx\,du) \right).
\end{equation}

With the notation above, we have the following:

\begin{proposition} \label{prop:two_sigmas}
Let $(\genm_t)_{t}$ and $(\genm^n_t)_{t}$, $n\in \N$, be setwise nondecreasing continuous mappings from $I$ to $\mea$ with $\genm_0=\genm^n_0=0$.
Let $\Genm$ be the $\sigma$-finite Borel measure defined by \eqref{eq:two_sigmas} and \eqref{LK3}.
For each $n$, similarly, let $\Genm^n$ be the measure defined by \eqref{eq:two_sigmas} and \eqref{LK3} with the superscript ${}^n$ put on $\genm$ and $\Genm$.
Then the following are equivalent as $n\to \infty$:
\begin{enumerate}
\item \label{G1}
\eqref{eq:luc_eta} holds, i.e., $(\genm^n_t)_{t}$ converges weakly to $(\genm_t)_{t}$ locally uniformly on $I$;
\item \label{G2}
\eqref{eq:pwc_eta} holds, i.e., $\genm^n_t$ converges weakly to $\genm_t$ for each $t\in I$;
\item \label{G3}
$\Genm^n\rvert_{\R\times [0,T]}$ converges weakly to $\Genm\rvert_{\R\times [0,T]}$ for each $T \in I$. 
\end{enumerate}
Here, $\Genm\rvert_{\R\times [0,T]}$ is the measure on $\cB(\R \times [0,T])$ defined by $\Genm\rvert_{\R\times [0,T]}(B):=\Genm(B \cap (\R\times [0,T]))$, $B \in \cB(\R \times [0,T])$, and $\Genm^n\rvert_{\R\times [0,T]}$ is defined in the same way.
\end{proposition}

\begin{proof}
It suffices to prove that \eqref{G2} and \eqref{G3} are equivalent.
Since
\[
\int_{\R \times [0,T]} f(x) \,\Genm^n(dx\,dt) = \int_{\R} f(x) \,\eta^n_T(dx),
\qquad f\in C_{\rm b}(\R),
\]
we immediately see that \eqref{G3} implies \eqref{G2}.
Conversely, suppose \eqref{G2}.
Then repeating the proof of $\rm (G2)\Rightarrow (G3)$ in Hasebe and Hotta~\cite[Proposition~2.6]{HH22}, we see that $\Genm^n\rvert_{\R \times [0,T]}$ converges vaguely to $\Genm\rvert_{\R \times [0,T]}$ as $n \to \infty$ (see Appendix \ref{sec:conv_measures_appdx}).
Moreover, the total mass $\Genm^n(\R \times [0,T])=\eta^n_T(\R)$ converges to $\Genm(\R \times [0,T])=\eta_T(\R)$ by \eqref{G2}.
Hence \eqref{G3} follows from Proposition \ref{prop:vague_to_weak}.
\end{proof}

Finally Propositions~\ref{prop:JS10} and \ref{prop:two_sigmas} (and the comments after the former proposition) are summarized, from the perspective of $\ast$-CHs, as follows:

\begin{theorem} \label{th:conv_classical}
Let $(\classical_{s,t})_{s\le t}$ and $(\classical_{s,t}^n)_{s\le t}$, $n\in \N$, be $\ast$-CHs and $(\genc_t, \genm_t)_{t}$ and $(\genc^n_t, \genm^n_t)_{t}$, be their respective generators.
The following are then equivalent:
\begin{enumerate}
\item \label{C1}
$(\classical^n_{0,t})_{t}$ converges weakly to $(\classical_{0,t})_{t}$ locally uniformly on $I$;
\item \label{C2}
$(\classical^n_{s,t})_{s\le t}$ converges weakly to $(\classical_{s,t})_{s\le t}$ locally uniformly on $I^2_\le$;
\item \label{C3}
$(\genc^n_t)_{t}$ converges to $(\genc_t)_{t}$ locally uniformly on $I$, and the equivalent conditions \eqref{G1}--\eqref{G3} in Proposition~\ref{prop:two_sigmas} are fulfilled.
\end{enumerate}
\end{theorem}

\subsection{Reduced generators}
\label{sec:mgf_for_classical}

The mean and variance of a $\ast$-infinitely divisible distribution can be written with its generator as follows:

\begin{proposition}[e.g.\ {\cite[Corollary~25.8]{Sat13}}]
\label{prop:2nd_moment}
For a $\ast$-infinitely divisible distribution $\classical$ with L\'evy--Khintchine representation \eqref{eq:CLK0_intro}, the following hold:
\begin{enumerate}
\item \label{prop:2nd_moment_1}
$\int_{\R}x^2\,\classical(dx)<+\infty$ if and only if $\int_{\R}(1+x^2)\,\genm(dx)<+\infty$;
\item \label{prop:2nd_moment_2}
if either of the integrals in \eqref{prop:2nd_moment_1} is finite, then
\begin{align}
\Mean(\classical)&=\genc+\int_{\R}x\,\genm(dx), \label{eq:mean} \\
\var(\classical)&=\int_{\R}(1+x^2)\,\genm(dx). \label{eq:variance}
\end{align}
\end{enumerate}
\end{proposition}

Proposition~\ref{prop:2nd_moment} allows us to introduce the following definition of reduced generators for $\ast$-CHs with finite second moment:

\begin{definition}[reduced generator of a $\ast$-CH]
\label{def:classical_MGF}
Let $(\classical_{s,t})_{s\le t}$ be a $\ast$-CH and $(\genc_t, \genm_t)_{t}$ be its generator in Proposition~\ref{prop:LK_repr}.
Suppose that $\var(\classical_{0,t})=\int_{\R}(1+x^2)\,\genm_t(dx)<+\infty$ for each $t\in I$.
Setting
\begin{equation} \label{eq:gen_rgen}
m_t:=\genc_t+\int_{\R}x\,\genm_t(dx),\qquad \rgen_t(dx):=(1+x^2)\,\genm_t(dx),
\end{equation}
we call $(m_t,\rgen_t)_{t}$ the \emph{reduced generator} of $(\classical_{s,t})_{s\le t}$. 
In this case, the L\'evy--Khintchine representation takes the reduced form \eqref{eq:CLK_reduced_intro}.
\end{definition}

In the sequel, we consider vague convergence of measures.
Here, a sequence $(\nu_n)_{n \in \N}$ of finite Borel measures on $\R$ is said to be vaguely convergent to a finite measure $\nu$ if $\lim_{n \to \infty}\int_{\R} f \,d\nu_n=\int_{\R} f \,d\nu$ for any $f$ taken from the set $C_{\rm c}(\R)$ of continuous functions with compact support.
Note that in the case $\sup_{n \in \N} \nu_n(\R)<+\infty$ we obtain an equivalent definition, replacing $C_{\rm c}(\R)$ with the larger set $C_\infty(\R)$ of continuous functions vanishing at infinity; see Appendix~\ref{sec:weak_and_vague}.
Vague convergence is indicated by $\xrightarrow{\rm v}$.

The density transformation $\genm(dx) \mapsto (1+x^2) \,\genm(dx)$ in Definition~\ref{def:classical_MGF} translates weak convergence into vague convergence. 
 
\begin{lemma} \label{lem:vague-weak}
Let $\genm, \genm_n \in \mea$, $n\in \mathbb{N}$, and set $\rgen(dx):=(1+x^2) \,\genm(dx)$ and $\rgen_n(dx):=(1+x^2) \,\genm_n(dx)$.
Suppose that $\sup_n \rgen_n(\mathbb{R})<+\infty$.
Then $\genm_n \xrightarrow{\rm w} \genm$ as $n \to \infty$ if and only if $\rgen_n \xrightarrow{\rm v} \rgen$.
\end{lemma}

\begin{proof}
\textit{``Only if'' part}.
Suppose $\genm_n \stackrel{\rm w}{\to} \genm$.
Then for $f \in C_{\rm c}(\R)$ we have
\begin{align*}
\lim_{n \to \infty}\int_{\R} f(x) \,\rgen_n(dx)
&=\lim_{n \to \infty}\int_{\R} f(x) (1+x^2) \,\genm_n(dx) \\
&=\int_{\R} f(x)(1+x^2) \,\genm(dx)
=\int_{\R} f(x) \,\rgen(dx)
\end{align*}
because $f(x)(1+x^2) \in C_{\rm b}(\R)$.
This implies $\rgen_n \xrightarrow{\rm v} \rgen$.

\smallskip\noindent
\textit{``If'' part}.
Suppose $\rgen_n \stackrel{\rm v}{\to} \rgen$.
Then for $f \in C_{\rm b}(\R)$ we have
\begin{align*}
\lim_{n \to \infty}\int_{\R} f(x) \,\genm_n(dx)
&=\lim_{n \to \infty}\int_{\R} \frac{f(x)}{1+x^2} \,\rgen_n(dx) \\
&=\int_{\R} \frac{f(x)}{1+x^2} \,\rgen(dx)
=\int_{\R} f(x) \,\genm(dx)
\end{align*}
because $f(x)(1+x^2)^{-1} \in C_\infty(\R)$.
Here, we have used $\sup_n \rgen_n(\R)<+\infty$ and Proposition~\ref{prop:vague_bounded_set}.
Thus $\genm_n \xrightarrow{\rm w} \genm$.
\end{proof}

Let $(m_t, \rgen_t)_{t}$ be the reduced generator of a $\ast$-CH.
It is easy to see directly (or from some literature) that $(\rgen_t)_{t}$ inherits setwise continuity from $(\genm_t)$ in the (non-reduced) generator.
We can then define a unique Borel measure $\rGen$ on $\R \times I$ that satisfies
\begin{align}
&\rGen(B\times (s,t])=\rgen_t(B)-\rgen_s(B),
&&s\le t,\ B\in \cB(\R); \label{eq:two_thetas} \\
&\rGen(\R \times [0, t])<+\infty,\ \rGen(\R \times \{t\})=0,
&&t \in I, \label{LK3_theta}
\end{align}
as in \eqref{eq:two_sigmas} and \eqref{LK3}.
Proposition~\ref{prop:two_sigmas} and Theorem~\ref{th:conv_classical} are transferred to this setting. 

\begin{definition} \label{def:luvc}
Let $\para= I$ or $I^2_{\le}$. 
We say that a sequence $((\mu_{\tau}^n)_{\tau \in \para})_{n \in \N}$ in $C(\para; \mea)$ converges \emph{vaguely to $(\mu_{\tau})_{\tau \in \para}\in C(\para; \mea)$ locally uniformly on $\para$} if \eqref{eq:luwc_intro} holds for all compact $K \subset \para$ and $f \in C_{\rm c}(\R)$. 
\end{definition}

\begin{remark}
Suppose that $(\mu_\tau^n)$ converges vaguely to $(\mu_\tau)$ locally uniformly.
If the total masses are uniformly bounded, i.e., $\sup_{n, \tau} \mu^n_\tau(\R)<+\infty$, then $f$ in \eqref{eq:luwc_intro} can be chosen from $C_\infty(\R)$, as is the case with the usual vague convergence, by a simple approximation argument (cf.~Proposition \ref{prop:vague_bounded_set}).
\end{remark}

\begin{proposition} \label{prop:two_thetas}
Let $(\rgen_t)_{t}$ and $(\rgen^n_t)_{t}$, $n\in \N$, be setwise non-decreasing continuous mappings from $I$ to $\mea$ with $\rgen_0=\rgen^n_0=0$.
Let $\rGen$ be the measure defined by \eqref{eq:two_thetas} and \eqref{LK3_theta}.
For each $n$, similarly, let $\rGen^n$ be the measure defined by \eqref{eq:two_thetas} and \eqref{LK3_theta} with the superscript $^n$ put on $\rgen$ and $\rGen$.
Then the following are equivalent as $n\to \infty$:
\begin{enumerate}
\item \label{mG1}
$(\rgen^n_t)_{t}$ converges vaguely to $(\rgen_t)_{t}$ locally uniformly on $I$;
\item \label{mG2}
$\rgen^n_t$ converges vaguely to $\rgen_t$ for each $t\in I$;
\item \label{mG3}
$\rGen^n$ converges vaguely to $\rGen$.
\end{enumerate}
\end{proposition}

\begin{proof}
This is proved in the same way as Proposition~\ref{prop:two_sigmas}, so we omit the proof.
\end{proof}

\begin{theorem} \label{th:conv_classical_SMB}
Let $(\classical_{s,t})_{s\le t}$ and $(\classical_{s,t}^n)_{s\le t}$, $n\in \N$, be $\ast$-CHs that satisfy $\var(\classical_{0,t})<+\infty$ and $\sup_{n\in \N}\var(\classical^n_{0,t}) < +\infty$ for each $t\in I$.
Let $(m_t, \rgen_t)_{t}$ and $(m^n_t, \rgen^n_t)_{t}$ be their respective reduced generators.
Then the following are equivalent:
\begin{enumerate}
\item \label{mC1}
$(\classical^n_{0,t})_{t}$ converges weakly to $(\classical_{0,t})_{t}$ locally uniformly on $I$;
\item \label{mC2}
$(m^n_t)_{t}$ converges to $(m_t)_{t}$ locally uniformly on $I$, and the equivalent conditions \eqref{mG1}--\eqref{mG3} in Proposition~\ref{prop:two_thetas} are fulfilled.
\end{enumerate}
\end{theorem}

\begin{proof}
Let $(\genc_t, \genm_t)_{t}$ and $(\genc^n_t, \genm^n_t)_{t}$, $n \in \N$, be the generators of $(\classical_{s,t})_{s\le t}$ and $(\classical^n_{s,t})_{s\le t}$, respectively.
To prove $\eqref{mC1} \Rightarrow \eqref{mC2}$, suppose \eqref{mC1}.
By Theorem~\ref{th:conv_classical}, $\genm^n_t \xrightarrow{\rm w} \genm_t$ for each $t \in I$, and hence $\rgen^n_t \xrightarrow{\rm v} \rgen_t$ by Lemma~\ref{lem:vague-weak}.
This further implies, by Proposition~\ref{prop:two_thetas}, that $(\rgen^n_t)_{t}$ converges vaguely to $(\rgen_t)_{t}$ locally uniformly on $I$.
In particular, by \eqref{eq:mean},
\begin{align*}
&m^n_t=\int_{\R} x \,\genm^n_t(dx) = \int_{\R} x(1+x^2)^{-1} \,\rgen^n_t(dx) \\
&\longrightarrow \int_{\R} x(1+x^2)^{-1} \,\rgen_t(dx) = \int_{\R} x \,\genm_t(dx)=m_t,
\quad n\to \infty,
\end{align*}
locally uniformly in $t$.
Similarly, the implication $\eqref{mC2} \Rightarrow \eqref{mC1}$ follows easily from Propositions~\ref{prop:two_sigmas}, \ref{prop:2nd_moment} and Lemma~\ref{lem:vague-weak}; so we omit the details.
\end{proof}

\section{Pick functions and Cauchy transforms}
\label{sec:Pick_Cauchy}

A holomorphic function from the upper half-plane $\C^+$ to $\overline{\C^+}=\C^+ \cup \R$ is called a \emph{Pick} (or \emph{Nevanlinna}) \emph{function}.
(It takes a real value only if it is constant because every non-constant holomorphic function is an open mapping.)
This section collects its basic properties and studies the Cauchy and reciprocal Cauchy transforms \eqref{eq:def_Cauchy} and \eqref{eq:def_rCT} as special subclasses of Pick functions.

\subsection{Nevanlinna formula}
\label{sec:PN_to_Cauchy}

A Pick function $f$ has the \emph{Nevanlinna representation}
\begin{align} \label{eq:PNrep}
f(z)&=\alpha + \beta z + \int_{\R}\left(\frac{1}{x-z}-\frac{x}{1+x^2}\right)\rho(dx)  = \alpha + \beta z + \int_{\R}\frac{1+xz}{x-z}\,\frac{\rho(dx)}{1+x^2}. 
\end{align}
Here, $\alpha \in \R$, $\beta \in [0,+\infty)$, and $\rho$ is a Borel measure on $\R$ with $\int_{\R}(1+x^2)^{-1}\,\rho(dx)<+\infty$.
They also satisfy the formulae
\begin{align}
&\alpha=\Re f(i), & \label{eq:PN_alpha} \\
&\beta=\lim_{y \to +\infty}\frac{f(iy)}{iy}=\lim_{y \to +\infty}\frac{\Im f(iy)}{y}, & \label{eq:PN_beta} \\
&\frac1{\pi}\lim_{y \to +0}\int_a^b \Im f(x+iy) \,dx = \rho((a,b))+\frac{\rho(\{a, b\})}{2},
&a, b \in \R,\ a<b. \label{eq:PN_interval}
\end{align}
The last one is called the Stieltjes inversion formula.
By \eqref{eq:PN_beta} and \eqref{eq:PN_interval} we mean, in particular, that all the limits exist.
Proofs of \eqref{eq:PNrep}--\eqref{eq:PN_interval} can be found in, e.g., Tsuji~\cite[{\S}IV.6]{Tsu75} or Rosenblum and Rovnyak~\cite[\S5.3]{RR94}. 
In addition, we have $\Im f(z) \ge \beta \Im z$ and
\begin{equation} \label{eq:PN_total}
y \Im(f(iy) - i\beta y) = \int_{\R}\frac{y^2}{x^2+y^2}\,\rho(dx)
\uparrow \rho(\R)
\qquad (y \to +\infty)
\end{equation}
by the monotone convergence theorem.

The $\beta$ in \eqref{eq:PN_beta}, which is also denoted by $f^\prime(\infty)$, is called the \emph{angular derivative} of the Pick function $f$ at $\infty$.
To explain the meaning, let $\Gamma_{a,b}:=\{\, z=x+iy : y>a \lvert x \rvert,\ y>b \,\}$ for $a,b>0$.
We call $\Gamma_{a,b}$ a \emph{Stolz angle} (in $\C^+$) at $\infty$.
A function $f \colon \C^+ \to \C$ is said to have \emph{angular limit}\footnote{%
Appendix \ref{sec:anglim_and_deriv} summarizes the general definition and properties of angular limit required in this paper.}
$c \in \hat{\C}=\C \cup \{\infty\}$ at $\infty$ if $f(z) \to c$ as $z \to \infty$ through each Stolz angle.
In this case, we write $\angle\lim_{z \to \infty}f(z)=c$.
Now, suppose that $f$ is a Pick function of \eqref{eq:PNrep}; then
\begin{equation} \label{eq:angderiv_at_infty}
\angle\!\lim_{z \to \infty}\frac{f(z)}{z}
=\angle\!\lim_{z \to \infty}f^\prime(z)
=\beta
\end{equation}
\cite[Theorem~IV.19]{Tsu75}.

The case $\rho(\R)<+\infty$ in \eqref{eq:PNrep} is considered frequently in the sequel.
In this case, \eqref{eq:PNrep} reads
\begin{equation} \label{eq:PNrep2}
f(z)=\alpha^\prime + \beta z + \int_{\R}\frac{1}{x-z}\,\rho(dx)
\end{equation}
with $\alpha^\prime:=\alpha-\int_{\R}x(1+x^2)^{-1}\,\rho(dx)$.
In fact, the negative of the Cauchy transform $G_\mu(z)$ of $\mu \in \mea$ defined by \eqref{eq:def_Cauchy} has the representation \eqref{eq:PNrep2} with $\alpha^\prime=\beta=0$ and $\rho=\mu$, and hence it determines the measure $\mu$ uniquely via \eqref{eq:PN_interval}.
If $\mu \neq 0$, i.e., $G_\mu \neq 0$, then the reciprocal Cauchy transform $F_\mu=1/G_\mu$ is also a Pick function.

Here is a version of the well-known condition \cite[Theorem~3, {\S}VI.59]{AG93} \cite[Proposition~2.1]{Maa92} \cite[Theorem~10]{MS17} for a Pick function to be the (reciprocal) Cauchy transform of some finite measure; see Appendix \ref{sec:Pick_appdx} for proof.

\begin{proposition} \label{prop:Maassen_2.1}
Suppose that Pick functions $f$ and $g$ are both non-zero and satisfy $f=-1/g$.
Then the following are equivalent:
\begin{enumerate}
\item \label{Ma:g}
$\lim_{y \to +\infty}iy\, g(iy)=c$ for some $c \in \C$;
\item \label{Ma:f}
$\lim_{y \to +\infty}f(iy)/iy=d$ for some $d \in \C \setminus \{0\}$;
\item \label{Ma:G}
$g=-G_\mu$ for some $\mu \in \mea$;
\item \label{Ma:F}
$f=F_\mu$ for some $\mu \in \mea$.
\end{enumerate}
If one of these is true, then $cd=-1$ holds in \eqref{Ma:g} and \eqref{Ma:f}, the measures $\mu$ in \eqref{Ma:G} and \eqref{Ma:F} are the same and non-zero, and $c=\mu(\R)$.
Moreover, $y \mapsto y\Im g(iy)$ is non-decreasing with $\lim_{y \to +\infty}y\Im g(iy)=c$.
\end{proposition}

\begin{definition}[Function space $\PickA$]
\label{def:reciprocal}
We write the class of Pick functions $f$ with $f'(\infty)=\lim_{y\to+\infty}f(iy)/iy=1$ as $\PickA$. From Proposition \ref{prop:Maassen_2.1} it follows that
\[
\PickA = \{F_\mu: \mu \in \prob\}. 
\]
\end{definition}

\subsection{Angular residues at infinity}
\label{subsec:angular_residue}

In Proposition \ref{prop:Maassen_2.1}, we can say more about the second moment of $\mu$ if $\mu \in \prob$.
For the proof of the next proposition, see Appendix \ref{sec:Pick_appdx}.

\begin{proposition} \label{prop:class_P}
For a Pick function $f$, the following five are equivalent: 
\begin{enumerate}
\item \label{HN:ang}
The angular limit
\begin{equation} \label{HN2}
\AR(f):=\angle\! \lim_{z \to \infty} z(z-f(z))
\end{equation}
exists finitely;
\item \label{HN:rad}
$\lim_{y \to +\infty}(f(iy)-iy)=0$, and $\sup_{y>0}y(\Im f(iy)-y)<+\infty$;
\item \label{HN:PN}
there exists $\rho \in \mea$ such that
\[
f(z) = z + \int_{\R} \frac{1}{x - z} \,\rho(dx)
=z - G_\rho(z),
\qquad z \in \C^+;
\]
\item \label{HN:ineq}
there exists $C\ge 0$ such that
\begin{equation} \label{eq:HN_ineq}
\lvert f(z)-z \rvert \le \frac{C}{\Im z},
\qquad z \in \C^+;
\end{equation}
\item \label{HN:Cauchy}
there exists $\mu \in \prob$ with $\Mean(\mu)=0$ and $\var(\mu)<+\infty$ such that $f=F_\mu$.
\end{enumerate}
In \eqref{HN:PN} and \eqref{HN:Cauchy}, the measures $\rho$ and $\mu$ are unique.
Moreover, if one of the above holds, then
\begin{equation} \label{HN:four_quantities}
\AR(f) = \rho(\R)
= \inf \{\, C \ge 0 : C\ \text{\rm satisfies}\ \eqref{eq:HN_ineq} \,\}
= \var(\mu),
\end{equation}
and
\begin{equation} \label{HN1}
\lim_{\substack{z \to \infty \\ z \in \C^+_b}} (f(z)-z) = 0,
\qquad b>0.
\end{equation}
Here, $\C^+_b:=\{\, z : \Im z>b \,\}$.
\end{proposition}

Conditions \eqref{HN:ang}--\eqref{HN:Cauchy} have been studied classically \cite{AAS83,GB92,Maa92}; in particular, \eqref{HN1} is referred to as the \emph{hydrodynamic normalization}\footnote{%
This indeed has a meaning in fluid mechanics: since $f(z) \sim z$ near $\infty$, the function $f$ is the complex velocity potential of a planar flow which is almost uniform around infinity.}
\emph{at infinity}.
The constant $\AR(f)$ in \eqref{HN2} is called the \emph{angular residue} of $f$ at infinity.
We note that \eqref{HN2} is also written as the asymptotic expansion
\[
f(z)=z-\frac{\AR(f)}{z}+o(\lvert z \rvert^{-1})
\]
as $z$ goes to $\infty$ through any Stolz angle at $\infty$.

Taking Proposition \ref{prop:class_P} \eqref{HN:Cauchy} (and Proposition~\ref{prop:Cauchy_P_prime} below) into account, we introduce the following subclasses of $\PickA$:

\begin{definition}[Function spaces $\PickC$ and $\PickB$]
\label{def:class_P}
Let
\begin{align*}
\PickC &:= \{\, f : f\ \text{satisfies one of \eqref{HN:ang}--\eqref{HN:Cauchy} in Proposition~\ref{prop:class_P}} \,\}, \\
\PickB&:=\{\, f : f+m \in \PickC\ \text{for some}\ m \in \R \,\}.
\end{align*}
By Proposition~\ref{prop:class_P}~\eqref{HN:PN}, for each $f \in \PickB$ there exists a unique pair $(m,\rho) \in \R \times \mea$ such that
\begin{equation} \label{def:class_P_prime}
f(z) = z - m + \int_{\R} \frac{1}{x-z} \,\rho(dx)
=z - m - G_\rho(z),
\qquad z \in \C^+.
\end{equation}
We write $m=\mean(f)$, $\rho=\rho_f=\rho(f; \cdot)$, and $\AR(f):=\rho_f(\R)$; then the first equality in \eqref{HN:four_quantities} still holds trivially.
We call $\rho_f$ the \emph{characteristic measure} of $f$.
\end{definition}

\begin{remark} \label{rem:AR}\
\begin{enumerate}
\item \label{rem:disconti_of_ar}
The functional $\AR \colon \PickC \to [0,+\infty)$ is not continuous; to see this, consider the example in Remark~\ref{rem:mean-var_conti}.
To retrieve continuity, one has to restrict $\PickC$ to a proper subset, as we shall indeed do in Proposition~\ref{prop:DCT_for_AR}.

\item \label{rem:ar_and_area}
It is known that, if $f \in \PickC$ is univalent, then $\AR(f)$ is comparable with the \emph{Euclidean} area of a certain \emph{hyperbolic} neighborhood of the complement $\C^+ \setminus f(\C^+)$ in the hyperbolic plane $\C^+$; for details, see Section~\ref{sec:hcap_area}.
\end{enumerate}
\end{remark}

For $f \in \PickB$ and $a \in \R$ it is obvious that $f-a \in \PickB$ with $\mean(f-a)=\mean(f)+a$ and $\AR(f-a)=\AR(f)$.
In addition, the inner shift $\tilde{f}(z)=f(z-a)$ belongs to $\PickB$ with $\mean(\tilde{f})=\mean(f)+a$ and $\AR(\tilde{f})=\AR(f)$, which is seen from
\[
\tilde{f}(z)=z-a-\mean(f)+\int_{\R}\frac{1}{x-z+a}\,\rho_f(dx)=z-(\mean(f)+a)+\int_{\R}\frac{1}{x-z}\,\rho_{\tilde{f}}(dx).
\]
Here, $\rho_{\tilde{f}}$ is given by $\rho_{\tilde{f}}(B)=\rho_f(B-a)$, $B \in \cB(\R)$, with $B-a:=\{\, x-a : x \in B\,\}$.
Similarly it is easy to see the following:

\begin{proposition} \label{prop:Cauchy_P_prime}
$\PickB$ is the set of reciprocal Cauchy transforms of probability measures of finite second moment; i.e.,
\[
\PickB=\left\{\, F_\mu : \mu \in \mathbf{P}(\R),\ \int_{\R}x^2 \,\mu(dx)<+\infty \,\right\}.
\]
Moreover, $\mean(F_\mu)=\Mean(\mu)$ and $\AR(F_\mu)=\var(\mu)$.
\end{proposition}

\subsection{Convergence of Cauchy transforms}
\label{sec:conti_Cauchy_trans} 

The Cauchy transform is known to be continuous with respect to vague or weak convergence.
For the proofs of the  two basic properties below, see Appendix \ref{sec:Pick_appdx}.

\begin{proposition}[e.g.\ {\cite[Corollary~14, Chapter~3]{MS17}}]
\label{prop:continuity_Cauchy}
Suppose that $\mu_n \in \mea$, $n \in \N$, are bounded: $\sup_{n \in \N}\mu_n(\R)<+\infty$.
Let $\AC \subset \C^+$ be a subset which has an accumulation point in $\C^+$. 
Then $\mu_n \xrightarrow{\rm v} \mu$ as $n \to \infty$ for some $\mu \in \mea$ if and only if $(G[\mu_n])_{n \in \N}$ converges to a function $G$ pointwise on $\AC$.
Moreover, if either of these conditions holds, then $G=G_\mu$, and $G[\mu_n] \to G_\mu$ locally uniformly on $\C^+$.
\end{proposition}

\begin{corollary}[e.g.\ {\cite[Theorem~13, Chapter~3]{MS17}}]
\label{cor:continuity_Cauchy}
Let $\mu, \mu_n \in \mea$, $n \in \N$.
Then $\mu_n \xrightarrow{\rm w} \mu$ as $n \to \infty$ if and only if $\mu_n(\R) \to \mu(\R)$ and $G[\mu_n] \to G_\mu$ locally uniformly on $\C^+$.
In the case $\mu, \mu_n \in \prob$, in particular, $\mu_n \xrightarrow{\rm w} \mu$ as $n \to \infty$ if and only if $F[\mu_n] \to F_\mu$ locally uniformly on $\C^+$.
\end{corollary}

We extend Proposition~\ref{prop:continuity_Cauchy} and Corollary~\ref{cor:continuity_Cauchy}, considering the locally uniform convergence of continuous mappings from the parameter set $\para=I$ or $I^2_{\le}$ to $\prob$.
To this end, we recall that the three topologies on $\prob$---weak, vague, and $\sigma(\prob,C_\infty(\R))$-ones---are all identical; see Propositions~\ref{prop:vague_bounded_set}, \ref{prop:metrizability}, and \ref{prop:vague_to_weak}.
Based on this fact, the following equivalence is easily verified, either by analogy with Proposition~\ref{prop:gLevy_conti} or by considering uniform structures as in Appendix~\ref{sec:topology_of_luwc}.

\begin{proposition} \label{prop:three_luwc}
Let $(\mu_{\tau})_{\tau \in \para}$ and $(\mu^n_{\tau})_{\tau \in \para}$, $n \in \N$, be elements of $C(\para; \prob)$ with $\para=I$ or $I^2_{\le}$.
Then the following are equivalent as $n \to \infty$:
\begin{enumerate}
\item
$(\mu^n_{\tau})_{\tau}$ converges weakly to $(\mu_{\tau})_{\tau}$ locally uniformly on $\para$, i.e.,
\begin{equation} \label{eq:general_luwc}
\lim_{n \to \infty}\sup_{\tau \in K}\left\lvert \int_{\R} f(x) \,\mu^n_{\tau}(dx) - \int_{\R} f(x) \,\mu_{\tau}(dx) \right\rvert=0
\end{equation}
for any compact $K \subset \para$ and any $f \in C_{\rm b}(\R)$; 
\item
$(\mu^n_{\tau})_{\tau}$ converges vaguely to $(\mu_{\tau})_{\tau}$ locally uniformly on $\para$, i.e., \eqref{eq:general_luwc} holds for any compact $K$ and any $f \in C_{\rm c}(\R)$; 
\item
\eqref{eq:general_luwc} holds for any compact $K$ and any $f \in C_\infty(\R)$.
\end{enumerate}
\end{proposition}

We now generalize the continuity of Cauchy transforms as follows:

\begin{proposition} \label{prop:luc_Cauchy} 
Let $\para=I$ or $I^2_{\le}$,
$(\mu_\tau)_{\tau \in \para}$, $(\mu^n_\tau)_{\tau \in \para}$, $n\in \N$, be members of $C(\para; \prob)$, and $\AC \subset \C^+$ be a subset which has an accumulation point in $\C^+$.
Then the following are equivalent as $n\to \infty$:
\begin{enumerate}
\item \label{i:luwc_Cauchy}
$(\mu^n_\tau)_{\tau \in \para}$ converges weakly to $(\mu_\tau)_{\tau \in \para}$ locally uniformly on $\para$;
\item \label{i:luc_Cauchy}
$G[\mu^n_\tau](z)$ converges to $G[\mu_\tau](z)$ locally uniformly on $\C^+\times \para$;
\item \label{i:pwc_Cauchy}
$G[\mu^n_\tau](z)$ converges to $G[\mu_\tau](z)$ locally uniformly on $\para$ for each $z \in \AC$.
\end{enumerate}
\end{proposition}

\begin{proof}
Similar to the proof of Proposition~\ref{prop:gLevy_conti}, the following observation is true:
given a compact $K \subset \para$ and a sequence $(\tau(n))_{n \in \N}$ in $K$ with $\tau(n) \to u$ for some $u \in K$, each of \eqref{i:luwc_Cauchy}--\eqref{i:pwc_Cauchy} implies $\mu^n_{\tau(n)} \xrightarrow{\rm w} \mu_u$.

Now let us assume that one of \eqref{i:luwc_Cauchy}--\eqref{i:pwc_Cauchy} fails.
This implies that there exist a compact $K \subset \para$, an $\epsilon>0$, and an increasing sequence $(n(k))_{k \in \N}$ such that $\sup_{\tau \in K} D(\mu^{n(k)}_\tau, \mu_\tau) \ge 2\epsilon$;
here $D(\mu, \nu)$ is one of
\begin{itemize}
\item
$\lvert \int f \,d\mu - \int f \,d\nu \rvert$ for some $f \in C_{\rm b}(\R)$,
\item
$\sup_{z \in B} \lvert G_\mu(z) - G_{\nu}(z) \rvert$ for some compact $B \subset \C$,
\item
$\lvert G_\mu(z) - G_{\nu}(z) \rvert$ for some $z \in \AC$,
\end{itemize}
according to which of \eqref{i:luwc_Cauchy}--\eqref{i:pwc_Cauchy} is assumed to fail.
Taking a further subsequence of $(n(k))_k$ if necessary, we can then find a sequence $(\tau(k))_{k \in \N}$ in $K$ converging to some $u \in K$ such that $D(\mu^{n(k)}_{\tau(k)}, \mu_{\tau(k)}) \ge \epsilon$.
By the observation at the beginning, this implies that all of \eqref{i:luwc_Cauchy}--\eqref{i:pwc_Cauchy} fail.
Considering the contraposition, we obtain the conclusion.
\end{proof}

To include the convergence of reciprocal Cauchy transforms into the equivalence in Proposition~\ref{prop:luc_Cauchy}, we note the following simple fact:

\begin{lemma} \label{lem:bicontinuity}
Let $D$ be a domain in $\C$, $\para=I$ or $I^2_{\le}$, and $f\colon  D \times \para \to \C$ be such that $z \mapsto f(z,\tau)$ is holomorphic for every fixed $\tau \in \para$, $\tau \mapsto f(z,\tau)$ is continuous for every fixed $z \in D$, and $f$ is locally bounded on $D \times \para$. Then $f$ is (jointly) continuous on $D \times \para$.  
\end{lemma}

\begin{proof}
This follows by applying the dominated convergence theorem in the identity
\[
f(w,\sigma) - f(z,\tau) = \frac1{2\pi i} \int_\Gamma \left(\frac{f(\zeta,\sigma)}{\zeta-w}  - \frac{f(\zeta,\tau)}{\zeta-z}\right)d\zeta,  
\]
where $\Gamma$ is a smooth simple closed curve such that $w, z \in \mathop{\rm ins}\Gamma \subset D$.
Here, $\mathop{\rm ins}\Gamma$ denotes the domain surrounded by $\Gamma$.
\end{proof}

\begin{corollary} \label{cor:luc_rCauchy}
In Proposition~\ref{prop:luc_Cauchy}, the equivalent conditions~\eqref{i:luwc_Cauchy}--\eqref{i:pwc_Cauchy} are also equivalent to either of the following:
\begin{enumerate}
\setcounter{enumi}{3}
\item \label{i:luc_rCauchy}
$F[\mu^n_\tau](z)$ converges to $F[\mu_\tau](z)$ locally uniformly on $\C^+\times \para$;
\item \label{i:pwc_rCauchy}
$F[\mu^n_\tau](z)$ converges to $F[\mu_\tau](z)$ locally uniformly on $\para$ for each $z\in \AC$.
\end{enumerate}
\end{corollary}

\begin{proof} 
We just mention the proof of $\eqref{i:luc_Cauchy} \Rightarrow \eqref{i:luc_rCauchy}$ briefly, for the other implications are proved quite similarly.
Assume \eqref{i:luc_Cauchy}; i.e., $G[\mu^n_\tau](z) \to G[\mu_\tau](z)$ locally uniformly on $\C^+ \times \para$.
Let $K \subset \C^+$ and $B \subset \para$ be compact in the respective spaces.
By the weak continuity of $(\mu_\tau)_{\tau \in \para}$ and Corollary~\ref{cor:continuity_Cauchy},
$\tau \mapsto G[\mu_\tau](z)$ is continuous for each $z$.
As $|G[\mu_\tau](z)| \le 1/\Im(z)$,   by Lemma~\ref{lem:bicontinuity}, the function $(z,\tau) \mapsto G[\mu_\tau](z)$ on $K \times \para$ is continuous.
In particular, we have $b:=-\sup\{\, \Im w : w=G[\mu_\tau](z)\ \text{for some}\ (z,\tau) \in K \times B \,\}>0$.
Then it follows in the same way as in the end of Appendix~\ref{sec:prf_conv_appdx} that $F[\mu^n_\tau](z) \to F[\mu_\tau](z)$ locally uniformly on $\C^+ \times \para$, which means \eqref{i:luc_rCauchy}.
\end{proof}

\section{Basics in Loewner theory for monotone probability}
\label{sec:mono-CHs}

In this section, we first give general results on Loewner chains.
After that, we study chains which correspond to $\rhd$-CHs, using the results in Section~\ref{sec:Pick_Cauchy}.

\subsection{Decreasing Loewner chains and reverse evolution families}
\label{sec:REF_DLC}

Let $I$ be as in Notation~\ref{not1} and $D$ be a simply connected proper subdomain of $\C$.

\begin{definition}[decreasing Loewner chain]
\label{def:DLC}
A one-parameter family $(f_t)_{t \in I}$ of holomorphic functions on $D$ is called a \emph{decreasing Loewner chain} (DLC for short) on $D$ if
\begin{enumerate}
\setlength{\leftskip}{1.5em}
\renewcommand{\theenumi}{(DLC\arabic{enumi})}
\renewcommand{\labelenumi}{\textup{\theenumi}}
\item \label{DLC1}
$f_t$ is univalent for each $t$, 
\item \label{DLC2}
$f_0=\mathrm{id}_D$, and $f_s(D) \supset f_t(D)$ if $s \le t$, 
\item \label{DLC3}
$t \mapsto f_t$ is continuous (in the sense of locally uniform convergence).
\end{enumerate}
Moreover, for a family $\mathcal{F}$ of holomorphic functions, we refer to $(f_t)$ as an \emph{$\mathcal{F}$-DLC} if $f_t \in \mathcal{F}$ for all $t \in I$. 
\end{definition}

We will first provide some results for general DLCs on $\C^+$ and then specialize in $\mathcal{F}$-DLCs for $\mathcal{F}=\PickA$, $\PickC$ and $\PickB$, which we have introduced in Definitions~\ref{def:reciprocal} and \ref{def:class_P}.

\begin{definition}[reverse evolution family]
\label{def:REF}
We call a two-parameter family $(f_{s,t})_{s,t \in I, s \le t}$ of holomorphic self-mappings on $D$ a \textit{reverse evolution family} (REF for short) on $D$ if
\begin{enumerate}
\setlength{\leftskip}{1.5em}
\renewcommand{\theenumi}{(REF\arabic{enumi})}
\renewcommand{\labelenumi}{\textup{\theenumi}}
\item \label{TM1}
$f_{s,s}=\mathrm{id}_D$ for each $s$, 
\item \label{TM2}
$f_{s,t} \circ f_{t,u} = f_{s,u}$ if $s \le t \le u$, 
\item \label{TM3}
$(s,t) \mapsto f_{s,t}$ is continuous.
\end{enumerate}
Moreover, given a family $\mathcal{F}$ we call $(f_{s,t})$ an \emph{$\mathcal{F}$-REF} if $f_{s,t} \in \mathcal{F}$ for all $s,t$.
\end{definition}

\begin{remark} \label{rem:reversal}
An intriguing contrast lies in the fact that many researchers in complex analysis focus on the case in which the composition order in \textup{\ref{TM2}} is reversed; namely, a two-parameter family $(\varphi_{s,t})_{s,t \in I, s \le t}$ of holomorphic self-mappings on $D$ is called an \emph{evolution family} (EF for short) if 
\begin{enumerate}
\setlength{\leftskip}{0.5em}
\renewcommand{\theenumi}{(EF\arabic{enumi})}
\renewcommand{\labelenumi}{\textup{\theenumi}}
\item \label{EF1} $\varphi_{s,s} = \id_D $ for all $s \in I$;
\item \label{EF2} $\varphi_{t,u} \circ \varphi_{s,t} = \varphi_{s,u}$ if $s \le t \le u$;
\item \label{EF3} $(s,t)\mapsto \varphi_{s,t}$ is continuous.
\end{enumerate}
We will also use the term $\mathcal{F}$-EF in the same way as $\mathcal{F}$-REF. 
Here we note that one can obtain some duality between REFs and EFs, modifying their parameters suitably.
For example, it is easily seen that $( f_{s,t} )_{s \le t}$ is a REF if and only if $( \varphi_{s,t} := f_{T-t, T-s} )_{s \le t \le T}$ is an EF with index set $[0,T]$ for every $T \in I$.  Alternatively, one could also consider the EF  $(f_{-t, -s} )_{s \le t}$ indexed by the negative set $-I$.  
\end{remark}

Although the definition of (R)EFs does not assume univalence \textit{a priori}, the following is true:

\begin{proposition} \label{th:univalent_REF}
Every EF consists of only univalent functions.
By Remark~\ref{rem:reversal}, any REF also consists of univalent functions only.
\end{proposition}

Pommerenke~\cite[Satz~3]{Pom65} first gave a beautiful proof of Proposition~\ref{th:univalent_REF}, using simple techniques in complex analysis, with the additional assumption that all elements of a given EF have a fixed point at the origin (see also Yanagihara~\cite[Theorem~2.7]{Yan19+}).
Later, his technique was slightly extended by Hoshinaga, Hotta and Yanagihara \cite[Theorem~1.6]{HHY22+} to show the univalence of the elements of an EF without any restriction.
Contreras and D\'{\i}az-Madrigal \cite[Proposition~2.4]{CDM21} treated the case where $D$ is a Riemann surface. 
For another proof based on the Loewner differential equation, see Franz, Hasebe and Schlei{\ss}inger~\cite[Theorem~3.16]{FHS20}.

\begin{proposition}[cf.\ {\cite[Theorem~3.1]{CDMG10} \cite[Theorem~4.1]{CDMG14}}]
\label{prop:DLC_to_REF}
If $(f_t)$ is a DLC on $D$, then the family
\begin{equation} \label{eq:REF-DLC}
f_{s,t}:=f_s^{-1} \circ f_t,
\qquad s,t \in I,\ s \le t,
\end{equation}
is an REF on $D$.
Conversely, if $(f_{s,t})$ is an REF on $D$, then the family $f_t:=f_{0,t}$, $t \in I$, is a DLC on $D$.
In particular, DLCs and REFs are in one-to-one correspondence via \eqref{eq:REF-DLC}.
\end{proposition}

\begin{proof}
Let $(f_t)$ be a DLC.
The only nontrivial part of this proposition is the continuity of $(s,t) \mapsto f_{s,t}$ with $f_{s,t}$ defined by \eqref{eq:REF-DLC}.
Let $s,t \in I$ be such that $s \le t$, $K$ be a compact subset of $D$, and $(D_k)_{k \in \N}$ be an increasing sequence of analytic Jordan domains with $\overline{D_k} \subset D$ and $\bigcup_k D_k=D$.
As
\[
f_t(K) \subset f_t(D) \subset f_s(D) = \bigcup_{k \in \N}f_s(D_k),
\]
there exists $k \in \N$ such that $f_t(K) \subset f_s(D_k)$ by the compactness of $f_t(K)$.
We fix such a $k$ and put $\epsilon:=5^{-1}d(f_t(K),f_s(\partial D_k))>0$ with $d$ denoting the Euclidean distance.
By the continuity of $u \mapsto f_u$, there exist neighborhoods $N(s)$ and $N(t)$ of $s$ and $t$, respectively, in $I$ such that
\begin{equation} \label{eq:K_and_bD_k}
\sup_{\substack{z \in K \\ v \in N(t)}}\lvert f_v(z)-f_t(z) \rvert<\epsilon
\quad \text{and} \quad
\sup_{\substack{z \in \partial D_k \\ u \in N(s)}}\lvert f_u(z)-f_s(z) \rvert<\epsilon.
\end{equation}
Then we have $f_v(K) \subset f_u(D_k)$ and $d(f_v(K),f_u(\partial D_k)) \ge 3\epsilon$ 
for all $u \in N(s)$ and $v \in N(t)$.
In particular, for such a $(u,v)$ the Lagrange inversion formula yields
\begin{equation} \label{eq:Lagrange}
f_{u,v}(z)
=f_u^{-1}(f_v(z))
=\frac{1}{2\pi i}\int_{\partial D_k}\frac{w f_u^\prime(w)}{f_u(w)-f_v(z)}\,dw, \qquad z\in K. 
\end{equation}
Noting that $\lvert f_u(w)-f_v(z) \rvert \ge 3\epsilon$ in this formula, we have
\begin{align*}
&\sup\{\, \lvert f_{u,v}(z)-f_{s,t}(z) \rvert : z \in K,\ (u,v) \in N(s) \times N(t),\ u \le v \,\} \\
&\le \frac{l(\partial D_k)}{18\pi \epsilon^2}\sup \lvert w \rvert \left\lvert f_u^\prime(w)(f_s(w)-f_t(z))-f_s^\prime(w)(f_u(w)-f_v(z)) \right\rvert. 
\end{align*}
Here, $l(\partial D_k)$ denotes the length of $\partial D_k$, and the latter supremum is taken over all $z \in K$, $w \in \partial D_k$, and $(u,v) \in N(s) \times N(t)$ with $u \le v$.
Clearly, the last expression goes to zero as $(u,v) \to (s,t)$, as was to be proved.
\end{proof}

We say that a DLC $(f_t)_{t}$ and a REF $(f_{s,t})_{s\le t}$ are \emph{associated with} each other if \eqref{eq:REF-DLC} holds.
Here is a convergence theorem for such DLCs and REFs, which generalizes part of the previous result \cite[Theorem~3.15]{HH22}.

\begin{theorem} \label{th:conv_DLC_REF}
Let $D \subsetneq \C$ be a simply connected domain.
Suppose that $(f_t)_{t}$ and $(f^n_t)_{t}$, $n \in \N$, are DLCs on $D$, and $(f_{s,t})_{s\le t}$ and $(f^n_{s,t})_{s\le t}$ are the REFs associated with them, respectively.
Then the following are equivalent as $n \to \infty$:
\begin{enumerate}
\item \label{i:conv_DLC}
$f^n_t(z)$ converges to $f_t(z)$ locally uniformly on $D \times I$;
\item \label{i:conv_REF}
$f^n_{s,t}(z)$ converges to $f_{s,t}(z)$ locally uniformly on $D \times I^2_\le$.
\end{enumerate}
\end{theorem}

\begin{proof}
Since \eqref{i:conv_REF} trivially implies \eqref{i:conv_DLC}, we only prove the converse.
The idea is similar to the proof of Proposition~\ref{prop:DLC_to_REF}.
As in that proof, let $s,t \in I$ be such that $s \le t$, $K$ be a compact subset of $D$, $D_k$ be an analytic Jordan domain with $\overline{D_k} \subset D$ and $f_t(K) \subset f_s(D_k)$, and $\epsilon=5^{-1}d(f_t(K),f_s(\partial D_k))$.
In addition, let $N(s)$ and $N(t)$ be neighborhoods of $s$ and $t$, respectively, such that \eqref{eq:K_and_bD_k} holds.
Moreover, by \eqref{i:conv_DLC} we can take $n_0 \in \N$, replacing $N(s)$ and $N(t)$ with smaller ones if necessary, so that
\begin{equation} \label{eq:K_and_bD_k_at_n}
\sup_{\substack{z \in K \\ v \in N(t)}}\lvert f^n_v(z)-f_v(z) \rvert<\epsilon
\quad \text{and} \quad
\sup_{\substack{z \in \partial D_k \\ u \in N(s)}}\lvert f^n_u(z)-f_u(z) \rvert<\epsilon
\end{equation}
for all $n \ge n_0$.
By \eqref{eq:K_and_bD_k} and \eqref{eq:K_and_bD_k_at_n}, we have $f^n_v(K) \subset f^n_u(D_k)$ and $d(f^n_v(K),f^n_u(\partial D_k)) \ge \epsilon$ for all $u \in N(s)$, $v \in N(t)$, and $n \ge n_0$.
Formula \eqref{eq:Lagrange} also holds with the superscript $^n$ put to all the appearances of $f_{u,v}$, $f_u$, and $f_v$ in it.
Since $\lvert f^n_u(w)-f^n_v(z) \rvert \ge \epsilon$ in this formula, we have
\begin{align*}
&\sup\{\, \lvert f^n_{u,v}(z) - f_{u,v}(z) \rvert : z \in K,\ (u,v) \in N(s) \times N(t),\ u \le v \,\} \\
&\le \frac{l(\partial D_k)}{2\pi \epsilon^2}
\sup \lvert w \rvert \left\lvert (f^n_u)^\prime(w)(f_u(w)-f_v(z))-f_u^\prime(w)(f^n_u(w)-f^n_v(z)) \right\rvert
\end{align*}
for all $n \ge n_0$.
Here, the latter supremum is taken over all $z \in K$, $w \in \partial D_k$, and $(u,v) \in N(s) \times N(t)$ with $u \le v$.
The last expression goes to zero as $n \to \infty$ by \eqref{i:conv_DLC}, which implies \eqref{i:conv_REF}.
\end{proof}

\subsection{DLCs and REFs of reciprocal Cauchy transforms}
\label{sec:ALC}

By \eqref{eq:mono_convolution}, Proposition~\ref{prop:Maassen_2.1}, and Corollary~\ref{cor:continuity_Cauchy}, we have the following:

\begin{proposition} \label{prop:additive_REF}
For a $\rhd$-CH $(\monotone_{s,t})_{s\le t}$, the family $(F[\monotone_{s,t}])_{s\le t}$ of reciprocal Cauchy transforms is a $\PickA$-REF on $\C^+$.
Conversely, for a $\PickA$-REF $(f_{s,t})_{s\le t}$ there exists a unique $\rhd$-CH $(\monotone_{s,t})_{s\le t}$ such that $f_{s,t}=F[\monotone_{s,t}]$.
\end{proposition}

For DLCs on $\C^+$ the same kind of assertion is less obvious but still true by virtue of Corollary~\ref{cor:anglim_composite}.

\begin{proposition} \label{prop:additive_DLC}
If  $(f_t)_{t}$ is a $\PickA$-DLC\footnote{%
Franz, Hasebe and Schlei{\ss}inger~\cite[Definition~3.1~(2)]{FHS20} call  a $\PickA$-DLC an additive Loewner chain. } 
then the associated REF $(f_{s,t})_{s\le t}$ is a $\PickA$-REF.
Hence, in particular, there exists a unique $\rhd$-CH $(\monotone_{s,t})_{s\le t}$ such that $f_{s,t}=F[\monotone_{s,t}]$.
\end{proposition}

Theorem~\ref{th:conv_DLC_REF} is transferred to the $\PickA$-valued case as follows:

\begin{theorem} \label{th:conv_monotone}
Let $(\monotone_{s,t})_{s\le t}$ and $(\monotone^n_{s,t})_{s\le t}$, $n\in \N$, be $\rhd$-CHs.
Suppose that $(f_t)_{t}$ and $(f_{s,t})_{s\le t}$ (resp.\ $(f^n_t)_t$ and $(f^n_{s,t})_{s\le t}$) are the $\PickA$-DLC and $\PickA$-REF associated with $(\monotone_{s,t})$ (resp.\ $(\monotone^n_{s,t})$).
Then the following are equivalent as $n\to \infty$:
\begin{enumerate}
\item \label{M1}
$(\monotone^n_{0,t})_{t}$ converges weakly to $(\monotone_{0,t})_{t}$ locally uniformly on $I$;
\item \label{M2}
$(\monotone^n_{s,t})_{s\le t}$ converges weakly to $(\monotone_{s,t})_{s\le t}$ locally uniformly on $I^2_\le$;
\item \label{M3}
$f^n_t(z)$ converges to $f_t(z)$ locally uniformly on $\C^+ \times I$;
\item \label{M4}
$f^n_{s,t}(z)$ converges to $f_{s,t}(z)$ locally uniformly on $\C^+ \times I^2_\le$.
\end{enumerate}
\end{theorem}

\begin{proof}
$\eqref{M1} \Leftrightarrow \eqref{M3}$ and $\eqref{M2} \Leftrightarrow \eqref{M4}$ follow from Corollary~\ref{cor:luc_rCauchy}, and $\eqref{M3} \Leftrightarrow \eqref{M4}$ is a direct consequence of Theorem~\ref{th:conv_DLC_REF}.
\end{proof}

\begin{remark} \label{rem:antimonotone_EF}
It is easy to show analogues of Proposition \ref{prop:additive_REF} and Theorem \ref{th:conv_monotone} for $\lhd$-CHs and $\PickA$-EFs.
However, analogues of Propositions \ref{prop:DLC_to_REF} and \ref{prop:additive_DLC} hold only partially.
More precisely, an ``increasing'' Loewner chain $(\varphi_t)\subset\PickA$ naturally induces the $\PickA$-EF $\varphi_{s,t}:=\varphi_t^{-1} \circ \varphi_s$, but given a $\PickA$-EF, the uniqueness and existence of an increasing LC inside $\PickA$ are unclear. 
If the index set is $I =[0,T]$, then $\varphi_t := \varphi_{t,T}$ will be a natural choice of such an increasing LC. 
In Contreras, D\'{\i}az-Madrigal and Gumenyuk~\cite{CDMG10}, this kind of existence and uniqueness are discussed in a more general setting of complex analysis.

\end{remark}


\section{Main results: the case of finite second moment}
\label{sec:finite_2nd_case}

In this section we study $\PickA$-REFs and $\PickA$-DLCs such that the associated $\rhd$-CHs have finite second moment.
By Proposition \ref{prop:Cauchy_P_prime} such REFs and DLCs are in $\PickB$.

\subsection{Continuity of angular residue for DLCs}
\label{sec:capacity_conti}

We begin with $\PickC$-REFs and $\PickC$-DLCs.
By \eqref{eq:REF-DLC} the DLC associated with a $\PickC$-REF is clearly a $\PickC$-DLC. 
The proof of the next lemma, which ensures the opposite correspondence, requires some effort; see Bauer~\cite[Lemma~4.1]{Bau05} or Murayama~\cite[Appendix~B]{Mur23}.

\begin{lemma} \label{lem:chordal_DLC_to_REF}
Suppose that $f, g \in \PickC$, $g$ is univalent, and $f(\C^+) \subset g(\C^+)$.
Then $g^{-1} \circ f \in \PickC$, and the angular residues at $\infty$ satisfy $\AR(g^{-1} \circ f)=\AR(f)-\AR(g)$.
In particular, the REF associated with a $\PickC$-DLC is a $\PickC$-REF.
\end{lemma}

For a $\PickC$-DLC $(f_t)_t$, it is clear that $I\ni t\mapsto \AR(f_t) \in[0,\infty)$ is non-decreasing because $0 \le \AR(f_{s,t})=\AR(f_t)-\AR(f_s)$ by Lemma~\ref{lem:chordal_DLC_to_REF}. 
Moreover, the following holds:

\begin{proposition} \label{prop:capacity_conti}
Let $(f_t)_{t}$ be a $\PickC$-DLC.
The angular residue $\AR(f_t)$ at $\infty$ is then a non-decreasing continuous function of $t \in I$. 
\end{proposition}

Proposition~\ref{prop:capacity_conti} will follow from the
slightly more general Proposition~\ref{prop:DCT_for_AR} proved below. Indeed, once Proposition~\ref{prop:DCT_for_AR} is established, taking $\tilde{f}=f_T$ shows the continuity of $t \mapsto \AR(f_t)$ on any interval $[0,T] \subset I$ for a $\PickC$-DLC $(f_t)$.

\begin{proposition} \label{prop:DCT_for_AR}
Suppose that $\tilde{f}$, $f$ and $f_n$, $n \in \N$, are functions in $\PickC$.
If each $f_n$ is univalent, $f_n \to f$ ($n \to \infty$) locally uniformly on $\C^+$, and $\tilde{f}(\C^+) \subset f(\C^+) \cap \bigcap_{n \in \N} f_n(\C^+)$, then
\[
\lim_{n \to \infty}\AR(f_n)=\AR(f).
\]
\end{proposition}

\begin{proof}
We build several functions from $\tilde{f}$, $f$ and $f_n$ to make their characteristic measures more tractable.

First, note that $f$ is also univalent by Hurwitz's theorem.
Then by assumption, the composites $\varphi:=f^{-1} \circ \tilde{f}$ and $\varphi_n:=f_n^{-1} \circ \tilde{f}$ are well-defined and, by Lemma~\ref{lem:chordal_DLC_to_REF}, belong to $\PickC$ with
\begin{equation} \label{eq:composed_AR}
\AR(\varphi)=\AR(\tilde{f})-\AR(f),
\quad
\AR(\varphi_n)=\AR(\tilde{f})-\AR(f_n).
\end{equation}
We have a trivial inequality
\begin{equation} \label{eq:dom_func_1}
\Im \varphi_n(z) = \Im (f_n^{-1}(\tilde{f}(z))) \le \Im \tilde{f}(z), \quad z \in \C^+,\ n \in \N.
\end{equation}
Moreover, it is easy to see that $\varphi_n \to \varphi$ ($n \to \infty$) locally uniformly, using the Lagrange inversion formula as in the proof of Proposition~\ref{prop:DLC_to_REF}.

Next, we fix $R>0$; it is unnecessary to specify the value of $R$ because any positive number will do.
We define
\[
\tilde{h}(z):=\tilde{f}(z+iR)-iR,\quad
\psi(z):=\varphi(z+iR)-iR,\quad
\psi_n(z):=\varphi_n(z+iR)-iR.
\]
These are again functions in $\PickC$; in fact,
\begin{align*}
\AR(\tilde{h})
&=\angle\!\lim_{z\to\infty} z(z-\tilde{h}(z))
=\angle\!\lim_{z\to\infty} z(z+iR-\tilde{f}(z+iR)) \\
&=\angle\!\lim_{z\to\infty} (z+iR)(z+iR-\tilde{f}(z+iR))
=\AR(\tilde{f})
\end{align*}
and similarly
\begin{equation} \label{eq:ver-shifted_AR}
\AR(\psi)=\AR(\varphi),\quad
\AR(\psi_n)=\AR(\varphi_n).
\end{equation}
Note that these shifted functions are defined on $\{\, z : \Im z > -R \,\}$, which contains $\R$.
By \eqref{eq:dom_func_1}, we have
\[
\Im \psi_n(x) \le \Im \tilde{h}(x),\quad x \in \R.
\]
In addition, $\psi_n(x) \to \psi(x)$ for every $x \in \R$ as $\varphi_n \to \varphi$.

By \eqref{eq:composed_AR} and \eqref{eq:ver-shifted_AR}, the proof finishes if we show
\[
\lim_{n \to \infty}\AR(\psi_n)=\AR(\psi).
\]
To see this, we observe from the Stieltjes inversion formula \eqref{eq:PN_interval} that the characteristic measures of $\tilde{h}$, $\psi$ and $\psi_n$ have their imaginary parts (divided by $\pi$) as the densities.
In particular, $x \mapsto \Im \tilde{h}(x)$ is integrable on $(\R, dx)$.
The properties in the preceding paragraph thus allow us to apply the Lebesgue dominated convergence theorem, which yields, as $n \to \infty$,
\begin{align*}
\AR(\psi_n)=\rho(\psi_n; \R)&=\frac{1}{\pi}\int_{\R}\Im \psi_n(x)\,dx \\
&\to \frac{1}{\pi}\int_{\R}\Im \psi(x)\,dx=\rho(\psi; \R)=\AR(\psi). \qedhere
\end{align*}
\end{proof}

\begin{remark}
The reader can observe that the assumption of Proposition~\ref{prop:DCT_for_AR} prohibits the ``escape to infinity'' that causes discontinuity in Remark~\ref{rem:mean-var_conti}.
\end{remark}

By Proposition~\ref{prop:capacity_conti} and Corollary~\ref{cor:continuity_Cauchy}, we have the following corollary:

\begin{corollary} \label{cor:rho_conti}
For a $\PickC$-DLC $(f_t)_{t}$, the mapping $t \mapsto \rho_{f_t}$ from $I$ to $\mea$ is weakly continuous. 
\end{corollary}

The angular residue at infinity governs a part of the structure of a $\PickC$-DLC $(f_t)_t$, especially its modulus of continuity in $t$.
For example, we can prove a kind of converse to Proposition~\ref{prop:capacity_conti}.
Actually, this direction is easier to prove and known to experts.

\begin{proposition}[e.g.\ {\cite[Theorem~5.1]{Bau05}}]
\label{prop:converse_to_cap_conti}
Let $(f_t)_{t}$ be a one-parameter family in $\PickC$ that satisfies the following:
\begin{enumerate}
\item $f_t$ is univalent for $t \in I$;
\item $f_0=\mathrm{id}_{\C^+}$, and $f_s(\C^+) \supset f_t(\C^+)$ for $s,t \in I$ with $s \le t$;
\item $t \mapsto \AR(f_t)$ is continuous.
\end{enumerate}
Then $t \mapsto f_t$ is continuous, i.e., $(f_t)_{t}$ is a DLC in the sense of Definition~\ref{def:DLC}.
Moreover, if $\AR(f_t)=t$ for all $t \in I$, then $t \mapsto f_t(z)$ is locally Lipschitz on $I$ for each $z \in \C^+$.
\end{proposition}

\begin{proof}
Let $f_{s,t}:=f_s^{-1} \circ f_t$ for $s,t \in I$ with $s \le t$.
Then
\begin{align*}
f_t(z)-f_s(z)
&=f_s(f_{s,t}(z))-f_s(z) \\ \displaybreak[1]
&=f_{s,t}(z)-z+\int_{\R}\left(\frac{1}{x-f_{s,t}(z)}-\frac{1}{x-z}\right)\rho(f_s;dx) \\ \displaybreak[1]
&=(f_{s,t}(z)-z)\left\{1+\int_{\R}\frac{\rho(f_s;dx)}{(x-f_{s,t}(z))(x-z)}\right\}.
\end{align*}
Here, we know $f_{s,t} \in \PickC$ and $\AR(f_{s,t})=\AR(f_t)-\AR(f_s)$ from Lemma~\ref{lem:chordal_DLC_to_REF}.
Hence, by Proposition~\ref{prop:class_P}~\eqref{HN:ineq},
\begin{align}
\lvert f_t(z)-f_s(z) \rvert
&\le \frac{\AR(f_{s,t})}{\Im z}\left\{1+\frac{\AR(f_s)}{(\Im f_{s,t}(z))(\Im z)}\right\}
\notag \\
&\le \frac{\AR(f_t)-\AR(f_s)}{\Im z}\left\{1+\frac{\AR(f_s)}{(\Im z)^2}\right\}.
\label{eq:dc-abs_conti}
\end{align}
From the last inequality the proposition is obvious.
\end{proof}

\begin{corollary} \label{cor:constant_DLC}
Let $(f_t)_{t}$ be a $\PickC$-DLC.
If $\AR(f_a)=\AR(f_b)$ for some $a,b \in I$ with $a<b$, then $f_t=f_a$ for every $t \in [a,b]$.
\end{corollary}

\begin{proof}
Obvious from \eqref{eq:dc-abs_conti}.
\end{proof}

We next consider $\PickB$-REFs and $\PickB$-DLCs.
By \eqref{eq:REF-DLC}, the DLC associated with a $\PickB$-REF is a $\PickB$-DLC.
The next lemma gives the converse.

\begin{lemma} \label{lem:quasi-chordal_DLC_to_REF}
Suppose that $f,g \in \PickB$ are univalent and satisfy $f(\C^+) \subset g(\C^+)$.
Then $g^{-1} \circ f \in \PickB$, $\mean(g^{-1} \circ f)=\mean(f)-\mean(g)$, and $\AR(g^{-1} \circ f)=\AR(f)-\AR(g)$.
In particular, if $(f_t)_{t}$ is a $\PickB$-DLC, then the associated REF is a $\PickB$-REF, and $t \mapsto \AR(f_t)$ is nondecreasing.
\end{lemma}

\begin{proof}
Let $\varphi:=g^{-1} \circ f$.
Since the inner shifts $\tilde f(z)=f(z+\mean(f))$ and $\tilde g(z)=g(z+\mean(g))$ are members of $\PickC$,
we have $\tilde \varphi:=\tilde{g}^{-1} \circ \tilde{f} \in \PickC$ and $\AR(\tilde \varphi)=\AR(f)-\AR(g)$ by Lemma~\ref{lem:chordal_DLC_to_REF}.
In addition, the relation of $\varphi$ and $\tilde \varphi$ is seen from
\[
f(z)=\tilde f(z-\mean(f))
=\tilde g(\tilde \varphi(z-\mean(f)))
=g(\tilde \varphi(z-\mean(f))+\mean(g));
\]
that is, $\varphi(z)=\tilde \varphi(z-\mean(f))+\mean(g)$.
This expression implies $\varphi \in \PickB$, $\mean(\varphi)=\mean(f)-\mean(g)$, and $\AR(\varphi)=\AR(\tilde \varphi)$.
\end{proof}

\begin{theorem} \label{th:capacity_conti_prime}
For a $\PickB$-DLC $(f_t)_{t}$, the mapping $t\mapsto \mean(f_t)$ is continuous, and $t \mapsto \rho_{f_t}$ is weakly continuous.
In particular, $t \mapsto \AR(f_t)$ is also continuous (and non-decreasing).
\end{theorem}

\begin{proof}
We put $m_t:=\mean(f_t)$, $\rho_t:=\rho_{f_t}$, and $\ar_t:=\AR(f_t)$.
For $s,t \in I$ and $v>0$ we have
\begin{align*}
\lvert m_t-m_s \rvert 
&\le\lvert f_s(iv) - f_t(iv) \rvert +\left\lvert \int_\R \frac1{x-iv}\,\rho_s(dx) -  \int_\R \frac1{x-iv}\,\rho_t(dx)\right\rvert \\
&\le \lvert f_s(iv)-f_t(iv) \rvert + \frac{\ar_s}{v} + \frac{\ar_t}{v},
\end{align*}
and hence
\[
\limsup_{t\to s}\,\lvert m_t -m_s \rvert \le \frac{\ar_s+\lim_{t\downarrow s}\ar_{t}}{v},\qquad s \in I,\ v>0. 
\]
Letting $v \to +\infty$ now yields $\lim_{t\to s}\lvert m_t-m_s \rvert=0$, which shows that $m_t$ is continuous.

To see the weak continuity of $\rho_t$, let $\tilde f_t(z) := f_t(z+m_t)$, which is in $\PickC$.
The characteristic measure of $\tilde f_t$ is given by $\tilde\rho_t(B):=\rho_t(B+m_t)$, $B \in \cB(\R)$.
It is easy to see that $(\tilde f_t)_{t}$ is a $\PickC$-DLC, and hence the mapping $t \mapsto \tilde{\rho}_t$ is weakly continuous by Corollary \ref{cor:rho_conti}.
For any $T \in I$, any $s,t \in [0,T]$, and any bounded uniformly continuous function $f \colon \R \to \R$, we have
\begin{align*}
&\left\lvert \int_{\R} f(x) \,\rho_s(dx) - \int_{\R} f(x) \,\rho_t(dx) \right\rvert \\ \displaybreak[1]
&=\left\lvert \int_{\R} f(x+m_s) \,\tilde\rho_s(dx) - \int_{\R} f(x+m_t) \,\tilde\rho_t(dx) \right\rvert \\ \displaybreak[1]
&\le \left\lvert \int_{\R} f(x+m_t) \,\tilde\rho_s(dx) - \int_{\R} f(x+m_t) \,\tilde\rho_t(dx) \right\rvert
+ \left\lvert \int_{\R} \bigl( f(x+m_s) - f(x+m_t) \bigr) \,\tilde\rho_s(dx) \right\rvert \\ \displaybreak[1]
&\le \left\lvert \int_{\R} f(x+m_t) \,\tilde\rho_s(dx) - \int_{\R} f(x+m_t) \,\tilde\rho_t(dx) \right\rvert
+ \ar_T \sup_{x \in \R}\lvert f(x+m_s)-f(x+m_t) \rvert.
\end{align*}
The last expression tends to zero as $s \to t$, and hence $t \mapsto \rho_t$ is weakly continuous by the portmanteau theorem (Proposition~\ref{prop:portmanteau}).
\end{proof}

\subsection{Chordal Loewner differential equation in the classical setting}
\label{sec:chordal_LDE}

There are several works \cite{AAS83,GB92,Bau05} that treat the evolution equation for $\PickC$-(R)EFs and $\PickC$-(D)LCs in the absolutely continuous case.
We summarize these results in a way suitable for our context.

We begin with a brief mention of holomorphic semigroups.
Let $D \subsetneq \C$ be a simply connected domain.
A continuous family $(f_t)_{0\le t <+\infty}$ of holomorphic self-mappings of $D$ is a \emph{semigroup} if $f_0=\id_D$ and $f_{t+s}=f_t \circ f_s$ for all $0\le s,t <+\infty$.
It is known \cite[Theorem~10.1.4]{BCDM20} that $(f_t)$ satisfies the autonomous evolution equation
\[
\frac{d f_t(z)}{dt}=p(f_t(z)),\qquad p(z):=\lim_{t \to +0}\frac{f_t(z)-z}{t},\ z \in D,
\]
with $p$ called the \emph{infinitesimal generator} of $(f_t)$.
For example, if $D=\C^+$ and if $(f_t)$ consists of reciprocal Cauchy transforms (i.e., $f_t^\prime(\infty)=1$ as in Definition \ref{def:reciprocal}), then as is already seen in Proposition \ref{thm:MLK}, the infinitesimal generator of $(f_t)$ is given by \eqref{eq:MLK}.
If further $(f_t)_t \subset \PickC$, then the infinitesimal generator is exactly a Cauchy transform (with negative sign) \cite[Theorem~2]{GB92}:

\begin{proposition}
\label{prop:P-generator}
A holomorphic function $p\colon \C^+\to \C$ is the  infinitesimal generator of a semigroup in $\PickC$ if and only if there exist $a \ge 0$ and $\nu \in \prob$ such that $p(z)=-aG_\nu(z)=-G_{a\nu}(z)$.
\end{proposition}

As a non-autonomous variant, a function $p \colon D \times I \to \C$ is called a \emph{Herglotz vector field} on $D$ if $p(z,t)$ is an infinitesimal generator (i.e., the infinitesimal generator of a semigroup) on $D$ for each $t \in I$ and has certain regularity in $t$.
In the general theory \cite{BCDM12, CDMG10, BCDM09}, the non-autonomous equations for evolution families and Loewner chains are written in terms of this vector field.

As one can imagine from Proposition~\ref{prop:P-generator}, to a $\PickC$-DLC a Herglotz vector field $p$ of the following form corresponds:
\begin{equation} \label{eq:reduced_HVF}
p(z,t)=-G[\nu_t](z)=\int_{\mathbb{R}}\frac{1}{x-z}\,\nu_t(dx),
\qquad z \in \C^+,~t \in I. 
\end{equation}
Here, $(\nu_t)_{t \in I}$ is a family in $\mea$ such that $t \mapsto p(z,t)$ is measurable for each fixed $z \in \C^+$.
If a Herglotz vector field $p$ has the form \eqref{eq:reduced_HVF}, let us call $(\nu_t)_{t}$ the \emph{driving kernel} of $p$.
Moreover, if $\nu_t$ are all in $\prob$, then we say that $p$ is \emph{normalized}. 
For later use, we here provide equivalent conditions for the $t$-measurability of $p$ in terms of $(\nu_t)_t$.
Below, let $\overline{\bf M}_1(\R)$ be the set of Borel sub-probability measures on $\R$ and $\cB(\overline{\bf M}_1(\R))$ be the Borel $\sigma$-algebra of $\overline{\bf M}_1(\R)$ endowed with vague or weak topology; here, either topology generates the same $\sigma$-algebra \cite[Corollary~D.6]{Mur23}.
$\mathrm{Hol}(\C^+,-\overline{\C^+})$ denotes the set of holomorphic functions on $\C^+$ taking values in $-\overline{\C^+}$ equipped with the topology of locally uniform convergence.

\begin{lemma} \label{lem:HVF_measurability}
Let $(\nu_t)_{t}$ be a family in $\overline{\bf M}_1(\R)$.
Then the following are equivalent:
\begin{enumerate}
\item \label{i:def_Borel_mble}
the mapping $t \mapsto \nu_t$ is $\cB(I)/\cB(\overline{\bf M}_1(\R))$-measurable;
\item \label{i:lu_Cauchy_mble}
$t \mapsto G[\nu_t]$ is $\cB(I)/\cB(\mathrm{Hol}(\C^+,-\overline{\C^+}))$-measurable;
\item \label{i:pw_Cauchy_mble}
$t \mapsto G[\nu_t](z)$ is $\cB(I)/\cB(-\overline{\C^+})$-measurable for each $z \in \C^+$;
\item \label{i:bd_conti_mble}
$t \mapsto \int_{\R}f(x)\,\nu_t(dx)$ is $\cB(I)/\cB(\C)$-measurable for each $f \in C_{\rm b}(\R)$;
\item \label{i:cpt_conti_mble}
$t \mapsto \int_{\R}f(x)\,\nu_t(dx)$ is $\cB(I)/\cB(\C)$-measurable for each $f \in C_{\rm c}(\R)$.
\end{enumerate}
If $\nu_t \in \prob$ for all $t \in I$, then $\cB(\overline{\bf M}_1(\R))$ in \eqref{i:def_Borel_mble} can be replaced by $\cB(\prob)$.
\end{lemma}

We include the proof that \eqref{i:def_Borel_mble}--\eqref{i:pw_Cauchy_mble} in Lemma~\ref{lem:HVF_measurability} are equivalent in Appendix~\ref{sec:measurability}.
For the proof that either \eqref{i:bd_conti_mble} or \eqref{i:cpt_conti_mble} is equivalent to \eqref{i:def_Borel_mble}, see Murayama~\cite[Appendix~D]{Mur23} and references therein.
We simply say that $(\nu_t)_{t}$ is \emph{measurable}, omitting the reference to Borel $\sigma$-algebras, if one of \eqref{i:def_Borel_mble}--\eqref{i:cpt_conti_mble} in Lemma~\ref{lem:HVF_measurability} holds.

Based on a Herglotz vector field $p$ of the form \eqref{eq:reduced_HVF}, the \emph{chordal Loewner differential equation} is formulated as in the next proposition; for the proof, see Bauer \cite[Theorem 5.3]{Bau05} (cf.\ Goryainov and Ba \cite[Theorem~3]{GB92}).

\begin{proposition}
\label{th:chordal_LDE}
Let $(f_t)_{t}$ be a $\PickC$-DLC with $\AR(f_t)=t$ and $(f_{s,t})_{s\le t}$ the associated REF.
Then there exist a normalized Herglotz vector field $p$ of the form \eqref{eq:reduced_HVF} and a Lebesgue-null set $N \in \cB(I)$ such that
\begin{equation} \label{eq:chordal_LDE}
\frac{\partial f_t(z)}{\partial t}=\frac{\partial f_t(z)}{\partial z}p(z,t),
\qquad t \in I \setminus N,\, z \in \C^+,
\end{equation}
and for each $t \in I$,
\begin{equation} \label{eq:chordal_LODE}
\frac{\partial f_{s,t}(z)}{\partial s}=-p(f_{s,t}(z),s),
\qquad s \in [0,t] \setminus N,\ z \in \C^+.
\end{equation}
Moreover, $p$ is unique in the sense that, if $q$ is a normalized Herglotz vector field of the form \eqref{eq:reduced_HVF} and $M\subset I$ with Lebesgue measure $0$ such that \eqref{eq:chordal_LDE} or \eqref{eq:chordal_LODE} holds with $p$ and $N$ replaced by $q$ and $M$, respectively, then $p(\cdot,t) = q(\cdot,t)$ for a.e.\ $t$.  
\end{proposition}

\begin{proposition}[{\cite[Theorem~5.6]{Bau05}, \cite[Theorem~5.5]{Bau05}; cf.\ \cite[Theorem~4]{GB92}}]
\label{th:unique_sol_to_LDE}
Let $p$ be a normalized Herglotz vector field of the form \eqref{eq:reduced_HVF}. 
Then there exists a unique $\PickC$-DLC $(f_t)$ with $\AR(f_t)=t$ that satisfies \eqref{eq:chordal_LDE}.
Moreover, there exists a unique $\PickC$-REF $(f_{s,t})$ that enjoys \eqref{eq:chordal_LODE}. The families $(f_t)$ and $(f_{s,t})$ are associated with each other. 
\end{proposition}

By Proposition~\ref{th:chordal_LDE}, there exists an a.e.\ unique family $(\nu_t)_{t}$ that satisfies \eqref{eq:chordal_LDE} with $p(z,t)=-G[\nu_t](z)$.
We call it the \emph{driving kernel} ``of $(f_t)_{t}$'' instead ``of $p$''.
In Section~\ref{sec:LIE-additive}, we shall see that allowing arbitrary parametrization instead of $\AR(f_t)=t$ in Proposition~\ref{th:chordal_LDE} naturally leads to an integro-differential equation. 
\begin{remark} \label{lem:no_mass_escape}
In the existing proofs \cite[Theorem~3]{GB92} \cite[Theorem~5.3]{Bau05} \cite[Theorem~3.1]{Sch17} of Proposition \ref{th:chordal_LDE}, the driving kernel is obtained as the vague limit
\[
\nu_s=\lim_{t \downarrow s}\frac{\rho_{f_{s,t}}}{\AR(f_{s,t})}=\lim_{t \downarrow s}\frac{\rho_{f_{s,t}}}{\rho_{f_{s,t}}(\R)} \in \overline{\bf M}_1(\R).
\]
The existence of this limit is ensured by a certain argument involving the uniqueness of Cauchy transform and Corollary~\ref{cor:continuity_Cauchy}.
To see further that $p$ is normalized, i.e., $\nu_s \in \prob$, we need the following argument posed earlier by Aleksandrov, Aleksandrov and Sobolev~\cite{AAS83}.
By the absolute continuity in Proposition~\ref{prop:converse_to_cap_conti}, we can integrate \eqref{eq:chordal_LDE} with $z=iy$ to obtain
\[
f_t(iy)-iy=-\int_0^t f_s^\prime(iy)\, G[\nu_s](iy) \,ds.
\]
Hence
\begin{equation} \label{eq:AAS83}
y\Im(f_t(iy)-iy)
\le \int_0^t \left\lvert f_s^\prime(iy) \right\rvert \left\lvert iy\, G[\nu_s](iy) \right\rvert ds.
\end{equation}
By \eqref{eq:angderiv_at_infty} and Proposition~\ref{prop:Maassen_2.1},
passage to the limit $y \to +\infty$ in \eqref{eq:AAS83} yields $t \le \int_0^t \nu_s(\R) \,ds$.
This inequality implies $\nu_s(\R)=1$ for Lebesgue-a.e.\ $s \in [0,t]$ since $\nu_s(\R) \le 1$.
\end{remark}

\subsection{Loewner integral and integro-differential equations}
\label{sec:LIE-additive}

We now introduce a distributional generalization of ``$p(z,t)\,dt$'' for a Herglotz vector field $p$.
Here ``distributional'' means that we focus on the ``measure'' $P(z,B)=\int_B p(z,t)\,dt$; indeed, given such a $P$ one can write down the integral form of the Loewner differential equation.
We consider a function $P \colon \C^+ \times \cB_{\rm c}(I) \to \overline{\C^+}$ such that 
\begin{enumerate}
\setlength{\leftskip}{0.5em}
\renewcommand{\theenumi}{(HM\arabic{enumi})}
\renewcommand{\labelenumi}{\textup{\theenumi}}
\item\label{CMF1}
$B \mapsto P(z,B)$ is countably additive on the collection $\cB_{\rm c}(I)$ of precompact Borel subsets of $I$ for each fixed $z \in D$;
\item\label{eq:conti_HFM}
$P(z,\{t\})=0$ for all $z \in \C^+$ and $t \in I$;
\item \label{CMF2.5} $z\mapsto P(z,B)$ is holomorphic for each $B \in \cB_{\rm c}(I)$; 
\item\label{CMF3}
$\lim_{y \to +\infty}y\, P(iy,B)$ exists finitely for each $B \in \cB_{\rm c}(I)$.
\end{enumerate}
See Hasebe and Hotta~\cite[Definition~3.4]{HH22} for the corresponding concept in the case of the unit disk, which is called a Herglotz family of measures. 
Condition \ref{CMF1} implies that, for fixed $z \in D$ and $B_0 \in \cB_{\rm c}(I)$, the set function $B \mapsto P(z, B \cap B_0)$ is a complex Borel measure on $B_0$.
Condition \ref{eq:conti_HFM} means the continuity of $P$ as a measure, and conditions \ref{CMF2.5} and \ref{CMF3} correspond to the form \eqref{eq:reduced_HVF} of Herglotz vector fields.
Such a family of measures $P$ is characterized in several ways as follows. 

\begin{proposition} \label{prop:CMF}
For a function $P \colon \C^+ \times \cB_{\rm c}(I) \to \overline{\C^+}$, the following four conditions are equivalent:
\begin{enumerate}
\item \label{i:CMF_Pick}
$P$ satisfies \textup{\ref{CMF1}}--\textup{\ref{CMF3}}; 
\item \label{i:CMF_Theta}
there exists a Borel measure $\varTheta$ on $\R \times I$ with
\begin{equation} \label{eq:continuity_of_CMF}
\varTheta(\R \times [0,t])<+\infty
\quad \text{and} \quad
\varTheta(\R \times \{t\})=0
\quad \text{for every}\ t \in I
\end{equation}
such that
\begin{equation} \label{eq:CMF_Theta}
P(z,B)
=\int_{\R}\frac{1}{x-z}\,\varTheta(dx \times B)
=-G_{\varTheta({\cdot} \times B)}(z),
\quad z \in \C^+,\ B \in \cB_{\rm c}(I); 
\end{equation}
\item \label{i:CMF_nu}
there exist an atomless Radon measure $\itm=\itm_P$ on $I$ and measurable family $(\nu_t)_{t}$ in $\prob$ such that
\begin{equation} \label{eq:CMF_nu}
P(z,B)
=\int_B \int_{\R}\frac{1}{x-z} \,\nu_t(dx) \,\itm(dt)
=-\int_B G[\nu_t](z) \,\itm(dt),
\quad z \in \C^+,\ B \in \cB_{\rm c}(I); 
\end{equation}
\item \label{i:CMF_q}
there exist an atomless Radon measure $\itm=\itm_P$ on $I$ and a normalized Herglotz vector field $p \colon \C^+ \times I \to \overline{\C^+}$ of  the form \eqref{eq:reduced_HVF}  such that
\begin{equation} \label{eq:CMF_q}
P(z,B)
=\int_B p(z,t) \,\itm(dt),
\quad z \in \C^+,\ B \in \cB_{\rm c}(I). 
\end{equation}
\end{enumerate}
Moreover, if one of \eqref{i:CMF_Pick}--\eqref{i:CMF_q} is the case, then
\begin{enumerate}
\renewcommand{\labelenumi}{\rm (\alph{enumi})}
\item
the measure $\varTheta$ on $\R \times I$ in \eqref{i:CMF_Theta} is unique, 
\item
the measures $\itm$ on $I$ in \eqref{i:CMF_nu} and \eqref{i:CMF_q} are the same and unique, 
\item
the measure $\nu_t$ on $\R$ in \eqref{i:CMF_nu} is unique for $\itm$-a.e.\ $t \in I$, 
\item
for $A \in \cB(\R)$ and $B \in \cB_{\rm c}(I)$,
\begin{equation} \label{eq:CMF_Theta-nu}
\varTheta(A \times B) = \int_B \nu_t(A) \,\itm(dt), 
\end{equation}
\item
for each $B \in \cB_{\rm c}(I)$,
\begin{equation} \label{eq:mass_of_CMF}
\itm(B)
=\varTheta(\R \times B)
=-\lim_{y \to +\infty}iy\, P(iy,B)
=\lim_{y \to +\infty}y\Im P(iy,B).
\end{equation}
\end{enumerate}
We call $\varTheta$ the \emph{compound driving measure} of $P$.
\end{proposition}

\begin{proof}
Most of the proof is devoted to $\eqref{i:CMF_Pick} \Rightarrow \eqref{i:CMF_Theta}$.
Suppose that $P \colon \C^+ \times \cB_{\rm c}(I) \to \overline{\C^+}$ satisfies \ref{CMF1}--\ref{CMF3}.

Let $B \in \cB_{\rm c}(I)$.
One case is $P(z,B)=0$ for some $z \in \C^+$.
In this case, $P(z,B) \equiv 0$ on $\C^+$ since it is Pick.
Then letting $\theta_B$ be the zero measure on $\R$ we have
\begin{equation} \label{eq:setwise_HVF}
P(z,B)=-G[\theta_B](z)=\int_{\R}\frac{1}{x-z}\,\theta_B(dx).
\end{equation}
The other case is $P(z,B) \not\equiv 0$.
In this case, let $\itm(B):=-\lim_{y \to +\infty}iy\, P(iy,B)$, which exists finitely by definition.
By Proposition~\ref{prop:Maassen_2.1} there exists a unique measure $\theta_B \in \mea \setminus \{0\}$ that satisfies $\theta_B(\R)=\itm(B)$ and \eqref{eq:setwise_HVF}.
Thus, we obtain the unique family $(\theta_B : B \in \cB_{\rm c}(I))$ of finite Borel measures that fulfills \eqref{eq:setwise_HVF} for any $B \in \cB_{\rm c}(I)$.

We show that $\theta_B$ is countably additive in $B \in \cB_{\rm c}(I)$; namely, if $B_n$, $n \in \N$, are mutually disjoint Borel subsets of $I$ such that $B^\prime:=\bigsqcup_n B_n$ is precompact, then $\theta_{B^\prime}=\sum_n \theta_{B_n}$.
To this end, we look at the total masses of these measures.
Since $y \mapsto y\Im P(iy,B_n)$ is non-decreasing by Proposition~\ref{prop:Maassen_2.1}, the monotone convergence theorem yields
\[
\sum_{n \in \N}\itm(B_n)
=\lim_{y \to +\infty}y\Im \sum_{n \in \N}P(iy,B_n)
=\lim_{y \to +\infty}y\Im P(iy,B^\prime)
=\itm(B^\prime)<+\infty.
\]
Then by assumption and the dominated convergence theorem we have
\begin{align*}
&G[\theta_{B^\prime}](z)
=-P(z,B^\prime)
=-\sum_{n \in \N}P(z,B_n) \\
&=-\sum_{n \in \N}\int_{\R}\frac{1}{x-z}\,\theta_{B_n}(dx)
=-\int_{\R}\frac{1}{x-z}\,\sum_{n \in \N}\theta_{B_n}(dx)
=G\left[\sum_n\theta_{B_n}\right](z). 
\end{align*}
Note that the fourth equality above can be justified, e.g., by \cite[Lemma A.4]{HH22}. 
Hence $\theta_{B^\prime}=\sum_n \theta_{B_n}$ follows from the uniqueness of Cauchy transform.

Now the function $(A,B)\mapsto \theta_B(A)$, $A \in \cB(\R)$, $B \in \cB_{\rm c}(I)$,  is countably additive on each component and the measure $B \mapsto \theta_B(\R)$ is $\sigma$-finite. Then there exists a measure $\varTheta$ on $\cB(\R \times I)$ such that $\varTheta(A\times B) = \theta_B(A)$; see \cite[Proposition 2.4]{RR89} or \cite[Appendixes, Theorem 8.1]{EK86}. 
Then \eqref{i:CMF_Theta} and the latter two equalities in \eqref{eq:mass_of_CMF} follow from the construction of $\varTheta$ and the continuity assumption $P(z,\{t\})=0$ for $t \in I$.

The remaining assertion simply follows from the decomposition of the measure $\varTheta$, which is called \emph{disintegration}, \emph{fibering}, or the existence of \emph{regular conditional (probability) measure}; for instance, see Dudley~\cite[\S10.2]{Dud02} or Stroock~\cite[Theorem~9.2.2]{Str11}.
Indeed, for the marginal measure $\itm(dt):=\varTheta(\R \times dt)$, which obviously extends the set function $\itm$ defined on $\cB_{\rm c}(I)$ in the previous paragraphs, there exist $\itm$-a.e.\ unique regular conditional measures $\nu_t$, $t \in I$, on $\cB(\R)$ that satisfy \eqref{eq:CMF_Theta-nu}.
\end{proof}

We are now in a position to give one of the main objects in this paper, the \emph{Loewner integro-differential equation}, for $\PickC$-DLCs.

\begin{theorem} \label{th:DLC_to_LIE}
Let $(f_t)_{t}$ be a $\PickC$-DLC. 
Then there exists a unique function $P \colon \C^+ \times \cB_{\rm c}(I) \to \overline{\C^+}$ with \textup{\ref{CMF1}}--\textup{\ref{CMF3}} such that the following (chordal) Loewner integro-differential equation holds:
\begin{equation}\label{eq:LIE-additive}
f_t(z)=z+\int_0^t \frac{\partial f_s(z)}{\partial z} \,P(z,ds),\qquad t \in I,\, z \in \C^+. 
\end{equation}
Moreover, let $\varTheta$ be the compound driving measure of $P$ given by \eqref{eq:CMF_Theta} and $(f_{s,t})$ be the REF associated with $(f_t)$. Then the following integral equation holds: 
\begin{equation}\label{eq:LIE2}
    f_{s,t}(z)=z+\int_{\R\times[s,t]} \frac{1}{x-f_{r,t}(z)} \,\varTheta(dx\,dr),\qquad s\le t,\, z \in \C^+.  
\end{equation}
\end{theorem}

\begin{proof}
Define $\ar(t):=\AR(f_t)$.
Then $J:=\ar(I)$ is a left-closed interval with left endpoint $0$.
We introduce the right-continuous inverse
\begin{equation}
\label{eq:reparametrization}
\tau(u)
=\inf\{\, t \in I : \ar(t)>u \,\}
=\sup\{\, t \in I: \ar(t)=u \,\}, \qquad u \in J,
\end{equation}
which is strictly increasing.
Proposition~\ref{prop:converse_to_cap_conti} implies that the time-changed family $(f_{\tau(u)})_{u \in J}$ is a $\PickC$-DLC with $\AR(f_{\tau(u)})=u$.
Hence by Proposition~\ref{th:chordal_LDE} there exists a Herglotz vector field $p^\tau(z,t)$, which has the form \eqref{eq:reduced_HVF} and is normalized, such that
\begin{equation} \label{eq:LDE_for_time-changed_DLC}
f_{\tau(u)}(z)
=z+\int_0^u \frac{\partial f_{\tau(v)}(z)}{\partial z}p^\tau(z,v) \,dv,
\qquad u \in J.
\end{equation}
We set
\begin{equation}\label{eq:Qq}
P(z, B):=\int_{\tau^{-1}(B)}p^\tau(z,v)\,dv,
\qquad z \in \C^+,\ B \in \cB_{\rm c}(I).
\end{equation}
It is easy to see that $P \colon \C^+ \times \cB_{\rm c}(I) \to \overline{\C^+}$ satisfies all conditions \ref{CMF1}--\ref{CMF3}.
Moreover, by the change of variables we have
\begin{equation} \label{eq:preLIE}
f_{\tau(u)}(z)
=z+\int_0^{\tau(u)} \frac{\partial f_s(z)}{\partial z}\,P(z,ds),
\qquad u \in J.
\end{equation}
Here, we notice that $I \setminus \tau(J)$ is the union of left-closed and right-open intervals on which $\ar(t)$ takes constant values.
On such an interval $t \mapsto f_t(z)$ is also constant by Corollary~\ref{cor:constant_DLC}.
In addition, $P(z, B)$ is zero if $B \subset I \setminus \tau(J)$.
Thus, \eqref{eq:preLIE} extends to \eqref{eq:LIE-additive}.

The uniqueness of $P$ follows in almost the same way as the proof of the standard fact that a probability measure on $\R$ is uniquely determined by its distribution function; see Hasebe and Hotta~\cite[Theorem~3.6]{HH22} for further details.

The integral equation \eqref{eq:LIE2} can be proved similarly. To be begin with, integrating  \eqref{eq:chordal_LODE} applied to the REF $(f_{\tau(u),\tau(v)})_{u\le v}$ yields   
\begin{align*} 
f_{\tau(u),\tau(v)}(z) 
&= z + \int_u^v p^\tau (f_{\tau(w),\tau(v)}(z),w) \,dw  \\
&=z+\int_{\R \times [u,v]} \frac{1}{x-f_{\tau(w),\tau(v)}(z)}\, \nu_w^\tau(dx) \,dw,
\qquad u,v \in J,
\end{align*}
where the following representation is employed: 
\begin{equation}\label{eq:ptau}
   p^\tau (z,u) = \int_\R \frac{1}{x-z} \,\nu_u^\tau(dx).  
\end{equation}
Combining \eqref{eq:CMF_Theta}, \eqref{eq:Qq} and \eqref{eq:ptau} shows  that $\varTheta$ is the pushforward of the measure $\nu_w^\tau(dx)\,dw$ by the mapping $\R \times J \ni (x,w)\mapsto (x,\tau(w)) \in \R\times I$. This readily implies, by the change of variables, that  
\begin{align*} 
f_{\tau(u),\tau(v)}(z) 
=z+\int_{\R \times [\tau(u),\tau(v)]} \frac{1}{x-f_{r,\tau(v)}(z)}\, \varTheta(dx \,dr),
\qquad u,v \in J. 
\end{align*}
The remaining proof is similar to the paragraph following \eqref{eq:preLIE} and is omitted. 
\end{proof}

\begin{theorem} \label{th:LIE_to_DLC}
Let $P$, $\varTheta$ and $\itm$ be as in the equivalent conditions in Proposition~\ref{prop:CMF}.
Then there exists a unique $\PickC$-DLC $(f_t)_{t}$  that satisfies \eqref{eq:LIE-additive}.
Moreover, there is a unique REF $(f_{s,t})$ that satisfies \eqref{eq:LIE2}. In addition, $\AR(f_t)=\itm_P([0,t])$ holds for $t \in I$  and $(f_t)$ and $(f_{s,t})$ are associated with each other. 
\end{theorem}

\begin{proof}
Consider the function $\ar(t):=\itm_P([0,t])$, which is continuous as $\itm_P$ is atomless, and the interval $J:=\ar(I)$.
We define $\tau$ as in \eqref{eq:reparametrization}. 
Then the function $P^\tau \colon \C^+ \times \cB_{\rm c}(J) \to \overline{\C^+}$ defined by $P^\tau(z,B)=P(z, \tau(B))$ is countably additive in $B \in \cB_{\rm c}(J)$ because this is the pushforward of the countably additive set function $A \mapsto P(z, A \cap \tau(J))$ by the mapping $\tau^{-1}\colon\tau(J)\to J$.
Hence $P^\tau$ satisfies \ref{CMF1}--\ref{CMF3} and has the expression
\begin{equation} \label{eq:pf_form_of_R}
P^\tau(z,B)=\int_B p^\tau(z,v) \,\itm_{P^\tau}(dv),
\qquad B \in \cB_{\rm c}(J),
\end{equation}
for a normalized Herglotz vector field $p^\tau$ of the form \eqref{eq:reduced_HVF} and Borel measure $\itm_{P^\tau}$ on $J$ by Proposition~\ref{prop:CMF}~\eqref{i:CMF_q}.
Moreover, by \eqref{eq:mass_of_CMF} we have
\[
\itm_{P^\tau}([0,u])
=\lim_{y \to +\infty}y \Im P(iy,\tau([0,u]))
=\itm_P(\tau([0,u]))
=u,
\]
which implies that $\itm_{P^\tau}$ is the Lebesgue measure on $J$.
Then, by Proposition~\ref{th:unique_sol_to_LDE}, the  Loewner differential equation
\begin{equation} \label{eq:pf_R-LDE}
\frac{\partial \tilde{f}_u(z)}{\partial u}=\frac{\partial \tilde{f}_u(z)}{\partial z}p^\tau(z,u)
\end{equation}
generates a $\PickC$-DLC $(\tilde{f}_u)_{u \in J}$ with $\AR(\tilde{f}_u)=u$.
We integrate \eqref{eq:pf_R-LDE} and substitute \eqref{eq:pf_form_of_R} to obtain
\[
\tilde{f}_u(z)=z+\int_0^u \frac{\partial \tilde{f}_v(z)}{\partial z} \,P^\tau(z,dv).
\]
Now we define $f_t:=\tilde{f}_{\ar(t)}$ for $t \in I$.
We then have $\tilde{f}_v=f_{\tau(v)}$, $v \in J$, and
\begin{equation} \label{eq:pf_pre_Q-LIE}
f_t(z)
=z+\int_0^{\ar(t)} \frac{\partial f_{\tau(v)}(z)}{\partial z} \,P^\tau(z,dv)
=z+\int_0^{\tau(\ar(t))} \frac{\partial f_s(z)}{\partial z} \,P(z,ds)
\end{equation}
by the change of variables.
Since
\[
\bigl\lvert P(z, [t,\tau(\ar(t))]) \bigr\rvert \le \frac{\itm_P([t,\tau(\ar(t))])}{\Im z}=0,
\qquad t \in I,
\]
follows from the definition of $\ar$ and $\tau$, we see that \eqref{eq:pf_pre_Q-LIE} yields \eqref{eq:LIE-additive}.

To see the uniqueness, note that any $\PickC$-DLC $(f_t)_{t}$ obeying \eqref{eq:LIE-additive} satisfies $\AR(f_t)=\ar(t)=\itm_P([0,t])$.
We can easily see this, generalizing the argument in Remark~\ref{lem:no_mass_escape} with  \eqref{eq:mass_of_CMF} and \eqref{eq:LIE-additive}.
Then the time-changes performed in the proof of Theorems~\ref{th:DLC_to_LIE} and \ref{th:LIE_to_DLC} are the same one
(indeed, we have deliberately used the same symbols for this reason).
Therefore, via time-change Proposition~\ref{th:unique_sol_to_LDE} yields the uniqueness.

Proposition~\ref{th:unique_sol_to_LDE} implies the ODE 
\[
\frac{\partial \tilde{f}_{u,v}(z)}{\partial u}=- p^\tau(\tilde f_{u,v}(z),u). 
\]
From this we can deduce the existence and uniqueness of a solution to \eqref{eq:LIE2}. The calculations are similar to the corresponding part of the proof of Theorem \ref{th:DLC_to_LIE}.  Finally, Proposition~\ref{th:unique_sol_to_LDE} implies that $(\tilde f_t)$ and $(\tilde f_{s,t})$ are associated with each other and so  are $(f_t)$ and $(f_{s,t})$. 
\end{proof}

By Proposition~\ref{prop:CMF}, Theorems~\ref{th:DLC_to_LIE} and \ref{th:LIE_to_DLC}, $\PickC$-DLCs $(f_t)$, functions $P$ with \ref{CMF1}--\ref{CMF3}, and their compound driving measures $\varTheta$ are in one-to-one correspondence. We say that they are \emph{associated} with each other.

\begin{remark} \label{rem:time-change}  The proofs of Theorems \ref{th:DLC_to_LIE} and \ref{th:LIE_to_DLC} employ a time-change to relax the assumption
 $\AR(f_t)=t$ of Proposition \ref{th:chordal_LDE}.
There is another method, based on generalized differentiation with respect to the ``intrinsic time'' $\ar(t)=\AR(f_t)$.
Suppose that $\ar$ is strictly increasing for simplicity, and let $\itm$ be the atomless Borel measure on $I$ such that $\itm([0,t])=\ar(t)$.
Then the function $t \mapsto f_t(z)$ is of finite variation on each compact subinterval of $I$ by \eqref{eq:dc-abs_conti}, and the complex measure determined by this function is absolutely continuous with respect to $\itm$.
Hence it has a Radon--Nikodym density with respect to $\itm$.
If we write this density as $\partial f_t(z)/\partial \ar(t)$, then the chordal Loewner equation
\begin{equation} \label{eq:LDE_by_density}
\frac{\partial f_t(z)}{\partial \ar(t)}=\frac{\partial f_t(z)}{\partial z}p(z,t)
\qquad \text{for $\itm$-a.e.}\ t \in I
\end{equation}
follows from a proof similar to that of Proposition~\ref{th:chordal_LDE}.
This is equivalent to the Loewner integro-differential equation \eqref{eq:LIE-additive}; indeed, by integrating \eqref{eq:LDE_by_density} with respect to $\itm$, the function $P$ in \eqref{eq:CMF_q} appears.
For such equations involving generalized derivatives as \eqref{eq:LDE_by_density}, we refer the reader to Yanagihara~\cite{Yan19+} and Murayama~\cite{Mur23}; the former derived the radial Loewner equation\footnote{%
Moreover, Yanagihara relaxed the assumption that each element of a Loewner chain is univalent, partially following an earlier work by Pommerenke~\cite{Pom65}.},
and the latter derived a version of the chordal Loewner equation on certain multiply connected domains.
\end{remark}

General $\PickB$-DLCs do not satisfy reasonable integro-differential equations, but $\PickB$-REFs do satisfy an integral equation that is a simple generalization of \eqref{eq:LIE2}. 

\begin{corollary} \label{cor:LIE}
Let $(f_{s,t})$ be a $\PickB$-REF.  Then there exist a continuous function $t\mapsto m_t$ with $m_0=0$ and a Borel measure $\rGen$ on $\R \times I$ with
 \eqref{LK3_theta} such that 
 \begin{equation}\label{eq:LIE3}
    f_{s,t}(z)=z - m_t +m_s + \int_{\R\times[s,t]} \frac{1}{x-f_{r,t}(z)} \,\rGen(dx\,dr),\qquad s\le t,\, z \in \C^+.  
 \end{equation}
 The pair $((m_t)_t, \rGen)$ is unique. Conversely, given such a pair $((m_t)_t, \rGen)$, there exists a unique $\PickB$-REF $(f_{s,t})$ that solves the integral equation \eqref{eq:LIE3}. 
\end{corollary}
\begin{proof}  Let $m_t:=\mean(f_t)$. Then $\mean(f_{s,t}) = m_t - m_s$ by Lemma \ref{lem:quasi-chordal_DLC_to_REF}. 
According to the calculations in the paragraph before Proposition  \ref{prop:Cauchy_P_prime}, the mappings $g_{s,t}(z) := -m_s+f_{s,t}(z+m_t)$ form  a $\PickC$-REF.  
Therefore $(g_{s,t})$ satisfies an equation of the form \eqref{eq:LIE2}: 
 \begin{equation}\label{eq:LIE4}
    g_{s,t}(z)=z  + \int_{\R\times[s,t]} \frac{1}{x-g_{r,t}(z)} \,\varSigma(dx\,dr),\qquad s\le t,\, z \in \C^+.  
 \end{equation}
 Setting $\rGen$ to be the push-forward of $\varSigma$ by the mapping $(x,t)\mapsto (x+m_t,t)$, one sees that \eqref{eq:LIE4} implies \eqref{eq:LIE3}. The uniqueness of $(m_t)$ follows from the asymptotic behavior $f_{s,t}(z) = z - m_t +m_s + O(z^{-1})$ as $z\to\infty$ nontangentially.  The uniqueness of $\rGen$ follows from the uniqueness of  $\varSigma$, which in turn follows from the uniqueness of $P$ in Theorem \ref{th:DLC_to_LIE}. 
 The converse statement can be proved by tracing the above proof backwards. 
\end{proof}

We can easily obtain versions of Theorems \ref{th:DLC_to_LIE}, \ref{th:LIE_to_DLC} and Corollary \ref{cor:LIE} for EFs by time-reversal.
For example, we have the next corollary.

\begin{corollary} \label{cor:LIE2}
Let $(\varphi_{s,t})$ be a $\PickB$-EF.  Then there exist a continuous function $t\mapsto m_t$ with $m_0=0$ and a Borel measure $\rGen$ on $\R \times I$ with
 \eqref{LK3_theta} such that 
 \begin{equation}\label{eq:LIE5}
    \varphi_{s,t}(z)=z - m_t +m_s + \int_{\R\times[s,t]} \frac{1}{x-\varphi_{s,u}(z)} \,\rGen(dx\,du),\qquad s\le t,\, z \in \C^+.  
 \end{equation}
 The pair $((m_t)_t, \rGen)$ is unique. Conversely, given such a pair $((m_t)_t, \rGen)$, there exists a unique $\PickB$-EF $(\varphi_{s,t})$ that solves the integral equation \eqref{eq:LIE5}. 
\end{corollary}

\begin{proof}
In the case $I=[0,T]$, the time-reversal $(f_{s,t}) \mapsto (f_{T-t,T-s})$ gives a bijection between $\PickB$-REFs and $\PickB$-EFs; hence Corollary \ref{cor:LIE} applies.
If $I=[0,T)$ with $0<T\le \infty$, then one can work first with compact subintervals of $I$ and then extend to the whole $I$. 
\end{proof}

\subsection{Reduced generators of monotone CHs} \label{subsec:reducedGF}

We return to $\rhd$-CHs.
Let $(\monotone_{s,t})_{s\le t}$ be a $\rhd$-CH with $\int_{\R} x^2 \,\monotone_{s,t}(dx)<+\infty$ for all $s\le t$. The associated $\PickB$-REF $(f_{s,t}):=(F[\monotone_{s,t}])$ satisfies the integral equation in \eqref{eq:LIE3} in which the pair  $((m_t)_t,\allowbreak \rGen)$ appears. 
\begin{definition}[reduced generator of a $\rhd$-CH]
\label{def:monotone_MGF}
We call both the pairs $((m_t),\allowbreak \rGen)$ and $((m_t),\allowbreak (\rgen_t))$, where $\rgen_t(\cdot):=\rGen(\cdot \times [0,t])$, the \emph{reduced generators} of the $\rhd$-CH $(\monotone_{s,t})_{s\le t}$ and also of the associated $\PickB$-DLC $(f_t)_t$ and the REF $(f_{s,t})_{s\le t}$.
\end{definition}

Identifying the reduced generator of a $\ast$-CH (Definition~\ref{def:classical_MGF}) with the one in Definition~\ref{def:monotone_MGF}, we define the bijection from the set of all $\ast$-CHs of finite second moment to the set of all $\rhd$-CHs of finite second moment.

Note that, in Definition~\ref{def:monotone_MGF}, $m_t$ is the mean and $\ar(t)=\rgen_t(\R)$ is the variance of $\monotone_{0,t}$.
It is also worth noting that $(\rgen_t)$ has setwise continuity as in Section~\ref{sec:classical_CH}. In addition, a $\lhd$-CH $(\check\monotone_{s,t})_{s\le t}$ with finite second moment is similarly characterized by the same pair $((m_t)_t,\allowbreak \rGen)$ through Corollary \ref{cor:LIE2} for $\varphi_{s,t}:=F[\check\monotone_{s,t}]$. 

The $\PickB$-DLC associated with a reduced generator $((m_t), \rGen)$ does not satisfy an integro-differential equation in general (see Remark \ref{rem:ill} below), but a shifted DLC does. 
Let $\ar(t):=\AR(f_t)$ and $m_t:=\mean(f_t)$.
The family of inner shifts $g_t(z):=f_t(z+m_t)$, $t \in I$, then form a $\PickC$-DLC as we have seen in the proof of Corollary \ref{cor:LIE}. 
By Theorem~\ref{th:DLC_to_LIE}, there exists a unique function $Q \colon \C^+ \times \cB_{\rm c}(I) \to \overline{\C^+}$ such that 
\begin{equation}\label{eq:LIE-additive_g}
g_t(z)=z+\int_0^t \frac{\partial g_s(z)}{\partial z} \,Q(z,ds),\qquad t \in I,\, z \in \C^+. 
\end{equation}
Let $\varSigma$ be the the compound driving measure of $Q$ (see  Proposition~\ref{prop:CMF}~\eqref{i:CMF_Theta}), i.e., it is determined by 
\[
Q(z,B) = \int_\R \frac{1}{x-z} \,\varSigma(dx \times B), \qquad B\in \cB_{\rm c}(I),\ z\in \C^+. 
\]
The pushforward of $\varSigma$ by the mapping $(x,t)\mapsto (x+m_t,t)$ coincides with $\rGen$.

\begin{example}[Absolutely continuous case] \label{rem:AC} Suppose that $(f_t)$ is a $\PickB$-DLC that has a reduced generator $(m_t,\rgen_t)_t$ of the form
\[
m_t = \int_0^t \dot{m}_s \,ds \qquad  \text{and}\qquad  \rgen_t = \int_0^t \dot{\rgen}_s\, ds, 
\] 
where $t\mapsto(\dot{m}_t, \dot{\rgen}_t) \in\R\times \mea$ is such that $t\mapsto \dot{m}_t$ and $t\mapsto \dot{\rgen}_t(B)$ are locally integrable for all $B\in \cB(\R)$.
This means that $\rGen(dx\,dt)=\dot{\rgen}_t(dx)\, dt$, so that \eqref{eq:LIE-additive_g} reads  
\begin{align*}
g_t(z) 
&= z+  \int_{\R \times [0,t]} \frac{\partial g_s}{\partial z}(z) \frac{1}{x-z}\, \varSigma(dx\,ds) =  z+  \int_{\R \times [0,t]} \frac{\partial g_s}{\partial z}(z) \frac{1}{x-m_s-z}\, \rGen(dx\,ds) \\
&=   z+  \int_0^t\left[\int_{\R} \frac{\partial g_s}{\partial z}(z) \frac{1}{x-m_s-z}\, \dot{\rgen}_s(dx)\right] ds, 
\end{align*}
or equivalently, for almost every $t \in I$ and all $z\in \C^+$, 
\[
\partial_t g_t(z) =  q(z,t)\partial_z g_t(z), \quad 
\text{where} \quad  
q(z,t) := \int_{\R} \frac{1}{x-m_t-z}\, \dot{\rgen}_t(dx).   
\]
This can be translated into the Loewner differential equation for $f_t$ of the form 
\begin{equation}\label{eq:LIEf}
\partial_t f_{t}(z) =  p(z,t)\partial_z f_{t}(z), 
\quad \text{where} \quad 
p(z,t) = -\dot{m}_t +\int_\R \frac{1}{x-z}\, \dot{\rgen}_t(dx).  
\end{equation}
 On the other hand, the integral equation in Corollary \ref{cor:LIE} for the $\PickA$-REF $(f_{s,t})$ associated to $(f_t)$ can be directly differentiated: 
\[
\partial_s f_{s,t}(z) = -p(f_{s,t}(z),s). 
\]
This is exactly the non-autonomous generalization of the time-homogeneous case in the present context; cf.~Proposition \ref{thm:MLK}.  
\end{example}

\begin{remark}\label{rem:ill}
If $(f_t)$ is a general $\PickB$-DLC, then the product $\dot{m_t}\partial_z f_{t}(z)$ that appears in \eqref{eq:LIEf} seems ill-defined even as a distribution.
If we assume some regularity, e.g., assume that $t\mapsto m_t$ is of bounded variation on each compact subinterval, then we can safely establish an integro-differential equation; cf.~Remark~\ref{rem:LIDE_general}. 
\end{remark}

\subsection{Convergence of Loewner chains and reduced generators}
\label{subsec:conv_Loewner_rgf}

We have seen in Theorem~\ref{th:conv_classical_SMB} that the correspondence between $\ast$-CHs and reduced generators is bi-continuous in a sense under the assumption that the second moments involved are uniformly bounded.
Here the same property is proved for $\rhd$-CHs via Theorem~\ref{th:conv_monotone}, in which we have seen that the sequences of associated $\rhd$-CHs, $\PickA$-DLCs, and $\PickA$-REFs are all convergent if one of them is so.
Actually what we are about to show is that the correspondence between $\PickB$-DLCs and and reduced generators is bi-continuous provided that the angular residues at infinity are uniformly bounded.

\begin{theorem} \label{thm:main_conv}
Let $(\monotone_{s,t})_{s \le t}$ and $(\monotone^n_{s,t})_{s \le t}$, $n \in \N$, be $\rhd$-CHs with finite second moment.
Let $(f_t)_t$ (resp.\ $(f^n_t)_t$) be the $\PickB$-DLC and $(m_t, \rgen_t)_{t}$ (resp.\ $(m^n_t, \rgen^n_t)_{t}$) the reduced generator of $(\monotone_{s,t})$ (resp.\ $(\monotone^n_{s,t})$).
Suppose that, for each $t \in I$,
\begin{equation} \label{eq:mono_var_bound}
\sup_{n \in \N}\var(\monotone^n_{0,t})=\sup_{n \in \N}\AR(f^n_t)<+\infty.
\end{equation} 
Then the equivalent conditions \eqref{M1}--\eqref{M4} in Theorem~\ref{th:conv_monotone} are satisfied if and only if
\begin{enumerate}
\setcounter{enumi}{4}
\item \label{M5}
$(m^n_t)_{t}$ converges to $(m_t)_{t}$ locally uniformly on $I$, and $(\rgen^n_t)_{t}$ converges vaguely to $(\rgen_t)_{t}$ locally uniformly on $I$ as $n \to \infty$.
\end{enumerate}
\end{theorem}

\begin{proof}
We put $g_t(z):=f_t(z+m_t)$ and $g^n_t(z):=f^n_t(z+m^n_t)$.
In the sequel, $\varSigma$ denotes the unique $\sigma$-finite Borel measure on $\R \times I$ such that
\begin{equation} \label{eq:LIE_by_Theta}
g_t(z)=z+\int_{\R \times [0,t]}\frac{\partial g_s(z)}{\partial z}\cdot \frac{1}{x-z} \,\varSigma(dx\,ds),
\qquad t \in I,
\end{equation}
whose existence is ensured by Theorem~\ref{th:DLC_to_LIE} and Proposition~\ref{prop:CMF}.
Let $\sigma_t(dx):=\varSigma(dx \times [0,t])$.
$\varSigma^n$ and $\sigma^n_t$ are also defined in the same way with $g_t$ replaced by $g^n_t$.
By definition, we have $\varSigma^n(\R \times [0,t])=\AR(f^n_t)$.

\smallskip\noindent
\textit{``Only if'' part}.
Assume \eqref{M3} in Theorem~\ref{th:conv_monotone}, the convergence of the $\PickB$-DLCs $(f^n_t)$.
From \eqref{eq:mono_var_bound} we can derive the following two consequences in the same way as in the proof of Theorem~\ref{th:capacity_conti_prime}:
\begin{itemize}
\item $\lim_{n \to \infty}\sup_{t \in [0,T]}\lvert m^n_t-m_t \rvert=0$ for each $T\in I$;
\item $\rgen^n_t$ converges vaguely to $\rgen_t$ if $\sigma^n_t$ converges vaguely to $\sigma_t$.
\end{itemize}
Moreover, for each $T\in I$, the sequence $(\mathbf{1}_{\R \times [0,T]}\varSigma^n)_n$ of measures is bounded in total mass by \eqref{eq:mono_var_bound}.
Hence it has a subsequence which converges vaguely to some measure $\varSigma^\prime$ by Proposition~\ref{prop:Alaoglu_for_vague}.
By passage to the limit $g^n_t \to g_t$ along this subsequence, we obtain \eqref{eq:LIE_by_Theta} with $\varSigma$ replaced by $\varSigma^\prime$ for $t \in [0,T]$.
By the uniqueness of $P$ in Theorem~\ref{th:DLC_to_LIE}, this implies $\varSigma^\prime=\mathbf{1}_{\R \times [0,T]}\varSigma$.
Hence the original sequence $(\mathbf{1}_{\R \times [0,T]}\varSigma^n)_n$ converges vaguely to $\mathbf{1}_{\R \times [0,T]}\varSigma$.
On account of the definition of vague convergence, this convergence is true even if the indicator $\mathbf{1}_{\R \times [0,T]}$ is removed.
Then we have $\sigma^n_t \xrightarrow{\rm v} \sigma_t$ by Proposition~\ref{prop:two_thetas}, which means \eqref{M5}.

\smallskip\noindent
\textit{``If'' part}.
Let us assume \eqref{M5} to derive \eqref{M3} in Theorem~\ref{th:conv_monotone}.
As is easily seen, it suffices to prove that $(g^n_t)_{t}$ converges to $(g_t)_{t}$ locally uniformly on $\C^+ \times I$.
As in the proof of ``$\rm (M5)\Rightarrow (M1)$'' in Hasebe and Hotta~\cite[Theorem 3.15]{HH22}, we apply the Arzel\`a--Ascoli theorem to the sequence of functions $(z,t) \mapsto g^n_t(z)$, $n \in \N$.
The local boundedness follows immediately from \eqref{eq:mono_var_bound} and Proposition~\ref{prop:class_P}~\eqref{HN:ineq};
so we observe the equicontinuity. Let $K \subset \C^+$ be any compact subset and $T\in I$. 
In the triangle inequality
\[
\lvert g^n_t(z)-g^n_s(w) \rvert
\le \lvert g^n_t(z)-g^n_t(w) \rvert + \lvert g^n_t(w)-g^n_s(w) \rvert,
\]
the first term on the right-hand side can be dealt with Cauchy's integral formula in the same way as the unit disk case \cite[Eq.\ (3.37)]{HH22}.
For the second term, the integro-differential equation yields 
\[
\lvert g^n_t(w)-g^n_s(w) \rvert
\le \int_{\R \times (s,t]}\frac{\lvert (g^n_u)^\prime(w) \rvert}{\lvert x-w \rvert} \,\varSigma^n(dx \,du) \le C_{K,T} \int_\R \frac{\varSigma^n(dx \times(s,t])}{|x|+1}, 
\]
where the constant 
\[
C_{K,T} := \sup_{\substack{w \in K, u\in [0,T] \\ x\in \R, n\in\N}}\frac{\lvert (g^n_u)^\prime(w) \rvert (1+|x|)}{|x-w|}
\] is finite due to Proposition \ref{prop:class_P} \eqref{HN:ineq}, assumption \eqref{eq:mono_var_bound} and Cauchy's integral formula for $(g_u^n)'(w)$. Since $(\rgen_t^n)$ converges vaguely to $(\rgen_t)$ uniformly on $[0,T]$ (recall that one can use test functions in $C_\infty(\R)$ thanks to assumption \eqref{eq:mono_var_bound}), we deduce 
\[
\lim_{\substack{0\le s \le t \le T\\  t-s \to 0}}\sup_{n\in\N}\int_\R \frac{\varSigma^n(dx \times(s,t])}{|x|+1} 
= \lim_{\substack{0\le s \le t \le T\\  t-s \to 0}}\sup_{n\in\N} \int_\R \frac{\rgen_t^n(dx) - \rgen_s^n(dx)}{|x|+1} =0,  
\]
verifying the equicontinuity. We can therefore extract a subsequence from $(g_t^n), n\in\N$ that converges  to a function, say $\tilde{g}_t(z)$, locally uniformly on $\C^+ \times I$. The uniqueness of limits follows from Theorem \ref{th:LIE_to_DLC} and the fact that $\tilde g$ fulfills \eqref{eq:LIE-additive_g} since 
we can pass to the limit in 
\[
g_t^n(z)=z+\int_{\R\times [0,t]} \frac{\partial g_s^n(z)}{\partial z} \cdot \frac{1}{x-z} \,\varSigma^n(dx\,ds),\qquad t \in [0,T],\, z \in \C^+. 
\]
This verifies the convergence $g_t^n(z)\to g_t(z)$ and hence $f_t^n(z)\to f_t(z)$. 
\end{proof}

We note that, instead of assuming that the second moments involved are uniformly bounded, one can include the convergence of the moments into the equivalent conditions that have been considered up to this point, as in the following corollary:

\begin{corollary} \label{cor:moment_conv}
Let the symbols in this corollary be as in Theorems~\ref{th:conv_classical_SMB} and \ref{thm:main_conv}.
In particular, all the CHs are assumed to have finite second moment.
Then the following are equivalent as $n\to \infty$: 
\begin{enumerate}
\item\label{C2_strong}
$(\classical_{s,t}^n) \stackrel{\rm w}{\to} (\classical_{s,t})$ locally uniformly and $\var(\classical_{s,t}^n) \to \var(\classical_{s,t})$ locally uniformly on $I^2_\le$;  
\item\label{M2_strong}
$(\monotone_{s,t}^n) \stackrel{\rm w}{\to} (\monotone_{s,t})$ locally uniformly and $\var(\monotone_{s,t}^n) \to \var(\monotone_{s,t})$ locally uniformly on $I^2_\le$; 
\item\label{M3_strong}
$(f^n_t(z)) \to (f_t(z))$ locally uniformly in $(z,t) \in \C^+ \times I$ and $\AR(f^n_t)\to \AR(f_t)$ locally uniformly on $I$;
\item\label{M4_strong}
$(f^n_{s,t}(z)) \to (f_{s,t}(z))$ locally uniformly in $(z,s,t) \in \C^+ \times I^2_\le$ and $\AR(f^n_{s,t}) \to \AR(f_{s,t})$ locally uniformly on $I^2_\le$;
\item\label{M5_strong}
$(m^n_t) \to (m_t)$ locally uniformly and $(\rgen^n_t) \stackrel{\rm w}{\to} (\rgen_t)$ locally uniformly on $I$.
\end{enumerate}
\end{corollary}

\begin{proof}
Note that the following relation, which has appeared repeatedly before:
\[
\kappa_t(\R)-\kappa_s(\R)=\var(\classical_{s,t})=\var(\monotone_{s,t})=\AR(f_t)-\AR(f_s)=\AR(f_{s,t}).
\]
From this it is easy to see that \eqref{C2_strong}--\eqref{M4_strong} are equivalent and implied by \eqref{M5_strong}.
Conversely, if we assume one of \eqref{C2_strong}--\eqref{M4_strong}, then by Theorem~\ref{thm:main_conv} $(m^n_t)_{t}$ converges to $(m_t)_{t}$ locally uniformly and $(\rgen^n_t)_{t}$ converges vaguely to $(\rgen_t)_{t}$ locally uniformly on $I$.
Moreover, by Proposition~\ref{prop:vague_to_weak} and the relation above, $\rgen^n_t$ converges weakly to $\rgen_t$ at every $t\in I$.
Thus, we can apply Proposition~\ref{prop:two_sigmas} to conclude that $(\rgen^n_t)_{t}$ converges weakly to $(\rgen_t)_{t}$ locally uniformly on $I$.
\end{proof}

The sets of $\ast$- and $\rhd$-CHs with finite second moment, subsets of $C(I^2_\le; \prob)$, can be endowed with the topology induced by the metric
\begin{equation} \label{eq:2-luw_metric}
\tilde \rho((\mu_{s,t})_{s\le t}, (\nu_{s,t})_{s\le t}) := \sum_{j\ge1}\frac1{2^j}\left[1 \wedge \sup_{0\le s\le t \le T_j} (\rho+\mathcal{V})(\mu_{s,t}, \nu_{s,t})\right].
\end{equation}
Here, $\{T_j\}_{j\ge1} \subset I$ is any sequence with $T_j\uparrow \sup I$, $\rho$ is any metric on $\prob$ compatible with weak convergence, e.g., the L\'evy distance, and $\mathcal{V}(\mu, \nu):=\lvert \Mean(\mu)-\Mean(\nu) \rvert+\lvert \var(\mu)-\var(\nu) \rvert$.
As is easily seen, conditions~\eqref{C2_strong} and \eqref{M2_strong} of Corollary~\ref{cor:moment_conv} are the sequential convergence of $(\classical^n_{s,t})$ and $(\monotone^n_{s,t})$, respectively, in this topology; see Appendix~\ref{sec:topology_of_luwc}.
As a result, our bijection, restricted on the sets of $\ast$- and $\rhd$-CHs with finite second moment, is homeomorphic in the $\tilde\rho$-topology.
The metric $\rho+\mathcal{V}$ in \eqref{eq:2-luw_metric} can be replaced with another metric such as $2$-Wasserstein distance, which is well known in optimal transport theory; for such metrics, see Villani~\cite[Chapter~6]{Vil09} and references therein.
Corollary \ref{cor:moment_conv} also implies that the set of $\rhd$-CHs with finite second moment is homeomorphic to the set of $\PickB$-DLCs endowed with the metric 
\begin{align*}
&\hat\rho((f_t)_{t}, (g_t)_{t}) \\
&:=\sum_{j\ge1}\frac1{2^j}\left\{ 1 \wedge \left[ \sup_{(z,t) \in K_j}\lvert f_t(z) - g_t(z) \rvert + \sup_{0 \le t \le T_j}\bigl( \lvert \mean(f_t)-\mean(g_t) \rvert + \lvert \AR(f_t)- \AR(g_t) \rvert \bigr) \right] \right\},  
\end{align*}
where $\{K_j\}_{j\ge1}$ is any sequence of compact subsets of $\C^+ \times I$ such that $K_j \uparrow \C^+ \times I$.

\subsection{Infinitesimal growth of moments under the bijection} \label{subsec:generator_moments}

In this subsection, we shall observe a probabilistic aspect of our bijection by a rather direct calculation of moments, putting assumptions which are stronger than those in Example~\ref{rem:AC}. 
Let $\Mean_n(\mu)$ denote the $n$th moment of $\mu \in \mea$ (if it exists) for $n\ge0$, i.e., 
\[
\Mean_n(\mu) = \int_\R x^n \,\mu(dx).  
\]

\begin{proposition}\label{prop:moment_growth}
Let $(\classical_{s,t})$ and $(\monotone_{s,t})$ the $\ast$- and $\rhd$-CHs having a common reduced generator $(m_t,\rgen_t)_{t}$.
Assume that 
\begin{itemize}
\item the mapping $t\mapsto m_t$ is absolutely continuous and the derivative is denoted by $\dot{m}_t$,
\item for every $T\in I$, the supports of $\rgen_t$, $0\le t\le T$, are contained in a common compact set $K_T \subset \R$ (cf.~Proposition \ref{thm:cpt_supp}), and 
\item $(\rgen_t)_{t}$ is absolutely continuous with respect to $t$ in the sense that 
\[
\rgen_t(B) = \int_0^t \dot{\rgen}_s(B)\, ds, \qquad t\in I,~ B \in \cB(\R),
\] 
for some measurable family of finite Borel measures $(\dot{\rgen}_t)_t$, which means $\rGen(dx\,dt)=\dot{\rgen}_t(dx)\,dt$. Note that $\frac{d}{dt}\Mean_n(\rgen_t) = \Mean_n(\dot{\rgen}_t)$ a.e.\ $t\in I$, $n\ge0$.
\end{itemize}
Then we have 
\begin{equation} \label{eq:gen_moment}
 \left. \frac{\partial}{\partial t}\right|_{t=s} \Mean_n(\classical_{s,t})=  \left. \frac{\partial}{\partial t}\right|_{t=s}\Mean_n(\monotone_{s,t}), \qquad n \in \N,~ \text{a.e.}~s\in I,   
\end{equation}
i.e., the bijection is defined so that the infinitesimal generators of all moments coincide. 
\end{proposition}
\begin{proof}
We begin with computing the left-hand side of \eqref{eq:gen_moment}, which is easier.
The L\'evy--Khintchine representation \eqref{eq:CLK_reduced_intro} yields, with convention $\Mean_{-1}(\rgen_t -\rgen_s):= m_t- m_s$, 
\begin{align*}
\widehat\classical_{s,t}(\xi) 
&=\exp\left(i (m_t-m_s) \xi + \int_{\R}\frac{e^{i\xi x}-1- i\xi x }{x^2} \, (\rgen_t-\rgen_s)(dx)\right) \\
&= \exp\left(\sum_{n=1}^\infty \frac{(i\xi)^n}{n!}\Mean_{n-2}(\rgen_t-\rgen_s)\right), 
\end{align*}
so that we have 
\begin{align*}
 \Mean_n(\classical_{s,t}) &= \frac1{i^n} {\widehat \classical_{s,t}}^{\,\,(n)} (0) \\
 &= 
 \begin{cases} \Mean_{-1}(\rgen_t -\rgen_s), & n =1, \\ 
 \Mean_{n-2}(\rgen_t -\rgen_s) + P_n(\Mean_{-1}(\rgen_t -\rgen_s), \dots, \Mean_{n-3}(\rgen_t -\rgen_s)), & n\ge2,   \\
  \end{cases}
\end{align*}
where $P_n$ is a polynomial without constant term or linear term. This implies  
\[
 \left.\frac{\partial}{\partial t} \right|_{t=s} \Mean_n(\classical_{s,t}) = \Mean_{n-2}(\dot{\rgen}_s), \qquad n\ge1,
\]
with convention $\Mean_{-1}(\dot{\rgen}_s):= \dot{m}_s$.

On the monotone side, let $f_{s,t}:= F[\monotone_{s,t}]=f_s^{-1}\circ f_t$.
From calculations in Example~\ref{rem:AC}, it follows that 
\[
\partial_t f_{s,t}(z) =  p(z,t)\partial_z f_{s,t}(z), 
\]
where
\[
p(z,t) := -\dot{m}_t+ q (z-m_t,t) = -\dot{m}_t +\int_\R \frac{1}{x-z}\, \dot{\rgen}_t(dx).  
\]
Furthermore, with the shorthand notation $G_{s,t}=G[\monotone_{s,t}],$ the relation $f_{s,t}(z) = 1/G_{s,t}(z)$ yields 
\begin{equation*}\label{eq:LDE_G}
\partial_t G_{s,t} (z)  = p(z,t) \partial_z G_{s,t}(z), \qquad s \le t.  
\end{equation*}
Combining this with the series expansions 
\[
G_{s,t}(z) = \sum_{n\ge0} \frac{\Mean_n(\monotone_{s,t})}{z^{n+1}}, \qquad  p(z,t) = -\dot{m}_t  -  \sum_{n\ge0}\frac{\Mean_n(\dot{\rgen}_t)} {z^{n+1}},
\]
we obtain 
\[
\sum_{n\ge0} \frac{1}{z^{n+1}} \left.\frac{\partial}{\partial t} \right|_{t=s} \Mean_n(\monotone_{s,t}) = \left( -\dot{m}_s  -  \sum_{n\ge0}\frac{\Mean_n(\dot{\rgen}_s)} {z^{n+1}} \right) \left( -\frac1{z^2}\right),
\]
and so 
\[
 \left.\frac{\partial}{\partial t} \right|_{t=s}\Mean_n(\monotone_{s,t})  =   \Mean_{n-2}(\dot{\rgen}_s),  \qquad n\ge 1, 
\]
as desired. 
\end{proof}

\begin{remark}
The same conclusion holds for $\lhd$-CHs $(\check\monotone_{s,t})$: starting from \eqref{eq:LIE5} with $\varphi_{s,t} := F[\check\monotone_{s,t}]$, computing $\partial_t \varphi_{s,t}|_{t=s}$ leads to the formula 
\[
 \left.\frac{\partial}{\partial t} \right|_{t=s}\Mean_n(\check\monotone_{s,t})  =   \Mean_{n-2}(\dot{\rgen}_s),  \qquad n\ge 1.
\]
\end{remark}

\section{Main results: the general case}
\label{sec:general_case}

In the case of finite second moment in Section~\ref{sec:finite_2nd_case}, the mean and variance play the roles of shift and time-change, respectively, to reduce the problem to differential equations.
In the general case, we lose such a canonical choice; nevertheless, suitable shift and time-change allow us to reduce a general $\PickA$-REF to a (noncanonical) absolutely continuous REF.

\subsection{Preliminaries}
\label{sec:modern_prel}

To implement the reduction of general $\PickA$-REFs, we need recent results on (R)EFs.

\begin{definition}[absolutely continuous evolution family] \label{def:EF}
Let $D \subsetneq \C$ be a simply connected domain. 
A two-parameter family $(\varphi_{s,t})_{s\le t}$ of holomorphic self-mappings $\varphi_{s,t} \colon D \to D$ is called  an \emph{absolutely continuous EF} if it satisfies \textup{\ref{EF1}}, \textup{\ref{EF2}} in Remark \ref{rem:reversal} and the following condition: 
\begin{enumerate}
\setlength{\leftskip}{0.5em}
\renewcommand{\theenumi}{(EF\arabic{enumi}')}
\renewcommand{\labelenumi}{\textup{\theenumi}}
\setcounter{enumi}{2}
\item \label{EF3'}
for every $z \in D$ there exists a nonnegative $k_z \in L^1_{\rm loc}(I)$ such that 
\[
\lvert \varphi_{s,u}(z)-\varphi_{s,t}(z) \rvert \le \int_t^u k_z(r)\,dr,\qquad s \le t  \le u.  
\]
\end{enumerate}

\end{definition}

\begin{definition}[absolutely continuous reverse evolution family]
In the same way as Definition \ref{def:EF}, we say that a two-parameter family $(f_{s,t})_{s\le t}$ of holomorphic self-mappings $f_{s,t} \colon D \to D$ is an \textit{absolutely continuous REF} if it satisfies \textup{\ref{TM1}}, \textup{\ref{TM2}} in Definition~\ref{def:REF}, and 
\begin{enumerate}
\setlength{\leftskip}{1.5em}
\renewcommand{\theenumi}{(REF\arabic{enumi}')}
\renewcommand{\labelenumi}{\textup{\theenumi}}
\setcounter{enumi}{2}
\item \label{TM3'} for every $z \in D$ there exists a nonnegative $k_z \in L^1_{\rm loc}(I)$ such that 
\[
\lvert f_{s,u}(z)- f_{s,t}(z) \rvert \le \int_t^u k_z(r)\,dr,\qquad s \le t  \le u.  
\]
\end{enumerate} 
\end{definition}

\begin{remark}\label{rem:ACEF_REF}
Let $I=[0,T]$ with $T < \infty$.
It was proved by Contreras, D\'{\i}az-Madrigal and Gumenyuk \cite[Proposition 4.3]{CDMG14} that a two-parameter family $(\varphi_{s,t})$ of holomorphic self-mappings $\varphi_{s,t} \colon D \to D$ is an absolutely continuous EF if and only if the family $(f_{s,t})$ defined by $f_{s,t}=\varphi_{T-t,T-s}$ is an absolutely continuous REF.
This is not trivial since a mere time-reversal in \textup{\ref{EF3'}} does not imply \textup{\ref{TM3'}}.
\end{remark}

\begin{remark}
By Bracci, Contreras and D\'{\i}az-Madrigal~\cite[Proposition 3.5]{BCDM12} an absolutely continuous EF is an EF, i.e., condition \textup{\ref{EF3}} holds.
Combined with Remark~\ref{rem:ACEF_REF}, an absolutely continuous REF is an REF.
\end{remark}

There are two more key results from the literature. 
One is the following result of Gumenyuk, Hasebe and P\'erez~\cite[Theorem 5]{GHP22+}.
We also provide a rather self-contained proof suggested by a referee in Appendix \ref{app:GHP}, which actually generalizes this result. 

\begin{lemma} \label{TH_diff-one-trajectory}
Suppose that $(\varphi_{s,t})$ is a family of holomorphic self-mappings of $\C^+$ with the properties \textup{\ref{EF1}} and \textup{\ref{EF2}}, and each $\varphi_{s,t}$ has Denjoy--Wolff fixed point at $\infty$.
Let $z_0\in \C^+$. 
If the map $t \mapsto \varphi_{0,t}(z_0)$ is absolutely continuous on each compact subinterval, then $(\varphi_{s,t})$ is an absolutely continuous EF.
\end{lemma}

\begin{remark}\label{rem:DW}
For the definition of Denjoy--Wolff fixed point, we refer the reader to, e.g., a book of Bracci, Contreras and D\'{\i}az-Madrigal~\cite[Definition~1.8.5]{BCDM20}. 
In the context of Lemma~\ref{TH_diff-one-trajectory}, a holomorphic self-mapping $\varphi \colon \C^+ \to \C^+$ has Denjoy--Wolff fixed point at $\infty$ if and only if the angular derivative at infinity (see Section \ref{sec:PN_to_Cauchy}) satisfies $ \varphi'(\infty) \in [1,+\infty)$; see the same book~\cite[Remark~1.9.7]{BCDM20} and also the computations in Appendix~\ref{sec:anglim_at_infty}.  
In particular, this condition is satisfied by every element $F \in \PickA$, i.e., every reciprocal Cauchy transform $F=F_\mu$ ($\mu \in \prob$), due to the property $F'(\infty)=1$. 
\end{remark}

 Another key result is due to Franz, Hasebe and Schleissinger \cite[Remark 3.11, Proposition 3.12]{FHS20}. We consider Herglotz vector fields $p\colon \C^+\times I \to \overline{\C^+}$ of the form
 \begin{equation}\label{eq:HVFg}
p(z,t) = -\dot\gamma_t + \int_\R \frac{1+xz}{x-z} \,\dot\eta_t(dx), 
\end{equation}
where $\dot\gamma_t \in \R$ and $\dot\eta_t$ is a finite Borel measure on $\R$ for all $t\in I$ such that the functions $t\mapsto\dot\gamma_t$ and $t\mapsto \int_\R f(x)\,\dot\eta_t(dx)$ are in $L^1_{\rm loc}(I)$ for each $f \in C_{\rm b}(\R)$.  

\begin{lemma} \label{lem:absPREF}
Let $(f_{s,t})$ be an absolutely continuous $\PickA$-REF. Then there exists a unique Herglotz vector field $p$ of the form \eqref{eq:HVFg} such that for each $z \in \C^+$ and $t\in I$ the equation 
\begin{equation}\label{eq:mLDE}
\frac{\partial }{\partial s} f_{s,t}(z) = -p(f_{s,t}(z),s), \qquad \text{a.e.~}s \in [0,t] 
\end{equation}
holds. Conversely, given a Herglotz vector field $p$ of the form \eqref{eq:HVFg}, a unique solution $(f_{s,t})$ to \eqref{eq:mLDE} exists and is an absolutely continuous $\PickA$-REF. 
\end{lemma}

\begin{remark}
The uniqueness of $p$ is not stated in the cited paper~\cite{FHS20}, but it follows from Contreras, D\'{\i}az-Madrigal and Gumenyuk~\cite[Theorem~4.2~(i)]{CDMG14}. 
\end{remark}

\subsection{The integral equation and generator}
\label{sec:LIE_for_P-REF}

We begin with reduction of $\PickA$-REFs to absolutely continuous ones. 

\begin{proposition} \label{thm:AC}
Let $(f_{s,t})_{0 \le s \le t \le T}$ be a $\PickA$-REF with $0<T<+\infty$.
Set $a_{s,t}:= \Re[f_{s,t}(i)]$, $L_t(z):= z + a_{t,T}$, and $h_{s,t} := L_s^{-1} \circ f_{s,t} \circ L_t$.
Then the following hold:
\begin{enumerate}
\item\label{item:AC1}
$(h_{s,t})_{0 \le s \le t \le T}$ is a $\PickA$-REF;
\item\label{item:AC2}
the mapping $[0,T] \ni t\mapsto \Im[h_{t,T}(i)] \in [0,\infty)$ is continuous and non-increasing;
\item\label{item:AC3}
if $\Im[h_{s,T}(i)]=\Im[h_{t,T}(i)]$ for some $0\le s \le t \le T$, then $h_{s,T} = h_{t,T}$ and further $h_{s',s''}=\mathrm{id}_{\C^+}$ for all $s', s'' \in [s,t]$ with $s' \le s''$. 
\end{enumerate}
Moreover, put $S:=\Im [h_{0,T}(i)-i]$ and define
\begin{align*}
\tau(s) &:= \sup\{t\in [0,T]: \Im [h_{t,T}(i)-i] =S-s\},\qquad 0 \le s \le S, \\
h^\tau_{s,t} &:=  h_{\tau(s), \tau(t)}, \qquad 0 \le s \le t \le S.
\end{align*}
Then
\begin{enumerate}
\setcounter{enumi}{3}
\item\label{item:AC4}
$\tau\colon [0,S]\to [0,T]$ is strictly increasing, 
\item\label{item:AC5}
$(h^\tau_{s,t})_{0 \le s \le t \le S}$ is an absolutely continuous $\PickA$-REF. 
\end{enumerate}
\end{proposition}

\begin{proof}
\eqref{item:AC1} is easy to check. 

\smallskip\noindent
\eqref{item:AC2} 
As being a $\PickA$-REF, each $f_{s,t}$ has the form 
\[
f_{s,t}(z) = z + a_{s,t} + \int_{\R}\frac{1+xz}{x-z}\,\sigma_{s,t}(dx),  
\]
where $\sigma_{s,t}$ is a finite Borel measure on $\R$.
Since $a_{t,t}=0$ we have 
\begin{equation}\label{eq:PN}
\Re[h_{t,T}(i)] = \Re[L_t^{-1} (f_{t,T}(i))] = 0,\qquad 0\le t \le T. 
\end{equation} 
As $h_{s,T}(i) = h_{s,t}( h_{t,T}(i))$, the mapping $t \mapsto \Im[h_{t,T}(i)]$ is non-increasing.
This mapping is obviously continuous. 

\smallskip\noindent
\eqref{item:AC3}
Suppose that $0\le s \le t \le T$ and $\Im[h_{s,T}(i)]=\Im[h_{t,T}(i)]$.
The relation $h_{s,T}= h_{s,t} \circ h_{t,T}$ implies $h_{s,t}(z)= z+ \tilde a_{s,t}$ for some $\tilde a_{s,t}\in \R$.
Hence $h_{s,T}= h_{s,t} \circ h_{t,T}= h_{t,T} +  \tilde a_{s,t}$. As $\Re[h_{s,T}(i)]=\Re[h_{t,T}(i)]=0$, we conclude that $ \tilde a_{s,t}=0$. 

Take any $s', s'' \in [s,t]$ with $s' \le s''$.
We can use arguments similar to the above: the relation $\mathrm{id}_{\C^+}=h_{s,t}=h_{s,s'}\circ h_{s',s''}\circ h_{s'',t}$ implies that
\[
\Im i = \Im[h_{s,s'}(h_{s',s''}(h_{s'',t}(i)))] \ge \Im[h_{s',s''} ( h_{s'',t}(i)))] \ge  \Im[h_{s'',t}(i))]\ge \Im i
\]
and hence the inequalities here are all equalities.
This means that $h_{s',s''}(z)= z + \tilde a_{s',s''}$ for some $\tilde a_{s',s''} \in \R$.
Using $h_{s',T}= h_{s',s''} \circ h_{s'',T}= h_{s'',T} + \tilde a_{s',s''}$ and $\Re[h_{s',T}(i)]=\Re[h_{s'',T}(i)]=0$ we conclude that $\tilde a_{s',s''}=0$. 

\smallskip\noindent
\eqref{item:AC4}, \eqref{item:AC5}
Property \eqref{item:AC2} readily implies that $\tau\colon [0,S]\to[0,T]$ is strictly increasing.
By the definition of $\tau$, we have $h^\tau_{S-t,S}(i) = i (t+1)$ for all $0 \le t \le S$, which is trivially absolutely continuous on $[0,S]$.
By Lemma~\ref{TH_diff-one-trajectory}, $(\varphi_{s,t})_{0\le s \le t \le S}:=(h_{S-t,S-s}^\tau)_{0\le s \le t \le S}$ is an absolutely continuous EF, and by Remark \ref{rem:ACEF_REF}, $(h_{s,t}^\tau)$ is an absolutely continuous REF. 
\end{proof}

The family $(h^\tau_{s,t})_{0\le s \le t \le S}$ in Proposition~\ref{thm:AC} satisfies a Loewner differential equation driven by a Herglotz vector field.
This will yield an integral equation for $(f_{s,t})$. 

\begin{theorem}\label{thm:integral}
Let $(f_{s,t})_{s\le t}$ be a $\PickA$-REF.
Then there exist a continuous function $I\ni t\mapsto \gamma_t \in \R$ with $\gamma_0=0$ and a Borel measure $\Genm$ on $\R\times I$ with \eqref{LK3} such that
\begin{equation}\label{eq:integral}
f_{s,t} (z) = z  + \gamma_s-\gamma_t + \int_{\R \times [s,t]} \frac{1+x f_{r,t}(z)}{x-f_{r,t}(z)} \,\Genm(dx\,dr), \qquad 0 \le s \le t, ~z \in \C^+. 
\end{equation}
The pair $((\gamma_t),\Genm)$ is unique.
Conversely, given such a pair $((\gamma_t),\Genm)$, there is a unique $\PickA$-REF $(f_{s,t})$ such that \eqref{eq:integral} holds.  
\end{theorem}

In this theorem, we call both the pairs $((\gamma_t), \Genm)$ and $((\gamma_t),(\eta_t))$, where $\eta_t(\cdot):= \Genm(\cdot\times[0,t])$, the \emph{generators} of the $\PickA$-REF $(f_{s,t})$ and also of the corresponding $\PickA$-DLC and $\rhd$-CH.

\begin{proof}
We first establish the  existence and uniqueness of the generator, given a $\PickA$-REF $(f_{s,t})$.
We begin with the case $I=[0,T]$ with $T<+\infty$.
Let $h_{s,t}$ and $h^\tau_{s,t}$ be as in Proposition \ref{thm:AC}.
According to Lemma \ref{lem:absPREF}, there exists a Herglotz vector field $p^\tau(z,t)$ of the form 
\[
p^\tau(z,t) = -\dot\gamma_{t}^\tau + \int_\R \frac{1+xz}{x-z} \, \dot{\genm}_t^{\tau}(dx) 
\]
such that the differential equation 
\[
\frac{\partial }{\partial s} h_{s,t}^\tau(z) = -p^\tau(h^\tau_{s,t}(z),s) 
\]
holds for a.e.~$s \in [0,t]$ for each $t\in[0,S]$ and $z \in \C^+$.
We can perform the change of variables as in the proof of Theorem \ref{th:DLC_to_LIE}, using Proposition~\ref{thm:AC}~\eqref{item:AC3} instead of Corollary \ref{cor:constant_DLC}.
Then we obtain
\[
h_{s,t}(z) = z - \varGamma([s,t])+ \int_{\R\times[s,t]} \frac{1+x h_{r,t}(z)}{x-h_{r,t}(z)}\,\varTheta(dx\,dr),  
\]
where $\varGamma$ is a signed measure on $[0,T]$ and $\varTheta$ is a finite Borel measure on $\R\times [0,T]$ given by 
\[
\varGamma(B) := \int_{\tau^{-1}(B)} \dot{\gamma}^\tau_{v}\, dv, \qquad \varTheta (A \times B):=  \int_{\tau^{-1}(B)} \dot{\genm}^\tau_{v}(A)\, dv. 
\]
For notational brevity we write $a_t:=a_{t,T}$ and set 
\begin{align*}
\Genm (dx\,dt) &:=  \frac{1+(x-a_{t})^2}{1+x^2}\,\varTheta(d(x-a_{t})\,dt) , \\
  \gamma_{t} &:= \varGamma([0,t]) + a_{t}-a_{0} + \int_{\R\times [0,t]} \left[  \frac{(x-a_{r})(1+x^2)}{1+(x-a_{r})^2}-x\right]\Genm(dx\,dr). 
\end{align*}
Then a straightforward calculation leads to the desired formula:
\begin{align*}
f_{s,t}(z) 
&= L_s\circ h_{s,t}\circ L_t^{-1}(z) = h_{s,t}(z-a_t)+a_s \\ \displaybreak[1]
&= z - a_t + a_s  -\varGamma([s,t]) + \int_{\R\times[s,t]} \frac{1+x h_{r,t}(z-a_t)}{x-h_{r,t}(z-a_t)}\,\varTheta(dx\,dr) \\ \displaybreak[1]
&= z +\gamma_s-\gamma_t + \int_{\R\times [s,t]} \left[ \frac{(x-a_{r})(1+x^2)}{1+(x-a_{r})^2}-x\right]\Genm(dx\,dr) \\ \displaybreak[1]
 &\quad + \int_{\R\times[s,t]} \left[\frac{1+x^2}{x+a_r-f_{r,t}(z)}-x\right] \frac{1+(x+a_r)^2}{1+x^2}\,\Genm(d(x+a_r)\,dr) \\ \displaybreak[1]
 &= z +\gamma_s-\gamma_t + \int_{\R\times [s,t]} \left[  \frac{(x-a_{r})(1+x^2)}{1+(x-a_{r})^2}-x\right]\Genm(dx\,dr) \\ \displaybreak[1]
 &\quad + \int_{\R\times[s,t]} \left[\frac{1+(x-a_r)^2}{x-f_{r,t}(z)}-(x-a_r)\right] \frac{1+x^2}{1+(x-a_r)^2}\,\Genm(dx\,dr) \\ \displaybreak[1]
 &= z +\gamma_s-\gamma_t + \int_{\R\times[s,t]} \frac{1+x f_{r,t}(z)}{x-f_{r,t}(z)}\,\Genm(dx\,dr). 
\end{align*}
As $\tau$ is injective, we see that $t\mapsto \gamma_t$ is continuous, $\gamma_0=0$, and $\Genm(\R\times \{t\})=0$. 

To see the uniqueness of the pair $((\gamma_t)_t, \Genm)$, suppose that $((\gamma_t^i)_{0\le t \le T}, \Genm^i)$, $i=1,2$, are two generators of a common $\PickA$-REF $(f_{s,t})$.
Consider the disintegration 
\[
\Genm^i(dx\,dr) = \nu_r^i(dx)\lambda^i(dr), 
\]
where $\nu_r^i$ is a probability measure on $\R$ and $\lambda^i(dr):=\Genm^i(\R\times dr)$ is an atomless finite measure on $[0,T]$. 
Let $\lambda := \lambda^1+\lambda^2$ and $\beta_t:=\gamma_t^2-\gamma_t^1$. Since $\lambda^i$ are absolutely continuous with respect to $\lambda$, the obvious equality 
\[
\gamma_s^2-\gamma_T^2+ \int_{\R \times [s,T]} \frac{1+x f_{r,T}(z)}{x-f_{r,T}(z)} \,\Genm^2(dx\,dr) = \gamma_s^1-\gamma_T^1+ \int_{\R \times [s,T]} \frac{1+x f_{r,T}(z)}{x-f_{r,T}(z)} \,\Genm^1(dx\,dr)
\]
may be written as 
\begin{align*}
&& \beta_s = \beta_T + \int_{s}^T \left[\int_\R\frac{1+x f_{r,T}(z)}{x-f_{r,T}(z)} \,\left(\nu_r^1(dx)\frac{d\lambda^1}{d\lambda}(r) -  \nu_r^2(dx)\frac{d\lambda^2}{d\lambda}(r) \right) \right]\lambda(dr), \label{eq:LB0} \\
&& 0 \le s \le T,~z\in \C^+.   \notag
\end{align*}
The Lebesgue--Besicovitch differentiation theorem (see, e.g., \cite[Theorem 1.32]{EG92}) yields 
\begin{equation}\label{eq:LB}
\lim_{h\to0^+}\frac{\beta_{r+h}-\beta_{r-h}}{\lambda((r-h,r+h))} = -\int_\R\frac{1+x f_{r,T}(z)}{x-f_{r,T}(z)} \,\left(\nu_r^1(dx)\frac{d\lambda^1}{d\lambda}(r) -  \nu_r^2(dx)\frac{d\lambda^2}{d\lambda}(r) \right)
\end{equation}
for all $r \in [0,T]\setminus N_{z}$, where $N_{z}$ is a $\lambda$-null set depending on  $z$. Now we take a countable set $A \subseteq \C^+$ that has an accumulation point in $\C^+$, say $A=\{i+1/n: n \in \N\}$, and set $N:= \bigcup_{z \in A} N_{z}$.  Then \eqref{eq:LB} holds for all $r \in [0,T]\setminus N$ and all $z\in A$. For each $r \in [0,T]\setminus N$, the set $\{f_{r,T}(z): z \in A\}$ also has an accumulation point in $\C^+$ as $f_{r,T}$ is univalent. From this fact and the identity theorem, we obtain 
\begin{equation*}\label{eq:LB2}
\lim_{h\to0^+}\frac{\beta_{r+h}-\beta_{r-h}}{\lambda((r-h,r+h))} = \int_\R\frac{1+x w}{x-w} \,\left(\nu_r^1(dx)\frac{d\lambda^1}{d\lambda}(r) -  \nu_r^2(dx)\frac{d\lambda^2}{d\lambda}(r) \right)
\end{equation*}
for all $r \in [0,T]\setminus N$ and all $w\in \C^+$. Since the imaginary part of the left-hand side is
zero, by the Stieltjes inversion formula \eqref{eq:PN_interval} we obtain 
\[
\nu_r^1(dx)\frac{d\lambda^1}{d\lambda}(r) =  \nu_r^2(dx)\frac{d\lambda^2}{d\lambda}(r), \qquad r \in [0,T ]\setminus N, 
\]
which in turn implies $\Genm^1=\Genm^2$. It is then easy to show $\gamma_t^1=\gamma_t^2$ using \eqref{eq:integral} and the initial conditions $\gamma_0^1=\gamma_0^2=0$. 

We now consider the case  $I=[0,T_0)$ with $T_0\le+\infty$.
Let $(f_{s,t})_{0\le s \le t <T_0}$ be a $\PickA$-REF.
We take a compact subset $[0,T] \subseteq I$ and then apply the established result to the restricted $\PickA$-REF $(f_{s,t})_{0\le s \le t \le T}$ to obtain $(\gamma_t^T)_{0\le t \le T}$ and a finite measure $\Genm^T$ on $\R\times [0,T]$.
In fact, these parameters are independent of $T$ in the sense that for $T<T'<T_0$ we have $\gamma_t^T = \gamma_t^{T'}$ and $\Genm^T(B)= \Genm^{T'}(B)$ for all $0\le t \le T$ and $B \in \cB(\R \times [0,T])$.
This follows from the uniqueness in the case $I=[0,T]$ already proven above.
Thus we are allowed to extend $(\gamma_t^T)_{0\le t \le T}$ and $\Genm^T$ to $(\gamma_t)_{t \in I}$ and $\Genm$ as desired.
The uniqueness of $((\gamma_t)_{t\in I},\Genm)$ can also be proved by the restriction argument. 

Conversely, suppose that  $(\gamma_t)_{t\in I}$ and $\Genm$ are given. As $(a_{s,t})$ is not given \textit{a priori}, we consider a differently modified REF $\tilde h_{s,t} := K_s^{-1} \circ f_{s,t}\circ K_t$, where $K_t(z):=z + \gamma_t$. Then the integral equation is equivalent to 
\[
\tilde h_{s,t} (z) = z + \int_{\R \times [s,t]} \frac{1+x (\tilde h_{r,t}(z)+\gamma_r)}{x-(\tilde h_{r,t}(z)+\gamma_r)} \,\Genm(dx\,dr), \qquad 0 \le s \le t, ~z \in \C^+ 
\]
and hence to 
\[
\tilde h_{s,t} (z) = z - \tilde\varGamma([s,t])+ \int_{\R \times [s,t]} \frac{1+x \tilde h _{r,t}(z)}{x-\tilde h_{r,t}(z)} \,\tilde\varTheta(dx\,dr), \qquad 0 \le s \le t, ~z \in \C^+,  
\]
where 
\begin{align*}
\tilde\varTheta(dx\,dr) &:=  \frac{1+(x+\gamma_r)^2}{1+x^2}\,\Genm(d(x+\gamma_r)\,dr), \\
\tilde\varGamma(B) &:= \int_{\R\times B} \left[ \frac{(1+x^2)(x+\gamma_r)}{1+(x+\gamma_r)^2}-x\right] \tilde\varTheta(dx\,dr). 
\end{align*}
An appropriate time-change further reduces this equation to a Loewner differential equation in a way analogous to Theorem \ref{th:LIE_to_DLC}. The point is that, if we write the disintegration $\tilde\varTheta(dx\,dt) = \tilde\nu_t(dx)\tilde\lambda(dt)$ with each $\tilde\nu_t$ a probability measure on $\R$ as in Proposition \ref{prop:CMF}, then $\tilde \varGamma$ is absolutely continuous with respect to the atomless measure $\tilde \lambda$.
Therefore, we can perform the time-change by the continuous function $\tilde \ar(t):=\tilde\lambda([0,t])$.
Then the converse part of Lemma \ref{lem:absPREF} completes the proof.  
\end{proof}

\begin{remark}\label{rem:LIDE_general}
The $\PickA$-REFs $(h_{s,t})$ and $(\tilde h_{s,t})$ in the above proof do satisfy integro-differential equations too; for example,  
\begin{equation*}
h_{s,t}(z) = z + \int_{[s,t]} \frac{\partial h_{r,t}}{\partial z}(z)\left[ -\varGamma(dr) + \int_{\R}\frac{1+xz}{x-z}\,\varTheta(dx\,dr)\right], \qquad s\le t, ~z\in \C^+.  
\end{equation*}
The proof is similar to the case of finite second moment in Theorem \ref{th:DLC_to_LIE}, i.e., the application of the time-change argument to the differential equation for $h_{s,t}^\tau$. 
\end{remark}

A version for $\PickA$-EFs can be easily obtained. 
\begin{corollary}\label{thm:integral_anti-monotone}
Let $(\varphi_{s,t})_{s\le t}$ be a $\PickA$-EF. 
Then there exist a continuous function $I\ni t\mapsto \gamma_t \in \R$ with $\gamma_0=0$ and a Borel measure $\Genm$ on $\R\times I$ with \eqref{LK3} such that
\begin{equation}\label{eq:integral_EF}
\varphi_{s,t} (z) = z  + \gamma_s-\gamma_t + \int_{\R \times [s,t]} \frac{1+x \varphi_{s,u}(z)}{x-\varphi_{s,u}(z)} \,\Genm(dx\,du), \qquad s \le t, ~z \in \C^+. 
\end{equation}
The pair $((\gamma_t),\Genm)$ is unique.
Conversely, given such a pair $((\gamma_t),\Genm)$, there is a unique $\PickA$-EF $(\varphi_{s,t})$ such that \eqref{eq:integral_EF} holds.  

The pairs $((\gamma_t), \Genm)$ and $((\gamma_t),(\eta_t))$, where $\eta_t(\cdot):= \Genm(\cdot\times[0,t])$, are both called the \emph{generator} of the $\PickA$-EF $(\varphi_{s,t})$ and also of the corresponding $\lhd$-CH.

\end{corollary}
\begin{proof}
In the case $I=[0,T]$ with $T<+\infty$ this is an easy consequence of the fact that $f_{s,t}:=\varphi_{T-t, T-s}$ form a $\PickA$-REF. In the case $I=[0,T)$ with $0<T \le +\infty$, the conclusion follows by working with each compact subinterval and then extending the generators to the whole interval $I$ by the uniqueness. 
\end{proof}

\begin{remark}\label{rem:M_AM} Recall from Remark \ref{rem:antimonotone_EF} that $\lhd$-CHs correspond to $\PickA$-EFs. 
Theorem \ref{thm:integral} and Corollary \ref{thm:integral_anti-monotone} thus induce a bijection (actually a homeomorphism, see Theorem \ref{th:conv_P-REF_gen}) between $\rhd$-CHs and $\lhd$-CHs by identifying their generators. Note that this bijection is different from another canonical bijection $(\monotone_{s,t})\mapsto (\monotone_{T-t,T-s})$ in case $I=[0,T]$. Together with Proposition \ref{prop:LK_repr}, we obtain a bijection between $\ast$-CHs, $\rhd$-CHs and $\lhd$-CHs.  
\end{remark}

Recall that the integral equation in the case of finite second moment has been established in Corollary \ref{cor:LIE}. We can derive it directly from Theorem \ref{thm:integral}.

\begin{proposition}\label{prop:LIE_PickB} Let $(f_{s,t})$ be a $\PickA$-REF and  $((\gamma_t), (\genm_t))$ be its generator. Then the following conditions are equivalent: 
\begin{enumerate}
\item\label{item:pickB} $(f_{s,t})$ is a $\PickB$-REF; 
\item\label{item:eta_second} $\eta_t$ has finite second moment for every $t\in I$. 
\end{enumerate}
If these conditions hold then $(f_{s,t})$ satisfies the integral equation \eqref{eq:LIE3} with reduced generator $((m_t),\rGen)$ defined by \eqref{eq:gen_rgen}, \eqref{eq:two_thetas}, and \eqref{LK3_theta}. 
\end{proposition}

\begin{proof} 
\eqref{item:eta_second} $\Rightarrow$  \eqref{item:pickB}.  A straightforward  calculation yields \eqref{eq:LIE3}. Then using the inequality $\Im z \le \Im f_{r,t}(z)$ we get 
\begin{align*}
\left| f_{s,t}(z)-z + m_t -m_s\right| 
&\le \int_{\R\times [s,t]} \frac{1}{|x - f_{r,t}(z)|} \,\rGen(dx\,dr) \\
& \le  \int_{\R\times [s,t]} \frac{1}{\Im[ f_{r,t}(z)]} \,\rGen(dx\,dr) \le \frac{\rGen(\R\times [s,t])}{\Im z}. 
\end{align*}
By Proposition \ref{prop:class_P} \eqref{HN:ineq}, we conclude that $f_{s,t}+ m_t -m_s \in \PickC$ and hence $f_{s,t} \in \PickB$.  

\smallskip\noindent
\eqref{item:pickB} $\Rightarrow$ \eqref{item:eta_second}.   
By Fatou's lemma and the integral equation \eqref{eq:integral}, we have 
\begin{align*}
\int_{\R} (1+x^2) \,\eta_t(dx) 
&= \int_{\R\times[0,t]} (1+x^2) \,\Genm(dx\,dr) \\
&\le \liminf_{y\to+\infty} \int_{\R\times [0,t]} \frac{y \Im[f_{r,t}(iy)]}{|x-f_{r,t}(iy)|^2}(1+x^2) \,\Genm(dx\,dr) \\
&= \liminf_{y\to+\infty}y \Im[f_{0,t}(iy)-iy] = \AR(f_{0,t})<+\infty. 
\end{align*}
The last equality is a consequence of Proposition \ref{prop:class_P} \eqref{HN:ang}. 
\end{proof}

\subsection{Convergence of REFs and generators}
\label{sec:REF_gen_homeo}

We are now going to show that the correspondence between generators and $\PickA$-REFs, and hence $\rhd$-CHs as well, is bi-continuous.
Our proof requires the following lemma adapted from Proposition 5.2 and the subsequent paragraph of Bercovici and Voiculescu~\cite{BV93}.
For the reader's convenience we include a proof in Appendix~\ref{app:conv}. 

\begin{lemma}\label{lem:inverse1}
Let $(\mu_\tau)_{\tau \in K}$ be a tight family of Borel probability measures on $\R$ with arbitrary index set $K$. Then the following hold:
\[
\sup_{\tau \in K} \lvert F[\mu_\tau](z) - z \rvert = o(z)
\] 
as $z \to\infty$ nontangentially. 
\end{lemma}

Using Lemma~\ref{lem:inverse1}, we prove the next theorem.
The second condition in this theorem, the convergence of generators, has several equivalent formulations (Proposition~\ref{prop:two_sigmas}), which we do not repeat here.

\begin{theorem} \label{th:conv_P-REF_gen}
    Let $((\gamma_t^n), \Genm^n), ((\gamma_t), \Genm)$ be the generators of $\PickA$-REFs $(f_{s,t}^n)$, $(f_{s,t})$, respectively. Then the following are equivalent:
    \begin{enumerate}
        \item\label{item:conv_general1} $f_{s,t}^n(z) \to f_{s,t}(z)$ locally uniformly on $I^2_\le \times \C^+$;
        \item\label{item:conv_general2} $\gamma_t^n\to\gamma_t$ uniformly on $[0,T]$ and $\Genm^n|_{\R\times[0,T]} \to \Genm\rvert_{\R\times[0,T]}$ weakly for each $T\in I$. 
    \end{enumerate}
\end{theorem}

\begin{remark} Combining Theorem \ref{th:conv_P-REF_gen}, Proposition \ref{prop:LIE_PickB} and Lemma \ref{lem:vague-weak}, we can easily deduce Theorem \ref{thm:main_conv} as a corollary. 
\end{remark}

\begin{proof}[Proof of Theorem~\ref{th:conv_P-REF_gen}]
\eqref{item:conv_general1} $\Rightarrow$ \eqref{item:conv_general2}.
Let $T <\infty$. We first show that the sequence $(\Genm^n \rvert_{\R\times[0,
T]})_n$ is uniformly bounded and tight. Since $f_{s,t}^n \to f_{s,t}$, the set $\{\, f_{s,t}^n(i): 0\le s \le t \le T,~n\in\N \,\}$ is relatively compact in $\C^+$. Therefore, there is a constant $c>0$ such that 
$(1+x^2)/\lvert x-f_{s,t}^n(i) \rvert^2 \ge c$ for all $n\in\N$ and $0\le s \le t \le T$. Together with $\Im[f_{r,T}^n(i)]\ge1$, this yields the estimate
\[
\Im[f_{s,T}^n(i)] = 1 + \int_{\R\times[s,T]} \frac{(1+x^2)\Im[f_{r,T}^n(i)]}{\lvert x-f_{r,T}^n(i) \rvert^2} \,\Genm^n(dx\,dr) \ge 1+c  \Genm^n(\R\times[0,T]),  
\]
which shows the uniform boundedness. 

Let $(\monotone_{s,t}^n)_{s \le t}$ be the underlying $\rhd$-CH; i.e., $f_{s,t}^n= F[\monotone_{s,t}^n]$ holds.
By the assumption \eqref{item:conv_general1} and Theorem~\ref{th:conv_monotone}, $\monotone_{s,t}^n$ converges weakly locally uniformly.
We can then easily check that $\{\, \monotone_{s,t}^n : 0\le s \le t \le T,\ n \in \N \,\}$ is tight.
By Lemma \ref{lem:inverse1} we have the estimate 
\[
f_{s,t}^n(iy) = iy +o(y),\qquad y\to \infty,
\]
uniformly in $n$ and $s,t$. 
This implies $\lvert \Re[f_{s,t}^n(iy)] \rvert \le y$ and $\Im[f_{s,t}^n(iy)]\le 2y$ for sufficiently large $y$, say $y \ge y_0 \ge1$. We have 
\begin{align*}
\frac{1+x^2}{\lvert x-f_{s,t}^n(iy) \rvert^2} 
&= \frac{1+x^2}{(x-\Re[f_{s,t}^n(iy)])^2 + \Im[f_{s,t}^n(iy)]^2} \\ \displaybreak[1]
&\ge \frac{1+x^2}{2(x^2+y^2) + 4y^2}\\ \displaybreak[1]
&\ge  \frac{1+x^2}{6(x^2+y^2)}. 
\end{align*}
Moreover, for $x,y$ such that $\lvert x \rvert \ge y \ge y_0$ we have 
\[
 \frac{1+x^2}{6(x^2+y^2)} \ge   \frac{1+y^2}{6(y^2+y^2)} \ge \frac{1}{12}. 
\]
Therefore, 
\begin{align*}
\Im[f_{0,T}^n(iy)] &= y + \int_{\R\times[0,T]} \frac{(1+x^2)\Im[f_{r,T}^n(iy)]}{\lvert x-f_{r,T}^n(iy) \rvert^2} \,\Genm^n(dx\,dr) \\ \displaybreak[1]
&\ge y + y\int_{(\R\setminus(-y,y))\times[0,T]} \frac{1}{12} \,\Genm^n(dx\,dr) \\ \displaybreak[1]
&=y+\frac{y}{12} \Genm^n((\R\setminus (-y,y))\times [0,T]),  
\end{align*}
and hence 
\[
\Genm^n((\R\setminus (-y,y))\times [0,T]) \le \frac{12(\Im[f_{0,T}^n(iy)]-y)}{y}=o(1),\qquad y\to\infty,    
\]
which shows the tightness of $(\Genm^n \rvert_{\R\times[0,T]})_{n\in\N}$. 

Prokhorov's theorem (see e.g.~\cite{Bil68}) ensures that there is a subsequence of $(\Genm^{n}\rvert_{\R\times [0,T]})_n$ which converges weakly to a finite Borel measure $\Genm'$. From the identity 
\begin{equation}\label{eq:master0}
     f_{0,t}^n (z) = z  -\gamma_t^n + \int_{\R \times [0,t]} \frac{1+x f_{r,t}^n(z)}{x-f_{r,t}^n(z)} \,\Genm^n(dx\,dr), 
\end{equation}
we deduce that $\gamma_t^n$ converges to some $\gamma_t' \in \R$ along this subsequence for each $t \in [0,T]$ since all the other terms converge.
Hence we have 
\[
 f_{s,t} (z) = z  +\gamma_s' -\gamma_t' + \int_{\R \times [s,t]} \frac{1+x f_{r,t}(z)}{x-f_{r,t}(z)} \,\Genm'(dx\,dr). 
\]
By this formula, $\Genm'(\R\times\{t\})=0$ for all $t$; otherwise the right-hand side of 
\[
 \Im[f_{0,t} (z)] = \Im z  + \int_{\R \times [0,t]} \Im\left[\frac{1+x f_{r,t}(z)}{x-f_{r,t}(z)}\right]\Genm'(dx\,dr)
\]
would be discontinuous, which would contradict the continuity of the left-hand side. 
Moreover, from the identity
\[
 f_{s,T} (z) = z +\gamma_s' -\gamma_T' + \int_{\R \times [s,T]} \frac{1+x f_{r,T}(z)}{x-f_{r,T}(z)} \,\Genm'(dx\,dr) 
\]
we conclude that $s\mapsto \gamma_s'$ is continuous as the other terms are all continuous.
Therefore, $((\gamma_t'),\Genm')$ is a generator of $(f_{s,t})_{0\le s \le t \le T}$.
The uniqueness of generators in Theorem \ref{thm:integral} yields $\gamma_t'=\gamma_t$ and $\Genm'=\Genm\rvert_{\R\times[0,T]}$.

The above argument shows that the whole sequence $(\gamma_t^n)_{n\in\N}$ converges to $\gamma_t$ for each $t \in[0,T]$ and $(\Genm^n \rvert_{\R\times[0,T]})_{n\in\N}$ converges weakly to $\Genm \rvert_{\R\times[0,T]}$.

It remains to show that the convergence $\gamma_t^n\to\gamma_t$ is uniform on $[0,T]$. 
From a standard $3\epsilon$ argument and the compactness of $[0,T]$, the pointwise convergence and equicontinuity imply the uniform convergence. Therefore, it suffices to show the equicontinuity of $\gamma_t^n$ on $[0,T]$, which follows from the master equation
 \begin{equation}\label{eq:equi0}
      f_{s,t}^n (z) = z  +\gamma_s^n -\gamma_t^n + \int_{\R \times [s,t]} \frac{1+x f_{r,t}^n(z)}{x-f_{r,t}^n(z)} \,\Genm^n(dx\,dr). 
 \end{equation}
Indeed, the uniform convergence $f_{s,t}^n(i)\to f_{s,t}(i)$ and uniform continuity of $(s,t)\mapsto f_{s,t}(i)$ ensure that for every $\epsilon>0$ there is $\delta>0$ such that 
\begin{equation}\label{eq:equi1}
    \sup_{\substack{n\in\N, \,0\le s \le t \le T\\t-s <\delta}}|f_{s,t}^n(i)-i| < \epsilon. 
\end{equation}
Since $t\mapsto \Genm(\R\times[0,t])$ is a continuous function, by P\'olya's theorem (Remark~\ref{rem:Polya}), $\Genm^n(\R\times[0,t])$ converges uniformly to $\Genm(\R\times[0,t])$, and hence by the uniform continuity of $t\mapsto \Genm(\R\times[0,t])$, there is $\delta'>0$ such that 
\begin{equation}\label{eq:equi2}
   \sup_{\substack{n\in\N, \,0\le s \le t \le T\\t-s <\delta'}} \Genm^n(\R\times (s,t])<\epsilon. 
\end{equation}
Finally, as $\{\, f_{s,t}^n(i) : n\in\N,\ 0\le s \le t \le T \,\}$ is a relatively compact subset of $\C^+$ we have a bound 
\begin{equation}\label{eq:equi3}
    C:=\sup_{n\in\N,\, x \in \R,\, 0\le r \le t \le T}\left\lvert \frac{1+x f_{r,t}^n(i)}{x-f_{r,t}^n(i)}\right\rvert <+\infty. 
\end{equation}
Substituting the estimates \eqref{eq:equi1}--\eqref{eq:equi3} into \eqref{eq:equi0} yields 
\[
\sup_{\substack{n\in\N,\, 0\le s\le t \le T\\t-s<\min\{\delta,\delta'\}}}\lvert \gamma_t^n-\gamma_s^n \rvert \le \epsilon+ C\epsilon, 
\]
which thereby verifies the equicontinuity of $\gamma_t^n$.

\smallskip\noindent
\eqref{item:conv_general2} $\Rightarrow$ \eqref{item:conv_general1}.
In a way similar to the usual proof of the continuous dependence of solutions on parameters in the theory of ordinary differential equations~\cite[Theorem 7.4]{CL55}, we show below that, for each $T>0$ and $z\in\C^+$, the function $f_{s,T}^n(z)$ converges to $f_{s,T}(z)$ as $n \to \infty$ uniformly in $s\in[0,T]$. 
Let $T>0$ and $z\in \C^+$ be fixed. 
For each $n\in\N$, 
let $f_k^n(s)$, $k=0,1,2,\dots$, be successively defined by 
\begin{align}
 f_0^n(s) &:=  f_{s,T}(z), \notag \\
 f_k^n(s) &:= z + \gamma_s^n - \gamma_T^n+ \int_{\R\times[s,T]} \frac{1+x f_{k-1}^n(r)}{x-f_{k-1}^n(r)}\,\Genm^n(dx\,dr). \notag
\end{align}
We have 
\begin{align*}
\lvert f_1^n(s)-f_0^n(s) \rvert &\le \lvert \gamma_s^n-\gamma_s \rvert + \lvert \gamma_T^n-\gamma_T \rvert + \left\lvert \int_{\R\times[s,T]} \frac{1+x f_{r,T}(z)}{x-f_{r,T}(z)}\,[\Genm^n(dx\,dr)-\Genm(dx\,dr)] \right\rvert.
\end{align*}
Here, the first and second terms on the right-hand side converge to zero uniformly on $[0,T]$ as $n\to\infty$. The third term converges to zero uniformly as well. Indeed, the pointwise convergence for each $s \in[0,T]$ is obvious by the assumption $\Genm^n \rvert_{\R\times[0,T]}\to \Genm \rvert_{\R\times[0,T]}$, and moreover, the equicontinuity follows from \eqref{eq:equi2} which can be proved in the same way as before.  Altogether, for a given $\epsilon>0$ there is $N\in\N$ such that 
\[
\sup_{s\in[0,T]}\lvert f_1^n(s)-f_0^n(s) \rvert \le \epsilon, \qquad n\ge N. 
\]
For each $k\in\N$ and $n\ge N$ we have 
\[
f_{k+1}^n(s)-f_k^n(s) = \int_{\R\times[s,T] } \frac{(1+x^2)[f_{k}^n(r)-f_{k-1}^n(r)]}{[x-f_k^n(r)][x-f_{k-1}^n(r)]}\,\Genm^n(dx\,dr). 
\]
Let $\lambda^n(B):=\Genm^n(\R\times B)$, $B\in \cB([0,T])$, $K:=\{\, f_{s,T}(z):s \in[0,T] \,\}$,  $K^r:=\{\, w \in \C^+: \text{dist}(w,K)\le r \,\}$, $\delta:=\text{dist}(K,\R)$, and   
\[
D := \sup_{w_1,w_2 \in K^{\delta/2},\,x\in \R} \frac{1+x^2}{\lvert x-w_1 \rvert \lvert x-w_2 \rvert}. 
\]
Provided that $f_0^n(r), f_1^n(r), \ldots, f_{k}^n(r)\in K^{\delta/2}$ for all $r \in [0,T]$ we have 
\[
\lvert f_{k+1}^n(s)-f_k^n(s) \rvert \le D\int_{[s,T]} \lvert f_{k}^n(r)-f_{k-1}^n(r) \rvert\,\lambda^n(dr)
\]
and iterating this inequality yields
\begin{align}
\lvert f_{k+1}^n(s)-f_k^n(s) \rvert 
&\le  D^k \int_{s \le r_1 \le \cdots \le r_k \le T} \lvert f_1^n(r_k) - f_0^n(r_k) \rvert (\lambda^n)^{\otimes k}(dr_1dr_2\cdots dr_k) \label{eq:Weierstrass}  \\
&\le \epsilon  \frac{(D\lambda^n([s,T]))^k}{k!} \le \epsilon  \frac{(DM)^k}{k!},  \notag
\end{align}
where $M:=\sup_{n}\lambda^n([0,T])<+\infty$. 
As the same inequality holds for nonnegative integers smaller than $k$, summing up them from $0$ to $k$ yields 
\begin{equation}\label{eq:Picard}
\lvert f_{k+1}^n(s) - f_0^n(s) \rvert \le \epsilon e^{DM}. 
\end{equation}
Hence, if we take $\epsilon >0$ such that $\epsilon e^{DM}<\delta/2$, then by induction, $f_k^n(s) \in K^{\delta/2}$ and \eqref{eq:Picard} holds for all $k\in \N$, $n\ge N$ and $s \in [0,T]$.
Moreover, by Weierstrass' M-test and \eqref{eq:Weierstrass}, the sequence $(f_k^n(s))_k$ converges to a function uniformly on $[0,T]$ as $k \to \infty$.
The limit function is actually $f_{s,T}^n(z)$ because it satisfies the integral equation driven by the generator $((\gamma_t^n), \Genm^n)$.
Passing to the limit $k\to\infty$ in \eqref{eq:Picard}, we deduce that $\lvert f_{s,T}^n(z) - f_{s,T}(z) \rvert \le \epsilon e^{DM}$ for all $s \in [0,T]$ and all $n \ge N$. 

Actually in the above, modifying the definition of the set $K$ slightly we can take the constant $D$ locally uniformly in $z$.
Thus, for each fixed $T \in I$, $f^n_{s,T}(z)$ converges to $f_{s,T}(z)$ as $n \to \infty$ locally uniformly in $(z,s) \in \C^+ \times [0,T]$.
In particular, $f^n_{0,T}=f^n_T$ converges to $f_T$ locally uniformly on $\C^+$.
Since $(f^n_T)^{-1}$ converges to $f_T^{-1}$ locally uniformly on $f_T(\C^+)$ (see, e.g., \cite[Proposition B.1]{HH22}), this further implies that $(f^n_s)^{-1}(z) =f^n_{s,T} ((f^n_T)^{-1}(z))$ converges to $f_s^{-1}(z)$ locally uniformly in $(z,s) \in f_T(\C^+) \times [0,T]$. 

The remaining part of the proof proceeds similarly to the
proof of Theorem \ref{th:conv_DLC_REF}. 
Let $s \in [0,T]$.
Take $c \in \C^+$ and $r>0$ in such a way that $\overline{\D(c, 2r)} \subset f_T(\C^+)$; here $\D(c, 2r)$ denotes the open disk with center $c$ and radius $2r$.
By the equicontinuity in $(z,u)$ of the family of functions $\{\, f_u^{-1}(z), (f^n_u)^{-1}(z) : n \in \N \,\}$, there exist $\delta>0$ and a neighborhood $N(s)$ of $s$ in $[0,T]$ such that $f_s^{-1}(\overline{\D(c,r)})$ is surrounded by the Jordan curves $f_u^{-1}(\partial \D(c,2r))$ and $(f^n_u)^{-1}(\partial \D(c,2r))$ for all  sufficiently large $n$ and $u \in N(s)$ with the distance between $f_s^{-1}(\overline{\D(c,r)})$ and each of these curves not less than the common $\delta$.
Now by the Lagrange inversion formula we have
\[
f_u(w)=\frac{1}{2\pi i}\int_{\partial \D(c,2r)} \frac{\zeta (f_u^{-1})'(\zeta)}{f_u^{-1}(\zeta)-w}\,d\zeta,
\qquad
f^n_u(w)=\frac{1}{2\pi i}\int_{\partial \D(c,2r)} \frac{\zeta [(f^n_u)^{-1}]'(\zeta)}{(f^n_u)^{-1}(\zeta)-w}\,d\zeta
\]
for all $w \in f_s^{-1}(\D(c,r))$, $u \in N(s)$, and $n \in \N$.
Noting that the denominators of these integrands are bounded from below by $\delta$, we can deduce from these identities that $f^n_u(w)$ converges to $f_u(w)$ locally uniformly in $(w, u) \in f_s^{-1}(\D(c,r)) \times N(s)$.
We now apply Corollary~\ref{cor:luc_rCauchy} (with the range of $u$-variable interpreted suitably in that statement), which implies that $f^n_u(w)$ converges to $f_u(w)$ locally uniformly in $(w, u) \in \C^+ \times N(s)$.
This convergence of the $\PickA$-DLCs is equivalent to condition \eqref{item:conv_general1} again by Theorem~\ref{th:conv_DLC_REF}; hence the proof is complete.
\end{proof}

\section{Free and boolean convolution hemigroups}
\label{sec:free_boolean-CHs}

For free and boolean CHs, we can establish a canonical bijection with classical CHs.
The definition of the bijection is much simpler than in the monotone case, because all the measures involved are infinitely divisible with respect to the respective convolutions, and so the Bercovici--Pata bijection \cite{BNT02,BP99} applies at each $(s,t)$.

Recall from Speicher and Woroudi~\cite{SW97} that  the \emph{energy function} $K_\mu$ (or $K[\mu]$) defined by the formula
\begin{equation} \label{eq:K-trans}
K_\mu(z) = z - F_\mu(z), \qquad z\in \C^+,
\end{equation}
linearizes boolean convolution: 
\[ 
K_{\mu\uplus \nu} =K_\mu+ K_\nu. 
\]

Free convolution is linearized as follows.
Let $\Gamma_{a,b} =\{\, x+ iy \in \C^+: a |x| < y,\ y >b \,\}$ for $a,b>0$.
It is known that for  any $a> a'>0$ there exist $b, b'>0$ (depending on $a,a'$ and $\mu$) such that $F_\mu$ is univalent in $\Gamma_{a', b'}$ and $F_\mu(\Gamma_{a', b'}) \supset \Gamma_{a, b}$.
This allows us to define the inverse mapping $F_\mu^{-1}\colon \Gamma_{a, b} \to \C^+$.  
The \emph{Voiculescu transform} $\phi_\mu$ (or $\phi[\mu]$) of $\mu$ is then defined by  
\begin{equation}
\phi_\mu(z) = F_\mu^{-1}(z) -z,   \qquad z \in \Gamma_{a,b},
\end{equation}
and the identity 
\begin{equation}
\phi_{\mu\boxplus \nu} =\phi_\mu+ \phi_\nu 
\end{equation}
holds on the intersection of the domains of the three transforms.
For further details the reader is referred to Bercovici and Voiculescu~\cite{BV93}.

\subsection{Convergence of Voiculescu transforms}
\label{subsec:luc_Voic-trans}

As well as the (reciprocal) Cauchy transform in Proposition~\ref{prop:luc_Cauchy} (and Corollary~\ref{cor:luc_rCauchy}), the Voiculescu transform characterizes the locally uniform weak convergence of families of probability measures. 
To see this, we need uniform estimates, given by the next lemma, of $F_\mu$ for $\mu$ in a tight family of probability measures.
This has been considered in the literature; see Proposition~5.4 and the paragraph subsequent to it in Bercovici and Voiculescu~\cite{BV93} and also Bercovici and Pata~\cite[Propositions~2.6 and 6.1]{BP99} \cite[Lemma~5]{BP00}.
We include the proof in Appendix \ref{app:conv} for the sake of reader's convenience.
Some techniques in that proof are also used in the proof of Proposition~\ref{prop:luc_Voiculescu} below. 
The index set $K$ in the following lemma can be arbitrary, but we use the case that $K$ is a compact subset of $\para=I$ or $I^2_{\le}$ only.

\begin{lemma}\label{lem:inverse2}
Let $(\mu_\tau)_{\tau \in K}$ be a tight family of Borel probability measures on $\R$ with arbitrary index set $K$. Then for every $a>\epsilon>0$ there exists $b_0>0$ such that, for all $b\ge b_0$ and $\tau \in K$, $F[\mu_\tau]$ are univalent in $\Gamma_{a,b}$ and $F[\mu_\tau](\Gamma_{a, b}) \supset \Gamma_{a+\epsilon,(1+\epsilon)b}$. 
\end{lemma}

Using Lemmas~\ref{lem:inverse1} and \ref{lem:inverse2} we can prove a uniform convergence of Voiculescu transforms. 
This is a generalization of the non-uniform version in the literature~\cite[Proposition~5.7]{BV93}\cite[Proposition~1]{BP96}.

\begin{proposition} \label{prop:luc_Voiculescu}
Let $\para=I$ or $I^2_{\le}$.
In Proposition~\ref{prop:luc_Cauchy}, the equivalent conditions~\eqref{i:luwc_Cauchy}--\eqref{i:pwc_Cauchy} are also equivalent to the following:
\begin{enumerate}
\setcounter{enumi}{5}
\item \label{item1:luc_Voiculescu}
for every compact $K \subset \para$ there exist $a>a'>0$ and $b>b'>0$ with the following four properties:
\begin{enumerate}
\item\label{i:F_univ} $F[\mu_\tau^n], F[\mu_\tau]$ are univalent in $\Gamma_{a',b'}$;
\item\label{i:F_univ2} both $F[\mu_\tau^n](\Gamma_{a',b'})$ and $F[\mu_\tau](\Gamma_{a',b'})$ contain $ \Gamma_{a,b}$ for all $\tau \in K$ and $n\in\N$;
\end{enumerate}
and with the notation $F[\mu_\tau^n]^{-1},F[\mu_\tau]^{-1}\colon \Gamma_{a,b} \to \Gamma_{a',b'}$ for the associated inverse mappings, 
\begin{enumerate}
\setcounter{enumii}{2}
\item \label{i:luc_F_inv}
$F[\mu_\tau^n]^{-1}$ converges to $F[\mu_\tau]^{-1}$ locally uniformly on $\Gamma_{a,b}\times K$;
\item \label{i:asymp_F_inv}
$\sup_{\tau \in K, n \in \N} \lvert F[\mu_\tau^n]^{-1}(w) - w \rvert =o(w)$ as $w \to\infty$, $w \in \Gamma_{a,b}$. 
\end{enumerate}
\end{enumerate}
\end{proposition}

\begin{proof}
We first suppose that $(\mu^n_\tau)_{\tau \in \para}$ converges weakly to $(\mu_\tau)_{\tau \in \para}$ locally uniformly on $\para$. Note that this implies the tightness of $\{\, \mu^n_\tau, \mu_\tau : \tau \in K, n\in\N \,\}$ for any compact $K \subset \para$ and hence Lemmas \ref{lem:inverse1} and \ref{lem:inverse2} are available. The first two conditions \eqref{i:F_univ} and \eqref{i:F_univ2} are shown in Lemma~\ref{lem:inverse2}.  Condition \eqref{i:asymp_F_inv} is an easy consequence of the estimate $\sup_{\tau \in K, n\in \N}\lvert F[\mu_\tau^n](z)-z \rvert = o(z)$ as $z \to \infty$ through $\Gamma_{a',b'}$, as shown in Lemma~\ref{lem:inverse1}.  

To derive \eqref{i:luc_F_inv} it suffices to demonstrate that for every $(w_0,\tau_0) \in\Gamma_{a,b}\times K$ there exists a neighborhood of $(w_0,\tau_0)$ on which $F[\mu_\tau^n]^{-1}(w)$ converges to $F[\mu_\tau]^{-1}(w)$ uniformly as $n \to \infty$. 
Let $z_0 := F[\mu_{\tau_0}]^{-1}(w_0)$ and take an open disk $D\ni z_0$ such that $\overline{D} \subset \Gamma_{a',b'}$. 
By virtue of Corollary~\ref{cor:luc_rCauchy} and the continuity of $(z,\tau)\mapsto F[\mu_\tau](z)$, which follows from Lemma~\ref{lem:bicontinuity}, there exist  neighborhoods $U \ni \tau_0$, $W \ni w_0$ and $n_0\in\N$ such that the distance between $W$ and $\{\, F[\mu_\tau^n](z) : z \in \partial D,\ n \ge n_0,\ \tau \in U \,\}$ is positive.
Using the Lagrange inversion formula yields 
\begin{align*}
\left\lvert F[\mu_\tau^n]^{-1}(w) - F[\mu_\tau]^{-1}(w) \right\rvert
&=\left\lvert \frac{1}{2\pi i}\int_{\partial D}\frac{z F[\mu_\tau^n]^\prime(z)}{F[\mu_\tau^n](z)- w }\,dz- \frac{1}{2\pi i}\int_{\partial D}\frac{z F[\mu_\tau]^\prime(z)}{F[\mu_\tau](z)- w }\,dz \right\rvert,  
\end{align*}
which converges to zero uniformly on $W \times U$ as $n \to \infty$.
Note here that $F[\mu_\tau^n]^\prime(z)$ converges to $F[\mu_\tau]^\prime(z)$ uniformly on $W \times U$ thanks to Cauchy's integral formula. 

Conversely, assume that \eqref{item1:luc_Voiculescu} holds.  Similar to the proof of Lemma~\ref{lem:inverse2}, for any $a''> a$ there exist $b'' > b_1 >  b > 0 $ such that both $F[\mu_\tau^n]^{-1}(\Gamma_{a,b_1})$ and $F[\mu_\tau]^{-1}(\Gamma_{a,b_1})$ contain $\Gamma_{a'',b''}$ for all $\tau \in K, n\in\N$.
Using the Lagrange inversion formula in a way similar to the above, we see that $F[\mu_\tau^n](z)$ converges to $F[\mu_\tau](z)$ locally uniformly on $\Gamma_{a'',b''} \times K$, which implies Corollary~\ref{cor:luc_rCauchy}~\eqref{i:pwc_rCauchy} as desired. 
\end{proof}

\subsection{L\'evy--Khintchine representations for free and boolean CHs}
\label{sec:c-f-real}

The L\'evy--Khintchine representations for $\boxplus$- and $\uplus$-CHs are given as follows. 

\begin{proposition}\label{thm:f-c-real} For a $\boxplus$-CH $(\free_{s,t})_{s\le t}$ on $\R$ there exists a unique family $(\genc_t,\genm_t)_{t}$ satisfying conditions \eqref{LK1} and \eqref{LK2} in Proposition~\ref{prop:LK_repr} and 
\begin{equation}\label{eq:FALKT}
\phi[\free_{0,t}](z) = \genc_t +  \int_\R\frac{1+x z}{z-x}\,\genm_t(dx).  
\end{equation}
Conversely, given a family $(\genc_t,\genm_t)_{t}$ satisfying conditions \eqref{LK1} and \eqref{LK2} in Proposition~\ref{prop:LK_repr} there exists a unique $\boxplus$-CH $(\free_{s,t})_{s\le t}$ on $\R$ for which \eqref{eq:FALKT} holds. We call $(\genc_t,\genm_t)_{t}$ the \emph{generator} of $(\free_{s,t})_{s\le t}$. 
\end{proposition}

\begin{proof} The proof is very analogous to the unit circle case \cite[Theorem~4.1]{HH22}.
In that proof we can use \cite[Theorem 3.8]{BNT02} instead of \cite[Lemma 4.4]{HH22}; the details are omitted. 
\end{proof}

\begin{proposition}\label{prop:BLK}
For a $\uplus$-CH $(\boole_{s,t})_{s\le t}$ on $\R$ there exists a unique family $(\genc_t,\genm_t)_{t}$ satisfying conditions \eqref{LK1} and \eqref{LK2} in Proposition~\ref{prop:LK_repr} and  
\begin{equation}\label{eq:BALKT}
K[\boole_{0,t}](z) = \genc_t +  \int_\R \frac{1+x z}{z-x}\,\genm_t(dx). 
\end{equation}
Conversely, given a family $(\genc_t,\genm_t)_{t}$ satisfying conditions \eqref{LK1} and \eqref{LK2} in Proposition~\ref{prop:LK_repr} there exists a unique $\uplus$-CH $(\boole_{s,t})_{s\le t}$ on $\R$ for which \eqref{eq:BALKT} holds. We call $(\genc_t,\genm_t)_{t}$ the \emph{generator} of $(\boole_{s,t})_{s\le t}$. 
\end{proposition}

\begin{proof}
The proof is almost identical to the free case. We only need to use \cite[Lemma 2.6]{FHS20} instead of \cite[Theorem 3.8]{BNT02}. 
\end{proof}

\begin{remark}\label{re:reduced_free} 
If $\free_{0,t}$'s have finite second moment in Proposition \ref{thm:f-c-real}, then $\genm_t$'s have finite second moment too \cite[Theorem 6.2]{Maa92}, and \eqref{eq:FALKT} can be reduced to
\begin{equation}\label{eq:FALKT2}
\phi[\free_{0,t}](z) = m_t +  \int_\R\frac{1}{z-x}\,\rgen_t(dx).   
\end{equation}
We shall refer to $(m_t,\rgen_t)_t$ as the reduced generator of $\boxplus$-CH $(\free_{s,t})_{s\le t}$. A similar remark applies to Proposition \ref{prop:BLK}. 
\end{remark}

Through the common parametrization $(\genc_t,\genm_t)_{t}$ one can define bijections between $\ast$-, $\boxplus$-, $\uplus$-, $\rhd$- and $\lhd$-CHs, which generalize the known Bercovici--Pata bijections \cite{AW14,BP99}. 
These bijections are homeomorphic with respect to the topology of locally uniform weak convergence, which is subsumed in Theorems~\ref{th:conv_classical}, \ref{th:conv_P-REF_gen} and the following two results. Recall that the convergence of generators has several equivalent formulations (see Proposition~\ref{prop:two_sigmas}), which we do not repeat here.

\begin{theorem}\label{th:luc_boole}
Let $\AC$ be a fixed subset of $\C^+$ that contains an accumulation point in $\C^+$. 
For each $n\in\N$, let $(\boole_{s,t}^n)_{s\le t}$ be the  $\uplus$-CH having generator $(\genc_t^n,\genm_t^n)_{t}$.
Similarly, let $(\boole_{s,t})_{s\le t}$ be the $\uplus$-CH with generator $(\genc_t,\genm_t)_{t}$. 
Then the following conditions are equivalent:
\begin{enumerate} 
\item\label{item:conv_boole}
$(\boole_{s,t}^n)$ converges weakly to $ (\boole_{s,t})$ locally uniformly on $I^2_\le$;
\item \label{i:luc_rEnergy}
$K[\boole_{0,t}^n](z)$ converges to $K[\boole_{0,t}](z)$ locally uniformly on $\C^+\times I$; 
\item \label{i:pwc_rEnergy}
$K[\boole_{0,t}^n](z)$ converges to $K[\boole_{0,t}](z)$ locally uniformly on $I$ for each $z\in \AC$;
\item\label{item:conv3}
$(\genc_t^n)_t$ converges to $(\genc_t)_t$ locally uniformly on $I$ and $(\genm_t^n)_t$ converges weakly to $(\genm_t)_t$ locally uniformly on $I$. 
\end{enumerate}
\end{theorem}

\begin{proof} 
The equivalence of \eqref{item:conv_boole}, \eqref{i:luc_rEnergy}, and \eqref{i:pwc_rEnergy} follows from Corollary \ref{cor:luc_rCauchy} and the definition of energy function \eqref{eq:K-trans}.
The implication \eqref{item:conv3} $\Rightarrow$ \eqref{i:pwc_rEnergy} is clear from the defining formula \eqref{eq:BALKT}.
It remains to show the implication \eqref{i:luc_rEnergy} $\Rightarrow$ \eqref{item:conv3}.
By Franz, Hasebe and Schlei{\ss}inger~\cite[Lemma~2.6]{FHS20}, condition \eqref{i:luc_rEnergy} implies that $\genc_t^n$ converges to $\genc_t$ and $\genm_t^n$ to $\genm_t$ weakly for every $t\in I$.
The latter convergence is, in fact, locally uniform on $I$ (see Proposition~\ref{prop:two_sigmas}).
The identity $\genc_t^n = \Re[K[\boole_{0,t}^n](i)]$ also implies that $\genc_t^n$ indeed converges to $\genc_t$ locally uniformly on $I$, which completes the proof. 
\end{proof} 

The following generalizes  \cite[Corollary 3.9]{BNT02} on a characterization of the convergence of a sequence of freely
infinitely divisible distributions.

\begin{theorem}\label{thm:free}
Let $\AC$ be a fixed subset of $\C^+$ that contains an accumulation point in $\C^+$.
For each $n\in\N$,  let $(\free_{s,t}^n)_{s\le t}$ be the $\boxplus$-CH having generator $(\genc_t^n,\genm_t^n)_{t}$. Similarly let $ (\free_{s,t})_{s\le t}$ the $\boxplus$-CH with generator $(\genc_t,\genm_t)_{t}$.
Then the following conditions are equivalent: 
\begin{enumerate}
\item\label{item:conv_free}
$(\free_{s,t}^n)$ converges weakly to $ (\free_{s,t})$ locally uniformly on $I^2_\le$;
\item \label{i:luc_rVoiculescu}
$\phi[\free^n_{0,t}](z)$ converges to $\phi[\free_{0,t}](z)$ locally uniformly on $\C^+\times I$;
\item \label{i:pwc_rVoiculescu}
$\phi[\free^n_{0,t}](z)$ converges to $\phi[\free_{0,t}](z)$ locally uniformly on $I$ for each $z\in \AC$;
\item\label{item:conv_generating}
$(\genc_t^n)_t$ converges to $(\genc_t)_t$ locally uniformly on $I$ and $(\genm_t^n)_t$ converges weakly to $(\genm_t)_t$ locally uniformly on $I$. 
\end{enumerate}
\end{theorem}

\begin{proof} 
Since \eqref{eq:FALKT} and \eqref{eq:BALKT} have the same expression on their right-hand side, the equivalence of \eqref{i:luc_rVoiculescu}, \eqref{i:pwc_rVoiculescu} and \eqref{item:conv_generating} follows from the same proof as that of Theorem~\ref{th:luc_boole}. 

Assuming \eqref{item:conv_free} we have a locally uniform convergence of inverse mappings $F[\free_{0,t}^n]^{-1}(w) \to F[\free_{0,t}]^{-1}(w)$ as described in Proposition~\ref{prop:luc_Voiculescu}.
On the other hand, $\phi[\free_{0,t}^n](w)=F[\free_{0,t}^n]^{-1}(w) -w$ and $\phi[\free_{0,t}](w)=F[\free_{0,t}]^{-1}(w)-w$ on the common domain, so that condition \eqref{i:pwc_rVoiculescu} holds. 

Assuming \eqref{i:luc_rVoiculescu} and \eqref{item:conv_generating} we prove \eqref{item:conv_free}.
It suffices to show condition \eqref{item1:luc_Voiculescu} in Proposition~\ref{prop:luc_Voiculescu}.
Taking a compact $K \subset I$ and $a>0$ we first establish
\begin{equation}\label{eq:varphi_infinity}
    \sup_{t \in K, n\in \N}|\phi[\free_{0,t}^n](z)| =o(z), \qquad z\to \infty, \ z \in \Gamma_{a,0}. 
\end{equation}

To derive \eqref{eq:varphi_infinity} observe that for $z \in \Gamma_{a,b}$ and $x \in \R$  
\begin{align*}
\left| \frac{1+x z}{z(z-x)} \right| &= \left| \frac{1+x z}{z(z-x)} +1 -1 \right| = \left| \frac{1+z^2}{z(z-x)}-1 \right| \le 1 + \left| \frac{z}{z-x}\right| \left(1+ \frac1{|z|^2} \right) \\
& \le 1 + \frac{\sqrt{1+a^2}}{a} \left( 1+\frac1{b^2}\right) =:C  
\end{align*}
and for $z \in \Gamma_{a,b}$ and $x \in [-R,R]$ 
\[
\left| \frac{1+x z}{z(z-x)} \right|\le  \frac{1+ R |z|}{|z|\Im z} \le \frac{1+ R \sqrt{1+a^{-2}}\Im z}{(\Im z)^2},  
\]
so that for all $n\in\N$, $t \in K$ and $z \in \Gamma_{a,b}$, 
\begin{align*}
\left|\frac{\phi[\free_{0,t}^n](z)}{z}\right| &\le \frac{|\genc_t^n|}{|z|}  + \int_{[-R,R]} \left| \frac{1+x z}{z(z-x)} \right| d\genm_t^n(x) + \int_{[-R,R]^c} \left| \frac{1+x z}{z(z-x)} \right| d\genm_t^n(x)\\
&\le  \frac{C'}{|z|} + C''\frac{1+ R \sqrt{1+a^{-2}}\Im z}{(\Im z)^2}  + C \genm_t^n([-R,R]^c), 
\end{align*}
where $C':= \sup_{t\in K, n \in \N} |\genc_t^n| <\infty$ and $C'' :=  \sup_{t\in K, n \in \N} \genm_t^n(\R) <+\infty$.
By the tightness of the family $\{\genm_t^n\mid t \in K, n\in \N\}$ we obtain \eqref{eq:varphi_infinity} as desired. 

Now \eqref{eq:varphi_infinity} readily implies that the mappings $H_t^n(w):= w + \phi[\free_{0,t}^n](w)$ and $H_t(w):= w + \phi[\free_{0,t}](w)$ satisfy $\sup_{n\in\N, t\in K}|H^n_t(w)-w|=o(w)$ and hence by an argument similar to the last part of the proof of Lemma~\ref{lem:inverse2}, we can find $a'>a>0$ and $b'>b>0$ such that 
 the right inverses $ (H^n_t)^{-1}, H_t^{-1}\colon \Gamma_{a',b'}\to \Gamma_{a,b}$ exist.
These inverses are equal to $F[\free_{0,t}^n], F[\free_{0,t}]$, respectively, on $ \Gamma_{a',b'}$ by the identity theorem.
These observations guarantee condition \eqref{item1:luc_Voiculescu} in Proposition~\ref{prop:luc_Voiculescu}, as desired.  
\end{proof}

\subsection{Infinitesimal growth of moments under the bijection}
\label{subsec:interpret_bij_free_etc}

Analogous to the monotone case studied in Proposition \ref{prop:moment_growth}, we calculate infinitesimal growths of moments of free and boolean CHs and compare them with the classical case.
Then the bijection that we have been considering in Sections~\ref{subsec:luc_Voic-trans} and \ref{sec:c-f-real} turns out to identify the infinitesimal generators of the moments of two convolution hemigroups under suitable assumptions including differentiability.

Given a generator $(\genc_t,\genm_t)_{t}$, we denote by $(\classical_{s,t})$, $(\free_{s,t})$, and $(\boole_{s,t})$ the $\ast$-,   $\boxplus$- and $\uplus$-CHs characterized by \eqref{eq:CLK}, \eqref{eq:FALKT}, and \eqref{eq:BALKT},  respectively. 
We further assume that $\genm_t$ has finite moment of every order and the mappings $t\mapsto \Mean_n(\genm_t)~(n\ge0)$ and $t\mapsto \genc_t$ are all differentiable.
We shall prove that the following identities hold: 
\begin{equation}\label{eq:c-f-moments}
 \left. \frac{\partial}{\partial t}\right|_{t=s}\Mean_n(\classical_{s,t})=  \left. \frac{\partial}{\partial t}\right|_{t=s} \Mean_n(\free_{s,t}) =  \left. \frac{\partial}{\partial t}\right|_{t=s} \Mean_n(\boole_{s,t}), \qquad n \in \N,~ s\in I.    
\end{equation}
 The first identity is deduced as follows.
 By the definition of the Bercovici--Pata bijection \cite[Remarks 3.10 and 4.3]{BNT02}, the classical cumulants $K_n$ and free cumulants $R_n$ are identified: 
\[
K_n(\classical_{s,t})= R_n(\free_{s,t}). 
\]
By the moment-cumulant formulas (see, e.g., \cite{AHLV15}), for each $n \in \N$ there exists a polynomial $P_n(x_1,\dots, x_{n-1})$ without constant or linear terms (with the convention $P_1=0$) such that 
\begin{equation}\label{eq:m-c}
 \Mean_n(\classical_{s,t}) = K_{n}(\classical_{s,t})  +P_n( K_{1}(\classical_{s,t}), \dots, K_{n-1}(\classical_{s,t})); 
\end{equation}
note that this $P_n$ is the same as that appeared in Section~\ref{subsec:generator_moments} and $K_n(\classical_{s,t})= \Mean_{n-2}(\kappa_t-\kappa_s)$. 
The hemigroup property and the additivity of cumulants imply that $K_n(\classical_{s,t})= K_n(\classical_{0,t}) - K_n(\classical_{0,s})$. Since $K_n(\classical_{0,t})$ can be written as 
\begin{align*}
&K_1(\classical_{0,t}) = \genc_t + \int_\R x \,\genm_t(dx), \\
& K_n(\classical_{0,t}) = \int_\R x^{n-2}(1+x^2) \,\genm_t(dx), \qquad n \ge2, 
\end{align*}
 it is differentiable in $t$ and hence $K_n(\classical_{s,t})=O(|t-s|)$ as $t\to s$. This fact and \eqref{eq:m-c} show that
\[
 \left. \frac{\partial}{\partial t}\right|_{t=s} \Mean_n(\classical_{s,t}) = \left.\frac{\partial}{\partial t}\right|_{t=s}K_n(\classical_{s,t}). 
\]
A similar argument shows that
\[
 \left. \frac{\partial}{\partial t}\right|_{t=s} \Mean_n(\free_{s,t}) = \left. \frac{\partial}{\partial t}\right|_{t=s}R_n(\free_{s,t}), 
\]
and hence the first identity in \eqref{eq:c-f-moments} holds. 
The second one is derived in a similar manner.

\subsection{Subordination property of free CHs} \label{subsec:embed}

We have seen in the paragraph before Theorem~\ref{th:luc_boole} that $\boxplus$- and $\rhd$-CHs are associated with each other via the bijection through generators.
In addition to this relation, another Loewner-theoretic aspect of $\boxplus$- and $\rhd$-CHs was observed by Franz~\cite[Corollary~5.3]{Fra09b} after contributions by Biane~\cite{Bia98} and Bauer~\cite{Bau04}: for a $\boxplus$-CH $(\free_{s,t})$ and for each $s \le t$, there exists a probability measure $\tilde{\monotone}_{s,t}$ such that 
\begin{equation}\label{eq:free_monotone}
\free_{0,t} = \free_{0,s} \rhd \tilde{\monotone}_{s,t}, 
\end{equation} 
or equivalently, $F[\free_{0,t}] = F[\free_{0,s}] \circ F[\tilde{\monotone}_{s,t}]$.  
This is a consequence of a more general subordination property for free convolution; see Biane~\cite{Bia98}, Belinschi and Bercovici~\cite{BB07}, and also the original article of Voiculescu~\cite{V93}.  
Since $F[\free_{0,t}]$ is univalent on $\C^+$ (see, e.g., \cite[Proposition~6.20]{FHS20}; recall that all $\free_{s,t}$ are $\boxplus$-infinitely divisible), the relation \eqref{eq:free_monotone} implies that $(F[\free_{0,t}])_{t}$ is a $\PickA$-DLC and $(F[\tilde{\monotone}_{s,t}])_{s\le t}$ gives the associated REF.
These phenomena were studied in the literature \cite{Bia98} \cite{Bia19} \cite[{\S}3.5]{Sch17} \cite[{\S}4.7]{FHS20}.
In particular, in Biane~\cite{Bia98}, a process with free increments having laws $(\free_{s,t})_{0\le s\le t <+\infty}$ is called a free additive L\'evy process of the second kind if the $\rhd$-CH $(\tilde{\monotone}_{s,t})_{0\le s \le t <+\infty}$ in \eqref{eq:free_monotone} is time-homogeneous. 
This class of processes can be characterized by the ``free nonlinear L\'evy--Khintchine formula'' \cite{Bia19}. 

Below we discuss the general time-inhomogeneous case.
For computational simplicity, we first assume that the above $\free_{s,t}$'s have finite second moment.
This implies that $\tilde{\monotone}_{s,t}$ also have finite second moment, and so the REF $(\tilde{\monotone}_{s,t})_{s\le t}$ has a reduced generator.
We shall relate it to the (reduced) generator of the $\boxplus$-CH $(\free_{s,t})$.  

\begin{theorem}\label{thm:fCH_mCH}
Let $(\free_{s,t})_{s\le t}$ be the $\boxplus$-CH with finite second moment having reduced generator $(m_t,\rgen_t)_t$ (see Remark \ref{re:reduced_free}) and let $(\tilde m_t, \tilde \rgen_t)_t$ be the reduced generator of the $\rhd$-CH $(\tilde{\monotone}_{s,t})_{s\le t}$ defined by \eqref{eq:free_monotone}. Then 
\begin{align}
&\tilde m_t = m_t, \label{eq:free-subord_mean} \\
&\tilde \rgen_t(B) =  \int_{\R\times [0,t]} (\delta_x \rhd \free_{0,s})(B) \,\rGen(dx\,ds), \qquad t\in I,\ B\in \cB(\R),\label{eq:embed_G}  
\end{align}
where $\rGen$ is the measure on $\R\times I$ determined by $(\rgen_t)$ through \eqref{eq:two_thetas}--\eqref{LK3_theta}.
\end{theorem} 

\begin{proof}
Formula \eqref{eq:free_monotone} implies that $\Mean(\tilde{\monotone}_{s,t}) = \Mean(\free_{0,t}) - \Mean(\free_{0,s}) = m_t - m_s$. 
Therefore, $\tilde m_t = \Mean(\tilde{\monotone}_{0,t}) = m_t$. 

To prove \eqref{eq:embed_G} on $\tilde \rgen_t$, it is convenient to first assume smoothness on generators and later remove the constraint by approximation. 

\smallskip\noindent
\textit{Smooth case}.
Suppose that 
\[
m_t = \int_0^t \dot{m}_s \, ds \qquad \text{and}\qquad \rgen_t(B)= \int_{B \times [0,t]} k(x,s) \,dx\,ds, \qquad B\in \cB(\R),\ t\in I,
\]
for some continuous functions $t\mapsto \dot{m}_t \in \R$ and $k\colon \R \times I \to [0,\infty)$. The latter means $\rGen(dx\, dt) = k(x, t) \,dx\,dt$. 

Let $f_{s,t} := F[\tilde{\monotone}_{s,t}]$,  $f_t := f_{0,t}$, and $\phi_t := \phi[\free_{0,t}]$ for short. 
Differentiating the formula $\phi_t(z) = f_t^{-1}(z) - z$ (this formula holds in a truncated cone that can be chosen uniformly on each compact interval; see Proposition~\ref{prop:luc_Voiculescu}),  we obtain
\[
\frac{\partial }{\partial t} f_t(z) = - \dot{\phi}_t(f_t(z)) \frac{\partial}{\partial z}f_t(z). 
\]
Compared with \eqref{eq:chordal_LDE}, the Herglotz vector field for $(f_t)$ is thus identified with $p(z,t):=- \dot{\phi}_t(f_t(z))$. 
We see from \eqref{eq:FALKT2} that 
\begin{align}
-p(z,t) 
 &=  \dot{m}_t + \int_{\R} \frac{1}{f_t(z)-x}\,k(x,t)\,dx =  \dot{m}_t + \int_\R G_{\delta_x \rhd \free_{0,t}}(z) \,k(x,t)\,dx \notag \\
&=   \dot{m}_t +  \int_\R \left[\int_{\R} \frac{1}{z-y} \,(\delta_x \rhd \free_{0,t})(dy)\right] k(x,t)\,dx =   \dot{m}_t +  \int_{\R} \frac{1}{z-y} \,\dot{\tilde\rgen}_t(dy), \label{eq:calc_p} 
\end{align}
where $\dot{\tilde\rgen}_t$ is a finite Borel measure defined by 
\[
\dot{\tilde\rgen}_t(B) :=  \int_\R (\delta_x \rhd \free_{0,t})(B) k(x,t)\,dx, \qquad B\in \cB(\R). 
\]
Then 
\begin{align*}
\tilde \rgen_t(B) := \int_0^t \dot{\tilde\rgen}_s(B)\,ds &= \int_{\R\times [0,t]} (\delta_x \rhd \free_{0,s})(B) k(x,s)\,dx\,ds \\
&=  \int_{\R\times [0,t]} (\delta_x \rhd \free_{0,s})(B) \,\rGen(dx\,ds),
\end{align*}
as desired (cf.~Example~\ref{rem:AC}).  

\smallskip\noindent
\textit{General case}.
To derive the formulae \eqref{eq:free-subord_mean} and \eqref{eq:embed_G} for a general reduced generator $(m_t,\rgen_t)_t$, it suffices to consider the case that $I$ is compact, because the problem is local.
We take a sequence of reduced generators $(m_t^n,\rgen_t^{n})_{t \in I}$, $n=1,2,3,\dots,$ each of which satisfies the smoothness condition as above and such that $m_t^n$ converges to $m_t$ uniformly on $I$ and $\rgen_t^n$ converges vaguely to $\rgen_t$ uniformly on $I$ and, for every $t \in I$, $\sup_{n\in\N}\rgen_t^n(\R)<+\infty$.
(For example, one can use a mollifier to smooth the finite measure $\rGen$ on $\R\times I$.)
As \eqref{eq:embed_G} is valid for this approximating sequence, taking the Cauchy transform yields (with the help of calculations in \eqref{eq:calc_p}) 
\[
\int_\R \frac{1}{z-y}\, \tilde \rgen_t^n(dy) =  \int_{\R\times [0,t]} \frac1{f_s^n(z)-x}\, \rGen^n(dx\,ds). 
\]
We can pass to the limit in this identity, for $\rgen_t^n$, $f_s^n$, and $\rGen^n$ all converge as $n \to \infty$ by the following observations:
\begin{itemize}
\item by Theorem \ref{thm:free}, $\tilde{\monotone}_{0,t}^n =\free_{0,t}^n$ converges weakly to $\tilde\monotone_{0,t}=\free_{0,t}$ uniformly on $I$, and then by Theorem \ref{thm:main_conv}, $\tilde\rgen_t^n$ converges vaguely to $\tilde \rgen_t$ uniformly;

\item by Theorem \ref{thm:free}, for each $z\in \C^+$, $f_t^n(z)$ converges to $f_t(z)$ uniformly on $ I$, and so the convergence $
  [f_s^n(z)-x]^{-1} \to [f_s(z)-x]^{-1}$ 
holds uniformly on $(x,s)\in \R\times I$;

\item since $\rgen_t^n$ converges vaguely to $\rgen_t$  uniformly on $I$, Proposition~\ref{prop:two_thetas} yields that $\rGen^n$ converges vaguely to $\rGen$. 
\end{itemize}
Thus we get  
\begin{equation}\label{eq:approx1}
\int_\R \frac{1}{z-y}\, \tilde \rgen_t(dy) =  \int_{\R\times [0,t]} \frac1{f_s(z)-x}\, \rGen(dx\,ds). 
\end{equation}
The right-hand side of \eqref{eq:approx1} is further calculated as 
\[
\int_{\R\times [0,t]} G_{\delta_x \rhd \free_{0,s}} (z) \,\rGen(dx\, ds) = \int_{\R\times [0,t]}\left[ \int_\R \frac1{z-y}\, (\delta_x \rhd \free_{0,s})(dy)\right] \rGen(dx\,ds).  
\]
\eqref{eq:embed_G} now follows from the Stieltjes inversion formula~\eqref{eq:PN_interval}. 
\end{proof}

Extending the arguments above, we generalize the result in the case without finite second moment. 

\begin{theorem}
Let $(\free_{s,t})_{s\le t}$ be the $\boxplus$-CH with generator $(\gamma_t,\genm_t)_t$ and let $(\tilde \gamma_t, \tilde \genm_t)_t$ be the generator of the $\rhd$-CH $(\tilde{\monotone}_{s,t})_{s\le t}$ defined by \eqref{eq:free_monotone}. Then 
\begin{align}
&\tilde \gamma_t = \gamma_t+ \int_{\R\times[0,t]}\left[\int_\R\left(x- \frac{y(1+x^2)}{1+y^2}\right) (\delta_x \rhd \free_{0,s})(dy)\right] \Genm(dx\,ds) , \label{eq:free-subord_mean2} \\
&\tilde \genm_t(B) =  \int_{\R\times [0,t]}\left[ \int_B \frac{1+x^2}{1+y^2}  (\delta_x \rhd \free_{0,s})(dy)\right] \Genm(dx\,ds), \qquad t\in I,\ B\in \cB(\R),\label{eq:embed_G2}  
\end{align}
where $\Genm$ is the measure on $\R\times I$ determined by $(\genm_t)$ through \eqref{eq:two_sigmas}--\eqref{eq:sigma2}. 
\end{theorem} 

\begin{proof} We only discuss the main difference from the proof of Theorem \ref{thm:fCH_mCH}. We consider the case $\gamma_t = \int_0^t \dot\gamma_s\,ds$ and $\Genm(dx\,dt) = \eta(x,t)\,dx\,dt$, where $\dot\gamma_t$ and $\eta(x,t)$ are continuous. A straightforward calculation of $-p(z,t)=\dot{\phi}_t(f_t(z))$ amounts to the desired formulas \eqref{eq:free-subord_mean2} and \eqref{eq:embed_G2}. One needs to pay attention that these integrals make sense. Indeed, one can see that 
\begin{align*}
\int_{\R}\left(x- \frac{y(1+x^2)}{1+y^2}\right) (\delta_x \rhd \free_{0,s})(dy) 
&= x + (1+x^2)\Re[G[\delta_x\rhd \free_{0,s}](i)] \\
&= x + (1+x^2) \frac{\Re[f_s(i)]-x}{|f_s(i)-x|^2} 
\end{align*}
and 
\[
 \int_{\R} \frac{1+x^2}{1+y^2}  (\delta_x \rhd \free_{0,s})(dy) = - (1+x^2)\Im[G_{\delta_x\rhd \free_{0,s}}(i)]  = \frac{(1+x^2)\Im[f_s(i)]}{|f_s(i)-x|^2}
\]
are both bounded continuous functions of $(x,s) \in \R \times [0,t]$, so that the integrals in \eqref{eq:free-subord_mean2} and \eqref{eq:embed_G2} exist finitely. The rest of the proof is similar to Theorem \ref{thm:fCH_mCH}. 
\end{proof}

\section{Measures with compact support and reciprocal Cauchy transforms}
\label{sec:univ_Cauchy}

We know very well that the normal (Gaussian) distribution, which has unbounded support, plays a central role in classical probability.
In contrast, typical distributions in non-commutative probability, including the ``normal'' ones, are known to have compact support.
In this last section, we first study a $\rhd$-CH $(\monotone_{s,t})_{s \le t}$ with the support of $\monotone_{s,t}$ bounded (locally) uniformly in $(s,t)$.
The reciprocal Cauchy transform $f_{s,t}=F[\monotone_{s,t}]$ is then univalent as in Section \ref{sec:mono-CHs}, and in addition $\C^+ \setminus f_{s,t}(\C^+)$ is bounded, as seen below.
Bearing these properties in mind, we conclude this paper with a brief survey through Sections \ref{sec:transfinite_diameter}--\ref{sec:hcap_area} on some geometric properties of a univalent reciprocal Cauchy transform $F_\mu$ (mainly in the case that $\C^+ \setminus F_\mu(\C^+)$ is bounded).

\subsection{Compactly supported monotone CHs}
\label{sec:Q-ALC}

As usual, for a Borel measure $\mu$ on $\R$, let $\supp\mu$ denote its (topological) support
\[
\supp\mu:=\{\, x \in \R : \mu(N(x))>0\ \text{for every neighborhood}\ N(x)\ \text{of}\ x \,\},
\]
which is closed in $\R$.
An easy consequence of the Stieltjes inversion formula \eqref{eq:PN_interval} is that, for a Pick function $f$ with Nevanlinna representation \eqref{eq:PNrep}, the following are equivalent for each $a<b$:
\begin{itemize}
\item $f$ extends to a continuous function on $\overline{\C^+} \setminus [a,b]$ which takes real values on $\R \setminus [a,b]$;
\item $\supp \rho \subset [a,b]$.
\end{itemize}
Suppose that these properties hold.
Then, by the latter we have
\[
\frac{1}{1+\max\{ a^2, b^2 \}} \rho([a,b]) \le \int_{[a,b]}\frac{1}{1+x^2}\,\rho(dx)<+\infty,
\]
which means that $\rho$ is a finite measure.
By the former property, $f$ extends to a holomorphic function on $\C \setminus [a,b]$ by the Schwarz reflection.
If further $f^\prime(\infty)=1$, i.e., $f=F_\mu$ for some $\mu \in \prob$, then $f$ belongs to $\PickB$ since $f$ has a representation of the form \eqref{eq:PNrep2} with $\beta=1$.
Moreover, it has the genuine Laurent expansion
\begin{equation} \label{eq:F-trans_Lau}
f(z)=F_\mu(z)=z-\Mean(\mu)-\frac{\var(\mu)}{z}+\sum_{n=2}^\infty \frac{a_{-n}}{z^n},
\qquad \lvert z \rvert>\max\{\lvert a \rvert, \lvert b \rvert\},
\end{equation}
with all the $a_{-n}$'s real.

\begin{lemma} \label{lem:cpt_supp}
Let $\mu \in \mea\setminus\{0\}$.
For $a<b$, the following are equivalent:
\begin{enumerate}
\item \label{i:supp_is_cpt}
$\supp\mu \subset [a,b]$;
\item \label{i:G_is_extdd}
$G_\mu$ extends to a continuous function on $\overline{\C^+} \setminus [a,b]$ which takes real values on  $\R \setminus [a,b]$;
\item \label{i:F_has_nonzero_extn}
$F_\mu$ extends to a continuous function on $\overline{\C^+} \setminus [a,b]$ which takes nonzero real values on $\R \setminus [a,b]$.
\end{enumerate}
\end{lemma}

\begin{proof}
$\eqref{i:supp_is_cpt} \Leftrightarrow \eqref{i:G_is_extdd}$ is just a particular case of the above-mentioned consequence of the inversion formula.
$\eqref{i:F_has_nonzero_extn} \Rightarrow \eqref{i:G_is_extdd}$ is obvious as well.
To see $\eqref{i:G_is_extdd} \Rightarrow \eqref{i:F_has_nonzero_extn}$, suppose that \eqref{i:G_is_extdd} holds.
Since it implies \eqref{i:supp_is_cpt}, $G_\mu$ never takes zero on $\R \setminus [a,b]$ by
\[
G_\mu(\xi)=\int_{[a,b]}\frac{1}{\xi-x}\,\mu(dx)
\begin{cases}
\le -(b-\xi)^{-1}\mu(\R)<0 & \text{if}\ \xi<a, \\
\ge (\xi-a)^{-1}\mu(\R)>0 & \text{if}\ \xi>b. \\
\end{cases}
\]
Hence \eqref{i:F_has_nonzero_extn} follows.
\end{proof}

\begin{proposition} \label{thm:cpt_supp}
Let $(\monotone_{s,t})_{s \le t}$, $(f_t)_t$, and $(\gamma_t, \genm_t)_t$ be a $\rhd$-CH, $\PickA$-DLC, and generator, respectively, associated with each other.
Then the following are equivalent:
\begin{enumerate}
\item \label{item:luc_p.m.}
for each $T \in I$ there exists $R=R_T>0$ such that $\bigcup_{0 \le s \le t \le T}\supp\monotone_{s,t} \subset [-R, R]$;
\item \label{item:luc_reciprocal_cauchy}
for each $T \in I$ there exists $R=R_T>0$ such that, for every $t \le T$, $f_t$ extends to a continuous function on $\overline{\C^+} \setminus [-R, R]$ which takes real values on $\R \setminus [-R, R]$;
\item \label{item:luc_generator}
for each $T \in I$ there exists $R=R_T>0$ such that $\bigcup_{0 \le t \le T}\supp\genm_t \subset [-R, R]$.
\end{enumerate}
\end{proposition}

\begin{proof}
$\eqref{item:luc_p.m.} \Rightarrow \eqref{item:luc_reciprocal_cauchy}$.
This follows immediately from Lemma~\ref{lem:cpt_supp}.

\smallskip\noindent
$\eqref{item:luc_reciprocal_cauchy} \Rightarrow \eqref{item:luc_p.m.}$.
Assume \eqref{item:luc_reciprocal_cauchy} and let $T>0$.
Then, as mentioned before Lemma~\ref{lem:cpt_supp}, $(f_t) \subset \mathcal{P}^2$, and in the Nevanlinna representation 
\begin{equation}\label{eq:repr}
f_{t}(z) = z -m_t + \int_\R \frac{1}{x-z}\, \rho_{t}(dx), 
\end{equation}
there exists $R=R_T>0$ such that $\bigcup_{0 \le t \le T}\supp \rho_t \subset [-R,R]$.
In addition, $t\mapsto m_t \in \R$ is continuous and $t\mapsto \rho_t \in\mea$ is weakly continuous by Theorem \ref{th:capacity_conti_prime}.
We set $C:=\sup_{t \in [0,T]}\rho_t(\R)$, $M:=\sup_{t \in [0,T]} |m_t|$, and $R^\prime:= R+C+2M+1$.
For $0\le s \le t \le T$ we have on the one hand 
\begin{equation} \label{eq:one-hand}
f_{t}(R^\prime)\ge R^\prime-M - \int_\R \frac{1}{R^\prime-\xi}\, \rho_t(d\xi) \ge  R^\prime -M   - \frac{C}{R^\prime-R} > R^\prime-M -1
\end{equation}
and on the other hand  
\begin{equation} \label{eq:other-hand}
f_{s}(R+0)\le R + M <R^\prime-M-1.   
\end{equation}
Here, $f_s$ and $f_t$ are strictly increasing and injective\footnote{%
This simple behavior of the boundary values is also a general consequence of the sense-preservation and boundary correspondence by the conformal mappings $f_s$ and $f_t$.}
on $(R,+\infty)$, which follows again from \eqref{eq:repr}.
The inequalities \eqref{eq:one-hand} and \eqref{eq:other-hand} then imply that $f_{s,t}=f_s^{-1} \circ f_t\colon (R^\prime,+\infty)\to (R,+\infty)$ is well-defined. 
Similarly, $f_{s,t} \colon (-\infty, -R^\prime) \to (-\infty, -R)$ is well-defined, and consequently $f_{s,t}$ is holomorphic and has no zeros on $\C\setminus [-R^\prime,R^\prime]$.
Lemma \ref{lem:cpt_supp} now yields the desired property \eqref{item:luc_p.m.}.

\smallskip\noindent
$\eqref{item:luc_reciprocal_cauchy} \Rightarrow \eqref{item:luc_generator}$.
Assume \eqref{item:luc_reciprocal_cauchy}.
We have already seen in the proof of $\eqref{item:luc_reciprocal_cauchy} \Rightarrow \eqref{item:luc_p.m.}$ that $f_t \in \PickB$.
By Proposition \ref{prop:LIE_PickB}, the DLC $(f_t)$ has the reduced generator $((m_t), (\rgen_t))$ defined by \eqref{eq:gen_rgen}, \eqref{eq:two_thetas}, and \eqref{LK3_theta}.
As $\supp\genm_t=\supp\rgen_t$, it suffices to show that $\cup_{0 \le t \le T}\supp\rgen_t$ is bounded.
We consider the $\PickC$-DLC $g_t(z):=f_t(z+m_t)$, the time-change $\tau$ in the proof of Theorem \ref{th:DLC_to_LIE}, and the time-changed DLC $\tilde g_u := g_{\tau(u)}$.
We then have the differential equation 
\begin{equation} \label{eq:LDE_cpt_supp}
\partial_u \tilde g_u(z)
= \frac{\partial \tilde g_u(z)}{\partial z}q(z,u),
\qquad  u \in J\setminus N, \ z\in \C\setminus \R,
\end{equation}
for some Herglotz vector field $q$ of the form \eqref{eq:reduced_HVF} and some $N$ of Lebesgue measure zero.
Differentiating $\tilde g_{u,v}:=\tilde g_u^{-1}\circ \tilde g_v$ in $v$ at $v=u$, we have
\begin{equation}\label{eq:q}
\lim_{h \to +0}\frac{\tilde g_{u,u+h}(z) - z}{h} = q(z,u), \qquad z\in \C\setminus \R, \ u \in J\setminus N. 
\end{equation}
Using this formula, we shall provide an analytic continuation of $z \mapsto q(z,u)$ across the real axis.
By the assumption \eqref{item:luc_reciprocal_cauchy} and the construction of $\tilde{g}_u$,
for each compact subinterval $[0,U]$ of $J$, there exists $R>1$ such that $\tilde g_u$ extends to a holomorphic function on $\C\setminus [-R, R]$ which takes real values on $(-\infty, -R) \cup(R,\infty)$.
In addition, modifying the proof of Proposition~\ref{prop:class_P}~\eqref{HN:ineq} slightly we have
\[
|\tilde g_{u,u+h}(z)- z| \le \frac{h}{|z| - R }, \qquad |z| >3R, \ h\in(0,1). 
\]
By this inequality, $\{(\tilde g_{u,u+h}- \text{id})/h\}_{h \in (0,1)}$ is bounded in $|z|>3R$.
Hence the convergence \eqref{eq:q} is locally uniform in $|z|>3R$ by Vitali's convergence theorem.
This limit gives an analytic continuation of $z \mapsto q(z, u)$  taking real values in $(-\infty, -3R) \cup (3R,\infty)$.
In particular, the driving kernel $(\nu_u)_{u\in J}$ of $q$ is supported on $[-3R,3R]$ for all $u \in [0,U] \setminus N$. 

Considering \eqref{eq:Qq} and $q(z,u)=-G[\nu_u](z)$, the measure
\[
\varSigma(A\times B) = \int_{\tau^{-1}(B)}\nu_u(A) \, du     
\]
is the compound driving measure of $Q(z,B):= \int_{\tau^{-1}(B)} q(z,u)\,du$ associated with $(g_t)$. 
Moreover, from Definition \ref{def:monotone_MGF} we see  $\rgen_t(A) =  \varSigma(A_t)$, where $A_t := \{(x,s)\in \R \times I: (x +m_s, s) \in A \times [0,t] \}$.
Therefore, putting $M:=\sup_{t \in [0,T]} |m_t|$ we have $\rgen_t(\R \setminus[-3R-M,3R+M])=0$ for all $t\in [0,\tau(U)]$, as desired. 

\smallskip\noindent
$\eqref{item:luc_generator}  \Rightarrow   \eqref{item:luc_reciprocal_cauchy}$.
Assume $\eqref{item:luc_generator}$.
Again by Proposition \ref{prop:LIE_PickB} we can use the reduced generator $((\mean_t), (\rgen_t))$ with $\supp\rgen_t=\supp\genm_t$. 
As before, the essential part is to work on the DLC $(\tilde g_u)_{u \in J}$ that satisfies PDE \eqref{eq:LDE_cpt_supp}.
By the assumption \eqref{item:luc_generator}, for every $U\in J$ there exists $R>0$ such that $q(\cdot,u)$ is holomorphic in $\C\setminus[-R,R]$ for all $0 \le u \le U$. 
We suppose $R^2 \ge U$ for later use. 

By Proposition~\ref{th:chordal_LDE}, the PDE \eqref{eq:LDE_cpt_supp} implies that, for each $v \in [0,U]$,
\[
\partial_u \tilde g_{u,v}(z) = - q(\tilde g_{u,v}(z),u), \qquad  \text{for a.e.}\ u \in [0,v]\ \text{and for every}\ z\in \C\setminus\R,
\]
or 
\[
\tilde g_{u,v}(z) = z + \int_u^v q(\tilde g_{r,v}(z),r) \, dr, \qquad 0\le u\le v \le U,\ z\in \C\setminus \R. 
\]
To construct an analytic continuation of $\tilde g_{u,v}$ across the real axis, we go back to Picard's iteration: for a fixed $v\in[0,U]$ and $z \in \C\setminus [-R,R]$ we set $\kk^0_{u,v}(z):= z$ and for $n\ge1$
\begin{align*}
&\kk^n_{u,v}(z) := z + \int_u^v q(\kk^{n-1}_{r,v}(z),r) \, dr \\
&\text{as long as}\quad \kk_{r,v}^{n-1}(z) \in \C\setminus[-R,R] \quad \text{for all $u \le r  \le v$}.  
\end{align*}
As $\Im(q) \ge0$ on $\C^+$ we have $\Im(\kk_{u,v}^n(z)) >0$ on $\C^+$, and thus $\kk_{u,v}^n$ is well-defined on $\C^+$ for every $n\in\N$ and $u\in[0,v]$. 
The same holds on $-\C^+$ by complex conjugation. 
The problem is to show that $\kk_{u,v}^n$ is well-defined as a holomorphic function in $|z|>3R$, i.e., $\kk_{u,v}^n(z)$ never hits $[-R,R]$ if $\lvert z \rvert>3R$. 
Below we show it and also establish a uniform boundedness on suitable compact sets. 

Using the integral representation 
\[
q(z,v) = \int_{[-R,R]} \frac1{x-z} \,\nu_v(dx), \qquad |z|>R, 
\]
we deduce that $|q(z,v)| \le (|z|-R)^{-1}$ for all $|z|> R$.
Let $R' > R$. 
By induction on $n$ we can prove that
\begin{equation} \label{eq:sol_bound}
2R \le |\kk_{u,v}^n(z)|\le 4R'
\qquad \text{whenever}\ 3R \le |z|\le 3R'\ \text{and}\ 0\le u\le v \le U.
\end{equation} 
Indeed, assuming \eqref{eq:sol_bound} for $n-1$ we get the lower bound 
\[
|\kk_{u,v}^n(z)| \ge |z| - \int_u^v \frac{1}{|\kk_{r,v}^{n-1}(z)|-R} \, dr \ge 3R - \frac{U}{2R-R} \ge 3R - \frac{R^2}{2R} \ge 2R.
\]
The upper bound is similarly obtained:  
\[
|\kk_{u,v}^n(z)| \le |z| + \int_u^v \frac{1}{|\kk_{r,v}^{n-1}(z)|-R} \, dr \le 3R' + \frac{U}{2R-R} \le 4R'. 
\]
Hence \eqref{eq:sol_bound} holds for all $n$.
In particular, $\{\, \kk_{u,v}^n \mid 0\le u\le v \le U,\ n \in \N \,\}$ are well-defined holomorphic functions in $|z| >3R$.

Now the successive approximations $\kk_{u,v}^n(z)$ converge to the unique solution $\kk_{u,v}(z)$ of the original ODE as $n \to \infty$ for every fixed $z\in\C\setminus[-3R,3R]$ uniformly on $u\in[0,v]$, and obviously $\kk_{u,v}(z)=\tilde g_{u,v}(z)$ on $\C^+$.
Moreover, $\kk_{u,v}^n(z) \to \kk_{u,v}(z)$ locally uniformly on $\lvert z \rvert>3R$ by \eqref{eq:sol_bound} and Vitali's convergence theorem.
Hence the function $\kk_{u,v}$ is holomorphic there.
This implies that $\tilde g_{u,v}$ (and hence $\tilde g_v=\tilde g_{0,v}$) extends to the holomorphic function $\kk_{u,v}$ on $|z|>3R$ for all $0\le u\le v \le U$.
Translating this conclusion into the original DLC $(f_t)_t$ yields the desired property \eqref{item:luc_reciprocal_cauchy}. 
\end{proof}

\begin{remark}
We give two comments related to the proof above.
\begin{enumerate}
\item
Let $\hat{\C}=\C \cup \{\infty\}$, $\extD=\hat{\C} \setminus \overline{\D}$, and
\[
\Sigma=\{\, f \colon \extD \to \hat{\C} : f\ \text{is univalent},\ f(\infty)=\infty,\ f^\prime(\infty)=1 \,\}.
\]
In other words, the class $\Sigma$ is the set of univalent functions $f$ on $\extD$ with Laurent expansion
$f(z)=z+a_0+\sum_{n=1}^\infty a_{-n}z^{-n}$.
For such $f$, we have the following consequences of Gronwall's area theorem $\sum_{n=1}^\infty n \lvert a_{-n} \rvert^2 \le 1$:
\begin{gather}
\lvert z \rvert - \lvert f(z) \rvert
\le \lvert z-f(z) \rvert
\le \lvert a_0 \rvert + \frac{\lvert z \rvert}{\lvert z \rvert^2-1}; \label{eq:Sigma_ineq} \\
f(\extD) \supset 2\extD - a_0. \label{eq:f_contains_2extD}
\end{gather}
For \eqref{eq:f_contains_2extD}, see Goluzin \cite[Theorem~3, {\S}4, Chapter II]{Gol69} or Pommerenke \cite[Theorem~1.4]{Pom75}.
\eqref{eq:Sigma_ineq} follows from the Cauchy--Schwarz inequality:
\begin{align*}
\lvert f(z)-z \rvert
\le \lvert a_0 \rvert + \left\lvert\sum_{n=1}^\infty a_{-n}z^{-n}\right\rvert
\le \lvert a_0 \rvert + \left(\sum_{n=1}^\infty n\lvert a_{-n} \rvert^2 \right)^{\!\!\frac{1}{2}} \left(\sum_{n=1}^\infty \frac{n}{\lvert z \rvert^{2n}}\right)^{\!\!\frac{1}{2}}.
\end{align*}
In the proof of $\eqref{item:luc_reciprocal_cauchy} \Rightarrow \eqref{item:luc_p.m.}$ in Proposition~\ref{thm:cpt_supp},
let $L \ge R$.
Then the functions $z \mapsto L^{-1}f_s(Lz)$ and $z \mapsto L^{-1}f_t(Lz)$ belong to $\Sigma$, and hence, taking $L$ dependent on $R$ nicely, we can use \eqref{eq:Sigma_ineq} and \eqref{eq:f_contains_2extD} as alternatives of \eqref{eq:one-hand} and \eqref{eq:other-hand}, respectively.
Application of the property of $\Sigma$ of this kind is sometimes useful in the analysis of a chordal Loewner chain $(f_t)_t$ with $\C^+ \setminus f_t(\C^+)$ bounded (cf.\ del Monaco and Gumenyuk \cite{dMG16}).

\item
Consider the initial value problem for the ODE
\begin{equation} \label{eq:Cara_ODE}
\partial_u \kk_u(z)=-q(\kk_u(z), u), \qquad \kk_0(z)=z,
\end{equation}
as in the proof of $\eqref{item:luc_generator} \Rightarrow \eqref{item:luc_reciprocal_cauchy}$ in Proposition \ref{thm:cpt_supp}.
There we have chosen a proof based on Picard's iteration so that, on the way, $z \mapsto \kk_u(z)$ is proved to be holomorphic on $\C \setminus [-3R,3R]$ for each $u \le R^2$.
Given \textit{a priori} lower bound of $\lvert \kk_u(z) \rvert$ similar to \eqref{eq:sol_bound}, this fact is a well-known consequence of the theory of analytic ODEs, at least, in the case that $u \mapsto q(z,u)$ is continuous.
Although \eqref{eq:Cara_ODE} is interpreted in the $u$-a.e.\ sense, generalization is possible.
For concrete statements, see Contreras, D\'{\i}az-Madrigal and Gumenyuk \cite[\S2]{CDMG13}.
\end{enumerate}
\end{remark}

\subsection{Transfinite diameter and length of support}
\label{sec:transfinite_diameter}

We again recall that, for a $\rhd$-CH $(\monotone_{s,t})$ and fixed $t \in I$, the reciprocal Cauchy transform $f_t=F[\monotone_{0,t}]$ is univalent by Proposition~\ref{th:univalent_REF} and \ref{prop:additive_REF}.
Some kinds of converse are also known \cite[Theorems~3.1 and 3.2]{Bau05} \cite{FHS20};
for example, Franz, Hasebe and Schlei{\ss}inger \cite[Theorems~1.16 and 1.20]{FHS20} showed that, if $F_\mu$ is a univalent reciprocal Cauchy transform, then there exists a $\rhd$-CH $(\monotone_{s,t})_{0 \le s \le t \le T}$ such that $\mu=\monotone_{0,T}$.
We are thus interested in the univalence of reciprocal Cauchy transforms.
In the case that $\mu \in \prob$ has compact support, Bauer~\cite[\S3]{Bau05} mentioned two criteria for $F_\mu$ to be univalent: one by means of transfinite diameter and the other in terms of Grunsky coefficients.
In this subsection, we explain the former criterion for comparison with the contents in Sections~\ref{sec:var_relcap} and \ref{sec:hcap_area}.

\begin{definition}[transfinite diameter]
\label{def:transfinite_diameter}
Let $E \subset \C$ be a compact set.
The diameter of $E$ of order $n \ge 2$ is defined by
\[
d_n(E):=\max\left\{\, \prod_{1 \le i<j \le n}\lvert z_i-z_j \rvert^{\frac{2}{n(n-1)}} : z_1, z_2, \ldots, z_n \in E \,\right\}.
\]
It can be seen that the sequence $(d_n(E))_{n \ge 2}$ is non-increasing.
The limit $d_\infty(E):=\lim_{n \to \infty}d_n(E)$ is called the \emph{transfinite diameter} of $E$.
\end{definition}

Let $E \subset \C$ be a compact non-polar set and $G_D \colon D\times D \to (-\infty,+\infty]$ be the Green function for the connected component $D$ of $\hat{\C} \setminus E$ containing $\infty$.
The Robin constant of $D$ at $z \in D$ is defined by
\begin{equation} \label{eq:Robin}
\operatorname{Rob}(D,z)=\begin{cases}\displaystyle
\lim_{\substack{w \to z \\ w \neq z}}\left(G_D(w,z)+\log\lvert w-z \rvert\right) & \text{if}\ z \neq \infty, \\
\displaystyle \lim_{w \to \infty}\left(G_D(w,\infty)-\log\lvert w \rvert\right) & \text{if}\ z=\infty.
\end{cases}
\end{equation}
The \emph{logarithmic capacity} of $E$ is then given by $\operatorname{cap}(E)=\exp(-\operatorname{Rob}(D,\infty))$.
The following properties are known classically (see, e.g., Ahlfors~\cite[\S2]{Ahl73}):
\begin{itemize}
\item
The transfinite diameter and logarithmic capacity are equal: $d_\infty(E)=\operatorname{cap}(E)$.
\item
For $j=1,2$, let $E_j \subset \C$ be a compact set and $D_j$ be the connected component of $\hat{\C} \setminus E_j$ containing $\infty$.
If there exists a conformal mapping $f \colon D_1 \to D_2$ with $f^\prime(\infty)=\alpha$, i.e., with $f(z)=\alpha z+a_0+a_{-1}z^{-1}+\cdots$ around $z=\infty$, then $\operatorname{cap}(E_2)=\lvert \alpha \rvert \operatorname{cap}(E_1)$.
\item
$\operatorname{cap}(\D)=1$, and $\operatorname{cap}(J)=1/4$ for an interval $J$ of unit length.
\end{itemize}

Translating Hayman's result~\cite[Theorem~I]{Hay51} through the inversion $z \mapsto 1/z$, we have the following:

\begin{proposition} \label{th:Hayman51}
Suppose that a holomorphic function $f \colon \C \setminus \overline{\D} \to \C$ has Laurent expansion $f(z)=z+a_0+a_{-1}z^{-1}+\cdots$ around $z=\infty$.
Then $E=\C \setminus f(\C \setminus \overline{\D})$ satisfies $d_\infty(E) \le 1$, and the equality is achieved if and only if $f$ is univalent.
\end{proposition}

For $\mu \in \prob$ with $\supp\mu \subset [a,b]$, $F_\mu$ extends analytically across $\R \setminus [a,b]$ and has Laurent expansion \eqref{eq:F-trans_Lau}.
Hence Proposition~\ref{th:Hayman51}, combined with the properties listed above, yields the following:

\begin{proposition}[Bauer~{\cite[Corollary~3.1]{Bau05}}]
\label{prop:Bauer_cor3.1}
Let $\mu \in \prob$ and $a<b$ be such that $\operatorname{supp}\mu \subset [a,b]$.
Then the reciprocal Cauchy transform $F_\mu(z)$ is univalent if and only if $E=\C \setminus F_\mu(\C \setminus [a,b])$ satisfies $d_\infty(E)=(b-a)/4$\footnote{%
The original statement seems to lack the multiplicative constant $1/4$.}.
\end{proposition}

\subsection{Relative capacity and variance}
\label{sec:var_relcap}

Following Dubinin and Vuorinen~\cite{DV14}, we now introduce another ``capacity.''
Let $D \subset \hat{\C}$ be a domain whose boundary contains a free one-sided analytic arc\footnote{%
To be precise, it is a free one-sided, regular, simple, analytic boundary arc; for these terms, see Ahlfors~\cite[Chapter~6, \S1.3]{Ahl79}. The one-sided assumption is put just for simplicity.}
$\gamma$.
Since $\gamma$ is non-polar, $D$ has the Green function.
With $\operatorname{Rob}(D,z)$ given by \eqref{eq:Robin}, the \emph{inner radius} of $D$ at $z \in D$ is defined by $r_0(D,z)=\exp(\operatorname{Rob}(D,z))$.

\begin{lemma}[{\cite[Theorem~2.1]{DV14}}]
\label{lem:relcap_exist}
In the above, let $z_0 \in \gamma \cap \C$.
Suppose that $E$ is a relatively closed subset of $D$ and has positive inner distance to $z_0$ in $D$ (i.e., the infimum of the length over paths connecting $E$ and $z_0$ in $D$ is positive).
Then there exists $c \ge 0$ such that, for any path $\gamma_\perp$ in $D$ ending at $z_0$ and perpendicular to $\gamma$, the following asymptotic relation holds:
\begin{equation} \label{eq:def_relcap}
\frac{r_0(D \setminus E,z)}{r_0(D,z)}
=1-c\lvert z-z_0 \rvert^2+o(\lvert z-z_0 \rvert^2),
\quad z \to z_0,\ z \in \gamma_\perp.
\end{equation}
\end{lemma}

\begin{definition}[relative capacity~{\cite{DV14}}]
\label{def:relcap}
In Lemma~\ref{lem:relcap_exist}, the \emph{relative capacity} of $E$ with respect to $D$ at $z_0$ is defined by $\operatorname{relcap}(E;D,z_0):=c/2$.
If $z_0=\infty$, then the relative capacity is defined in the same way with $z-z_0$ in \eqref{eq:def_relcap} replaced by $1/z$.
\end{definition}

Relative capacity is related to Schwarzian derivative.
To see this, let $\mathbf{Dub}$ be the set of all holomorphic functions $g$ on $\D$ such that the asymptotic expansion
\begin{equation} \label{eq:Dubinin_class}
g(z)=1+a_1(z-1)+a_2(z-1)^2+a_3(z-1)^3+o((z-1)^3)
\end{equation}
holds as $z \to 1$ through every Stolz angle at $1$ with conditions $a_1>0$ and $\Re(2a_2+a_1(1-a_1))=0$.
The \emph{Schwarzian derivative} of $g$ at $z=1$ is defined by $S_g(1):= 6(a_1a_3-a_2^2)/a_1^2$.

\begin{proposition}[{\cite[Theorem~4.1]{DV14}}]
\label{th:DV14}
Suppose that $g \in \mathbf{Dub}$ has the asymptotic expansion \eqref{eq:Dubinin_class}, and let $E$ be a relatively closed subset of $\D$ which lies in $\D \setminus g(\D)$ and has positive inner distance to $1$ in $\D$.
Then the inequality
\begin{equation} \label{eq:DV14}
-\frac{1}{6}S_g(1) \ge a_1^2 \operatorname{relcap}(E; \D, 1)
\end{equation}
holds.
The equality is achieved if $g$ is a conformal mapping from $\D$ onto $\D \setminus E$.
\end{proposition}

\begin{remark}
Dubinin and Vuorinen proved Proposition~\ref{th:DV14} using a version of Proposition~\ref{th:Hayman51} in Hayman's book~\cite[Theorem~4.7]{Hay94}.
However, the authors do not know whether the equality in \eqref{eq:DV14} is also sufficient for $g$ to be conformal, because one has to take the limit to the boundary point $1$ after applying Hayman's result.
\end{remark}

Proposition~\ref{th:DV14} is applied to functions in $\PickC$ via the Cayley transform $T(z)=i(z+1)/(1-z)$ and the bijection $\operatorname{ad}T \colon \mathrm{Hol}(\C^+, \C^+) \to \mathrm{Hol}(\D, \D)$ defined by \eqref{eq:induced_Cayley}.
In fact, direct calculations show that
\[
\operatorname{ad}T(\PickC)=\{\, g \in \mathbf{Dub} : a_1=1\ \text{in the asymptotic \eqref{eq:Dubinin_class}} \,\}
\]
and that
\[
\AR(f)=\angle\!\lim_{z \to \infty}z(z-f(z))
=-\frac{2}{3}S_{\operatorname{ad}T(f)}(1),
\qquad f \in \PickC.
\]
Moreover, for a bounded, relatively closed subset $E$ of $\C^+$, we have
\[
\operatorname{relcap}(E; \C^+,\infty)=4\operatorname{relcap}(T^{-1}(E); \D,1).
\]
In summary, we obtain the following:

\begin{proposition} \label{th:DV14_half-plane}
Let $f \in \PickC$ and $E$ be a bounded, relatively closed subset of $\C^+$ such that $E \subset \C^+ \setminus f(\C^+)$.
Then the inequality
\[
\AR(f) \ge \operatorname{relcap}(E; \C^+,\infty)
\]
holds, and the equality is achieved if $f$ is a conformal mapping from $\C^+$ onto $\C^+ \setminus E$.
\end{proposition}

\begin{corollary} \label{cor:DV14_half-plane}
The variance of $\mu \in \prob$ is not less than the relative capacity of any bounded subset of $\C^+ \setminus F_\mu(\C^+)$.
In particular, if $\mu$ has compact support and if $F_\mu$ is univalent, then $\var(\mu)=\operatorname{relcap}(\C^+ \setminus F_\mu(\C^+); \C^+,\infty)$.
\end{corollary}

\subsection{Half-plane capacity and area of image complement}
\label{sec:hcap_area}

We now introduce one more concept of capacity, which is common in SLE theory.

\begin{definition}[half-plane capacity~{\cite[p.920]{LSW03}}]
\label{def:original_hcap}
A relatively closed subset $E$ of $\C^+$ is called a \emph{$\C^+$-hull} if $\C^+ \setminus E$ is a simply connected domain.
Let $E$ be a \emph{bounded} $\C^+$-hull.
Then there exists a unique conformal mapping $g_E \colon \C^+ \setminus E \to \C$ that obeys the hydrodynamic normalization at infinity
\begin{equation} \label{eq:Laurent_MOF}
g_E(z)=z+\frac{b_{-1}}{z}+o(z^{-1}),
\qquad z \to \infty,
\end{equation}
with $b_{-1} \ge 0$.
We refer to $\operatorname{hcap}(E):=b_{-1}$ as the \emph{half-plane capacity} of $E$.
\end{definition}

In Definition~\ref{def:original_hcap}, the mapping $g_E$ is called the \emph{mapping-out function} of $E$.
The inverse map of $g_E$ satisfies $g_E^{-1} \in \PickC$ and $\AR(g_E^{-1})=b_{-1}$, and hence
\begin{equation} \label{eq:hcap_relcap}
\operatorname{hcap}(E)=\operatorname{relcap}(E; \C^+,\infty)
\end{equation}
follows from Proposition~\ref{th:DV14_half-plane}.
Indeed, this equality is a topic in Dubinin~\cite[Lemma~2]{Dub10} and Dubinin and Vuorinen~\cite[Theorem~2.6]{DV14}.

Both in Definitions~\ref{def:relcap} and \ref{def:original_hcap}, $E \subset \C^+$ is required to be bounded.
Nevertheless, half-plane capacity can be defined for unbounded $E$.
A trivial way is as follows: for a (unbounded) $\C^+$-hull $E$, we define $\operatorname{hcap}(E):=\AR(f_E)$ if there exists a conformal mapping $f_E \colon \C^+ \to \C^+ \setminus E$ such that $f_E \in \PickC$.
The inverse $f_E^{-1} \colon \C^+ \setminus E \to \C$ will then be called the mapping-out function of $E$.
Another non-trivial and more general way is a use of planar Brownian motion.
Let $Z=((Z_t)_{t}, (\mathbb{P}_z)_{z \in \C})$ be a planar Brownian motion; i.e., $Z$ is a random element with values in $C([0,+\infty), \C)$ and $\mathbb{P}_z$ is the law of Brownian motion on $\C$ starting at $z \in \C$.
The expectation with respect to $\mathbb{P}_z$ is denoted by $\mathbb{E}_z$.
For a bounded $\C^+$-hull $E$ and a constant $b>\sup\{\, \Im z : z \in E \,\}$, the following identities \cite[p.567]{LLN09} hold:
\begin{align}
\operatorname{hcap}(E)
&=\lim_{y \to +\infty}y\mathop{\mathbb{E}_{iy}}[\Im Z_{\tau_{\C^+ \setminus E}}]
\label{eq:probabilistic_hcap} \\
&=\frac{1}{\pi}\int_{\R}\mathbb{E}_{x+ib}[\Im Z_{\tau_{\C^+ \setminus E}}]\,dx.
\label{eq:SMP_hcap}
\end{align}
Here, $\tau_{\C^+ \setminus E}:=\inf\{\, t \ge 0 : Z_t \notin \C^+ \setminus E \,\}$ is the (first) exit time of the process $Z$ from the set $\C^+ \setminus E$.
Actually, these two expressions do not require $E$ to be bounded and are equal to $\AR(f_E)$ if the mapping $f_E$ above exists.

Using the probabilistic expressions \eqref{eq:probabilistic_hcap} and \eqref{eq:SMP_hcap}, Lalley, Lawler and Narayanan~\cite{LLN09} gave a comparison of $\operatorname{hcap}(E)$ and the area of a certain neighborhood of $E$.

\begin{proposition}[{\cite[Theorem~1]{LLN09}}]
\label{th:LLN09}
For every $\C^+$-hull $E$,
\begin{equation} \label{eq:LLN09_disk}
\frac{1}{66}\operatorname{area}\left(\bigcup_{x+iy \in E}\D(x+iy,y)\right)
<\operatorname{hcap}(E)
<\frac{7}{2\pi}\operatorname{area}\left(\bigcup_{x+iy \in E}\D(x+iy,y)\right).
\end{equation}
Here, $\operatorname{area}$ denotes the usual two-dimensional Lebesgue measure.
\end{proposition}

In Proposition~\ref{th:LLN09}, $\bigcup_{x+iy \in E}\D(x+iy,y)$ is the covering of $E$ by the disks with center in $E$ and tangent to the real axis.
The area of this covering is comparable with the area of a hyperbolic neighborhood of $E$.
More precisely, let $d_{\rm hyp}$ denote the hyperbolic distance on $\C^+$, and put
\[
N(E):=\{\, z \in \C^+ : d_{\rm hyp}(z,E) \le 1 \,\}.
\]
Then it holds \cite[p.571]{LLN09} that
\begin{equation} \label{eq:LLN09_hyp}
\frac{1}{100}\operatorname{area}(N(E))
<\operatorname{hcap}(E)
<\frac{20}{\pi(e-e^{-1})^2}\operatorname{area}(N(E)).
\end{equation}
There are still other choices of neighborhood of $E$ whose area is comparable with $\operatorname{hcap}(E)$; see Rohde and Wong~\cite[pp.931--932]{RW14}.

The coefficients in \eqref{eq:LLN09_disk} and \eqref{eq:LLN09_hyp} are not expected to be optimal.
In view of Corollary~\ref{cor:DV14_half-plane} and \eqref{eq:hcap_relcap}, the problem of improving these bounds may have a certain meaning in non-commutative probability as well as in geometric function theory.

\begin{remark} \label{rem:continuity}
By means of the probabilistic expressions \eqref{eq:probabilistic_hcap} and \eqref{eq:SMP_hcap}, the following property was proved~\cite{Mur22}:
\begin{quote}
Let $\tilde{E}$ be a $\C^+$-hull with $\tilde{b}:=\sup_{z \in \tilde{E}}\Im z<+\infty$ and $\tilde{c}:=\operatorname{hcap}(\tilde{E})<+\infty$.
If a sequence of $\C^+$-hulls $E_n$, $n \in \N$, converges to a $\C^+$-hull $E$ in Carath\'eodory's sense with $E \cup \bigcup_n E_n \subset \tilde{E}$, then $\lim_{n \to \infty}\operatorname{hcap}(E_n)=\operatorname{hcap}(E)$. 
\end{quote}
Here, by ``Carath\'eodory's sense'' we mean that the sequence of complementary domains $\C \setminus E_n$ converges to $\C \setminus E$ in the sense of Carath\'eodory kernel convergence (with respect to the base point $i\tilde{b}$).
Previously, the third named author mentioned that the relation between the property introduced here and Proposition~\ref{prop:DCT_for_AR} in the current paper is not clear, because the Carath\'eodory kernel theorem for the class $\PickC$ has not been formulated properly so far \cite[{\S}4.2]{Mur22}.
Actually, however, the theorem for the subclass $\{\, f \in \PickC : f\ \text{is univalent},\ \AR(f) \le \tilde{c} \,\}$ can be easily formulated; cf.\ Bauer~\cite[Theorem~4.2]{Bau05}.
As compactness and uniqueness properties needed to verify the kernel theorem are not difficult to obtain for this class, one can prove it in a way similar to a general version in Goluzin~\cite[Theorem~1, {\S}5, Chapter~V]{Gol69}.
Thus, as long as $\C^+ \setminus E_n = f_n(\C^+)$ for $f_n \in \PickC$, either the property above or Proposition~\ref{prop:DCT_for_AR} implies the other.
\end{remark}

\appendix


\section{Measures and measure-valued functions}
\label{sec:conv_measures_appdx}

Appendix~\ref{sec:weak_and_vague} provides a quick review of weak and vague convergences of finite Borel measures; the reader can find a large part of them in Bauer's book~\cite{Bau01}.
Based on this review, convergence of measure-valued continuous functions is discussed in Appendix~\ref{sec:topology_of_luwc}.
Appendix~\ref{sec:measurability} is devoted to a specific topic: we discuss part of the equivalent conditions, listed in Lemma~\ref{lem:HVF_measurability}, on measurability of the measure-valued function $t \mapsto \nu_t$.

\subsection{Weak and vague convergences}
\label{sec:weak_and_vague}

Let $S$ be a locally-compact second-countable Hausdorff space.
Then $S$ is Polish \cite[Theorem~31.5]{Bau01}, i.e., separable and metrizable with a complete metric.
Any finite Borel measure $\mu$ on $S$ is Radon \cite[Theorem~29.12]{Bau01} in the sense that
\[
\mu(B)=\sup\{\, \mu(K) : K \subset B,\ K\ \text{is compact} \,\},
\quad B \in \cB(S).
\]
Basically, vague convergence suits Radon measures on a locally compact Hausdorff space whereas weak convergence suits (locally) finite Borel measures on a (complete) separable metric space.
This is the reason why we have taken $S$ as above, and indeed, $S$ is just a (subspace of) Euclidean space in this paper.
On $S$ we consider the following classes of test functions:
the set $C_{\rm b}(S)$ of real bounded continuous functions, the set $C_{\rm c}(S)$ of continuous functions with compact support, and the set $C_\infty(S)$ of continuous functions vanishing at infinity.
Here, $f \colon S \to \R$ is said to vanish at infinity if the set $\{\, x \in S : \lvert f(x) \rvert \ge \varepsilon \,\}$ is compact for each $\varepsilon>0$.

Now let $\mathbf{M}(S)$ be the set of finite Borel measures on $S$.
We say that a net $(\mu_\alpha)_{\alpha \in A}$ in $\mathbf{M}(S)$ converges to $\mu \in \mathbf{M}(S)$ \emph{weakly} if
\begin{equation} \label{eq:def_conv_measures}
\lim_{\alpha} \int_S f(x) \,\mu_\alpha(dx)=\int_S f(x) \,\mu(dx)
\end{equation}
for any $f \in C_{\rm b}(S)$.
If this holds, we write $\mu_\alpha \xrightarrow{\rm w} \mu$.
This convergence defines the \emph{weak topology} on $\mathbf{M}(S)$, which is interpreted as $\sigma(\mathbf{M}(S),C_{\rm b}(S))$-topology (see Remark~\ref{rem:pairing} below) in view of the pairing $(\mu, f) \mapsto \int f \,d\mu$.
Similarly, we say that $(\mu_\alpha)_{\alpha \in A}$ converges to $\mu$ \emph{vaguely},
which is denoted by $\mu_\alpha \xrightarrow{\rm v} \mu$, if \eqref{eq:def_conv_measures} holds for any $f \in C_{\rm c}(S)$\footnote{%
This definition is slightly different from the one in the previous paper by two of the authors~\cite[Definition~2.1]{HH22}. The current one has the advantage that vague compactness is easily available (see Proposition~\ref{prop:Alaoglu_for_vague}).}.
$\sigma(\mathbf{M}(S),C_{\rm c}(S))$-topology is called the \emph{vague topology}.

\begin{remark}[topology and uniformity from a pairing]
\label{rem:pairing}
We recall that, given vector spaces $X$ and $Y$ over $\R$ and a pairing (i.e., bilinear form) $\boldsymbol{b} \colon X \times Y \to \R$, \emph{$\sigma(X,Y)$-topology} on $X$ is defined as the coarsest topology that makes $x \mapsto \boldsymbol{b}(x,y)$ continuous for any $y \in Y$.
By definition, a fundamental neighborhood system of $x_0 \in X$ is given by
\[
\{\, x \in X : \lvert \boldsymbol{b}(x-x_0,y_j) \rvert < \epsilon,\ j=1,2,\ldots,n \,\}
,\quad n \ge 1,\ y_j \in Y,\ \epsilon>0.
\]
Similarly, the pairing $\boldsymbol{b}$ naturally defines a uniform structure on $X$ such that one of its basis is given by
\[
\{\, (x_1,x_2) \in X^2 : \lvert \boldsymbol{b}(x_1-x_2,y_j) \rvert < \epsilon,\ j=1,2,\ldots,n \,\}
,\quad n \ge 1,\ y_j \in Y,\ \epsilon>0.
\]
This kind of uniform structures will appear in several times in Appendix~\ref{sec:topology_of_luwc}.
\end{remark}

In vague convergence, total masses do not increase.
Indeed, let $\mu_\alpha \xrightarrow{\rm v} \mu$.
By Urysohn's lemma, we can take a sequence $f_k \in C_{\rm c}(S)$ with $0 \le f_k \uparrow 1$ ($k \to \infty$), so that
\[
\mu(S)
=\lim_{k \to \infty}\int_S f_k(x) \,\mu(dx)
=\lim_{k \to \infty}\lim_\alpha \int_S f_k(x) \,\mu_\alpha(dx)
\le \liminf_\alpha \mu_\alpha(S).
\]
By this inequality, the subset $\overline{\mathbf{M}}_r(S):=\{\, \mu \in \mathbf{M}(S) : \mu(S) \le r \,\}$ is vaguely closed for each $r \ge 0$.

\begin{proposition}[{\cite[Theorem~30.6]{Bau01}}]
\label{prop:vague_bounded_set}
For each $r \ge 0$, $\sigma(\overline{\mathbf{M}}_r(S), C_{\rm c}(S))$-topology coincides with $\sigma(\overline{\mathbf{M}}_r(S),\allowbreak C_\infty(S))$-topology.
\end{proposition}

\begin{proof}
This is proved easily by approximating $f \in C_\infty(S)$ with elements of $C_{\rm c}(S)$.
\end{proof}

Proposition~\ref{prop:vague_bounded_set} is no longer valid if $\overline{\mathbf{M}}_r(S)$ is replaced by $\mathbf{M}(S)$.
In other words, $\sigma(\mathbf{M}(S),C_{\rm c}(S))$-topology is different from $\sigma(\mathbf{M}(S),C_\infty(S))$-topology.
For example, let $S=\R$ and $\mu_n:=(1+n^2)\delta_{\{n\}}$, $n \in \N$.
Then $\mu_n \xrightarrow{\rm v} 0$ but $\int_{\R} (1+x^2)^{-1} \,\mu_n(dx)=1$ for all $n$.
A similar example also shows that $\mu_n \xrightarrow{\rm v} \mu$ does not imply $\mu_n(S) \to \mu(S)$.

By virtue of Proposition~\ref{prop:vague_bounded_set}, we can make use of basic results in functional analysis.
Indeed, the dual of the Banach space $C_\infty(S)$ with uniform norm is identified with the space of complex Radon measures with total variation norm by the Riesz--Markov--Kakutani representation theorem.
Therefore, Alaoglu's theorem yields the following:

\begin{proposition}[{\cite[Corollary~31.5]{Bau01}}]
\label{prop:Alaoglu_for_vague}
For each $r \ge 0$, the set $\overline{\mathbf{M}}_r(S)$ is vaguely compact.
\end{proposition}

This compactness is equivalent to the sequential one, because the bounded set $\overline{\mathbf{M}}_r(S)$ is metrizable by the usual argument in functional analysis (note that $C_\infty(S)$ is separable).
In fact, much stronger results are known with regard to metrizability in the present case, which allow us to work with only sequences in discussing the properties of weak and vague topologies.

\begin{proposition} \label{prop:metrizability}
The weak and vague topologies on $\mathbf{M}(S)$ are both Polish.
\end{proposition}

We refer the reader to, e.g., Stroock~\cite[Theorems~9.1.5 and 9.1.11]{Str11} or Kallenberg~\cite[Lemma~4.5]{Kal17} for the Polishness of weak topology and Bauer~\cite[Theorem~31.5]{Bau01} for vague topology.

The relation between weak and vague convergences is given as follows:

\begin{proposition}[{\cite[Theorem~30.8]{Bau01}}]
\label{prop:vague_to_weak}
Let $\mu, \mu_n \in \mathbf{M}(S)$, $n \in \N$.
Then $\mu_n \xrightarrow{\rm w} \mu$ if and only if $\mu_n \xrightarrow{\rm v} \mu$ and $\lim_{n \to \infty}\mu_n(S)=\mu(S)$.
\end{proposition}

In addition to Proposition~\ref{prop:vague_to_weak}, weak convergence has several characterizations, which are collectively called the \emph{portmanteau theorem} by Billingsley~\cite[Theorem~2.1]{Bil68}\footnote{%
The theorem goes back to A. D. Alexsandrov; see ``Remarks'' in the first edition of Billingsley's book~\cite[p.16]{Bil68} or Dudley~\cite[p.433]{Dud02}.}.

\begin{proposition}[portmanteau theorem]
 \label{prop:portmanteau}
Let $\mu, \mu_n \in \mathbf{M}(S)$, $n \in \N$.
Then the following are equivalent:
\begin{enumerate}
\item \label{port:convergence}
$\mu_n \xrightarrow{\rm w} \mu$ as $n \to \infty$;
\item \label{port:unif_conti}
$\lim_{n \to \infty} \int_S f(x) \,\mu_n(dx)=\int_S f(x) \,\mu(dx)$
for some metric $d$ that metrizes (the prescribed topology on) $S$ and for any bounded $d$-uniformly continuous function $f \colon S \to \R$;
\item \label{port:open_closed}
$\limsup_{n \to \infty} \mu_n(F) \le \mu(F)$ for any closed $F \subset S$, and $\liminf_{n \to \infty} \mu_n(G) \ge \mu(G)$ for any open $G \subset S$;
\item \label{port:Borel}
$\lim_{n \to \infty} \mu_n(B)=\mu(B)$ for any Borel $B \subset S$ with $\mu(\partial B)=0$.
\end{enumerate}
\end{proposition}

\begin{remark}[setwise convergence]
\label{rem:setwise_convergence}
We say that a sequence $(\mu_n)_{n \in \N}$ in $\mathbf{M}(S)$ converges to $\mu \in \mathbf{M}(S)$ \emph{setwise} (or \emph{strongly}) if $\lim_{n \to \infty}\mu_n(B)=\mu(B)$ for all $B \in \cB(S)$.
By the portmanteau theorem, setwise convergence implies weak convergence.
\end{remark}

\subsection{Topology on $C(\para; \mea)$}
\label{sec:topology_of_luwc}

On the basis of Appendix~\ref{sec:weak_and_vague}, we introduce a suitable topology on $C(\para; \mea)$, $\para=I$ or $I^2_{\le}$, which is consistent with the mode of convergence in \eqref{eq:luwc_intro}.
Here, we note that sequential convergence is in general not enough to single out a topology on a given set.
One way is to consider nets rather than sequences.
In what follows, we take another way, suitable for discussing ``uniform convergence'' of functions, using uniform structures on $\mea$ and $C(\para; \mea)$.
The reader can consult Kelley~\cite[Chapters 6 and 7]{Kel75} for the general theory.

Let $\mathbf{U}$ be the uniformity (or uniform structure) on $\mea$ generated by the collection of sets
\begin{equation} \label{eq:equi-distant_set}
U(f, \varepsilon)
:=\left\{\, (\mu, \nu)\in \mea^2 : \left\lvert \int f \,d\mu - \int f \,d\nu \right\rvert<\varepsilon \,\right\}
\end{equation}
indexed by $f \in C_{\rm b}(\R)$ and $\varepsilon>0$.
The topology of the uniform space $(\mea, \mathbf{U})$ is exactly the weak topology.
We can then define $\mathscr{U}$ as the uniformity on $C(\para; \mea)$ generated by the collection of sets
\begin{align}
V(K; f, \varepsilon)
&:=\left\{\, (\mu_\tau, \nu_\tau)_{\tau \in \para}\in C(\para; \mea)^2
: (\mu_\tau, \nu_\tau) \in U(f, \varepsilon)\ \text{for all}\ \tau \in K\,\right\}
\notag \\
&=\left\{\, (\mu_\tau, \nu_\tau)_{\tau \in \para} \in C(\para; \mea)^2
: \sup_{\tau \in K}\left\lvert \int f \,d\mu_\tau - \int f \,d\nu_\tau \right\rvert<\varepsilon \,\right\}
\label{eq:lu_equi-distant_set}
\end{align}
indexed by compact $K \subset \para$, $f \in C_{\rm b}(\R)$, and $\varepsilon>0$.
The topology $\mathscr{T}$ on $C(\para; \mea)$ induced by $\mathscr{U}$ is, by definition, the topology of ($\mathbf{U}$-)uniform convergence on compacta, or locally uniform convergence.
It is clear that $\lim_\alpha (\mu^{(\alpha)}_\tau)_{\tau \in \para}=(\mu_\tau)_{\tau \in \para}$ in $\mathscr{T}$ if and only if
\begin{equation} \label{eq:net-luwc}
\lim_\alpha\sup_{\tau \in K} \left\lvert \int_{\R} f(x) \, \mu^{(\alpha)}_\tau(dx) - \int_{\R} f(x) \, \mu_\tau(dx) \right\rvert=0
\end{equation}
for every compact set $K \subset \para$ and for every $f \in C_{\rm b}(\R)$.

By the following fact, adopted from Kelley~\cite[Theorem~11, Chapter~7]{Kel75}, $\mathscr{T}$ is indeed the \emph{compact-open topology} \cite[p.221]{Kel75} on $C(\para; \mea)$:

\begin{proposition} \label{prop:cpt-open}
Let $X$ be a topological space and $(Y, \mathcal{U})$ be a uniform space.
On the set $C(X; Y)$ of continuous mappings from $X$ to $Y$, the topology of $\mathcal{U}$-uniform convergence on compacta is identical with the compact-open topology.
In particular, the former topology depends only on the topology that $\mathcal{U}$ induces on $Y$.
\end{proposition}

Proposition~\ref{prop:cpt-open} enables us to obtain the topology $\mathscr{T}$ using a uniformity different from $\mathbf{U}$ above.
Take a distance $\rho$ on $\mea$ which induces the weak topology (e.g., the L\'evy--Prokhorov distance).
We replace $U(f, \varepsilon)$ in \eqref{eq:equi-distant_set} with
\[
U_\rho(\varepsilon):=\{\, (\mu, \nu) \in \mea^2 : \rho(\mu, \nu)<\varepsilon \,\}
\]
and define $V_\rho(K; \varepsilon)$ in a way similar to \eqref{eq:lu_equi-distant_set}, which leads to another uniformity $\mathscr{U}_\rho$ on $C(\para; \mea)$.
By Proposition~\ref{prop:cpt-open}, $\mathscr{U}_\rho$ induces the topology $\mathscr{T}$.
Hence we have the following:

\begin{corollary} \label{cor:metrizability}
Let $\rho$ be a distance which induces the weak topology on $\mea$.
\begin{enumerate}
\item \label{i:lu_LP}
For a net $((\mu^{(\alpha)}_\tau)_{\tau \in \para})_{\alpha \in A}$ in $C(\para; \mea)$ and $(\mu_\tau)_{\tau \in \para} \in C(\para; \mea)$, the following are equivalent:
\begin{itemize}
\item
\eqref{eq:net-luwc} holds for every compact $K \subset \para$ and for every $f \in C_{\rm b}(\R)$;
\item
$\lim_\alpha\sup_{\tau \in K}\rho(\mu^{(\alpha)}_\tau, \mu_\tau)=0$ holds for every compact $K$.
\end{itemize}

\item \label{i:luwc_metrizability}
The space $(C(\para; \mea), \mathscr{T})$ is metrizable.
\end{enumerate}
\end{corollary}

\begin{proof}
We have already seen \eqref{i:lu_LP}.
In \eqref{i:luwc_metrizability}, a desired distance is given by
\[
\mathscr{D}_\rho\left((\mu_\tau)_\tau, (\nu_\tau)_\tau \right):=\sum_{j\in \mathbb{N}}2^{-j}\left( \sup_{\tau \in K_j}\rho(\mu_\tau, \nu_\tau)\wedge 1\right),
\qquad (\mu_\tau)_{\tau \in \para}, (\nu_\tau)_{\tau \in \para} \in C(\para;\mea),
\]
as usual.
Here, $(K_j)_{j \in \N}$ is an exhaustion sequence of $\para$.
\end{proof}

We next discuss the topology on $C(\para; \prob)$.
Let $\mathbf{U}^\prime$ (resp. $\mathbf{U}^{\prime\prime}$) be the uniformity on $\mea$ generated by the collection of sets \eqref{eq:equi-distant_set} indexed by $f \in C_\infty(\R)$ (resp.\ $C_{\rm c}(\R)$) and $\varepsilon>0$.
The topology induced by $\mathbf{U}^{\prime\prime}$ is the vague topology on $\mea$.
Here, the vague topology coincides with the weak one, if they are restricted on $\prob$, by Propositions~\ref{prop:metrizability} and \ref{prop:vague_to_weak}.
Consequently, the uniformities $\mathbf{U}$, $\mathbf{U}^\prime$ and $\mathbf{U}^{\prime\prime}$ restricted to $\prob$ induce the same topology.
Hence, by Proposition~\ref{prop:cpt-open}, the topology of locally uniform convergence on $C(\para; \prob)$ is the same if we take any one of these uniformities.
This gives a way to verify Proposition~\ref{prop:three_luwc}.

One can include the convergence of moments into the discussion above.
For a fixed $p \ge 1$, let $\mathbf{P}^p(\R)=\{\, \mu \in \prob : m_p(\mu):=\int_{\R} \lvert x \rvert^p \,\mu(dx)<+\infty \,\}$.
On this set a uniformity $\mathbf{U}^p$ is generated by the collection of sets
\[
U_p(f,\varepsilon)
:=\left\{\, (\mu, \nu)\in \left( \mathbf{P}^p(\R) \right)^2 : \left\lvert \int f \,d\mu - \int f \,d\nu \right\rvert+\lvert m_p(\mu)-m_p(\nu) \rvert<\varepsilon \,\right\}
\]
indexed by $f \in C_{\rm b}(\R)$ and $\varepsilon>0$.
Let $\mathbf{T}^p$ denote the topology of the uniform space $(\mathbf{P}^p(\R), \mathbf{U}^p)$.
$\lim_\alpha \mu_\alpha=\mu$ in $\mathbf{T}^p$ if and only if the pair of the conditions $\mu_\alpha \stackrel{\rm w}{\to} \mu$ and $m_p(\mu_\alpha) \to m_p(\mu)$ holds.
This topology $\mathbf{T}^p$ can be given in different ways.
First, let $\rho$ be a distance which induces the weak topology on $\prob$; then the distance
\[
\rho_p(\mu, \nu):=\rho(\mu, \nu)+\lvert m_p(\mu)-m_p(\nu) \rvert,
\qquad \mu, \nu \in \mathbf{P}^p(\R),
\]
induces the topology $\mathbf{T}^p$.
It is also known that the $p$-Wasserstein distance $W_p$ on $\mathbf{P}^p(\R)$ induces $\mathbf{T}^p$ (see, e.g., Villani~\cite[Chapter~6]{Vil09}).
By Proposition~\ref{prop:cpt-open}, the topology of uniform convergence on compacta on the space $C(\para; \mathbf{P}^p(\R))$ of continuous mappings from $\para$ to $\mathbf{P}^p(\R)$ is the same if we take any one of the above constructions to give a uniformity on $\mathbf{P}^p(\R)$.
Moreover, $C(\para; \mathbf{P}^p(\R))$ is metrizable as in Corollary~\ref{cor:metrizability}~\eqref{i:luwc_metrizability}.
These properties legitimate our description at the end of Section~\ref{sec:intro_conti_bijec}.

\subsection{Some measurability issues}
\label{sec:measurability}

In this section, we prove the equivalence of conditions~\eqref{i:def_Borel_mble}--\eqref{i:pw_Cauchy_mble} in Lemma~\ref{lem:HVF_measurability}, which complements the description on the measurability of the driving kernel $\nu_t$ in Section~\ref{sec:chordal_LDE}.
To this end, we begin with a lemma on the measurable structure of the space of holomorphic functions.
Let $U,V$ be open sets in $\C$ and $\mathrm{Hol}(U,\overline{V})$ be the set of holomorphic functions on $U$ taking values in $\overline{V}$ equipped with the topology of locally uniform convergence.
For each $z \in U$ we denote the coordinate mapping $\mathrm{Hol}(U,\overline{V}) \to \overline{V},\ f \mapsto f(z)$ by $\pi^{U,V}_z$ or simply by $\pi_z$.
Then $\sigma(\pi_z : z \in U)$ is the usual symbol for the $\sigma$-algebra generated by all coordinate mappings; i.e., it is the smallest $\sigma$-algebra containing all sets of the form
\[
\pi_z^{-1}(B)=\{\, f \in \mathrm{Hol}(U,\overline{V}) : \pi_z(f)=f(z) \in B \,\},
\qquad z \in U,\ B \in \cB(\overline{V}).
\]

\begin{lemma} \label{lem:Borel_by_cylinder}
The identity $\cB(\mathrm{Hol}(U,\overline{V}))=\sigma(\pi^{U,V}_z : z \in U)$ holds.
\end{lemma}

\begin{proof}
Although this kind of assertion often appears in the theory of stochastic processes, we give a proof for completeness.
Since $\pi_z \colon \mathrm{Hol}(U,\overline{V}) \to \overline{V}$ is continuous, all the preimages by $\pi_z$ are Borel sets; namely, we have $\cB(\mathrm{Hol}(U,\overline{V})) \supset \sigma(\pi_z : z \in U)$.
To see the opposite inclusion, let
\begin{gather*}
N(f;K,r):=\{\, g \in \mathrm{Hol}(U,\overline{V}) : \sup_{z \in K}\lvert f(z)-g(z) \rvert < r \,\}, \\
\bar{N}(f;K,r):=\{\, g \in \mathrm{Hol}(U,\overline{V}) : \sup_{z \in K}\lvert f(z)-g(z) \rvert \le r \,\}
\end{gather*}
for a compact set $K \subset U$ and $r>0$.
Denoting the set of rational points in $K$ by $\mathrm{rat}K$, we have
\begin{align*}
\bar{N}(f;K,r)
&=\bigcap_{z \in \mathrm{rat}K}\{\, g \in \mathrm{Hol}(U,\overline{V}) : \lvert f(z)-g(z) \rvert \le r \,\} \\
&=\bigcap_{z \in \mathrm{rat}K}\pi_z^{-1}(\D(f(z),r))
\in \sigma(\pi_z : z \in U).
\end{align*}
Hence
\begin{equation} \label{eq:subbase_by_cylinder}
N(f;K,r)
=\bigcup_{n \ge 1}\bar{N}(f;K,r-n^{-1})
\in \sigma(\pi_z : z \in U).
\end{equation}
Here, the collection of all $N(f;K,r)$ is a subbase of the topology on $\mathrm{Hol}(U,\overline{V})$, which is second countable.
Thus, \eqref{eq:subbase_by_cylinder} implies $\cB(\mathrm{Hol}(U,\overline{V})) \subset \sigma(\pi_z : z \in U)$.
\end{proof}

In addition to Lemma~\ref{lem:Borel_by_cylinder}, we remind the reader that the Borel $\sigma$-algebra $\cB(\overline{\bf M}_1(\R))$ is the same if we endow $\overline{\bf M}_1(\R)$ with either the vague or weak topology \cite[Corollary~D.6]{Mur23}.

\begin{proof}[Proof of the equivalence of \eqref{i:def_Borel_mble}--\eqref{i:pw_Cauchy_mble} in Lemma~\ref{lem:HVF_measurability}]
We denote by $G[\overline{\bf M}_1(\R)]$ the image of $\overline{\bf M}_1(\R)$ by the Cauchy transform $G \colon \mea \to \mathrm{Hol}(\C^+,-\overline{\C^+})$.
Proposition~\ref{prop:continuity_Cauchy} implies that $G|_{\overline{\bf M}_1(\R)}$ is a homeomorphism from $\overline{\bf M}_1(\R)$ onto $G[\overline{\bf M}_1(\R)]$, a compact and hence Borel subset of $\mathrm{Hol}(\C^+,\allowbreak -\overline{\C^+})$.
In particular,
\[
\cB(G[\overline{\bf M}_1(\R)])
=\{\, \mathcal{C} \in \cB(\mathrm{Hol}(\C^+,-\overline{\C^+})) : \mathcal{C} \subset G[\overline{\bf M}_1(\R)] \,\}.
\]
Thus, if $t \mapsto \nu_t$ is measurable, then $t \mapsto G[\nu_t]$, the composite of $t \mapsto \nu_t$ and $\nu \to G[\nu]$, is also measurable, and vice versa.
Namely, \eqref{i:def_Borel_mble} and \eqref{i:lu_Cauchy_mble} are equivalent.
Similarly, since the coordinate mapping $\pi_z$ is continuous, $t \mapsto G[\nu_t](z)=\pi_z(G[\nu_t])$ is measurable if $t \mapsto G[\nu_t]$ is measurable.
This means $\eqref{i:lu_Cauchy_mble} \Rightarrow \eqref{i:pw_Cauchy_mble}$.
Finally, let us assume \eqref{i:pw_Cauchy_mble}.
Under this assumption, it is a routine to check that
\[
\bigl\{\, \mathcal{C} \subset \mathrm{Hol}(\C^+,-\overline{\C^+}) : \{\, t \in I : G[\nu_t] \in \mathcal{C} \,\} \in \cB(I) \,\bigr\}
\]
is a $\sigma$-algebra containing all $\mathcal{C}=\pi_z^{-1}(B)$ for $z \in \C^+$, $B \in \cB(-\overline{\C^+})$.
Hence this $\sigma$-algebra contains $\sigma(\pi^{\C^+,-\C^+}_z : z \in \C^+)=\cB(\mathrm{Hol}(\C^+,-\overline{\C^+}))$, which implies \eqref{i:lu_Cauchy_mble}.

To see the last remark of Lemma~\ref{lem:HVF_measurability} on the replacement of  $\overline{\bf M}_1(\R)$ with $\prob$, it suffices to confirm $\prob \in \cB(\overline{\bf M}_1(\R))$.
Since $\prob$ is a \emph{weak} closed subset of $\overline{\bf M}_1(\R)$, this is trivial if we regard $\cB(\overline{\bf M}_1(\R))$ as the Borel $\sigma$-algebra with respect to the weak topology.
We can also use the identity
\[
\prob=\overline{\bf M}_1(\R) \setminus \left(\bigcup_{n=1}^\infty \overline{\bf M}_{1-1/n}(\R) \right)
\]
and the \emph{vague} compactness of $\overline{\bf M}_r(\R)$.
\end{proof}

\section{Angular limits and derivatives of holomorphic functions}
\label{sec:anglim_and_deriv}

This appendix is devoted to basic properties of angular limits and derivatives of holomorphic functions at a boundary point.
We first present classical results on the unit disk $\D$ and then apply them to Pick functions on $\C^+$ via Cayley transforms.
As in the other part of the paper, the symbol $\mathrm{Hol}(U,V)$ denotes the set of holomorphic functions on $U$ taking values in $V$.

\subsection{Definitions and basic properties}
\label{sec:def_of_anglim}

Following the literature \cite[{\S}4.3]{Pom92} \cite[{\S}V.5]{GM05} we state some definitions.
For $z_0 \in \C$, $\theta \in \R$, $\alpha \in (0,\pi/2)$, and $r>0$, we define the sector by
\begin{equation} \label{eq:sector_region}
S(z_0,\theta;\alpha,r):=\{\, z \in \C : \lvert \arg(z-z_0)-\theta \rvert < \alpha,\ 0 < \lvert z-z_0 \rvert < r \,\}.
\end{equation}
Suppose that the boundary of a domain $D \subset \C$ has an \emph{inner tangent} with \emph{inner normal} $e^{i\theta}$ at $z_0 \in \partial D$; namely, for any $\alpha \in (0,\pi/2)$ there exists $r>0$ such that $S(z_0,\theta;\alpha,r) \subset D$.
The inner normal $e^{i\theta}$ being specified, a sector $S(z_0,\theta;\alpha,r)$ in $D$ is called a \emph{Stolz angle} at $z_0$.
In this setting, we say that a function $f \colon D \to \C$ has an \emph{angular} (or \emph{non-tangential}) \emph{limit} $c \in \hat{\C}=\C \cup \{\infty\}$ at $z_0$, denoted by $\angle\lim_{z \to z_0}f(z)=c$, if
\begin{equation} \label{eq:angular_limit}
\lim_{\substack{z \to z_0 \\ z \in S}}f(z)=c
\end{equation}
holds for every Stolz angle $S$ at $z_0$.
We define $f(z_0):=\angle\lim_{z \to z_0}f(z)$.
Moreover, we say that $f$ has an \emph{angular derivative} at $z_0 \in \partial D$ if 
\[
f^\prime(z_0):=\angle\!\lim_{z \to z_0}\frac{f(z)-f(z_0)}{z-z_0}\ (\in \hat{\C})
\]
exists.
If $f^\prime(z_0) \neq 0$ and $\infty$, then $f$ is said to be \emph{conformal} at $z_0$.

\begin{lemma} 
\label{lem:Visser-Ostrowski}
Let $f \in \mathrm{Hol}(D,\C)$ and $z_0 \in \partial D$.
Suppose that $\partial D$ has an inner tangent with inner normal $e^{i\theta}$ at $z_0$ and $f$ is conformal at $z_0$.
Then for any $\alpha, \beta, \gamma$ with $0 < \gamma < \alpha < \beta < \pi/2$, there exist $\rho,r, R > 0$ such that $S(z_0,\theta;\alpha,r) \subset D$ and
\[
S(f(z_0),\theta+\arg f^\prime(z_0);\gamma,\rho) 
\,\subset\, 
f(S(z_0,\theta;\alpha,r)) 
\,\subset\, 
S(f(z_0),\theta+\arg f^\prime(z_0);\beta,R).
\]
In particular, $\partial f(D)$ has an inner tangent with inner normal $e^{i\theta}f^\prime(z_0)/\lvert f'(z_0)\rvert$ at $f(z_0)$.
\end{lemma}

\begin{proof}
While the proof is based on elementary calculus, we will give a brief sketch.
Let 
\[
\varepsilon(z) := \dfrac{f(z) - f(z_{0})}{z-z_{0}} -f'(z_{0}),
\]
that is well-defined because $f'(z_{0}) \neq\infty$.
Since $f'(z_{0}) \neq 0$, we have
\[
\arg(f(z) - f(z_{0})) = \arg(z-z_{0}) + \arg f'(z_{0}) + \arg \left(1 + \frac{\varepsilon(z)}{f'(z_{0})}\right).
\]
With this equation, observe the image of the boundary $\partial S$ of $S = S(z_0,\theta;\alpha,r)$ under $f$.
If $z$ lies on one of the radii of $S$, for instance $z = z_{0} + r e^{i(\theta + \alpha)}$, then taking $r$ small enough so that $\varepsilon(r) := \varepsilon(z_{0} + r e^{i(\theta + \alpha)})$ satisfies $\gamma - \alpha < \arg (1+ \varepsilon(r)/f'(z_{0})) < \beta -\alpha$, we obtain
\[
\gamma <\arg(f(z_{0} + r e^{i(\theta + \alpha)}) - f(z_{0})) - (\theta  + \arg f'(z_{0})) < \beta.
\]
Similarly, one can evaluate another radial and the arc of $S$. 
This concludes that a continuous curve $f(\partial S)$ lies in a closed set enclosed by two sectors sharing the common center $f(z_{0})$. Since $\partial f(S) \subset f(\partial S)$, an elementary topological argument shows our assertions.
\end{proof}

The chain rule can be formulated for angular derivatives in the usual way, except that the definition of angular derivatives depends on the choice of the inner normal $e^{i\theta}$.
The boundary of a domain $D \subset \C$ possibly has many inner normals; for example, $\partial(\C \setminus (-\infty,0])$ has inner normals $e^{i\theta}$, $-\pi/2 \le \theta \le \pi/2$, at $z_0=0$ and $\pm i$ at $z_0=-1$.
Noting this point, one can easily deduce the following from Lemma~\ref{lem:Visser-Ostrowski}:

\begin{proposition} \label{prop:angular_chain_rule}
For each $j=1,2$, let $D_j \subset \C$ be a domain such that $\partial D_j$ has an inner tangent with inner normal $e^{i\theta_j}$ at $z_j \in \partial D_j$.
Suppose that $f \in \mathrm{Hol}(D_1,D_2)$ satisfies $f(z_1)=z_2$ and $\arg f^\prime(z_1)=\theta_2-\theta_1$ and that $g \in \mathrm{Hol}(D_2,\C)$ has an angular derivative at $z_2$.
Then
\[
(g \circ f)^\prime(z_1)=g^\prime(z_2)f^\prime(z_1).
\]
\end{proposition}

As can be seen from the definition \eqref{eq:angular_limit} of the angular limit, only the opening angle of a Stolz angle $S$ is important.
When the angular limit is considered on $\D$, one can replace Stolz angles $S(z_0,-\arg z_0;\alpha,r)$, $\alpha \in (0,\pi/2)$, $r \in (0,2\cos\alpha]$, by sets of the form
\begin{equation} \label{eq:nt_approach_region}
\tilde{S}(z_0;a):=\{\, z \in \D : \lvert z-z_0 \rvert < a(1 - \lvert z \rvert) \,\},
\qquad a>1
\end{equation}
(see, e.g., \cite[p.23]{BCDM20}).
By abuse of terminology, we call $\tilde{S}(z_0;a)$ a Stolz angle at $z_0$ as well.
In some places below, the expression \eqref{eq:nt_approach_region} is preferred rather than \eqref{eq:sector_region}.

We now focus on holomorphic self-mappings on $\D$, mainly following the expository book by Bracci, Contreras and D{\'{\i}}az-Madrigal~\cite{BCDM20}.
Let $\D(c,r)$ denote the open disk with center $c$ and radius $r$.
For $z_0 \in \partial \D$ and $R>0$, the disk
\[
E(z_0,R)
:=\left\{\, z \in \D : \frac{\lvert 1-\overline{z_0}z \rvert^2}{1-\lvert z \rvert^2}<R \,\right\}
=\D\left(\frac{z_0}{1+R},\frac{R}{1+R}\right),
\]
internally tangent to $\partial \D$ at $z_0$, is called a \emph{horocycle} (or \emph{horodisk}) in $\D$.

\begin{proposition}[Julia's lemma {\cite[Theorem~1.4.7]{BCDM20}}]
\label{Julia_lemma}
Let $f \colon \D \to \D$ be holomorphic and $z_0 \in \partial \D$.
If
\begin{equation} \label{eq:Julia_quotient}
\alpha_f(z_0):=\liminf_{z \to z_0}\frac{1-\lvert f(z) \rvert}{1-\lvert z \rvert}<+\infty,
\end{equation}
then $f(z_0)$ exists and lies on $\partial \D$, and
$f(E(z_0,R)) \subset E(f(z_0),\alpha_f(z_0)R)$ for each $R>0$; in particular, $\alpha_f(z_0)>0$.
\end{proposition}

\begin{proposition}[Julia--Wolff--Carath\'eodory theorem {\cite[Theorem~1.7.3]{BCDM20}}]
\label{Julia-Caratheodory}
Let $f \colon \D \to \D$ be holomorphic and $z_0 \in \partial \D$.
The following are equivalent:
\begin{enumerate}
\item
$\displaystyle \alpha_f(z_0)$ in \eqref{eq:Julia_quotient} is finite;
\item
$f(z_0) \in \partial \D$, and $f^\prime(z_0)$ exists finitely.
\item
$f(z_0) \in \partial \D$, and $\angle\lim_{z \to z_0}f^\prime(z)$ exists finitely.
\end{enumerate}
Moreover, $f^\prime(z_0)=\alpha_f(z_0)f(z_0)/z_0$ and hence $f^\prime(z_0)$ is neither zero nor $\infty$.
\end{proposition}

If $\alpha_f(z_0)<+\infty$, then $f^\prime(z_0) \in \C \setminus \{0\}$ by the two propositions above.
On the other hand, if $\alpha_f(z_0)=+\infty$, then on each Stolz angle $\tilde{S}(z_0;a)$ we have
\[
\liminf_{\substack{z \to z_0 \\ z \in \tilde{S}(z_0;a)}}\left\lvert \frac{f(z_0)-f(z)}{z_0-z} \right\rvert \ge \frac{1}{a}\liminf_{z \to z_0}\frac{1-\lvert f(z) \rvert}{1-\lvert z \rvert}=+\infty,
\]
which implies $f^\prime(z_0)=\infty$.
Hence the following holds:

\begin{corollary}[{\cite[Proposition~1.7.4]{BCDM20}}] \label{cor:ang_deriv_exist}
Let $f \colon \D \to \D$ be holomorphic and $z_0 \in \partial\D$.
If $f$ has an angular limit $f(z_0) \in \partial \D$, then the angular derivative $f^\prime(z_0)$ exists in $\hat{\C} \setminus \{0\}$.
\end{corollary}

In Section~\ref{sec:ALC}, to prove Proposition~\ref{prop:additive_DLC}, we need the angular derivative of $g^{-1} \circ f$ for two univalent self-mappings $f$ and $g$. However, the chain rule obtained above (Proposition~\ref{prop:angular_chain_rule}) is not sufficient for this purpose, and further discussion is required (in fact we need Corollary~\ref{cor:anglim_composite}, which will be shown later).
To proceed, let us recall a part of the boundary correspondence induced by conformal mappings.
Let $D \subset \C$ be a domain, $z_1, z_2 \in \partial D$, and $\gamma_1$, $\gamma_2$ be paths in $D$ ending at $z_1$ and $z_2$, respectively.
We consider the pairs $(z_1,\gamma_1)$ and $(z_2,\gamma_2)$ to be equivalent if $z_1=z_2$ and if for any $r>0$ there exists a path in $D \cap \D(z_1,r)$ with one endpoint on $\gamma_1$ and the other on $\gamma_2$.
An equivalence class with respect to this equivalence relation is called an \emph{accessible boundary point} \cite[pp.35--39]{Gol69}.
We denote the set of accessible boundary points of $D$ by $\mathrm{Acc}(D)$.

\begin{lemma}[{\cite[Theorem~II.3.1]{Gol69}}] \label{lem:APE}
Let $\varphi \colon D \to \D$ be a conformal mapping.
Then there exists an injection $\hat{\varphi} \colon \mathrm{Acc}(D) \to \partial \D$ such that, for any representative $(z,\gamma)$ of $p \in \mathrm{Acc}(D)$, the path $\varphi(\gamma)$ ends at $\hat{\varphi}(p)$.
\end{lemma}

Using this lemma and some facts on asymptotic values of holomorphic functions, we prove the following:

\begin{proposition}[cf.\ {\cite[Theorem~4.14]{Pom92} \cite[\S4]{BCDMG15}}]
\label{prop:anglim_composite}
Suppose that $f$ and $g$ are univalent holomorphic self-mappings on $\D$ with the following properties:
\begin{enumerate}
\item \label{i:anglim_composite1}
$f(\D) \subset g(\D) \subset \D$;
\item \label{i:anglim_composite2}
$\zeta:=f(1)=g(1) \in \partial \D$;
\item \label{i:anglim_composite3}
$f^\prime(1)$ and $g^\prime(1)$ are finite and satisfy $-\pi<\mathrm{Arg}(g^\prime(1)/f^\prime(1))<\pi$.
\end{enumerate}
Then the composite $h:=g^{-1}\circ f$ satisfies $h(1)=1$ and $h^\prime(1)=f^\prime(1)/g^\prime(1)$.
\end{proposition}

\begin{proof}
Clearly, $\D$ has inner tangent at $\zeta$ with the unique inner normal $-\zeta$.
From this we can conclude $\mathrm{Arg}(g^\prime(1)/f^\prime(1))=0$;
otherwise there would be two different inner normals by Assumptions \eqref{i:anglim_composite1}, \eqref{i:anglim_composite2} and Lemma~\ref{lem:Visser-Ostrowski}.
Now let $\gamma$ be the line segment from $0$ to $1$.
Again by Lemma~\ref{lem:Visser-Ostrowski}, we can find a sector $S=S(\zeta,-\pi+\arg f^\prime(1);\beta,\rho) \subset g(\D)$ such that both the paths $f(\gamma)$ and $g(\gamma)$ approach $\zeta$ through $S$.
The pairs $(\zeta,f(\gamma))$ and $(\zeta,g(\gamma))$ thus represent the same accessible boundary point of $g(\D)$.
Applying Lemma~\ref{lem:APE} to $g^{-1} \colon g(\D) \to \D$, we see that the path $h(\gamma)=g^{-1}(f(\gamma))$ ends at $1$.
This implies that $h(1)=1$ by Lindel\"of's theorem\footnote{%
The statement and proof can be found, for example, in \cite[Theorem~VIII.4.4]{Gol69}, \cite[Theorem 3-5]{Ahl73}, or \cite[Theorem~1.5.7]{BCDM20}.},
and moreover, $h^\prime(1) \in \hat{\C} \setminus \{0\}$ by Corollary~\ref{cor:ang_deriv_exist}.
Hence we have
\[
\lim_{r \to 1-0}\frac{g(1)-g(h(r))}{1-h(r)}
=\lim_{r \to 1-0}\frac{1-r}{h(1)-h(r)} \cdot \frac{f(1)-f(r)}{1-r}
=\frac{f^\prime(1)}{h^\prime(1)} \in \C.
\]
By \cite[Corollary~7.3.11]{BCDM20} this limit equals $g^\prime(1)$, which completes the proof.
\end{proof}

\subsection{Angular derivatives of Pick functions at infinity}
\label{sec:anglim_at_infty}

Let $T(z):=i(z+1)/(1-z)$, which is the Cayley transform from $\D$ onto $\C^+$ with $T(1)=\infty$.
Its inverse is given by $T^{-1}(z)=(z-i)/(z+i)$.
We introduce the bijection
\begin{equation} \label{eq:induced_Cayley}
\operatorname{ad}T \colon \mathrm{Hol}(\C^+, \C^+) \to \mathrm{Hol}(\D, \D),\ f \mapsto T^{-1} \circ f \circ T.
\end{equation}
Recall from Section~\ref{sec:PN_to_Cauchy} that, for a Pick function $f$, we call $\beta$ in \eqref{eq:PNrep}, \eqref{eq:PN_beta}, and \eqref{eq:angderiv_at_infty} the angular derivative $f^\prime(\infty)$.

\begin{proposition} \label{prop:anglim_via_Cayley}
For $f \in \mathrm{Hol}(\C^+, \C^+)$, $f^\prime(\infty) \neq 0$ if and only if $(\operatorname{ad}T(f))^\prime(1) \neq \infty$.
Moreover, $f^\prime(\infty)(\operatorname{ad}T(f))^\prime(1)=1$ if either of the two holds.
\end{proposition}

\begin{proof}
Clearly, $f(\infty)=\infty$ if and only if $\operatorname{ad}T(f)(1)=1$ (in the sense of an angular limit).
Moreover, putting $w=T(z)$ we see from a direct calculation that
\[
\frac{f(z)}{z} = \frac{\operatorname{ad}T(f)(w)+1}{w+1} \cdot \frac{1-w}{1-\operatorname{ad}T(f)(w)}.
\]
Hence the conclusion follows.
\end{proof}

We note that Corollary~\ref{cor:ang_deriv_exist} and Proposition~\ref{prop:anglim_via_Cayley} imply the existence of the limits in \eqref{eq:angderiv_at_infty} if $\beta \neq 0$.

\begin{corollary} \label{cor:Julia}
Let $\C^+_M:=\{\, z : \Im z>M \,\}$, $M>0$.
If $f$ is a Pick function with $f^\prime(\infty)=\beta \neq 0$, then $f(\C^+_M) \subset \C^+_{\beta M}$ for every $M>0$.
\end{corollary}

\begin{proof}
A direct calculation shows $T(E(1,1/M))=\C^+_M$ for horocycles tangent to $\partial \D$ at $1$.
The conclusion then follows from Propositions~\ref{Julia_lemma}, \ref{Julia-Caratheodory} and \ref{prop:anglim_via_Cayley}.
\end{proof}

\begin{corollary} \label{cor:anglim_composite}
Let $f$ and $g$ be univalent Pick functions with $f(\C^+) \subset g(\C^+) \subset \C^+$.
Moreover, suppose $f^\prime(\infty) \neq 0$ and $g^\prime(\infty) \neq 0$.
Then the composite $h:=g^{-1} \circ f$ satisfies $h^\prime(\infty)=f^\prime(\infty)/g^\prime(\infty)$.
\end{corollary}

\begin{proof}
This follows immediately from Propositions~\ref{prop:anglim_composite} and \ref{prop:anglim_via_Cayley}.
\end{proof}

\section{Proofs and notes for Section~\ref{sec:Pick_Cauchy}}
\label{sec:Pick_appdx}

This appendix provides the proofs skipped in Section~\ref{sec:Pick_Cauchy}.

\subsection{Proof of Propositions~\ref{prop:Maassen_2.1} and \ref{prop:class_P}}
\label{sec:prf_char_appdx}

\begin{proof}[Proof of Proposition~\ref{prop:Maassen_2.1}]
$\eqref{Ma:G} \Leftrightarrow \eqref{Ma:F}$ is trivial; here, we note that $\mu \neq 0$ in these conditions due to the assumption that $g$ is a non-zero function.
$\eqref{Ma:g} \Leftrightarrow \eqref{Ma:f}$ is also obvious except for the case $c=0$; in fact, in the proof of $\eqref{Ma:g} \Rightarrow \eqref{Ma:G}$ below, we show that $c \neq 0$ in \eqref{Ma:g}.
Thus, it remains to prove $\eqref{Ma:g} \Leftrightarrow \eqref{Ma:G}$.

\smallskip\noindent
$\eqref{Ma:g} \Rightarrow \eqref{Ma:G}$.
Suppose that the Nevanlinna representation of $g$ is given by the right-hand side of \eqref{eq:PNrep}, and assume \eqref{Ma:g}.
In particular, $\lim_{y \to +\infty}y\Im g(iy)=\Re c$.
Comparing this to \eqref{eq:PN_total} with $f$ in this formula replaced by $g$, we have $\beta=0$ and $\rho(\R)=\Re c$.
Hence \eqref{eq:PNrep2} yields $g(z)=\alpha^\prime-G_\rho(z)$ for some $\alpha^\prime \in \R$.
Since the limits $\lim_{y \to +\infty}iy\, g(iy)$ and $\lim_{y \to +\infty}iy\, G_\rho(iy)$ exist, we have $\alpha^\prime=0$, and hence \eqref{Ma:G} follows with $\mu=\rho$.
Now $\mu \neq 0$ implies $\Re c \neq 0$.

\smallskip\noindent
$\eqref{Ma:G} \Rightarrow \eqref{Ma:g}$.
Assume \eqref{Ma:G};
then
\[
iy\, g(iy)
=-iy\, G_\mu(iy) = \int_{\R}\frac{-y^2+iyx}{x^2+y^2}\,\mu(dx) \to -\mu(\R)
\quad (y \to +\infty),
\]
i.e., \eqref{Ma:g} holds with $c=\mu(\R)$.
This also implies the last property of the proposition.
\end{proof}

\begin{proof}[Proof of Proposition~\ref{prop:class_P}]
We first show that \eqref{HN:ang}--\eqref{HN:PN} are equivalent. 

\smallskip\noindent
$\eqref{HN:ang} \Rightarrow \eqref{HN:rad}$.
Assume \eqref{HN:ang}. Then
\[
\angle\!\lim_{z \to \infty}\left( 1-\frac{f(z)}{z} \right)=\angle\!\lim_{z \to \infty}\frac{z(z-f(z))}{z^2}=0.
\]
Hence $f^\prime(\infty)=1$. By \eqref{eq:PN_total} we have
\[
y(\Im f(iy)-y)=\int_{\R}\frac{y^2}{x^2+y^2}\,\rho(dx).
\]
As $y^2/(x^2+y^2)=1-x^2/(x^2+y^2)$ is non-decreasing in $y \in (0,+\infty)$, it follows that
\begin{align*}
\sup_{y>0} y(\Im f(iy)-y)&=\lim_{y \to +\infty}y(\Im f(iy)-y) \\
&=\lim_{y \to +\infty}\Re[ iy(iy-f(iy))] \\
&= \AR(f)<+\infty.
\end{align*}
Thus we obtain \eqref{HN:rad}.

\smallskip\noindent
$\eqref{HN:rad} \Rightarrow \eqref{HN:PN}$.
This follows from \eqref{eq:PN_beta}, \eqref{eq:PN_total}, and \eqref{eq:PNrep2}.

\smallskip\noindent
$\eqref{HN:PN} \Rightarrow \eqref{HN:ang}$.
Assume \eqref{HN:PN}. Then
\begin{equation} \label{eq:high_ord_term_F}
z(z-f(z))=\int_{\R}\frac{z}{z-x}\,\rho(dx).
\end{equation}
For each $a>0$ and any $z \in \Gamma_{a,0}$ we have $\lvert z/(x-z) \rvert \le \lvert z \rvert/\Im z \le \sqrt{1+a^2}/a$.
Hence the bounded convergence theorem applies to \eqref{eq:high_ord_term_F} as $z \to \infty$ through $\Gamma_{a,0}$, which yields \eqref{HN:ang} and $\AR(f)=\rho(\R)<+\infty$. 

\smallskip
We have now seen that \eqref{HN:ang}--\eqref{HN:PN} are equivalent.
In a similar way, \eqref{HN1} is derived from \eqref{HN:PN}.
Indeed, for each $b>0$ and any $z \in \C^+_b$ we have $1/\lvert x-z \rvert \le 1/\Im z \le 1/b$, and hence the bounded convergence theorem yields \eqref{HN1}.

It is easy to see that \eqref{HN:PN} implies \eqref{HN:ineq} with $C=\rho(\R)$.
Conversely, \eqref{HN:ineq} implies \eqref{HN:rad} and also $\rho(\R) \le C$ by \eqref{eq:PN_total}.
Condition \eqref{HN:ineq} is thus equivalent to \eqref{HN:ang}--\eqref{HN:PN}, and also $\rho(\R)=\inf\{\, C : C\ \text{enjoys}\ \eqref{eq:HN_ineq}\,\}$.

It remains to prove that \eqref{HN:Cauchy} is equivalent to the other five conditions.
Assume \eqref{HN:rad}.
Then 
\[
\frac{f(iy)}{iy}-1 = \frac{1}{iy}(f(iy)-iy) \to 0
\qquad (y \to +\infty),
\]
and hence $f=F_\mu$ for some $\mu \in \prob$ by Proposition~\ref{prop:Maassen_2.1}.
Following Maassen \cite[pp.417--419]{Maa92} we introduce the function
\begin{align}
C_f(y)
&:=\frac{iy}{f(iy)}(f(iy)-iy)
=y^2\left(\frac{1}{f(iy)}-\frac{1}{iy}\right)
\label{eq:Maassen_function} \\
&=-\int_{\R}\frac{y^2x}{x^2+y^2}\,\mu(dx) + i\int_{\R}\frac{yx^2}{x^2+y^2}\,\mu(dx). \notag
\end{align}
As $y \mapsto y^2/(x^2+y^2)$ is non-decreasing, the monotone convergence theorem yields
\begin{equation} \label{eq:Maassen_Var}
\AR(f)=\lim_{y \to +\infty}y \Im C_f(y)
=\int_{\R}x^2 \,\mu(dx),
\end{equation}
which implies that $\mu$ has finite second moment.
The dominated convergence theorem then gives
\begin{equation} \label{eq:Maassen_mean}
-\lim_{y \to +\infty}\Re C_f(y)
=\int_{\R}x \,\mu(dx)=\Mean(\mu).
\end{equation}
The limit on the left-hand side of \eqref{eq:Maassen_mean} is zero by \eqref{HN:rad}.
Therefore, \eqref{eq:Maassen_Var} gives the variance $\var(\mu)$, which shows $\eqref{HN:rad} \Rightarrow \eqref{HN:Cauchy}$ and also \eqref{HN:four_quantities}.
Conversely, assume \eqref{HN:Cauchy};
then the limits in \eqref{eq:Maassen_Var} and \eqref{eq:Maassen_mean} exist by the dominated convergence theorem.
In particular, $\lim_{y \to +\infty}C_f(y)=0$.
As $\lim_{y \to +\infty}iy/f(iy)=1$ due to Proposition~\ref{prop:Maassen_2.1}, we have $\lim_{y \to +\infty}(f(iy)-iy)=0$.
The last property combined with \eqref{eq:Maassen_Var} implies \eqref{HN:rad}.
\end{proof}

\begin{remark}
We have utilized $C_f(y)$ defined by \eqref{eq:Maassen_function} as Maassen~\cite[Proposition~2.2]{Maa92},
but our proof is slightly different from his.
This is because the identity $\var(\mu)=\lim_{y \to +\infty} y \lvert C_f(y) \rvert$, which follows from \eqref{eq:PN_total}, \eqref{HN2}, and \eqref{HN:four_quantities} in our context, seems difficult to derive \emph{directly from condition \eqref{HN:Cauchy}}.
\end{remark}

\subsection{Proof of Proposition~\ref{prop:continuity_Cauchy} and Corollary~\ref{cor:continuity_Cauchy}}
\label{sec:prf_conv_appdx}

\begin{proof}[Proof of Proposition~\ref{prop:continuity_Cauchy}]
The ``only if'' part follows immediately from Proposition~\ref{prop:vague_bounded_set} because the function $\R \ni x \mapsto (x-z)^{-1}$ vanishes at infinity.
To prove the ``if'' part, assume $G[\mu_n] \to G$ pointwise on $\AC$.
Then any vague cluster point $\mu$ of $(\mu_n)_n$ enjoys $G_{\mu}=G$ on $\AC$ by the ``only if'' part, and hence on $\C^+$ by the identity theorem.
The cluster point $\mu$ is thus unique by the uniqueness of the Cauchy transform.
Since $\{\, \mu_n : n \in \N \,\}$ is vaguely compact by Proposition~\ref{prop:Alaoglu_for_vague}, we have $\mu_n \xrightarrow{\rm v} \mu$.
To see the remaining assertion, we note that $\sup_n \lvert G[\mu_n](z) \rvert \le C/\Im z$ for $C:=\sup_n \mu_n(\R)<+\infty$.
Hence $(G[\mu_n])_n$ is locally bounded on $\C^+$, and the pointwise convergence implies the locally uniform one by Vitali's convergence theorem.
\end{proof}

\begin{proof}[Proof of Corollary~\ref{cor:continuity_Cauchy}]
The only matter is the latter assertion about reciprocal Cauchy transforms in the case $\mu, \mu_n \in \prob$.
Suppose that $G[\mu_n] \to G_\mu$ locally uniformly on $\C^+$, and for a compact subset $K \subset \C^+$, put $b:=-\sup\{\, \Im w : w \in G_\mu(K) \,\}$.
As $G_\mu(K)$ is a compact subset of $-\C^+$, we have $b>0$.
Then there exists $N \in \N$ such that $\sup\{\, \lvert G[\mu_n](z)-G_\mu(z) \rvert : z \in K,\ n \ge N \,\}<b/2$.
For $n \ge N$ we have
\[
\sup_{z \in K}\left\lvert F[\mu_n](z)-F_\mu(z) \right\rvert
\le \frac{\displaystyle \sup_{z \in K }\left\lvert G[\mu_n](z)-G_\mu(z) \right\rvert}{\displaystyle \inf_{z \in K}\left\lvert G[\mu_n](z)G_\mu(z) \right\rvert} 
\le \frac{2}{b^2}\sup_{z \in K}\left\lvert G[\mu_n](z)-G_\mu(z) \right\rvert,
\]
which means that $F[\mu_n] \to F_\mu$ locally uniformly on $\C^+$.
The converse is proved in the same way.
\end{proof}

\section{A proof of Lemma \ref{TH_diff-one-trajectory}} \label{app:GHP}

Here we provide a generalization of Lemma \ref{TH_diff-one-trajectory} with a simple proof suggested by a referee.  Let $p \in [1,\infty]$. 
A two-parameter family $(\varphi_{s,t})_{s\le t}$ of holomorphic self-mappings $\varphi_{s,t} \colon \C^+ \to \C^+$ is called an \emph{EF of order $p$} \cite[Definition 3.1]{BCDM12} if it satisfies \ref{EF1}, \ref{EF2} in Remark \ref{rem:reversal} and \ref{EF3'} with $L^1_{\rm loc}(I)$ replaced by $L^p_{\rm loc}(I)$.  Note that the case $p=\infty$ corresponds to the local Lipschitz continuity. 

\begin{lemma} \label{TH_diff-one-trajectory2}
Let $p \in [1,\infty]$.
Suppose that $(\varphi_{s,t})_{s \le t}$ is a family of holomorphic self-mappings of $\C^+$ such that conditions \textup{\ref{EF1}} and \textup{\ref{EF2}} hold, and each $\varphi_{s,t}$ has Denjoy--Wolff fixed point at $\infty$.
If there are a function $k \in L^p_{\rm loc}(I)$ and $z_0 \in \C^+$ such that 
\[
|\varphi_{0,u}(z_0) - \varphi_{0,t}(z_0)| \le \int_{t}^u k(r)\,dr ,\qquad t \le u,
\] then $(\varphi_{s,t})$ is an EF of order $p$.
\end{lemma}
\begin{proof}
Since the problem is local, we may assume that $I=[0,T]$. Also we may assume $z_0=i$ by working with $L^{-1} \circ \varphi_{s,t} \circ L$ for a suitable affine transformation $L(z):= az+b,~a>0,~b\in\R$. For notational brevity, let $z_s := \varphi_{0,s}(z_0)$ and $\Phi_{s,t}(z) := \varphi_{s,t}(z+z_s)$. Observe that $\Phi_{s,t}(0) = \varphi_{s,t}(\varphi_{0,s}(z_0)) =  \varphi_{0,t}(z_0) = z_t$, and by the assumption on the Denjoy--Wolff point (see also Remark \ref{rem:DW}), we have $\Im[\varphi_{t,u}(z)] \ge \Im z$ and so 
\[
\Im[\varphi_{s,u}] - \Im[\varphi_{s,t}] = \Im[\varphi_{t,u} \circ \varphi_{s,t}] - \Im[\varphi_{s,t}] \ge 0, \qquad s\le t \le u. 
\]
 By the Harnack inequality \cite[Eq.~(11), Chapter 6]{Ahl79} for the nonnegative harmonic function $\Im[\Phi_{s,u}-\Phi_{s,t}]$~($s\le t \le u$), we have
\begin{align}\label{eq:g0}
0\le  \Im[\Phi_{s,u}(z) - \Phi_{s,t}(z)]  
\le  \frac{1+|z|}{1-|z|} \Im[z_u - z_t],   \qquad |z|<1 ~(\le |z_s|). 
\end{align}
This obviously yields 
\begin{align} \label{eq:g1}
\Big|\!\Im[\Phi_{s,u}(z) - \Phi_{s,t}(z)]\Big| 
\le  \frac{1+|z|}{1-|z|} \int_t^u k(r)\,dr,   \qquad |z|<1, ~s\le t \le u. 
\end{align}
By the Poisson--Schwarz formula \cite[Eq.~(66), Chapter 4]{Ahl79},  for each $0<R<1$, there exists  $\beta \in\R$ such that 
\begin{align}\label{eq:g2}
\Phi_{s,u}(z) - \Phi_{s,t}(z) = \beta +i \int_{-\pi}^\pi \frac{Re^{i\theta}+z}{Re^{i\theta}-z} & \Im[\Phi_{s,u}(Re^{i\theta}) - \Phi_{s,t}(Re^{i\theta})] \,\frac{d\theta}{2\pi}, \\
& \qquad |z|<R,  ~t,u \in [s,T], \notag 
\end{align}
Evaluating this at $z=0$ shows that $\beta=\Re[\Phi_{s,u}(0)-\Phi_{s,t}(0)]=0$. Combining \eqref{eq:g1} and \eqref{eq:g2} we obtain 
\begin{align}\label{eq:g3}
|\Phi_{s,u}(z) - \Phi_{s,t}(z)| \le \frac{(R+|z|)(1+R)}{(R-|z|)(1-R)}\int_t^u k(r)\,dr, \qquad |z|<R,~s\le t\le u.  
\end{align}
This kind of estimate can be obtained for $\varphi_{s,u}(w)-\varphi_{s,t}(w)$ at any point $w\in \C^+$. We set $R:= 1/2$ in \eqref{eq:g3}. Let $K:=\{z_s: s \in [0,T]\}$. Recalling that $K \subseteq \{z: \Im z \ge1\}$, we take open balls $B_1, B_2,\dots, B_n$ with radius $r_w:=8^{-1}\min\{\Im w, 1\}$ that cover $K \cup \{w\}$ such that $B_i$ contains the center of $B_{i+1}$ for every $1\le i \le n-1$ and $B_n$ contains $w$.  For each $s \in [0,T]$, there is a ball $B_i=B_{i(s)}$ containing $z_s$. This ball is contained in the ball $\{z: |z-z_s|<1/4\}$, so that by \eqref{eq:g3} we have 
\begin{align}\label{eq:g4}
|\varphi_{s,u}(z) - \varphi_{s,t}(z)| \le 9\int_t^u k(r)\,dr, \qquad z \in B_i,~s\le t\le u.  
\end{align}
Since \eqref{eq:g4} holds in particular at the center $c_{i+1}$ of $B_{i+1}$, we can repeat the above arguments in neighborhoods of $c_{i+1}$. First, inequality \eqref{eq:g0} can be modified to 
\begin{align*}
0\le\Im[\varphi_{s,u}(z) - \varphi_{s,t}(z)] 
&\le  \frac{3r_w +|z-c_{i+1}|}{3r_w-|z-c_{i+1}|} \Im[\varphi_{s,u}(c_{i+1}) - \varphi_{s,t}(c_{i+1})]\\
&\le 45\int_t^u k(r)\,dr, \qquad |z - c_{i+1}|\le 2r_w,~s\le t\le u. 
\end{align*}
Second, when modifying \eqref{eq:g2}--\eqref{eq:g3}, the number $\beta$ will be replaced with $\Re[\varphi_{s,u}(c_{i+1}) - \varphi_{s,t}(c_{i+1})]$, which can be estimated by \eqref{eq:g4}. Choosing $R=2r_w$ we obtain 
\[
|\varphi_{s,u}(z) - \varphi_{s,t}(z)| \le 144\int_t^u k(r)\,dr, \qquad z \in B_{i+1},~s\le t\le u.  
\]
Iterating these arguments, we obtain the inequality
\begin{equation*}
|\varphi_{s,u}(z) - \varphi_{s,t}(z)| \le 9\cdot 16^{n-i}\int_t^u k(r)\,dr, \qquad z \in B_n,~s\le t\le u.  
\end{equation*}
Note that $16^{n-i}\le 16^n$ and $n$ depends only on $w$ and $K$.   
\end{proof}

\section{Proof of Lemmas  \ref{lem:inverse1} and \ref{lem:inverse2}}\label{app:conv}

\begin{proof}[Proof of Lemma \ref{lem:inverse1}]
As is easily seen, there exists a constant $C_a$ (depending on $a$) such that $\lvert x/(z-x) \rvert \le C_a$ for all $z\in \Gamma_{a,0}$ and $x \in \R$. 
Thus for any $R>0$ we have 
\begin{align*}
\lvert z G[\mu_\tau](z)-1 \rvert
&=\left\lvert \int_{\R}\frac{x}{z-x}\,\mu_\tau(dx)\right\rvert 
\leq \int_{[-R,R]} \frac{\lvert x \rvert}{\Im z} \,\mu_\tau(dx) + C_a\mu_\tau([-R,R]^c) \\
&\leq \frac{R}{\Im z} + C_a \sup_{\tau \in K}\mu_\tau([-R,R]^c).
\end{align*}
From this and the tightness it follows that
\[
\lim_{\substack{z\to\infty\\ z \in \Gamma_{a,0}}}\sup_{\tau \in K}\lvert z G[\mu_\tau](z)-1 \rvert =0,
\]
which implies the desired conclusion. 
\end{proof}

\begin{proof}[Proof of Lemma \ref{lem:inverse2}]
\textit{Univalence}.
We first establish 
\begin{equation}\label{eq:tight_derivative_F}
\sup_{\tau \in K} \lvert F[\mu_\tau]'(z)-1 \rvert = o(1), \qquad z\to\infty,\ z \in \Gamma_{a,0}.   
\end{equation}
Using Lemma \ref{lem:inverse1} established above we get 
\[
F[\mu_\tau]'(z)= - F[\mu_\tau](z)^2  G[\mu_\tau]'(z) = (1+o(1)) \int_\R \frac{z^2}{(z-x)^2} \,\mu_\tau(dx), 
\]
where $o(1)$ is uniform in $\tau \in K$.
Therefore, it suffices to show that 
\begin{equation}\label{eq:tight_derivative_G}
\sup_{\tau \in K}\left\lvert \int_{\R} \frac{z^2}{(z-x)^2} \,\mu_\tau(dx) -1 \right\rvert =o(1). 
\end{equation}
It is easy to find a constant $C^\prime_a$ such that $\lvert z^2 / (z-x)^2\rvert \le C^\prime_a$ for all $z \in \Gamma_{a,0}$ and $x \in \R$.
Using this $C^\prime_a$ we have
\begin{align*}
\int_{\R} \left\lvert \frac{z^2}{(z-x)^2} -1 \right\rvert \,\mu_\tau(dx) 
&\le \int_{[-R,R]} \frac{x^2+2\lvert x \rvert \lvert z \rvert}{(\Im z)^2} \,\mu_\tau(dx) + (C^\prime_a+1) \mu_\tau([-R,R]^c) \\
&\le \frac{R^2+2R\lvert z \rvert}{(\Im z)^2}+ (C^\prime_a+1) \sup_{\tau \in K} \mu_\tau([-R,R]^c).
\end{align*}
As $\lvert z \rvert \asymp \Im z$ for $z \in \Gamma_{a,0}$, letting first $z \to \infty$ and then $R \to \infty$ we obtain \eqref{eq:tight_derivative_G} and hence \eqref{eq:tight_derivative_F}. 

By \eqref{eq:tight_derivative_F} we can find such a large $b>0$ that 
\[
\sup_{\tau \in K} \lvert F[\mu_\tau]'(z)-1 \rvert \le \frac{1}{2},
\qquad z \in \Gamma_{a,b}. 
\]
From this it follows that, for $\tau \in K$ and $z, w \in \Gamma_{a,b}$,
\[
\Big \lvert (F[\mu_\tau](z)-z)-(F[\mu_\tau](w)-w) \Big\rvert \le \frac{1}{2}\lvert z-w \rvert,
\]
hence $\lvert F[\mu_\tau](z)-F[\mu_\tau](w) \rvert \ge 2^{-1}\lvert z-w \rvert$.
This implies that $F[\mu_\tau]$ is univalent on $\Gamma_{a,b}$ for all $\tau \in K$. 

\smallskip\noindent 
\textit{The inclusion $F[\mu_\tau](\Gamma_{a, b}) \supset \Gamma_{a+\epsilon,(1+\epsilon)b}$}. 
By a simple geometric consideration we can observe that there exists a constant $c=c(a,\epsilon)$ such that, for any $b>0$ and any $\zeta \in \Gamma_{a+\epsilon,(1+\epsilon)b}$, the closed disk $\overline{\D}(\zeta, c\lvert \zeta \rvert)$ with center $\zeta$ and radius $c\lvert \zeta \rvert$ is contained in $\Gamma_{a,b}$.
For this $c$ we can find $b_0>0$ such that
\[
\sup_{\tau \in K,\, z \in \Gamma_{a,b_0}} \lvert F[\mu_\tau](z)-z \rvert < \frac{c}{1+c}\lvert z \rvert
\]
by Lemma \ref{lem:inverse1}.
Now let $b \ge b_0$ and fix an arbitrary $\zeta \in \Gamma_{a+\epsilon,(1+\epsilon)b}$.
For any $z \in \partial \D(\zeta, c\lvert \zeta \rvert)$, it is clear that $\lvert z \rvert \le (1+c)\lvert \zeta \rvert$ and hence
\[
\lvert F[\mu_\tau](z)-z \rvert < \frac{c}{1+c}\lvert z \rvert \le c \lvert \zeta \rvert=\lvert z-\zeta \rvert.
\]
Thus by Rouch\'e's theorem, the functions $z \mapsto z-\zeta$ and $z \mapsto F[\mu_\tau](z)-\zeta$ have the same number of zeros in $\D(\zeta, c\lvert \zeta \rvert)$, which shows $\zeta \in F[\mu_\tau](\D(\zeta, c\lvert \zeta \rvert)) \subset F[\mu_\tau](\Gamma_{a, b})$.
\end{proof}


\begin{thebibliography}{99}
\bibitem{ABCDM10} M.\ Abate, F.\ Bracci, M.\ D.\ Contreras and S.\ D{\'{\i}}az-Madrigal, The evolution of Loewner's differential equations, Eur.\ Math.\ Soc.\ Newsl., 2010.

\bibitem{Ahl79}
L.\ V.\ Ahlfors, \textit{Complex Analysis: An Introduction to the Theory of Analytic Functions of One Complex Variable}, 3rd ed., International Series in Pure and Applied Mathematics, McGraw-Hill Book Co., New York, 1979.

\bibitem{Ahl73}
L.\ V.\ Ahlfors, \textit{Conformal Invariants: Topics in Geometric Function Theory}, Reprint of the 1973 original, with a foreword by Peter Duren, F. W. Gehring and Brad Osgood, AMS Chelsea Publishing, Providence, RI, 2010.

\bibitem{AG93}
N.\ I.\ Akhiezer and I.\ M.\ Glazman, \textit{Theory of linear operators in Hilbert space I and II}, Two volumes bound as one, Dover Publications, Inc., New York, 1993.

\bibitem{ATZ21} V.\ Akhmedova, T.\ Takebe and A.\ Zabrodin, L\"owner equations and reductions of dispersionless hierarchies, J.\ Geom.\ Phys.\ \textbf{162} (2021), Paper No.\ 104100, 32 pp.

\bibitem{AAS83} I.\ A.\ Aleksandrov, S.\ T.\ Aleksandrov and V.\ V.\ Sobolev, Extremal properties of mappings of a half plane into itself (Russian), in: \textit{Complex Analysis (Warsaw, 1979)}, Banach Center Publ. \textbf{11}, PWN, Warsaw, 1983, pp.7--32.

\bibitem{AW14} M.\ Anshelevich and J.\ D.\ Williams, Limit theorems for monotonic convolution and the Chernoff product formula, Int.\ Math.\ Res.\ Notices \textbf{11} (2014), 2990--3021.

\bibitem{AHLV15} O.\ Arizmendi, T.\ Hasebe, F.\ Lehner and C.\ Vargas, Relations between cumulants in noncommutative probability, Adv.\ Math.\ \textbf{282} (2015), 56--92. 

\bibitem{BDDM86}
A.\ Baernstein, II, D.\ Drasin, P.\ Duren, and A.\ Marden (eds.), \emph{The {B}ieberbach conjecture}, Mathematical Surveys and Monographs, vol.~21, American Mathematical Society, Providence, RI, 1986.

\bibitem{BNT02} O.\ E.\ Barndorff-Nielsen and S.\ Thorbj{\o}rnsen, Self-decomposability and L\'evy processes in free probability, Bernoulli \textbf{8} (2002), 323--366. 

\bibitem{Bau01} H.\ Bauer, \textit{Measure and Integration Theory}, de Gruyter Studies in Math.\ 26, Walter de Gruyter, Berlin/New York, 2001. 

\bibitem{Bau04} R.\ O.\ Bauer, L\"owner’s equation from a noncommutative probability perspective, J.\ Theoret.\ Probab.\ \textbf{17} (2004), 435--456.

\bibitem{Bau05} R.\ O.\ Bauer, Chordal Loewner families and univalent Cauchy transforms, J.\ Math.\ Anal.\ Appl.\ \textbf{302} (2005), 484--501.

 \bibitem{Bel05} S.\ T.\ Belinschi, Complex analysis methods in noncommutative probability, PhD thesis, Indiana University, 2005. 
Available at arXiv:math/0602343v1. 

\bibitem{BB07} S.\ T.\ Belinschi and H.\ Bercovici, A new approach to subordination results in free probability,  J.\ Anal.\ Math.\ \textbf{101} (2007), 357--365. 

\bibitem{BP96} H. Bercovici and V. Pata, The law of large numbers for free identically distributed random variables, Ann.\ Probab.\ \textbf{24} (1996), 453--465.

\bibitem{BP99} H.~Bercovici and V.~Pata, Stable laws and domains of attraction in free probability theory (with an appendix by Philippe Biane), Ann.\ of Math.\ \textbf{149} (1999), 1023--1060.

\bibitem{BP00} H.~Bercovici and V.~Pata, A free analogue of Hin\v{c}in's characterization of infinite divisibility, Proc.\ Amer.\ Math.\ Soc.\ \textbf{128} (2000), 1011--1015. 

\bibitem{BV93} H.~Bercovici and D.~Voiculescu, Free convolution of measures with unbounded support, Indiana Univ.\ Math.\ J.\ \textbf{42} (1993), 733--773.

 \bibitem{BP78} E.\ Berkson and H.\ Porta, Semigroups of analytic functions and composition operators, Michigan Math.\ J.\ \textbf{25} (1978), 101--115. 

\bibitem{BW16} R.\ Bhattacharya and E.C.\ Waymire, \textit{A Basic Course in Probability Theory}, Springer International Publishing AG, Cham, Switzerland, 2016. 

\bibitem{Bia98} P.\ Biane, Processes with free increments, Math.\ Z.\ \textbf{227} (1998), 143--174. 

\bibitem{Bia19} P.\ Biane, Nonlinear free L\'evy-Khintchine formula and conformal mapping, J.\ Operator Theory \textbf{85} (2021), 79--99.  

\bibitem{Bil68} P.\ Billingsley, \textit{Convergence of Probability Measures}, John Wiley \& Sons, Inc., New York, 1968.

\bibitem{Bil12} P.\ Billingsley, \textit{Probability and Measure}, Anniversary Edition, John Wiley \& Sons, Inc., New Jersey, 2012.

\bibitem{BCDM09} F. Bracci, M. D. Contreras, and S. D{\'{\i}}az-Madrigal, Evolution families and the Loewner equation II: complex hyperbolic manifolds, Math.\ Ann.\ \textbf{344} (2009), 947--962.

\bibitem{BCDM12} F.~Bracci, M.~D. Contreras, and S.~D{\'{\i}}az-Madrigal, Evolution families and the Loewner equation I: The unit disc, J. Reine Angew. Math. \textbf{672} (2012), 1--37.

\bibitem{BCDM20} F.\ Bracci, M.\ D.\ Contreras, and S.\ D{\'{\i}}az-Madrigal, \textit{Continuous Semigroups of Holomorphic Self-maps of the Unit Disc}, Springer Monographs in Mathematics, Springer, Cham, 2020.

\bibitem{BCDMG15} F.\ Bracci, M.\ D.\ Contreras, S.\ Diaz-Madrigal and P.\ Gumenyuk, Boundary regular fixed points in Loewner theory, Ann.\ Mat.\ Pura Appl.\ \textbf{194} (2015), 221--245.

\bibitem{BCDMV14} F.\ Bracci, M.\ D.\ Contreras, S.\ Diaz-Madrigal and A.\ Vasil'ev, Classical and stochastic L\"owner--Kufarev equations, in: \textit{Harmonic and complex analysis and its applications}, Trends Math., Birkh\"auser/Springer, Cham, 2014, pp.\ 39--134.

\bibitem{CL55} E.\ A.\ Coddington and N.\ Levinson, \textit{Theory of Ordinary Differential Equations}, McGraw-Hill, New York, 1955.    

\bibitem{CDM21} M.\ D.\ Contreras, S.\ D{\'{\i}}az-Madrigal, Topological Loewner theory on Riemann surfaces, J.\ Math.\ Anal.\ Appl.\ \textbf{493} (2021), 124525. 

\bibitem{CDMG10}
M.\ D.\ Contreras, S.\ D{\'{\i}}az-Madrigal and P.\ Gumenyuk, Loewner chains in the unit disk, Rev.\ Mat.\ Iberoam.\ \textbf{26} (2010), 975--1012.

\bibitem{CDMG13}
M.\ D.\ Contreras, S.\ D{\'{\i}}az-Madrigal and P.\ Gumenyuk, Loewner theory in annulus I: evolution families and differential equations, Trans.\ Amer.\ Math.\ Soc.\ \textbf{365} (2013), 2505--2543.

\bibitem{CDMG14}
M.~D. Contreras, S.~D{\'{\i}}az-Madrigal, and P.~Gumenyuk, 
Local duality in Loewner equations, 
J. Nonlinear Convex Anal. \textbf{15} (2014), no.~2, 269--297.

\bibitem{dMG16}
A.\ del Monaco and P.\ Gumenyuk, Chordal Loewner equation, in: \textit{Complex analysis and dynamical systems VI. Part 2}, 63--77, Comtemp.\ Math.\ \textbf{667}, Israel Math.\ Conf.\ Proc., American Mathematical Society, Providence, RI, 2016.

\bibitem{Dud02} R.\ M.\ Dudley, \textit{Real Analysis and Probability}, Cambridge University Press, Cambridge, 2002.

\bibitem{Dub10}
V.\ N.\ Dubinin, Lower bounds for the half-plane capacity of compact sets and symmetrization (Russian), Mat.\ Sb.\ \textbf{201} (2010), no.11, 77--88; Translation in Sb.\ Math.\ \textbf{201} (2010), 1635--1646.

\bibitem{DV14}
V.\ N.\ Dubinin and M.\ Vuorinen, Ahlfors-Beuring conformal invariant and relative capacity of compact sets, Proc.\ Amer.\ Math.\ Soc.\ \textbf{142} (2014), 3865--3879.

\bibitem{EK86}  S.\ N.\ Ethier and T.\ G.\ Kurtz, \emph{Markov Processes: Characterization and Convergence}. John Wiley \& Sons, Inc., New York, 1986. 

\bibitem{EG92}
L. C. Evans and R. F. Gariepy, \emph{Measure Theory and Fine Properties of Functions},  Revised edition, CRC Press, Boca Raton, FL, 2015.

\bibitem{Fe78}
P.\ Feinsilver, Processes with independent increments on a Lie group, Trans.\ Amer.\ Math.\ Soc.\ \textbf{242} (1978), 73--121.

\bibitem{Fra09b} U.\ Franz, Monotone and boolean convolutions for non-compactly supported probability measures, Indiana Univ.\ Math.\ J.\ \textbf{58} (2009), 1151--1186.

\bibitem{FHS20} U.\ Franz, T.\ Hasebe and S.\ Schlei{\ss}inger,
Monotone increment processes, classical Markov processes and Loewner chains, Diss.\ Math.\ \textbf{552}, 119 (2020).

\bibitem{GM05} J.\ B.\ Garnett and D.\ E.\ Marshall, \textit{Harmonic Measure}, New Mathematical Monographs \textbf{2}, Cambridge University Press, Cambridge, 2005.

\bibitem{GK54} B.\ V.\ Gnedenko and A.\ N.\ Kolmogorov, \textit{Limit Distributions for Sums of Independent Random Variables}, translated and annotated by K. L. Chung, with an appendix by J. L. Doob, Addison-Wesley Publishing Co., Inc., Cambridge, MA, 1954.

\bibitem{GL21} B.\ Gustafsson and Y.-L.\ Lin, \textit{Laplacian Growth on Branched Riemann Surfaces}, Lecture Notes in Mathematics \textbf{2287}, Springer, Cham, 2021.

\bibitem{Gol69} G.\ M.\ Goluzin, \textit{Geometric Theory of Functions of a Complex Variable}, Translations of Mathematical Monographs \textbf{26}, American Mathematical Society, Providence, RI, 1969.

\bibitem{GB92} V.\ V.\ Goryainov and I.\ Ba, Semigroups of conformal mappings of the upper half-plane into itself with hydrodynamic normalization at infinity, Ukrainian Math.\ J.\ \textbf{44} (1992), 1209--1217.
Translation from Ukra\"{\i}n.\ Mat.\ Zh.\ \textbf{44} (1992), 1320--1329.

\bibitem{GHP22+}  P.\ Gumenyuk, T.\ Hasebe and J.-L.\ P\'erez, Loewner Theory for Bernstein functions I: evolution families and differential equations, Constr.\ Approx.\ \textbf{61} (2025), 379--412. 

\bibitem{GHP23+}  P.\ Gumenyuk, T.\ Hasebe and J.-L.\ P\'erez, Loewner Theory for Bernstein functions II: applications to inhomogeneous continuous-state branching processes. Ann.\ Appl.\ Probab.\ \textbf{36}, No.\ 3 (2026), 2163--2198. 

\bibitem{HH22}
T.\ Hasebe and I.\ Hotta, Additive processes on the unit circle and Loewner chains, Int.\ Math.\ Res.\ Not.\ \textbf{22} (2022), 17797--17848.

\bibitem{Ha23}
K.\ Hata, Uniform weak convergence of random walks to additive processes, Master's thesis, Hokkaido University, 2023.

\bibitem{Hay51}
W.\ K.\ Hayman, Some applications of the transfinite diameter to the theory of functions, J.\ Analyse Math.\ \textbf{1} (1951), 155--179.

\bibitem{Hay94}
W.\ K.\ Hayman, \textit{Multivalent Functions}, 2nd ed., Cambridge Tracts in Mathematics \textbf{110}, Cambridge University Press, Cambridge, 1994.

\bibitem{HHY22+}
S.\ Hoshinaga, I.\ Hotta and H.\ Yanagihara, Continuous evolution families, Proc.\ Amer.\ Math.\ Soc.\ \textbf{151} (2023), 5251--5263.

\bibitem{JS10}
J. Jacod and A. N. Shiryaev, \textit{Limit Theorems for Stochastic Processes}, 2nd ed., Grundlehren der mathematischen Wissenshaften \textbf{288}, Springer, 2010.

\bibitem{Jek20} D.\ Jekel, Operator-valued chordal Loewner chains and non-commutative probability, J.\ Funct.\ Anal.\ \textbf{278} (2020), 108452.

\bibitem{Kal17} O.\ Kallenberg, \textit{Random Measures, Theory and Applications}, Probability Theory and Stochastic Modelling \textbf{77}, Springer, Cham, 2017.

\bibitem{Kel75} J.\ L.\ Kelley, \textit{General Topology}, Reprint of the 1955 ed.\ by Van Nostrand, Graduate Texts in Mathematics \textbf{27}, Springer-Verlag, New York--Berlin, 1975.

\bibitem{Kem17} A.\ Kemppainen, \textit{Schramm--Loewner Evolution}, Springer Briefs in Mathematical Physics \textbf{24}, Springer, Cham, 2017.

\bibitem{LSW03}
G.\ Lawler, O.\ Schramm and W.\ Werner, Conformal restriction: the chordal case, J.\ Amer.\ Math.\ Soc.\ \textbf{16} (2003), 917--955.

\bibitem{LLN09}
S.\ Lalley, G.\ Lawler and H.\ Narayanan, Geometric interpretation of half-plane capacity, Electron.\ Commun.\ Probab.\ \textbf{14} (2009), 566--571.

\bibitem{Loe23}
K.\ L\"owner, Untersuchungen \"uber schlichte konforme Abbildungen des Einheitskreises. I. (in German),  Math.\ Ann.\ \textbf{89} (1923), 103--121.

\bibitem{Maa92} H.\ Maassen, Addition of freely independent random variables, J.\ Funct.\ Anal.\ \textbf{106} (1992), 409--438.

\bibitem{MS17} J.\ A.\ Mingo and R.\ Speicher, \textit{Free Probability and Random Matrices}, Fields Institute Monographs \textbf{35}, Springer, New York; Fields Institute for Research in Mathematical Sciences, Toronto, ON, 2017.

\bibitem{Mur00} N.\ Muraki, Monotonic convolution and monotonic L\'{e}vy-Hin\v{c}in formula, preprint (2000).  

\bibitem{Mur22} T.\ Murayama, On the continuity of half-plane capacity with respect to Carath\'eodory convergence, in: \textit{Dirichlet Forms and Related Topics}, Springer Proceedings in Mathematics \& Statistics \textbf{394}, Springer Nature Singapore Pte Ltd., Singapore, 2022, pp.379--399.

\bibitem{Mur23} T.\ Murayama, Loewner chains and evolution families on parallel slit half-planes, J.\ Math.\ Anal.\ Appl.\ \textbf{523} (2023), Paper No.\ 127180, 51 pp.

\bibitem{Pol20}
G.\ P\'olya, \"Uber den zentralen Grenzwertsatz der Wahrscheinlichkeitsrechnung und das Momentenproblem (in German), Math.\ Z.\ \textbf{8} (1920), 171--181.

\bibitem{Pom65}
Ch.\ Pommerenke, \"Uber die Subordination analytischer Funktionen, J.\ Reine Angew.\ Math.\ \textbf{218} (1965), 159--173.

\bibitem{Pom75} Ch.\ Pommerenke, \textit{Univalent Functions}, Vandenhoeck \& Ruprecht, G\"ottingen, 1975.

\bibitem{Pom92} Ch.\ Pommerenke, \textit{Boundary Behaviour of Conformal Maps}, Grundlehren der mathematischen Wissenschaften \textbf{299}, Springer-Verlag, Berlin, 1992. 

\bibitem{RR89} B.\ S.\ Rajput and J.\ Rosinski, Spectral representations of infinitely divisible processes, Probab.\ Theory Related Fields \textbf{82} (1989), 451--487. 

\bibitem{RW14}
S.\ Rohde and C.\ Wong, Half-plane capacity and conformal radius, Proc.\ Amer.\ Math.\ Soc.\ \textbf{142} (2014), 931--938.

\bibitem{RR94} M.\ Rosenblum and J.\ Rovnyak, \textit{Topics in Hardy Classes and Univalent Functions}, Birkh\"auser Advanced Texts: Basler Lehrb\"ucher, Birkh\"auser Verlag, Basel, 1994.

\bibitem{Sat13} K.\ Sato, \textit{L\'evy Processes and Infinitely Divisible Distributions}, corrected paperback edition, Cambridge Studies in Advanced Math.\ \textbf{68}, Cambridge University Press, Cambridge, 2013.

\bibitem{Sch17} S.\ Schlei{\ss}inger, The chordal Loewner equation and monotone probability theory, Infin.\ Dimens.\ Anal.\ Quantum Probab.\ Relat.\ Top.\ \textbf{20} (2017), 1750016, 17 pp. 

\bibitem{SW97} R.\ Speicher and R.\ Woroudi, Boolean convolution, in: \textit{Free Probability Theory}, Ed.\ D.\ Voiculescu, Fields Inst.\ Commun.\ \textbf{12}, Amer.\ Math.\ Soc. (1997), pp.267--280. 

\bibitem{Str11} D.\ W.\ Stroock, \textit{Probability Theory: An Analytic View}, 2nd ed., Cambridge University Press, Cambridge, 2011.

\bibitem{Tsu75}
M.~Tsuji, \textit{Potential theory in modern function theory}, Chelsea Publishing Co., New York, 1975.

\bibitem{Vil09} C.\ Villani, \textit{Optimal Transport: Old and New}, Grundlehren der mathematischen Wissenschaften \textbf{338}, Springer-Verlag, Berlin, 2009. 

\bibitem{V93} D.\ V.\ Voiculescu, The analogues of entropy and of Fisher's information measure in free probability theory I,  Commun.\ Math.\ Phys.\ \textbf{155} (1993), 71--92. 

\bibitem{Yan19+} H.\ Yanagihara, Loewner theory on analytic universal covering maps, preprint (2025). arXiv:1907.11987v2
\end{thebibliography}
\end{document}